\documentclass[11pt]{article}
\usepackage[a4paper,margin=1in]{geometry}
\usepackage{amsmath,amssymb,amsthm,mathtools}
\usepackage{enumitem}
\usepackage{array}
\usepackage[T1]{fontenc}
\usepackage{lmodern}
\usepackage{mathrsfs}
\usepackage{microtype}
\usepackage{hyperref}
\usepackage{aliascnt}
\usepackage[nameinlink,capitalise,noabbrev]{cleveref}
\hypersetup{colorlinks=false,pdfborder={0 0 0}}
\usepackage{xcolor}
\usepackage{seqsplit}

\newtheorem{theorem}{Theorem}[section]
\newaliascnt{lemma}{theorem}
\newtheorem{lemma}[lemma]{Lemma}
\aliascntresetthe{lemma}
\newaliascnt{proposition}{theorem}
\newtheorem{proposition}[proposition]{Proposition}
\aliascntresetthe{proposition}

\newaliascnt{corollary}{theorem}
\newtheorem{corollary}[corollary]{Corollary}
\aliascntresetthe{corollary}
\newaliascnt{conjecture}{theorem}

\aliascntresetthe{conjecture}
\theoremstyle{definition}
\newaliascnt{definition}{theorem}
\newtheorem{definition}[definition]{Definition}
\aliascntresetthe{definition}
\newaliascnt{remark}{theorem}
\newtheorem{remark}[remark]{Remark}
\aliascntresetthe{remark}
\crefname{theorem}{Theorem}{Theorems}
\crefname{lemma}{Lemma}{Lemmas}
\crefname{proposition}{Proposition}{Propositions}
\crefname{corollary}{Corollary}{Corollaries}
\crefname{conjecture}{Conjecture}{Conjectures}
\crefname{definition}{Definition}{Definitions}
\crefname{remark}{Remark}{Remarks}

\newcommand{\F}{\mathbb F}
\newcommand{\Z}{\mathbb Z}
\newcommand{\Alt}{\operatorname{Alt}}
\newcommand{\rad}{\operatorname{rad}}
\newcommand{\Gr}{\operatorname{Gr}}
\newcommand{\im}{\operatorname{im}}
\newcommand{\rk}{\operatorname{rank}}
\newcommand{\codim}{\operatorname{codim}}
\newcommand{\linspan}{\mathrm{span}}
\newcommand{\Ann}{\operatorname{Ann}}
\newcommand{\supp}{\operatorname{supp}}
\newcommand{\tr}{\operatorname{tr}}
\newcommand{\GL}{\operatorname{GL}}

\newcommand{\ad}{\operatorname{ad}}
\newcommand{\Sp}{\operatorname{Sp}}
\newcommand{\res}{\operatorname{res}}
\newcommand{\row}{\operatorname{row}}
\newcommand{\gr}{\operatorname{gr}}
\newcommand{\core}{\mathrm{core}}

\newcommand{\rmt}{\mathrm{t}}
\newcommand{\trans}[1]{{#1}^{\rmt}}
\newcommand{\cE}{\mathcal E}
\newcommand{\cF}{\mathcal F}
\newcommand{\cB}{\mathcal B}
\newcommand{\cC}{\mathcal C}
\newcommand{\cA}{\mathcal A}
\newcommand{\cH}{\mathcal H}
\newcommand{\cL}{\mathcal L}
\newcommand{\cS}{\mathcal S}

\newcommand{\cP}{\mathcal P}
\newcommand{\cR}{\mathcal R}

\newcommand{\cX}{\mathcal X}
\newcommand{\cZ}{\mathcal Z}

\newcommand{\qbinom}[3]{\genfrac{[}{]}{0pt}{}{#1}{#2}_{#3}}

\newaliascnt{assumption}{theorem}

\aliascntresetthe{assumption}
\crefname{assumption}{Input}{Inputs}
\crefname{part}{Part}{Parts}
\DeclareMathOperator{\Aut}{Aut}
\DeclareMathOperator{\Out}{Out}
\DeclareMathOperator{\Inn}{Inn}

\newcommand{\cc}{\frac{2}{27}}

\newcommand{\cK}{\mathcal K}
\newcommand{\N}{\mathbb N}
\newcommand{\R}{\mathbb R}
\newcommand{\fH}{\mathfrak H}
\newcommand{\fA}{\mathfrak A}
\newcommand{\fb}{\mathfrak b}
\newcommand{\rmnil}{\mathrm{nil}}
\newcommand{\rmcom}{\mathrm{com}}
\newcommand{\rmbal}{\mathrm{bal}}
\newcommand{\rmnonmin}{\mathrm{nonmin}}
\newcommand{\rmmin}{\mathrm{min}}
\newcommand{\rmand}{\mathrm{and}}
\newcommand{\rmns}{\mathrm{ns}}
\newcommand{\rmsat}{\mathrm{sat}}

\title{
From almost tight $p$-group enumeration \\
to $2$-group domination among finite groups
}
\author{
Yinan Li\thanks{Wuhan University} 
\and 
Youming Qiao\thanks{University of Technology Sydney}
}
\date{\today}

\begin{document}
\maketitle

\begin{abstract}
Let $f(n)$ denote the number of isomorphism classes of groups of order $n$. The pioneering works of Higman (\emph{Proc. Lond. Math. Soc.}, 1960) and Sims (\emph{Proc. Lond. Math. Soc.}, 1965), with improvement by Newman and Seeley, establish that 
\[
\frac{2}{27}m^3-\frac{4}{9}m^2+\Omega(m)\leq \log_p f(p^m)\leq \frac{2}{27}m^3+O(m^{5/2}).
\]
In the same paper in 1965, Sims speculated on improving the error term ($O(m^{8/3})$ at that time) in the upper bound. To show that $\log_p f(p^m)\leq \frac{2}{27}m^3+O(m^2)$ is now known as Sims' conjecture.
Our first main result proves Sims' conjecture in a strong form as follows:
\[
\log_p f(p^m)\leq \frac{2}{27}m^3-\frac{4}{9}m^2+O(m).
\]

Techniques developed to prove the above lead to progress on some well-known open questions in group enumeration in affirmative, including
\begin{itemize}
	\item 2-group domination in finite group enumeration (Pyber, \emph{Eur. Congr. Math.}, 1996): Let $F(n)$ be the number of groups of order $\leq n$, and $F_2(n)$ be the number of $2$-groups of order $\leq n$. Then 
	$$
	\lim_{n\to\infty}\frac{F_2(n)}{F(n)} =1. 
	$$
	\item A question attributed to Higman and Erd\H{o}s (see Blackburn--Neumann--Venkataraman, \emph{Cambridge Tracts in Maths.}, 2007) when $m$ is large enough: 
	$$
	\text{for } 1<n<2^m, f(n)<f(2^m).
	$$
\end{itemize}

Our proof refines Sims' enumeration framework using graded Lie
brackets, constraints on deeper commutator coordinates, and incidence
and Fourier estimates for alternating forms. Normalisation by changes
of generating bases then determines the quadratic term in the upper bound.
We also prove that
all but a $p^{-\Omega(m^2)}$ fraction of groups of order $p^m$ have
a central elementary abelian Frattini subgroup. Combining this
concentration with the sharp enumeration bound of
Bl\"aser--Li--Qiao--Rogovskyy (\emph{arXiv}, 2025) for $p$-groups
of Frattini class $2$ yields the linear error term. The same presentation estimates give fixed-quotient
extension bounds which, together with Pyber's prescribed-Sylow
theorem and the concentration result, establish domination by $2$-groups of frattini class $2$. 

\end{abstract}

\newpage 

\begingroup
\small
\paragraph{AI and concurrent work disclosure.} 
\begin{description}
    \item[9 October, 2026.] We began this project in April 2026 as a natural continuation of
our previous joint work with M.~Bl\"aser and A.~Rogovskyy
\cite{BLQR}. In August 2026, one of us obtained, through
interactions with GPT-5.5 in ChatGPT, an argument establishing
\[
\log_p f(p^m)=\frac{2}{27}m^3+O(m^2).
\]
Further work with GPT-5.5 and subsequently GPT-6 led to the
results of the present paper in September 2026.

Since August 2026, we have been checking the proofs and revising
the exposition to reflect our understanding of the arguments.
At present, we are convinced of the correctness of the proofs.
In particular, we have reworked the proofs in Parts~I and~II
through further discussions with ChatGPT. 

On 9 October 2026, we became aware of the preprint by Su and Xia
\cite{SuXia}, which contains overlapping enumeration and
class-$2$ concentration results, together with finer asymptotic
enumeration formulas for each fixed prime.
In their AI declaration, the authors state
that their work ``was carried out in collaboration with the
AI system GPT-6.'' Following the appearance of their paper,
we decided to release our manuscript as it stood on
9 October 2026. 
We will continue to refine the exposition,
particularly in the later parts, to make the arguments
clearer and easier to follow. 
\item[17 October, 2026.] We have updated the title and abstract, rewritten the first two sections, and polished Part III. We also announce further results on enumerating finite rings and commutative algebras. The manuscript remains a work in progress, and we will continue to revise and improve it.

GPT-5.5 in the statement on 9 October should be GPT-5.6 in retrospect.
\item[12 November, 2026.] We have substantially revised Parts I–III by reorganizing and restating the results, adding explanatory context, and rewriting proofs to improve clarity. The manuscript remains a work in progress, and we will continue to revise and improve it.

\end{description}

\par\endgroup

\clearpage
\setcounter{tocdepth}{2}
\tableofcontents

\newpage

\section{Introduction}\label{intro:section}

\subsection{Enumerating \texorpdfstring{$p$}{p}-groups}\label{intro:main-results}

\paragraph{Enumerating finite groups.} This paper concerns the number of finite groups of a fixed order up to isomorphism. The monograph by Blackburn, Neumann,
and Venkataraman \cite{BNV} provides an excellent introduction
to this classical topic in finite group theory.

Write $[G]$ for the isomorphism class of a finite group $G$.
For a positive integer $n$, let $n=\prod_{p\mid n}p^{a_p}$ be its prime decomposition. Our counting function and prime-factor
parameter are
\begin{equation}\label{intro:eq:group-counts}
 f(n):=\#\{[G]:|G|=n\},\qquad
 \mu(n):=\max_{p\mid n}a_p\quad(n>1),
\end{equation}
and we set $\mu(1)=0$.
The celebrated work of Pyber \cite{Pyber} establishes that 
\begin{equation}\label{eq:pyber}
f(n)\leq n^{\frac{2}{27}\mu(n)^2+O(\mu(n)^{5/3})}.
\end{equation}

The coefficient of $\mu(n)^2$ on the exponent is $\frac{2}{27}$, which is tight by considering $p$-groups. Indeed, $p$-groups are long regarded as the major contributing factor to the many isomorphic types of groups, and precise questions or conjectures were formulated by Higman and Erd\H{o}s \cite[Question 22.16]{BNV} and Pyber \cite[Conjecture 2.4]{Pyber1998} to reflect the intuition that ``most groups are $p$-groups.'' So we turn to review works on enumerating $p$-groups.

\paragraph{Enumerating $p$-groups by Higman and Sims.} 
The classical works of Higman \cite{Higman} and Sims \cite{Sims},
both with the main title \emph{Enumerating $p$-groups}, together with an improvement by Newman and Seeley,
established that
\begin{equation}\label{eq: old estimate}
	p^{\frac{2}{27}m^3-\frac49m^2+\Omega(m)}
	\le f(p^m)
	\le p^{\frac{2}{27}m^3+O(m^{5/2})},
\end{equation}
where the asymptotics are taken as $m\to\infty$ and the implied constants in $O(\cdot)$ and $\Omega(\cdot)$ are absolute. More specifically, the lower bound in \cref{eq: old estimate} due to Higman \cite{Higman} is achieved at $p$-groups of Frattini class $2$. 
For a finite $p$-group $G$, its Frattini subgroup is
\begin{equation}\label{intro:eq:frattini-subgroup}
 \Phi(G)=G^p[G,G].
\end{equation}
The group has \emph{Frattini class at most $2$} if
\begin{equation}\label{intro:eq:frattini-class}
 \Phi(G)\le Z(G),\qquad \Phi(G)^p=1.
\end{equation}
The upper bound in \cref{eq: old estimate} with a weaker error bound, that is $O(m^{8/3})$ instead of $O(m^{5/2})$, was first shown by Sims \cite{Sims}. The improvement to $O(m^{5/2})$ was by Newman and Seeley, whose proof was presented in \cite[Chapter~5]{BNV}.

For $n=p^m$, we have $\mu(n)=m$, and the lower bound in
\cref{eq: old estimate} gives
\[
n^{\frac{2}{27}\mu(n)^2-\frac49\mu(n)}
\le f(n)
\]
for sufficiently large $m$. Consequently, the coefficient
$\frac{2}{27}$ of $\mu(n)^2$ in the exponent of Pyber's upper
bound (\cref{eq:pyber}) is best possible.

\paragraph{Sims' conjecture.} In \cite[p.~153]{Sims}, Sims discussed the possibility of improving the error term in the upper bound, which was $O(m^{8/3})$ in his paper. He remarked that an error term of $O(m^2)$ ``should be possible'', and speculated that obtaining a smaller error, such as $o(m^2)$, would require estimating the average number of automorphisms of a group of order $p^m$, a task that ``appears to be a very difficult problem''.

Sims' paper \cite{Sims} appeared in 1965. The problem he raised---whether
the upper bound for $f(p^m)$ could be improved to
$p^{\frac{2}{27}m^3+O(m^2)}$---has long been regarded as a central
problem in group enumeration. It is known as Sims' conjecture \cite{Blackburn} in group enumeration
and is recorded as Question~22.3 in the monograph of Blackburn,
Neumann, and Venkataraman \cite{BNV}.

Blackburn's thesis \cite{Blackburn} provides important evidence
supporting Sims' conjecture. He proved that, for odd primes $p$,
the number of isomorphism classes of groups of order $p^m$
and nilpotency class at most $3$ is at most
$p^{\frac{2}{27}m^3+O(m^2)}$.

Echoing Sims' speculation that average number of automorphisms could be important in $p$-group enumeration, Bl\"aser, Li, Qiao, and Rogovskyy recently showed in \cite{BLQR} that the number of $p$-groups of Frattini class $2$ is 
\begin{equation}\label{intro:eq:frattini-enumeration}
 p^{\frac{2}{27}m^3-\frac49m^2+\Theta(m)},
\end{equation}
by improving the upper bound on the number of $p$-groups of Frattini class $2$ from $p^{\frac{2}{27}m^3+O(m)}$ to $p^{\frac{2}{27}m^3-\frac49m^2+O(m)}$. The main technical part of \cite{BLQR} tackles the average automorphism group orders of such groups.

\paragraph{Almost tight enumeration of $p$-groups.}
The main result of this paper is an almost tight upper bound for $f(p^m)$ as follows. 
\begin{theorem}\label{intro:sharp}
	For the counting function in~\eqref{intro:eq:group-counts}, we have 
	\begin{equation}\label{intro:eq:sharp}
	 f(p^m)
	 \le p^{\frac{2}{27}m^3-\frac49m^2+O(m)}.
	\end{equation}
\end{theorem}
Combining with Higman's lower bound of $f(p^m)$ \cite{Higman}, we obtain
$$
f(p^m)= p^{\frac{2}{27}m^3-\frac49m^2+\Theta(m)}.
$$
Thus \cref{intro:sharp} settles Sims' conjecture in a stronger
form: it determines the quadratic coefficient in the exponent
and shows that the remaining error has linear order.

\paragraph{Implications to general group enumeration.} In his thesis \cite{Blackburn}, Blackburn indicated that an affirmative answer to Sims' conjecture implies a sharpening of Pyber's general group enumeration (\cref{eq:pyber}). We record this consequence of \cref{intro:sharp} here. 
\begin{corollary}\label{cor:pyber-sharpening}
	With the notation in~\eqref{intro:eq:group-counts}, we have 
	$$
	f(n)\leq n^{\frac{2}{27}\mu(n)^2+O(\mu(n))}.
	$$
\end{corollary}
That is, \cref{cor:pyber-sharpening} improves the error term in \cref{eq:pyber} from $O(\mu(n)^{5/3})$ to $O(\mu(n))$.
The proof is given at the end of Section~\ref{p3:sec:linear}.

\subsection{Most groups are \texorpdfstring{$2$}{2}-groups of Frattini class \texorpdfstring{$2$}{2}}

\paragraph{Most groups are $2$-groups.} A familiar heuristic in group enumeration is that ``most groups
are $2$-groups''. This is illustrated by the fact that, up to
isomorphism, there are 11,759,892 groups of order at most $1023$
\cite{besche2002millennium}, compared with 49,487,367,289 groups
of order $1024$ \cite{burrell2022number}. Among the 11,759,892 many groups of order $\leq$ 1023, groups of order 512 take up the majority: there are 10,494,213 many of them.

Some conjectures or questions have been formulated to formalise this heuristic. 

For $x\ge1$, define the cumulative counts
\begin{equation}\label{intro:eq:cumulative-counts}
 \begin{aligned}
 F(x)&:=\sum_{1\le n\le x}f(n),\\
 F_{\rmnil}(x)&:=\#\{[G]:|G|\le x,\ G\text{ is nilpotent}\},\\
 F_2(x)&:=\#\{[G]:|G|\le x,\ G\text{ is a $2$-group}\}.
 \end{aligned}
\end{equation} 

Pyber \cite[Conjecture~2.4]{Pyber1998} conjectured that
\begin{equation}\label{eq:pyber-conj}
	\lim_{x\to\infty}\frac{F_{\rmnil}(x)}{F(x)}=1;
\end{equation}
see also \cite[Question~1]{Mann1999} and
\cite[Question~22.19]{BNV}. His enumeration theorem \cite{Pyber}
implies the weaker logarithmic statement
\[
\lim_{x\to\infty}
\frac{\log F_{\rmnil}(x)}{\log F(x)}=1.
\]
As observed by Mann \cite{Mann1999}, \cref{eq:pyber-conj} implies
that
\[
\lim_{x\to\infty}\frac{F_2(x)}{F(x)}=1,
\]
which gives a precise formulation of the heuristic that
``most groups are $2$-groups''.

In \cite{Mann1999}, Mann discussed a connection between Pyber's conjecture and $p$-group enumeration, which was attributed to Shalev \cite{Shalev1999}. While we were not able to reproduce this connection in the form described by Mann following \cite[Proposition 2.1]{Shalev1999}, this motivated us to explore the possibility of proving Pyber's conjecture after we obtained a preliminary version of \cref{intro:sharp}.

\paragraph{Most groups are $2$-groups of Frattini class $2$.}
Denote the family defined by~\eqref{intro:eq:frattini-class} by
\begin{equation}\label{intro:eq:frattini-family}
 \fH_p:=\{G:\ |G|\text{ is a power of }p,\
                  \Phi(G)\le Z(G),\ \Phi(G)^p=1\}.
\end{equation}
Every member has nilpotency class at most $2$, since
$[G,G]\le\Phi(G)\le Z(G)$. Its nonabelian members have both
Frattini class and nilpotency class exactly $2$.

Higman's lower bound in~\eqref{eq: old estimate} is obtained
by counting $p$-groups of Frattini class $2$ \cite{Higman}.
Taking the heuristic that ``most groups are $2$-groups'' one step
further, one may conjecture that ``most groups are $2$-groups of
Frattini class $2$''. A related conjecture was suggested by Holt
\cite{Holt2014MO} on MathOverflow. The strongest entry in his list
of ``reasonable'' conjectures was that ``Almost all finite groups
are $2$-groups of nilpotency class $2$ --- you can add extra
conditions to this''.

We prove that this is indeed the case, thereby also resolving
Pyber's conjecture~\eqref{eq:pyber-conj}. For the family in
\eqref{intro:eq:frattini-family}, define
\begin{equation}\label{intro:eq:frattini-counts}
 \begin{aligned}
 f_{\fH_p}(p^m)&:=\#\{[G]:|G|=p^m,\ G\in\fH_p\},\\
 F_{\fH_p}(x)&:=\sum_{k=0}^{\lfloor\log_p x\rfloor}
                        f_{\fH_p}(p^k)\qquad(x\ge1).
 \end{aligned}
\end{equation}
The terms with $k=0$ count the trivial group, with
$f(1)=f_{\fH_p}(1)=1$. 
\begin{theorem}\label{intro:domination}
	For the cumulative counts in
\cref{intro:eq:cumulative-counts,intro:eq:frattini-counts}, we have 
	\begin{equation}
		\lim_{x\to\infty}\frac{ F_{\fH_2}(x)}{F(x)}=1.
	\end{equation}
\end{theorem}

\paragraph{Concentration of $p$-groups of Frattini class $2$ within $p$-groups.}

The strengthening from ``most groups are $2$-groups'' to ``most groups are $2$-groups of Frattini class $2$'' suggests that within $p$-groups, $p$-groups of Frattini class $2$ should take up the majority. This is indeed the case, as reflected in the following result.
\begin{theorem}[Concentration in Frattini class $2$]
\label{intro:concentration}
There are absolute constants $\varepsilon>0$ and $m_0$ such
that, for every prime $p$ and every $m\ge m_0$,
\[
 f(p^m)-f_{\fH_p}(p^m)
 \le p^{\frac{2}{27}m^3-\frac49m^2-\varepsilon m^2}.
\]
Consequently,
\[
 1\le\frac{f(p^m)}{f_{\fH_p}(p^m)}
 \le1+p^{-\varepsilon m^2}.
\]
\end{theorem}

See \cref{p3:sec:linear} for the proof of \cref{intro:concentration}.

\cref{intro:concentration} is the technical result that leads to \cref{intro:sharp} and \cref{intro:domination}. 

To deduce~\cref{intro:sharp}, use the estimate announced in
\eqref{intro:eq:frattini-enumeration}, in the following form
\cite[Theorem~1.3]{BLQR}:
\begin{equation}\label{intro:eq:BLQR}
 p^{\frac{2}{27}m^3-\frac49m^2}\le f_{\fH_p}(p^m)
 \le p^{\frac{2}{27}m^3-\frac49m^2+O(m)}.
\end{equation}
Adding the exceptional count in Theorem~\ref{intro:concentration}
to this estimate proves Theorem~\ref{intro:sharp}.

For Theorem~\ref{intro:domination}, we first use fixed-quotient
extension estimates and Pyber's prescribed-Sylow theorem
\cite[p.~204]{Pyber} to show that groups whose orders are not
powers of $2$ make a negligible contribution to $F(x)$.
It therefore remains to consider $2$-groups. The concentration
estimate then shows that those outside $\fH_2$ also
make a negligible contribution.

\paragraph{More results on group enumeration.} Blackburn, Neumann, and Venkataraman made a nice list of open questions about group enumeration in \cite{BNV}. We summarise the impact of \cref{intro:sharp,intro:domination,intro:concentration} as follows.
\begin{itemize}
	\item Theorem~\ref{intro:sharp} resolves Sims' conjecture
\cite[Question~22.3]{BNV}.
	\item Theorem~\ref{intro:domination} implies that almost all finite groups
are nilpotent, and hence soluble. It therefore answers~\cite[Questions~22.14, 22.15, and~22.19]{BNV} in affirmative.
	\item Theorem~\ref{intro:concentration} gives a ratio tending to $1$,
strengthening the comparison requested in Question~22.9.
Since the abelian groups of order $p^m$ number at most $p^m$,
it also shows that almost all groups of order $p^m$ are nonabelian
and have nilpotency class exactly $2$, answering Question~22.10.
\end{itemize}

We can say even more. 
The two main theorems 
determine the cumulative count up to
asymptotic equivalence and answer both assertions of
\cite[Question~22.18]{BNV}. Together with the bound of Pantelidakis \cite{Pantelidakis2003} as
recorded in \cite[pp.~265]{BNV}, they answer the asymptotic versions of \cite[Questions~22.16 and~22.17]{BNV}. As mentioned in \cite[pp.~264]{BNV}, Question~22.16 there was attributed to people including Higman and Erd\H{o}s. We add that \cite[Questions~22.16 and~22.17]{BNV} were formulated for any $m$, while our results apply to only large enough $m$.

\begin{corollary}[Consequences for BNV Questions 22.16--22.18]
\label{intro:bnv-2218}
As $x\to\infty$ through real values,
\[
 F(x)\sim f\bigl(2^{\lfloor\log_2x\rfloor}\bigr).
\]
Consequently, as $m\to\infty$ through integers,
\[
 F(2^m-1)\sim f(2^{m-1}),\qquad
 \frac{F(2^m-1)}{f(2^m)}
 =2^{-2m^2/9+O(m)}\longrightarrow0.
\]
Moreover, there is an absolute integer $m_0$ such that, for every
integer $m\ge m_0$,
\[
 \begin{aligned}
 f(n)&<f(2^m)&&\text{for every integer }1\le n<2^m,\\
 f(2^m p)&<f(2^{m+2})&&\text{for every prime }p.
 \end{aligned}
\]
In particular, $F(2^m-1)<f(2^m)$ for all sufficiently large $m$.
\end{corollary}

The proof of \cref{intro:bnv-2218} is given in Subsection~\ref{p4:sec:bnv-2218}.

\subsection{Further results on finite rings and Poonen's conjecture}
\label{intro:ring-results}

Finite group enumeration belongs to the family of enumerating finite algebraic structures. After Higman and Sims' work \cite{Higman,Sims}, Kruse and Price studied enumerating finite rings \cite{KP70}. Recently, Poonen studied the moduli spaces of commutative algebras of finite rank in relation to Hilbert schemes, and enumerating commutative algebras appears there naturally \cite{Poonen} (see \cite{BM22} for a correction and improvement to the corresponding result). 

We announce the following results
from our work on finite-ring enumeration; their proofs will be presented separately in a work under preparation \cite{LiQiao27}.

All rings below are finite and associative, and all counts give
equal weight to isomorphism classes. For a prime $p$, let $A_p(n)$
and $C_p(n)$ count rings and commutative rings of order $p^n$,
respectively, with an identity optional in both counts. Let $U_p(n)$
count commutative rings with identity of order $p^n$.
The estimates with a subscript $p$ keep the prime fixed; their
constants may depend on $p$.

\paragraph{Rings with an optional identity.}
For $i\ge1$, write $R^i$ for the additive subgroup generated by
products of $i$ elements of $R$. A ring is nilpotent if $R^i=0$
for some $i$, and is cube-zero if $R^3=0$. Among rings of $p$-power
order, put
\[
 \cK_p=
 \{R:R^3=0,\quad pR\subseteq R^2,\quad pR^2=0\}.
\]
Thus the additive groups of $R/R^2$ and $R^2$ are elementary
abelian for $R\in\cK_p$. Let $K_p(n)$ and
$K_p^{\rmcom}(n)$ count the rings and commutative rings in
$\cK_p$ of order $p^n$, respectively.

\begin{theorem}[Finite-ring enumeration]\label{intro:finite-rings}
For each fixed prime $p$, as $n\to\infty$,
\[
 \begin{aligned}
 \log_p A_p(n)&=\frac4{27}n^3-\frac13n^2+O_p(n),\\
 \log_p C_p(n)&=\frac2{27}n^3-\frac29n^2+O_p(n).
 \end{aligned}
\]
Moreover,
\[
 \begin{aligned}
 \frac{A_p(n)-K_p(n)}{K_p(n)}
 &\le p^{-n^2/3+O_p(n^{5/3})},\\
 \frac{C_p(n)-K_p^{\rmcom}(n)}{K_p^{\rmcom}(n)}
 &\le p^{-n^2/9+O_p(n^{5/3})}.
 \end{aligned}
\]
In particular, the same enumeration estimates hold after restricting
to nilpotent rings or to cube-zero rings, in both the associative
and commutative settings.
\end{theorem}

\paragraph{Commutative rings with identity.}
For a commutative ring with identity, let $J(R)$ be its Jacobson
radical, the intersection of its maximal ideals. The ring is local
if it has a unique maximal ideal, which is then $J(R)$; its residue
field is $R/J(R)$. When this field is $\F_p$, the embedding
dimension is $\dim_{\F_p}J(R)/J(R)^2$.

Put
\[
 \kappa_p=\prod_{j=1}^{\infty}(1-p^{-j}),\qquad
 \eta_p(n)=
 \begin{cases}
 1+p^{-4},&3\mid n,\\
 1,&3\nmid n,
 \end{cases}
\]
and define $E_n$ by
\[
 E_n=
 \begin{cases}
 2k^3-4k^2+3k-3,&n=3k+1,\\
 2k^3-2k^2+k-1,&n=3k+2,\\
 2k^3,&n=3k+3.
 \end{cases}
\]

\begin{theorem}[Commutative rings with identity]
\label{intro:unital-rings}
For each fixed prime $p$, there is $c_p>0$ such that
\[
 U_p(n)=\kappa_p^{-2}p^{E_n}
       \bigl(\eta_p(n)+O_p(p^{-c_p n})\bigr)
 \qquad(n\to\infty).
\]
In particular,
\[
 \log_p U_p(n)=\frac2{27}n^3-\frac23n^2+
 \begin{cases}
 2n,&3\mid n,\\
 19n/9,&3\nmid n
 \end{cases}
 +O_p(1).
\]
All but a proportion $O_p(p^{-c_p n})$ of these rings are local,
have residue field $\F_p$, and satisfy
$J(R)^3=0\ne J(R)^2$ and $\operatorname{char}(R)=p^3$.
Outside the same exceptional proportion, their embedding dimension
$r$ belongs to the following list:
\begin{equation}\label{intro:eq:ring-embedding-dimensions}
 \begin{array}{c|c}
 n\bmod3&r\\ \hline
 0&2n/3-1\ \text{or}\ 2n/3-2\\
 1&(2n-5)/3\\
 2&(2n-4)/3.
 \end{array}
\end{equation}
When $3\mid n$, the limiting proportions of the two listed
dimensions are $1/(1+p^{-4})$ and $p^{-4}/(1+p^{-4})$, respectively.
\end{theorem}

\paragraph{Poonen's conjecture.}
Let
\[
 U(X)=\#\{[R]:|R|\le X,\quad
              R\text{ is commutative with identity}\}.
\]
Poonen~\cite[Conjecture~11.4]{Poonen} predicted six simultaneous
properties of almost all rings in this cumulative count.
The following result proves that conjecture and gives an asymptotic
formula for $U(X)$.

\begin{theorem}[Poonen's conjecture]\label{intro:poonen}
Put $n=\lfloor\log_2X\rfloor$. There is an absolute constant $c>0$
such that, among the isomorphism classes counted by $U(X)$, all but
a proportion $O(2^{-cn})$ satisfy the following properties:
\begin{enumerate}[label=\textup{(\roman*)}]
 \item $|R|=2^n$;
 \item $R$ is local;
 \item its residue field is $\F_2$;
 \item its maximal ideal $J$ satisfies $J^3=0\ne J^2$;
 \item $\dim_{\F_2}J/J^2=2n/3+O(1)$;
 \item $\operatorname{char}(R)=8$.
\end{enumerate}
More precisely, the embedding dimensions are those in
\eqref{intro:eq:ring-embedding-dimensions}; when $3\mid n$, the two
limiting proportions are $16/17$ and $1/17$.
Uniformly for $2^n\le X<2^{n+1}$,
\[
 U(X)=\kappa_2^{-2}2^{E_n}
       \bigl(\eta_2(n)+O(2^{-cn})\bigr).
\]
\end{theorem}

\subsection{Organization of the paper}
\label{intro:organization}

Section~\ref{p12:sec:introduction} introduces the recorded Sims
presentation data and explains how these data, together with the
power word and commutator coefficient arrays, determine a group
up to isomorphism. It also introduces the associated graded Lie
ring and uniform counting tools.
\Cref{p12:prop:sims-generating-data} bounds the number of choices
of recorded Sims presentation data.

Part~\ref{p12:part:valid-completions} counts pairs of power word
arrays and commutator coefficient arrays conditioned on fixed
recorded Sims presentation data.
Section~\ref{sec:bijuv} counts power word arrays.
For each fixed power word array,
Section~\ref{p12:subsec:lie-bracket-count} counts graded Lie
brackets, which determine the leading commutator coefficients,
and Section~\ref{p12:sec:filtered-HW} counts the deeper
commutator coefficients for each fixed bracket.
Section~\ref{p12:sec:saturated-optimisation} combines these estimates
into the joint array bound in~\cref{p12:thm:joint-enumeration}
for a fixed core breadth sequence, and distinguishes the saturated
and nonsaturated ranges.

Part~\ref{p12:part:leading-maps} bounds the number of alternating
bilinear maps with fixed Sims rank, including the binary endpoint.
Section~\ref{p12:sec:endpoint-absorption} combines these estimates
with the presentation bounds to show that groups in the
nonsaturated range make a negligible contribution.

Part~\ref{part:sharp-enumeration} treats the remaining saturated
range. Section~\ref{p3:sec:statement} optimizes the joint array bound.
Section~\ref{p3:sec:normalization} constructs families of ordered
Sims generating sets with common recorded Sims presentation data, bounds from below
the number of distinct array pairs produced by each group, and
refines the joint array count using the core centralizer.
Section~\ref{p3:sec:linear} uses these multiplicities to convert
array bounds into bounds on the number of group isomorphism
classes and prove \cref{intro:concentration}.
It then combines this concentration
with the enumeration theorem of~\cite{BLQR} to prove
\cref{intro:sharp}, and shows that the linear order of the error
term is necessary.

Section~\ref{p4:sec:extensions} combines extension estimates,
the cumulative cutoff, and Pyber's prescribed-Sylow theorem to prove
Theorem~\ref{intro:domination}: almost all finite groups are
nonabelian $2$-groups of Frattini class exactly two. Its final
subsection proves Corollary~\ref{intro:bnv-2218}.
The finite-ring results and the resolution of Poonen's conjecture
in Subsection~\ref{intro:ring-results} are announced here; their
proofs will be presented separately.

\section{Framework and tools for \texorpdfstring{$p$}{p}-group enumeration}\label{p12:sec:introduction}
We use the convention $[n]:=\{1,\dots,n\}$ for any nonnegative integer $n$, with $[0]=\varnothing$. Let $G$ be a finite group. For a subset $S\subseteq G$ of group elements, we use $\langle S\rangle$ to denote the subgroup generated by elements of $S$. We use the group commutator convention $[x,y]=x^{-1}y^{-1}xy$. Let $d(G)$ denote the minimum size of a generating set of $G$.

For a finite group $G$ of order
$p^m$ and of nilpotency class $c$, let $G_1=G$, and $G_{i+1}=[G_i,G]$ for $i\in[c]$. The subgroup chain $(G_i)$ satisfying
\begin{equation}\label{p12:eq:lower-central-series}
  G=G_1>G_2>\cdots>G_{c+1}=\{1\}
\end{equation}
is called the lower central series of $G$. We set $G_t=1$ for all $t\ge c+1$. When $c=1$, the group $G$ is abelian. 

For $i\in[c]$, define the abelian layers and their parameters by
\begin{equation}\label{p12:eq:layer-parameters}
 \begin{aligned}
 L_i&:=G_i/G_{i+1}, & r_i&:=d(L_i),\\
 s_i&:=\log_p|L_i|, & m_i&:=\log_p|G_i|=\sum_{j=i}^c s_j.
 \end{aligned}
\end{equation}
We set $L_t=0$ for $t>c$.
The \emph{rank} $r_i$ is the minimum number of generators of $L_i$.
We set $m_t=0$ for $t>c$. In particular, $m_1=m$, $\sum_{i=1}^c s_i=m$, and $s_i=m_i-m_{i+1}$ for $i\in[c]$.

The following notion in $p$-group enumeration was introduced by Sims~\cite{Sims}.
\begin{definition}\label{def:sims-rank}
	Let $G$ be a finite $p$-group. The \emph{Sims rank} $s(G)$ of $G$ is 
\begin{equation}\label{eq: sims rank}
  s(G):=\min\{d(H):~H\leq G,~G_2=[H,H]G_3\}.  
\end{equation}
\end{definition}
In other words, $s(G)$ is the least number of generators of a subgroup $H\le G$ satisfying $G_2=[H,H]G_3$. We use $s=s(G)$ when there is no confusion. 
If $G$ is nonabelian, then $2\le s\le r_1$ and
$1\le r_2\le\binom{s}{2}$. If $G$ is abelian, then $s=0$ and $r_2=0$.

A minimizing subgroup $H$ in~\eqref{eq: sims rank} is called a \emph{Sims core} of $G$, whose lower central series is closely related to that of $G$.
\begin{lemma}[Proposition 3.8 in~\cite{BNV}]\label{lem:Sims core lower central}
  Let $G$ be a nilpotent group and $H$ be a subgroup of $G$ satisfying $G_2=[H,H]G_3$. Let $(H_i)$ and $(G_i)$ be the lower central series of $H$ and $G$, respectively. Then we have $H_i=G_i$ for all $2\leq i\leq c$.
\end{lemma}

For the Frattini subgroup $\Phi(G)$ defined in~\eqref{intro:eq:frattini-subgroup}, we have 
\[
\begin{split}
  \Phi(G_1/G_2)&=(G_1/G_2)^p=G_1^pG_2/G_2=\Phi(G_1)/G_2\\
  \Phi(G_2/G_3)&=(G_2/G_3)^p=G_2^pG_3/G_3=\Phi(G_2)G_3/G_3
\end{split}
\] 
Define the Frattini quotients
\begin{equation}\label{p12:eq:frattini-quotients}
 \begin{split}
 V&:=G_1/\Phi(G_1)\cong (G_1/G_2)/\Phi(G_1/G_2)=L_1/pL_1,\\
 W&:=G_2/(G_3\Phi(G_2))\cong (G_2/G_3)/\Phi(G_2/G_3)=L_2/pL_2.
 \end{split}
\end{equation}
These isomorphisms identify $V$ and $W$ as vector spaces over $\F_p$, with dimensions $r_1$ and $r_2$, respectively.

Group commutation induces an
alternating bilinear map $\beta_G:V\times V\rightarrow W$ acting on $x\Phi(G),y\Phi(G)\in V$ as
\begin{equation}\label{p12:eq:group-alternating-map}
 \beta_G(x\Phi(G),y\Phi(G))=[x,y]G_3\Phi(G_2).
\end{equation}
Its values span $W$, so the induced linear map $\bigwedge^2 V\to W$ is surjective. For a subspace $U\le V$, we write $\beta_G(U,U)$ for the $\F_p$-linear span of the values $\beta_G(u,u')$ with $u,u'\in U$.
The Sims rank $s(G)$ of $G$ can be translated as the minimum generating dimension of $\beta_G$:
\begin{equation}\label{eq: linearised sims rank}
  s(G)=\min\{\dim(U):~U\leq V,~\beta_G(U,U)=W\}.
\end{equation}
A Sims core $H$ of $G$ determines a minimizing subspace $U=H\Phi(G)/\Phi(G)$ of~\eqref{eq: linearised sims rank} with $\dim(U)=s(G)$. The alternating bilinear map structure yields the following relation between $s(G)$ and $r_2$:
\begin{lemma}[Corollary~5.3 in~\cite{BNV}]\label{lem:relation between s(G) and r_2}
  Let $G$ be a finite $p$-group with Sims rank $s(G)$ and let $r_2$ be the rank of $G_2/G_3$. Then $s-1\leq r_2\leq \binom{s}{2}$. 
\end{lemma}

For fixed $r_1,r_2,s$, define the logarithmic count of surjective linear maps
\begin{equation}\label{p12:eq:leading-map-count-definition}
 \begin{split}
 N_p(r_1,r_2,s):=\log_p\#\bigl\{&\beta:\bigwedge^2\F_p^{r_1}\twoheadrightarrow\F_p^{r_2}:\\
 &\min\{\dim U:U\le\F_p^{r_1},\ \beta(\bigwedge^2U)=\F_p^{r_2}\}=s\bigr\}.
 \end{split}
\end{equation}
Thus $N_p(r_1,r_2,s)$ is the base-$p$ logarithm of the number of alternating bilinear maps whose values span $\F_p^{r_2}$ and whose Sims rank is $s$; if there are no such maps, set $N_p(r_1,r_2,s)=-\infty$.
The bounds in~\cite[Proposition~5.5]{BNV} give
\begin{equation}\label{eq:BNV's Nprs}
\begin{split}
  N_p(r_1,r_2,s)&\leq \frac{1}{2}r_1^2(r_2-(s-1))+O((r_1+r_2)^{8/3}),~~\text{and}\\
  N_p(r_1,r_2,s)&\leq \frac{1}{2}r_1^2(r_2-(s-1))+\frac{1}{2}r_1r_2(s-1)+O((r_1+r_2)^{5/2}).
\end{split}
\end{equation}

\subsection{Sims' generating set and relation arrays}\label{p12:par:fixed-sims-data}
The idea of Sims' enumeration framework for $p$-groups~\cite{Sims} is to show that any $p$-group has a presentation of a restricted type, and then give an upper bound for the number of such presentations. We recall this framework following
\cite[Chapter~5]{BNV}, which will be largely utilized throughout the proofs of our results.

Let $G$ be a nonabelian group of order $p^{m}$ with nilpotency class $c$ and Sims rank $s$. Use its lower central series from~\eqref{p12:eq:lower-central-series} and the layer parameters from~\eqref{p12:eq:layer-parameters}. We first specify a generating set of $G$ and then describe the relations among the chosen generating set.

\paragraph{The Sims generating set.}
The generating set of $G$ is described as an ordered set 
\[
(g_{i,j})_{i\in[c],j\in[r_i]}=(g_{1,1},\dots,g_{1,r_1},g_{2,1},\dots,g_{2,r_2},\dots,g_{c,1},\dots,g_{c,r_c}), 
\]
where $g_{i,j}\in G_i$ is called the \emph{$j$th weight-$i$ generator}. The generating set will satisfy the following: 
\[
  \langle g_{i,1}G_{i+1},\dots,g_{i,r_i}G_{i+1}\rangle=G_i/G_{i+1}~~\forall\ i\in[c].
\]
We now specify how these $g_{i,j}$ are chosen
layer by layer, using a fixed Sims core $H$.

  For $i=1$, let the first $s$ elements $g_{1,1},\ldots,g_{1,s}$ generate $H$; i.e.~$H=\langle g_{1,1},\ldots,g_{1,s}\rangle$.  In the Frattini quotient
  $V\cong L_1/pL_1$ from~\eqref{p12:eq:frattini-quotients}, the minimality characterization~\eqref{eq: linearised sims rank} implies that the images $g_{1,1}\Phi(G),\dots,g_{1,s}\Phi(G)$ are linearly independent in
  $V$.  Extend them to an ordered basis
  $(g_{1,j}\Phi(G))_{j\in[r_1]}$ of $V$ by choosing
  $g_{1,s+1},\ldots,g_{1,r_1}\in G$.  The Burnside basis
  theorem~\cite[Theorem~5.3.2]{Robinson} implies that
  $(g_{1,j})_{j\in[r_1]}$ is a minimum generating set of $G=G_1$.
  Thus, $\langle g_{1,1}G_{2},\dots,g_{1,r_1}G_{2}\rangle=G_1/G_{2}$.

  For $i=2$, we choose $g_{2,1},\dots,g_{2,r_2}$ using the structure of $\beta_G:V\times V\to W$ and the Sims core $H$. 
  Let $U=\linspan_{\F_p}\{g_{1,1}\Phi(G),\cdots,g_{1,s}\Phi(G)\}$, where these cosets are viewed as elements in $V\cong\F_p^{r_1}$. The Sims rank definition~\eqref{eq: linearised sims rank} gives 
  \[
    W=\beta_G(U,U)
  =\linspan_{\F_p}\{[g_{1,a},g_{1,b}]G_3\Phi(G_2):1\le a<b\le s\}
  \]
  For each $k\in[s]$, put
  \[
  W_k=\linspan_{\F_p}\{
    [g_{1,a},g_{1,b}]G_3\Phi(G_2):1\le a<b\le k\}.
  \]
  By \cite[Proposition~5.2]{BNV}, we can choose the Sims-core generators $g_{1,1},\ldots,g_{1,s}$ 
  such that 
  \[
   W_{k+1}=W_{k}+\linspan_{\F_p}\{
   [g_{1,a},g_{1,k+1}]G_3\Phi(G_2):1\le a <k+1\}
  \]
  for $k\in[s-1]$, and the resulting subspaces satisfy $0=W_1<W_2<\cdots<W_s=W$.
  Let $d_k=\dim W_k$ for $k\in[s]$, which satisfy $0=d_1<d_2<\cdots<d_s=r_2$. For convenience, let $d_0=0$.

  For each $2\le k\le s$, let $([g_{1,a_{k,\ell}},g_{1,k}]G_3\Phi(G_2)+W_{k-1})_{\ell\in[d_k-d_{k-1}]}$ be an ordered basis of $W_k/W_{k-1}$, where $a_{k,1},\dots,a_{k,d_k-d_{k-1}}\in[k-1]$. In particular, 
   \[
     W=\linspan_{\F_p}\{[g_{1,a_{k,\ell}},g_{1,k}]G_3\Phi(G_2):~2\leq k\leq s,~\ell\in[d_k-d_{k-1}]\}.
   \]
  We record this ordered basis by defining the
  auxiliary function $f:[r_2]\to[s-1]$ by
  \begin{equation}\label{eq:auxiliary function}
     f(d_{k-1}+\ell)=a_{k,\ell}~~\text{if}~~2\le k\le s~~\text{and}~~1\le\ell\le d_k-d_{k-1}.
  \end{equation}
  Thus $1\le f(j)<k$ whenever $d_{k-1}<j\le d_k$, and $f$ is
  injective on each such block.

  For each $j\in[r_2]$, let $k\in\{2,\ldots,s\}$ be the unique index with
  $d_{k-1}<j\le d_k$, and set
  \begin{equation}\label{p12:eq:second-layer-generators}
   g_{2,j}:=[g_{1,f(j)},g_{1,k}]\in G_2.
  \end{equation}  
   The elements $(g_{2,j}G_3\Phi(G_2))_{j\in[r_2]}$ form the chosen ordered basis
  of $W$, and $(g_{2,j}G_3)_{j\in[r_2]}$ is a minimum generating set of
  $G_2/G_3$. 

  For $i\geq 3$,~\cref{lem:Sims core lower central} (cf.~
  \cite[Proposition~3.8]{BNV}) gives
  $H_i=G_i$ for every $i\ge2$.  Thus for $i\ge2$, $G_{i+1}=H_{i+1}=[H_i,H]=[G_i,H]$.
  If $c\ge3$, choose $g_{3,1},\ldots,g_{3,r_3}\in G_3$ from 
  \[
  \{[g_{2,a},g_{1,b}]\in G_3:b\in[s],\ d_{b-1}<a\le r_2\},  
  \]
  so that their cosets modulo $G_4$ form a minimum generating set of
  $G_3/G_4$~\cite[Proposition~3.7 and Section~5.4, the argument following
  Equation~(5.7)]{BNV}.  

  For $4\le i\le c$, we recursively choose
  $g_{i,1},\ldots,g_{i,r_i}$ from
  \[
    \{[g_{i-1,a},g_{1,b}]\in G_i:a\in[r_{i-1}],\ b\in[s]\}
  \]
  so that their cosets modulo $G_{i+1}$ form a minimum generating set
  of $G_i/G_{i+1}$.  

\paragraph{Relations among a Sims generating set.} We now exhibit relations among the chosen Sims generating set $(g_{i,j})_{i\in[c],j\in[r_i]}$ which specify the $p$-group $G$. 

\emph{Relative orders of Sims generators.} For $i\in[c]$ and $j\in[r_i]$, we use $a(i,j)$ to denote the \emph{relative-order exponent} of $g_{i,j}$; that is, 
  \begin{equation}\label{eq: relative order of gij}
  a(i,j)=\min\{a\in \N:~g_{i,j}^{p^{a}}\in \langle g_{i,1},\ldots,g_{i,j-1},G_{i+1}\rangle\}.
  \end{equation}
  Equivalently, $p^{a(i,j)}$ is the exact order of
  $g_{i,j}G_{i+1}$ modulo
  $\langle g_{i,1}G_{i+1},\ldots, g_{i,j-1}G_{i+1}\rangle$ in $G_i/G_{i+1}$. We collect these exponents in the \emph{relative order array} $(a(i,j))$. The relative order array determines $s_i=\log_p|L_i|=\sum_{j=1}^{r_i}a(i,j)$ and $m=\sum_{i=1}^c s_i$.
  
  Relative to the Sims generating set, any $g\in G$ has a unique \emph{Sims normal form}
  \begin{equation}\label{p12:eq:sims_normal_form}
  g=\prod_{i=1}^c\prod_{j=1}^{r_i}g_{i,j}^{e(i,j)},~~0\le e(i,j)<p^{a(i,j)}
  \end{equation}
  where both products are taken in increasing index order.  For $i\in[c]$, we call $\prod_{j=1}^{r_i}g_{i,j}^{e(i,j)}$ the \emph{weight-$i$ block} of $g$ and $(e(i,j))_{j\in[r_i]}$ the \emph{weight-$i$ coordinates} of $g$.  
  In particular, if $g\in G_k$, then $e(i,j)=0$ for any $i\in[k-1]$ and $j\in[r_{i}]$. Moreover, if $g,g'\in G_i$ satisfy $gG_{i+1}=g'G_{i+1}$, then their weight-$i$ blocks and coordinates are the same.

\emph{Power words of Sims generators.} For $i\in[c]$ and $j\in[r_i]$, let the Sims normal form of $g_{i,j}^{p^{a(i,j)}}$ be
  \begin{equation}\label{p12:eq:power_word_normal_form}
  g_{i,j}^{p^{a(i,j)}}
  =\prod_{u=1}^{c}\prod_{v=1}^{r_u}g_{u,v}^{b(i,j,u,v)}
  =\left(\prod_{v=1}^{j-1}g_{i,v}^{b(i,j,i,v)}\right)
   \left(\prod_{u=i+1}^{c}\prod_{v=1}^{r_u}
     g_{u,v}^{b(i,j,u,v)}\right).
  \end{equation}
  We collect the Sims normal form coordinates of all $g_{i,j}^{p^{a(i,j)}}$ and call $(b(i,j,u,v))$ the \emph{power word array}.
  By~\eqref{p12:eq:sims_normal_form}, the coordinates $(b(i,j,u,v))$ satisfy $0\le b(i,j,u,v)<p^{a(u,v)}$ and
  $b(i,j,u,v)=0$ whenever $u<i$, or $u=i$ and $v\ge j$. 

\emph{Commutator coefficients of Sims generators.} By the choice of the Sims generating set, for $i\in[c-1]$, weight-($i+1$) generators are chosen as commutators of weight-$i$ generators and weight-$1$ generators. More generally, for $i\in[c-1]$, $j\in[r_i]$ and $k\in[r_1]$, write the Sims normal form of $[g_{i,j},g_{1,k}]$ as 
  \begin{equation}\label{p12:eq:relation_normal_form}
  [g_{i,j},g_{1,k}]
  =\prod_{u=1}^{c}\prod_{v=1}^{r_u}
  g_{u,v}^{c(i,j,k,u,v)},
  \end{equation}
  Record $c(i,j,k,u,v)$ as the exponent of $g_{u,v}$ in the Sims
  normal form of $[g_{i,j},g_{1,k}]$.
  The entries in the index ranges of~\cref{thm:sims_presentation_determines_isomorphism_classes_}
  form the \emph{commutator coefficient array}; the same coefficient
  notation is used for the other commutators when needed.  
  Again by~\eqref{p12:eq:sims_normal_form}, the coordinates satisfy
  $0\le c(i,j,k,u,v)<p^{a(u,v)}$ and
  $c(i,j,k,u,v)=0$ if $u\le i$.

The following result of~\cite[Theorem~5.6]{BNV} shows that the
listed integers and relation arrays obtained from an ordered Sims
generating set determine the isomorphism class of $G$.
\begin{proposition}[Theorem~5.6 in~\cite{BNV}]\label{thm:sims_presentation_determines_isomorphism_classes_}
  Let $G$ be a nonabelian finite $p$-group.
  The isomorphism class of the group $G$ is determined by the integers $p$, $c$, $s$, $(r_i)_{i\in[c]}$, $(d_k)_{k\in[s]}$ together with the Sims relation arrays --- relative order array $(a(i,j))$, power word array $(b(i,j,u,v))$ and commutator coefficient array $(c(i,j,k,u,v))$ with the following index ranges:
   \[
  \begin{array}{ll}
  a(i,j):
   &1\le i\le c,\quad 1\le j\le r_i;\\
  b(i,j,u,v):
   &1\le i\le c,\quad 1\le j\le r_i,\quad 1\le u\le c,\quad 1\le v\le r_u;\\
  c(1,j,k,u,v):
   &1\le j<k\le r_1,\quad 2\le u\le c,\quad 1\le v\le r_u;\\
  c(2,j,k,u,v):
   &1\le k\le s,\quad d_{k-1}<j\le r_2~(d_0=0),\quad
     3\le u\le c,\quad 1\le v\le r_u,\\
  c(i,j,k,u,v):
   &3\le i\le c-1,\quad 1\le j\le r_i,\quad 1\le k\le s,\ i+1\le u\le c,\quad 1\le v\le r_u.
  \end{array}
  \]
\end{proposition}

Following~\cite[Section 5.5]{BNV},~\cref{thm:sims_presentation_determines_isomorphism_classes_} can be utilized to enumerate nonabelian groups of order $p^m$ using the following strategy:
First, we enumerate the possible choices of the integers $c$, $s$,
$(r_i)_{i\in[c]}$ and $(d_k)_{k\in[s]}$ describing a Sims generating
set, together with the auxiliary function $f$ as in~\eqref{eq:auxiliary function} and the index pairs
specifying the higher-weight generators.
For each such choice, we enumerate the distinct relative order
arrays $(a(i,j))$, then the distinct power word arrays
$(b(i,j,u,v))$, and finally the distinct commutator coefficient
arrays $(c(i,j,k,u,v))$ arising from compatible ordered Sims
generating sets of groups of order $p^m$, each conditioned on all
preceding choices. 

\phantomsection\label{p12:par:recorded-sims-data}
In the rest of this paper, we work with a fixed subset of the above data, and count the choices of the remaining ones. In particular, by \emph{recorded Sims presentation data}, we mean
\begin{center}
\emph{the integers $p$, $c$, $s$,
$(r_i)_{i\in[c]}$, $(d_k)_{k\in[s]}$, the relative order array $(a(i,j))$, the auxiliary function $f$ from~\eqref{eq:auxiliary function} and the index pairs specifying the generators of weights
$3,\ldots,c$.}
\end{center}
For fixed recorded Sims presentation data, $(g_{i,j})$ denotes a
compatible ordered Sims generating set in each group under
consideration. The group and its actual generators may vary while
these data remain fixed. We count distinct arrays arising from
these realizations. A complete choice of recorded Sims presentation
data, a power word array, and a commutator coefficient array
determines one group isomorphism class by
\cref{thm:sims_presentation_determines_isomorphism_classes_}; a
single isomorphism class may yield several such complete choices.

The following proposition summarizes the discussion in~\cite[Section 5.5]{BNV}, which bounds the choices of recorded Sims presentation data.

\begin{proposition}[Counting recorded Sims presentation data]
\label{p12:prop:sims-generating-data}
Fix a prime $p$ and an integer $m\ge3$. For nonabelian groups
of order $p^m$, the number of choices of the recorded Sims presentation data is at
most $p^{O(m\log m)}$. More precisely,
\begin{itemize}
  \item the number of possible integers $c$, $s$, $(r_i)_{i\in[c]}$ and $(d_k)_{k\in[s]}$ is at most
$2^{O(m)}$. For each such choice of integers, the auxiliary function $f$ and
the index pairs defining the generators of weights $3,\ldots,c$ have
at most $m^{2m}$ possible choices in total.
  \item Once these choices are fixed, the number of possible relative order
arrays is at most
\[
 \binom{m-1}{(\sum_{i=1}^c r_i)-1}\le2^{m-1}.
\]
\end{itemize}
\end{proposition}

\begin{proof}
We first count the possible integers $c$, $s$, $(r_i)_{i\in[c]}$ and
$(d_k)_{k\in[s]}$. Since $2\le c\le m$ and
$2\le s\le r_1\le m$, there are at most $m$ choices for each of
$c$ and $s$. For fixed $c$ and $s$, the ranks are positive integers
satisfying $\sum_{i=1}^c r_i\le m$, so there are at most
$\binom{m}{c}\le2^m$ choices for $(r_i)_{i\in[c]}$.
For each feasible choice of these ranks, the positive increments
$d_2-d_1,\ldots,d_s-d_{s-1}$ form an ordered composition of $r_2$
into $s-1$ parts and determine $(d_k)$, since $d_1=0$.
Thus there are at most
$\binom{r_2-1}{s-2}\le2^{r_2-1}\le2^{m-1}$ choices for $(d_k)$.
Multiplying these bounds, the integers $c$, $s$, $(r_i)$ and $(d_k)$
have at most $m^2 2^{2m}=2^{O(m)}$ possible choices in total.

For each fixed choice of these integers, the function
$f:[r_2]\to[s-1]$ has at most $(s-1)^{r_2}\le m^{r_2}$ choices.
For each generator of weight $i\ge3$, there are at most
$sr_{i-1}\le m^2$ possible index pairs for its defining commutator. Thus the number of choices of the auxiliary function and the defining index pairs is at most
\[
 m^{r_2}\prod_{i=3}^c m^{2r_i}
 =m^{r_2+2\sum_{i=3}^c r_i}\le m^{2m}.
\]
Thus the integers $c$, $s$, $(r_i)$ and $(d_k)$, together
with the auxiliary function $f$ and the index pairs specifying
the higher-weight generators, have at most $2^{O(m)}m^{2m}$
possible choices.

For each such choice, we count the compatible relative order
arrays. The array $(a(i,j))$ has $\sum_{i=1}^c r_i$ positive
entries satisfying
\[
 \sum_{i=1}^c\sum_{j=1}^{r_i}a(i,j)=m.
\]
Consequently, the number of possible relative order arrays is at most
\[
 \binom{m-1}{(\sum_{i=1}^c r_i)-1}\le2^{m-1}.
\]
Combining these bounds gives at most $2^{O(m)}m^{2m}\le p^{O(m\log m)}$
choices of recorded Sims presentation data.
\end{proof}

\subsection{The associated graded Lie ring.}
\label{p12:subsec:filtered-preliminaries}

A Lie ring is an abelian group $L$, written additively, together with a Lie bracket $[\cdot,\cdot]_L: L\times L\to L$ such that for all $x,y,z\in L$:
\begin{itemize}
  \item $[x,x]_L=0$;
  \item $[x,y+z]_L=[x,y]_L+[x,z]_L$, $[x+y,z]_L=[x,z]_L+[y,z]_L$;
  \item $[[x,y]_L,z]_L+[[y,z]_L,x]_L+[[z,x]_L,y]_L=0$.
\end{itemize}
A Lie ring $L$ is \emph{graded} if $L$ is expressed as a direct sum $L=L_1\oplus\cdots\oplus L_c$ of subgroups $L_i$, with $L_t=0$ for $t>c$, such that $[L_i,L_j]_L\subseteq L_{i+j}$ for all $i,j\in[c]$.

For a $p$-group $G$ of nilpotency class $c$, one can associate a graded Lie ring as follows~\cite[Section 5.6]{Gorenstein}. Use the lower central series and layers from~\eqref{p12:eq:lower-central-series}--\eqref{p12:eq:layer-parameters}. We call $L_i=G_i/G_{i+1}$ the \emph{$i$-th layer} of $G$ and write its group operation additively:
\[
(gG_{i+1})+(hG_{i+1})=ghG_{i+1}~~\forall\ g,h\in G_i.
\]
By induction
from the three-subgroup lemma (cf.~\cite[Lemma~2.1(b)]{GlaubermanLargeAbelian}),
we have $[G_i,G_j]\le G_{i+j}$.  
Consider the abelian group
\begin{equation}\label{p12:eq:graded-abelian-group}
 L:=\bigoplus_{i=1}^c L_i.
\end{equation}
An element of $L$ is a tuple $(z_1,\ldots,z_c)$ with $z_i\in L_i$,
and addition is componentwise.  An element of $L$ supported only on
$L_i$ is called \emph{homogeneous of degree $i$}.  A \emph{lift} of
$z\in L_i=G_i/G_{i+1}$ is an element $g\in G_i$ such that $gG_{i+1}=z$.

For $i,j\in [c]$ with $i+j\le c$, define the Lie bracket of 
$z=gG_{i+1}\in L_i$ and $w=hG_{j+1}\in L_j$ by
\begin{equation}\label{p12:eq:graded-bracket-lifts}
 [z,w]_L=[gG_{i+1},hG_{j+1}]_L
 :=[g,h]G_{i+j+1}\in L_{i+j}.
\end{equation}
Note that the definition is independent of the chosen lifts: changing either lift
changes the commutator only by an element of $G_{i+j+1}$. 
For fixed $i,j$ with $i+j\le c$, this gives a biadditive map
$L_i\times L_j\to L_{i+j}$.  If $i+j>c$, then
$[G_i,G_j]\le G_{i+j}=1$, so set $[L_i,L_j]_L=0$.  Extend these maps
to arbitrary elements of $L$ by
\[
 \left[\bigoplus_{i=1}^{c} z_i,\bigoplus_{j=1}^{c} w_j\right]_L
 =\sum_{\substack{1\le i,j\le c\\ i+j\le c}}[z_i,w_j]_L.
\]
The resulting Lie bracket $[\mathord\cdot,\mathord\cdot]_L:L\times L\to L$
has the following properties:
\begin{itemize}
  \item It satisfies the \emph{grading rule}
  $[L_i,L_j]_L\subseteq L_{i+j}$ when $i+j\le c$, and
  $[L_i,L_j]_L=0$ when $i+j>c$.
Here $[L_i,L_j]_L$ denotes the additive subgroup generated by these
brackets.
  \item It is additive in each variable: for any $x,x',y,y'\in L$,
  \[
   [x+x',y]_L=[x,y]_L+[x',y]_L,
   \qquad
   [x,y+y']_L=[x,y]_L+[x,y']_L.
  \]
\item It is alternating and antisymmetric: for any $x,y\in L$,
$[x,x]_L=0$ and $[x,y]_L=-[y,x]_L$.
\item It satisfies the \emph{Jacobi identity}: for any $x,y,z\in L$,
\begin{equation}\label{p12:eq:graded-jacobi-explicit}
 [[x,y]_L,z]_L+[[y,z]_L,x]_L+[[z,x]_L,y]_L=0.
\end{equation}
\end{itemize}
Thus, the abelian group $L$ with the Lie bracket $[\cdot,\cdot]_L$ is a graded Lie ring, and we shall call it the
\emph{associated graded Lie ring of $G$}.
\begin{remark}\label{p12:rem:bracket-mod-p}
 Under the identifications $V\cong L_1/pL_1$ and $W\cong L_2/pL_2$ in~\eqref{p12:eq:frattini-quotients}, the alternating bilinear map $\beta_G$ from~\eqref{p12:eq:group-alternating-map} is precisely induced by the restriction $L_1\times L_1\to L_2$ of the Lie bracket~\eqref{p12:eq:graded-bracket-lifts} modulo $p$.
\end{remark}

\subsection{Sims presentation in the associated Lie ring}
Blackburn utilized Sims' enumeration framework to study the enumeration of graded Lie rings, leading to better enumeration bounds for $p$-groups of nilpotency class at most $3$~\cite{Blackburn} (see also~\cite[Sections~20.3 and~20.4]{BNV}). We record how Sims generating sets and Sims relation arrays are translated in the associated Lie ring.

\paragraph{The associated graded Lie ring generators from a Sims generating set.}
Let $(g_{i,j})_{i\in[c],j\in[r_i]}$ be a chosen Sims generating set with respect to a Sims core $H$, as constructed in~\cref{p12:par:fixed-sims-data}. Define its homogeneous generators by
\begin{equation}\label{p12:eq:homogeneous-generators}
 \bar g_{i,j}:=g_{i,j}G_{i+1}\in L_i
 \qquad(i\in[c],\ j\in[r_i]).
\end{equation}
By that construction, $L_i=\langle\bar g_{i,1},\ldots,\bar g_{i,r_i}\rangle$. We call $\bar g_{i,j}$ the $j$th homogeneous generator of $L_i$. 

The first $s$ weight-$1$ generators generate the chosen Sims core $H=\langle g_{1,1},\ldots,g_{1,s}\rangle$, as in~\cref{p12:par:fixed-sims-data}. Define its image in $L_1$ by
\begin{equation}\label{p12:eq:core-image}
 C_{\core}:=HG_2/G_2=\langle\bar g_{1,1},\ldots,\bar g_{1,s}\rangle\le L_1.
\end{equation} 

In $L_2$, the construction~\eqref{p12:eq:second-layer-generators} chooses each $g_{2,j}$, $j\in[r_2]$, as $[g_{1,f(j)},g_{1,k}]$, where $k$ is the unique index in $\{2,\dots,s\}$ satisfying $d_{k-1}<j\leq d_k$ and $f$ records the chosen core commutators used to define the weight-$2$ generators. Thus, we have $\bar g_{2,j}=[\bar g_{1,f(j)},\bar g_{1,k}]_L$ for all $j\in[r_2]$.  

For $i\ge3$, the construction in~\cref{p12:par:fixed-sims-data} chooses each weight-$i$ generator as a commutator of a weight-$(i-1)$ generator and a Sims core generator. Consequently, we have
\begin{equation}\label{p12:eq:sims-core-graded-generation}
 \begin{split}
  L_1&=\bigl\langle C_{\core},\bar g_{1,s+1},\dots,
  \bar g_{1,r_1}\bigr\rangle,\\
  L_2&=\bigl\langle[\bar g_{1,a},\bar g_{1,b}]_L:1\le a<b\le s\bigr\rangle,\\
  L_{i+1}&=[L_i,C_{\core}]_L
  \qquad(2\le i\le c-1).
 \end{split}
 \end{equation}
 In particular, for $2\leq i\leq c$, each $L_i$ is generated additively by
 \emph{left-normed Lie words} in $(\bar g_{1,k})_{k\in[s]}$.

\paragraph{Sims relation arrays in the associated graded Lie ring.} We now examine how the Sims relation arrays determine the graded Lie ring structure.

Given a homogeneous element $z\in L_i=G_i/G_{i+1}$ of degree $i$, note that the weight-$i$ coordinates of a lift $g\in G_i$ of $z$ are independent of the chosen lift. Take any such lift $g$ such that $z=gG_{i+1}$. Denote the weight-$i$ block of $g$ by $\prod_{j=1}^{r_i}g_{i,j}^{e(i,j)}$, where $0\le e(i,j)<p^{a(i,j)}$ for $j\in[r_i]$. These weight-$i$ coordinates \emph{uniquely} determine $z$ in $L_i$ as
 \begin{equation}\label{eq: homogeneous element decomposition}
 z=\sum_{j=1}^{r_i}e(i,j)\bar g_{i,j}\in L_i.
 \end{equation}
Based on this, define the \emph{Sims lift} of $z$ by
 \begin{equation}\label{p12:eq:sims-lift}
  g_z:=\prod_{j=1}^{r_i}g_{i,j}^{e(i,j)}\in G_i.
 \end{equation}

Passing the power relation~\eqref{p12:eq:power_word_normal_form}
 to $L_i=G_i/G_{i+1}$ kills its blocks of weight greater than $i$
 and gives the additive relation
 \begin{equation}\label{p12:eq:fibre-power-to-layer}
  p^{a(i,j)}\bar g_{i,j}
  =\sum_{v=1}^{j-1}b(i,j,i,v)\bar g_{i,v}\in L_i.
 \end{equation}
 Consequently, the relative order array $(a(i,j))$ and the partial power word array $(b(i,j,i,v))$ determine the
 abelian groups $(L_i)_{i\in[c]}$ and their direct sum $L=\bigoplus_{i\in[c]}L_i$. 
 
 With the graded abelian group $L$ and its ordered homogeneous generators fixed, we now explain how the commutator coefficient array determines the Lie bracket. 
 For $i\in[c-1]$, consider a group commutator of the form $[g_{i,j},g_{1,k}]$ for some $j\in[r_i]$ and $k\in[r_1]$. Grouping the Sims normal form from~\eqref{p12:eq:relation_normal_form}
 into its weight blocks, we have
\begin{equation}\label{p12:eq:commutator-weight-blocks}
 [g_{i,j},g_{1,k}]
 =\prod_{t=i+1}^{c}\omega_t(i,j,k)\in G_{i+1},~\text{where}~\omega_t(i,j,k):=\prod_{v=1}^{r_t}
 g_{t,v}^{c(i,j,k,t,v)}\in G_t.
\end{equation}
For $i+1\le u\le c$, define the \emph{$L_u$-component} of $[g_{i,j},g_{1,k}]$ as
\begin{equation}\label{p12:eq:L_u_component_relation}
\omega_u(i,j,k)G_{u+1}
 =\left(\prod_{v=1}^{r_u}g_{u,v}^{c(i,j,k,u,v)}\right)G_{u+1}
 =\sum_{v=1}^{r_u}c(i,j,k,u,v)\bar g_{u,v}\in L_u.  
\end{equation}
With respect to the chosen Sims generating set, the commutator $[g_{i,j},g_{1,k}]$ corresponds to its \emph{Sims normal-form layer decomposition} in the graded abelian group $L$:
\begin{equation}\label{eq:sims layer decomposition}
  [g_{i,j},g_{1,k}]~~\Longleftrightarrow~~\sum_{t=i+1}^{c}\omega_t(i,j,k)G_{t+1}\in L,
\end{equation}
where the summand $\omega_t(i,j,k)G_{t+1}$ lies in
$L_t=G_t/G_{t+1}$. 
In particular, the $L_{i+1}$-component of $[g_{i,j},g_{1,k}]$ is precisely the corresponding Lie bracket value:
\begin{equation}\label{p12:eq:fibre-commutator-to-bracket}
  \omega_{i+1}(i,j,k)G_{i+2}=\sum_{v=1}^{r_{i+1}}c(i,j,k,i+1,v)\bar g_{i+1,v}=[\bar g_{i,j},\bar g_{1,k}]_L\in L_{i+1}
\end{equation}
We note that the $L_u$-components for $i+2\le u\le c$ encode the deeper
commutator coefficients.

For $i=1$, the commutator coefficients $(c(1,j,k,2,v))$
determine the bracket on $L_1\times L_1$ by additivity and
alternation. If $c\ge3$, the coefficients $(c(2,j,k,3,v))$ determine
$[\bar g_{2,j},\bar g_{1,k}]_L$ for $k\in[s]$ and
$d_{k-1}<j\le r_2$.
We recover the remaining values by induction on $k$.
For $k=1$, all values are recorded, since $d_0=0$.
Suppose $k\ge2$ and all values with second core index less
than $k$ have been determined.
For an omitted index $j\le d_{k-1}$, let $k'$ be the unique
index satisfying $d_{k'-1}<j\le d_{k'}$.
By the construction of the weight-$2$ generators,
$f(j)<k'<k$ and $\bar g_{2,j}=[\bar g_{1,f(j)},\bar g_{1,k'}]_L$.
The Jacobi identity gives
\[
 [\bar g_{2,j},\bar g_{1,k}]_L
 =[[\bar g_{1,f(j)},\bar g_{1,k}]_L,
          \bar g_{1,k'}]_L-[[\bar g_{1,k'},\bar g_{1,k}]_L,
          \bar g_{1,f(j)}]_L.
\]
The inner brackets are determined by the known bracket on
$L_1\times L_1$. Since $f(j),k'<k$, the outer brackets
are determined by the induction hypothesis and biadditivity.
This recovers all omitted values and hence determines
the bracket on $L_2\times C_{\core}$.

For $3\le i\le c-1$, the corresponding partial commutator
coefficient array directly determines the bracket on
$L_i\times C_{\core}$.
Together, these restrictions determine the entire graded
Lie bracket by~\cref{p12:lem:relative-graded}\textup{(iii)}.

\subsection{Improving Sims' enumeration framework}
\cref{thm:sims_presentation_determines_isomorphism_classes_} can be
used to count isomorphism classes of nonabelian groups of order $p^m$;
the abelian groups contribute at most $p^m$ isomorphism classes.
By~\cref{p12:prop:sims-generating-data}, the recorded Sims presentation
data have at most $p^{O(m\log m)}$ choices. For each such choice,
the number of compatible power word arrays $(b(i,j,u,v))$ is bounded
naively by $p^{O(m^2)}$~\cite[Sections~5.4--5.5]{BNV}.

To derive Sims' bound for $f(p^m)$, we count the distinct commutator
coefficient arrays $(c(i,j,k,u,v))$ after fixing the recorded Sims
presentation data and a compatible power word array.
The residues $\{c(1,j,k,2,v)\bmod p:~1\leq j<k\leq r_1,~v\in[r_2]\}$ determine the alternating bilinear map $\beta_G$. Since $s$, $r_1$ and $r_2$ are fixed, there are at most $p^{N_p(r_1,r_2,s)}$ such maps. This count may include maps incompatible with the fixed flag dimensions $(d_k)$, which only increases the upper bound. Bounding the lifts of these residues and the remaining commutator coefficients naively gives the following upper bound for the base-$p$ logarithm of the number of choices of $(c(i,j,k,u,v))$:
\begin{equation}\label{p12:eq:book-master}
  N_p(r_1,r_2,s)
   +\left(\binom{r_1}{2}+(s-1)r_2-\frac{(s-1)^2}{2}\right)(m-r_1-r_2)+\frac{s-1}{2}(m-r_1-r_2)^2+O(m^2).
\end{equation}
See the derivation for Eq.~(5.14) in~\cite{BNV} for details. Substituting the bounds on $N_p(r_1,r_2,s)$ in~\eqref{eq:BNV's Nprs} and optimizing over the parameters gives Sims' and Newman-Seeley's bounds on $f(p^m)$.

Our technical contributions are two improvements within the Sims enumeration framework. 
\begin{itemize}
  \item In Part~I, we first count power word arrays for fixed recorded
  Sims presentation data. For each fixed power word array, we use
  the associated graded Lie ring to count commutator coefficient
  arrays. Combining these conditional counts gives a bound on the
  number of distinct pairs of these arrays, which Part~III uses to
  count group isomorphism classes.
  \item In Part~II, we improve the estimate of $N_p(r_1,r_2,s)$ when
  $r_2\leq\binom{s-1}{2}$. Combining this map count with the
  presentation counts shows that nonabelian $p$-groups in this range
  contribute a negligible proportion of all groups of order $p^m$.
  Part~III therefore focuses on $\binom{s-1}{2}<r_2\leq\binom{s}{2}$.
\end{itemize}

\subsection{Uniform counting tools}\label{p12:sec:uniformity}

We collect some useful Gaussian, rank, and
basis estimates. 

For a prime power $q\ge2$, put
\[
 \pi_0(q)=1,\qquad
 \pi_n(q)=\prod_{j=1}^n(1-q^{-j})\quad(n\ge1),\qquad
 \pi_\infty(q)=\prod_{j=1}^{\infty}(1-q^{-j}).
\]
The elementary inequality
$\prod_j(1-x_j)\ge1-\sum_j x_j$ for $0\le x_j\le1$ gives
\begin{equation}\label{p12:eq:pi-infinity-two-lower}
 \begin{aligned}
 \pi_\infty(2)
 &=\frac12\cdot\frac34\prod_{j\ge3}(1-2^{-j})\ge\frac12\cdot\frac34
       \left(1-\sum_{j\ge3}2^{-j}\right)
  =\frac12\cdot\frac34\cdot\frac34
  =\frac9{32}>\frac14.
 \end{aligned}
\end{equation}
Since each factor increases with $q$ and
each finite product decreases to its infinite counterpart,
\begin{equation}\label{p12:eq:uniform-pi}
 \pi_\infty(2)\le \pi_n(q)\le1
 \qquad(q\ge2,\ n\ge0).
\end{equation}

\begin{lemma}[Universal finite-field estimates]\label{p12:lem:universal-field}
The following bounds hold with absolute constants, simultaneously for every prime power $q\ge2$.
\begin{enumerate}[label=\textup{(\roman*)}]
 \item For $0\le k\le n$,
 \begin{equation}\label{p12:eq:uniform-gaussian}
  q^{k(n-k)}\le \qbinom nkq
  \le \pi_\infty(2)^{-1}q^{k(n-k)}.
 \end{equation}
 \item For integers $R\ge0$ and $0\le a\le\lfloor R/2\rfloor$, the number of alternating
 forms on a fixed $R$-dimensional $\F_q$-space having rank at most $2a$ is at most $2\pi_\infty(2)^{-2}q^{a(2R-2a-1)}$.
 \item For every integer $d\ge0$, one has $|\GL_d(q)|\le q^{d^2}$.
 If $p$ is prime, $G$ is an abelian $p$-group of order $p^e$, and
 $0\le b\le e$, then the number of its subgroups of order $p^b$ is at most $p^{b(e-b+1)}$.
 \item For every prime $p$ and integer $m\ge1$, the number of
 isomorphism classes of abelian groups of order $p^m$ is the partition
 number of $m$, and is at most $2^{m-1}\le p^m$.
\end{enumerate}
\end{lemma}

\begin{proof}
The Gaussian binomial coefficient formula gives
\[
\qbinom nkq=q^{k(n-k)}
 \frac{\pi_n(q)}{\pi_k(q)\pi_{n-k}(q)}.
\]
The ratio is at least one term by term and is at most
$\pi_\infty(2)^{-1}$ by \eqref{p12:eq:uniform-pi}, proving (i).

For each $0\le j\le a$, the number of alternating forms of rank exactly $2j$ is
\[
 \qbinom R{2j}q\frac{|\GL_{2j}(q)|}{|\Sp_{2j}(q)|}
 =
  \frac{\pi_R(q)}{\pi_{R-2j}(q)\pi_j(q^2)}\cdot q^{j(2R-2j-1)}.
\]
See~\cite[Section~5, equation~(24)]{FulmanGoldstein} for this formula.
The product factor is at most $\pi_\infty(2)^{-2}$.  The exponent
$E(j)=j(2R-2j-1)$ satisfies
$E(j)-E(j-1)=2R-4j+1\ge1$ for $j\le R/2$.  Thus the largest
power-of-$q$ bound in the range $0\le j\le a$ occurs at $j=a$.
Let $N_j$ be the number of alternating forms of rank exactly
$2j$.  For $0\le s\le a$, the consecutive-exponent inequality gives
$E(a-s)\le E(a)-s$ and $N_{a-s}\le \pi_\infty(2)^{-2}q^{E(a)-s}$.
Consequently,
\[
 \sum_{j=0}^{a}N_j
 =\sum_{s=0}^{a}N_{a-s}
 \le \pi_\infty(2)^{-2}q^{E(a)}
     \sum_{s=0}^{a}q^{-s}
 \le 2\pi_\infty(2)^{-2}q^{E(a)},
\]
where the last inequality follows from $\sum_{s\ge0}q^{-s}=(1-q^{-1})^{-1}\le2$ for $q\ge2$.  This proves (ii).

Part (iii) is immediate for the general linear group from
$|\GL_d(q)|=\prod_{i=0}^{d-1}(q^d-q^i)\le q^{d^2}$.  The subgroup
estimate is from \cite[Lemma~20.15]{BNV}.  For the classical exact enumeration of
subgroups of prescribed type, see~\cite{Birkhoff}.

Part (iv) follows by encoding a partition by its partial sums.
\end{proof}

By~\cref{p12:lem:universal-field}\textup{(iv)}, the number
of isomorphism classes of abelian groups of order $p^m$
is at most $p^m$. Together with Higman's lower bound
$f(p^m)\ge p^{\frac{2}{27}m^3-\frac49m^2}$~\cite{Higman},
this shows that the proportion of abelian groups among
all groups of order $p^m$ is at most
\[
 \frac{p^m}{f(p^m)}
 \le p^{-\frac{2}{27}m^3+\frac49m^2+m}
 \longrightarrow0
\]
when $m\to\infty$. Thus abelian groups contribute a negligible proportion,
uniformly in the prime $p$.

In the rest of this paper, we shall mostly focus on counting nonabelian $p$-groups. Write
\begin{equation}\label{p12:eq:nonabelian-count}
 f_{nab}(p^m):=\#\{\text{isomorphism classes of nonabelian groups of order }p^m\}.
\end{equation}

\clearpage
\part{Counting Sims relation arrays}\label{p12:part:valid-completions}

\section{Counting power word arrays}\label{sec:bijuv}

Fix a prime
$p$ and an integer $m\ge3$. Fix recorded Sims presentation data
occurring for a nonabelian group of order $p^m$, as counted
in~\cref{p12:prop:sims-generating-data}. We count distinct power
word arrays and commutator coefficient arrays arising from all
nonabelian groups of order $p^m$ equipped with ordered Sims
generating sets realizing these recorded Sims presentation data. By the relative-order definition~\eqref{eq: relative order of gij}, the fixed array determines $s_i=\sum_{j=1}^{r_i}a(i,j)$, and hence the parameters $m_i$ from~\eqref{p12:eq:layer-parameters}. In each realization, let $(g_{i,j})_{i\in[c],j\in[r_i]}$ denote its chosen Sims generating set.

We first count the power word arrays $(b(i,j,u,v))$ conditioned on a fixed choice of recorded Sims presentation data.
Sims bounds the number of power word arrays by
$p^{O(m^2)}$~\cite[p.~164]{Sims}
(see also~\cite[Section~5.5, p.~41]{BNV}).
\phantomsection\label{p12:par:naive-power-count}
Indeed, there are $\sum_i r_i\le m$ power words, each with at most
$p^m$ possible Sims normal forms, so the naive count is at most
$p^{m\sum_i r_i}\le p^{m^2}$. We refine this estimate by
counting the same-layer coefficients and the deeper blocks separately. 
\begin{theorem}[Counting power word arrays]
\label{p12:lem:power-ledger}
For nonabelian groups of order $p^m$, fix recorded Sims presentation
data as counted in~\cref{p12:prop:sims-generating-data}.
The base-$p$ logarithm of the number of distinct power word arrays
$(b(i,j,u,v))$ arising from groups whose Sims generating sets
realize these fixed data is at most
\begin{align}\label{p12:eq:power-ledger}
 \sum_{i=1}^{c}\left[
       r_i m_{i+1}+
       \sum_{v=1}^{r_i}(r_i-v)(a(i,v)-1)\right]\le\sum_{i=1}^{c}\left[
       r_i m_{i+1}+(r_i-1)(s_i-r_i)\right].
\end{align}

\end{theorem}

\begin{proof}
Recall the power word relation~\eqref{p12:eq:power_word_normal_form}:
\[
 g_{i,j}^{p^{a(i,j)}}
 =\left(\prod_{v=1}^{j-1}g_{i,v}^{b(i,j,i,v)}\right)
  \left(\prod_{u=i+1}^{c}\prod_{v=1}^{r_u}
    g_{u,v}^{b(i,j,u,v)}\right).
\]
Its weight-$i$ block involves only the earlier generators
$g_{i,1},\ldots,g_{i,j-1}$, and its remaining blocks have weights
greater than $i$. The deeper coordinates have at most $p^{m_{i+1}}$
possible tuples, since their ranges are fixed and
$m_{i+1}=\sum_{u=i+1}^c\sum_{v=1}^{r_u}a(u,v)$.

For $u\ge i+1$, every deeper generator satisfies
$g_{u,v}\in G_u\le G_{i+1}$. Thus applying the quotient map
$G_i\to L_i=G_i/G_{i+1}$ to the power relation makes all deeper
blocks trivial and gives the additive relation
\eqref{p12:eq:fibre-power-to-layer}. Since $a(i,j)\ge1$, its left
side lies in $pL_i$ and becomes zero in $L_i/pL_i$. Reducing that
relation modulo $pL_i$ therefore gives
\[
 0=\sum_{v=1}^{j-1}\bigl(b(i,j,i,v)\bmod p\bigr)
       (\bar g_{i,v}+pL_i)\in L_i/pL_i,
\]
where the homogeneous generators $\bar g_{i,v}$ are those in~\eqref{p12:eq:homogeneous-generators}. By the Sims generating-set construction in~\cref{p12:par:fixed-sims-data}, the elements $\bar g_{i,1},\ldots,\bar g_{i,r_i}$ form a minimum
generating set of the abelian $p$-group $L_i$, so their images modulo
$pL_i$ form an $\F_p$-basis of $L_i/pL_i$.
By linear independence, we have
$b(i,j,i,v)\equiv0\pmod p$ for every $v<j$. This fact has already been observed by
Sims~\cite[p.~160]{Sims}.

For each $v<j$, the allowed interval
$0\le b(i,j,i,v)<p^{a(i,v)}$ contains exactly $p^{a(i,v)-1}$
multiples of $p$. Hence this same-layer coordinate has at most
$p^{a(i,v)-1}$ choices. Combining these bounds with the bound for
the deeper coordinates, the power word for $g_{i,j}$ has at most $p^{m_{i+1}+\sum_{v=1}^{j-1}(a(i,v)-1)}$
possible choices.

To count the complete power word array, multiply this bound over all
$i\in[c]$ and $j\in[r_i]$. For each fixed $i$, the deeper term
$m_{i+1}$ occurs $r_i$ times. For each fixed $v\in[r_i]$, the
same-layer coordinates $b(i,j,i,v)$ with target $g_{i,v}$ occur for
$j=v+1,\ldots,r_i$. There are $r_i-v$ such coordinates, each
contributing at most $a(i,v)-1$ to the exponent. This gives the first bound
in~\eqref{p12:eq:power-ledger}.
Finally, $a(i,v)-1\ge0$, $r_i-v\le r_i-1$, and
$\sum_{v=1}^{r_i}(a(i,v)-1)=s_i-r_i$ give the second bound.
\end{proof}

\section{Counting leading commutator coefficients}
\label{p12:subsec:lie-bracket-count}
Now fix recorded Sims presentation data counted
in~\cref{p12:prop:sims-generating-data}, together with a
compatible power word array $(b(i,j,u,v))$ counted
in~\cref{sec:bijuv}. We consider groups $G$ equipped with
Sims generating sets $(g_{i,j})_{i\in[c],j\in[r_i]}$
realizing these fixed data.
By~\eqref{p12:eq:fibre-power-to-layer}, these choices
determine a common graded abelian group
$L=\bigoplus_{i=1}^c L_i$ with the homogeneous generators $\bar g_{i,j}$ from~\eqref{p12:eq:homogeneous-generators}. In each realization, we identify
$G_i/G_{i+1}$ with $L_i$ by
$g_{i,j}G_{i+1}\mapsto\bar g_{i,j}$.
Thus $L$ and its labeled homogeneous generators remain
fixed as the commutator coefficient array varies.

The core image $C_{\core}\le L_1$ from~\eqref{p12:eq:core-image} is generated by the first $s$ homogeneous generators $\bar g_{1,1},\ldots,\bar g_{1,s}$. Thus $C_{\core}$ is also fixed in the common graded abelian group $L$.

In this section, we count the graded Lie brackets on $L$ induced by groups realizing
these fixed data, which will determine the commutator coefficient subarray
$(c(i,j,k,i+1,v))$ by~\eqref{p12:eq:fibre-commutator-to-bracket}.
We follow Blackburn's treatment of graded Lie rings \cite[Chapter~6]{Blackburn}, as presented in \cite[Sections~20.3--20.4]{BNV}, with the relevant bracket maps restricted to the Sims core. 

By
\eqref{p12:eq:sims-core-graded-generation}, every graded Lie bracket $[\cdot,\cdot]_L$ on $L=\bigoplus_{i=1}^c L_i$ arising from groups with these fixed data satisfies 
 \[
 L_2=\bigl\langle[u,v]_L:u,v\in C_{\core}\bigr\rangle,~~\text{and}~~L_{i+1}=[L_i,C_{\core}]_L~~\text{for}~~2\le i\le c-1.
 \]
The following lemma shows that the bracket $L_1\times L_1\to L_2$
and the core bracket maps
$L_i\times C_{\core}\to L_{i+1}$ for $2\le i\le c-1$
determine the whole Lie bracket. A similar result with
$C_{\core}$ replaced by $L_1$ was proved
in~\cite[Lemma~20.13]{BNV}.
\begin{lemma}
\label{p12:lem:relative-graded}
Let $c\ge2$ and let $L=\bigoplus_{i=1}^{c}L_i$ be a graded abelian group, with $L_i=0$ for $i>c$. Let
$C_{\core}\le L_1$.  For $\nu\in\{0,1\}$, let
$[-,-]_\nu$ be a graded Lie bracket on $L$ satisfying $L_2=\bigl\langle[u,v]_\nu:u,v\in C_{\core}\bigr\rangle$ and $L_{i+1}=[L_i,C_{\core}]_\nu$ for $2\leq i\leq c-1$.
We have the following:
\begin{enumerate}[label=\textup{(\roman*)}]
\item If the two brackets agree on
$L_1\times C_{\core}$ and
$L_2\times C_{\core}$, then they agree on
$L_2\times L_1$.
\item More generally, if they agree on
$L_i\times C_{\core}$ for every $i\ge1$, then they agree on
$L_a\times L_b$ whenever $\max\{a,b\}\ge2$.
\item If, in addition to the hypothesis of~\textup{(ii)}, they agree
on $L_1\times L_1$, then they agree on $L\times L$.
\end{enumerate}
\end{lemma}

\begin{proof}
For~\textup{(i)}, let $w\in L_2$ and $x\in L_1$.  By
the generation hypothesis, and because the two brackets agree on
$L_1\times C_{\core}$, we may choose
$u_\alpha,v_\alpha\in C_{\core}$ such that $w=\sum_\alpha[u_\alpha,v_\alpha]_0
  =\sum_\alpha[u_\alpha,v_\alpha]_1$.
For each $\nu\in\{0,1\}$, the Jacobi
identity~\eqref{p12:eq:graded-jacobi-explicit} gives
\[
 [[u_\alpha,v_\alpha]_\nu,x]_\nu
 =-[[v_\alpha,x]_\nu,u_\alpha]_\nu
  -[[x,u_\alpha]_\nu,v_\alpha]_\nu.
\]
The inner brackets on the right are equal for $\nu=0,1$ by the agreement on $L_1\times C_{\core}$ and alternation. The outer brackets are equal because the two brackets also agree on $L_2\times C_{\core}$. Hence
$[w,x]_0=[w,x]_1$, proving~\textup{(i)}.

We follow the induction of
\cite[Lemma~20.13]{BNV} to prove~\textup{(ii)}.  The hypothesis and additivity imply that
the two brackets agree on $L\times C_{\core}$.  We prove by
induction on $j\ge2$ that they agree on $L\times L_j$.  For $j=2$,
let $y\in L_2$ and choose a common core-commutator decomposition of
$y$ as in the proof of part~(i). If $x\in L_i$, the same Jacobi calculation uses only
the common restrictions on $L_i\times C_{\core}$ and
$L_{i+1}\times C_{\core}$, so $[x,y]_0=[x,y]_1$.

Now let $j\ge3$ and assume the conclusion for $L\times L_{j-1}$.
The common restriction on $L_{j-1}\times C_{\core}$ and the
generation hypothesis give a common expression
\[
 y=\sum_\alpha[z_\alpha,u_\alpha]_0
  =\sum_\alpha[z_\alpha,u_\alpha]_1,
 \qquad z_\alpha\in L_{j-1},\quad
 u_\alpha\in C_{\core},
\]
and, for $\nu\in\{0,1\}$, the Jacobi identity
\eqref{p12:eq:graded-jacobi-explicit} gives
\[
 [x,[z_\alpha,u_\alpha]_\nu]_\nu
 =[[x,z_\alpha]_\nu,u_\alpha]_\nu
  -[[x,u_\alpha]_\nu,z_\alpha]_\nu.
\]
The first term on the right is equal for $\nu=0,1$ by induction and
then the common core restriction.  The second is equal first by the
common core restriction and then by induction.  Thus the two
brackets agree on $L\times L_j$.  Alternation gives agreement when
the higher-layer input occurs first, proving~\textup{(ii)}.  The only
remaining domain is $L_1\times L_1$, so~\textup{(iii)} follows.
\end{proof}
By~\cref{p12:lem:relative-graded}\textup{(iii)}, it suffices to count
the higher core bracket maps $L_i\times C_{\core}\to L_{i+1}$ for
$2\le i\le c-1$, and then the restrictions $L_1\times L_1\to L_2$
for each fixed family of core maps. We carry out these conditional
counts for a fixed \emph{breadth sequence}
$\boldsymbol{\fb}=(\fb_i)_{i=2}^{c}$, defined in~\eqref{p12:eq:core-breadth}.

For $1\le i\le c$, the core bracket map $L_i\times C_{\core}\longrightarrow L$ evaluates $[z,u]_L$ for all $z\in L_i$ and $u\in C_{\core}$. Fix $u\in C_{\core}$ and denote its slice by $\ad(u):L_i\longrightarrow L$, $z\mapsto[z,u]_L$. For $i<c$, its image lies in $L_{i+1}$, and we also regard it as a map into that layer.
Thus $\im\ad(u)=[L_i,u]_L$.  The \emph{breadth} of
the slice $\ad(u)$ is defined as $\log_p|\im\ad(u)|$, and we write $\ker\ad(u)=\{z\in L_i:[z,u]_L=0\}$.

For $2\le i\le c$, define the maximum breadth among the slices of the core
bracket map on $L_i$, and choose a maximizing element, by
\begin{equation}\label{p12:eq:core-breadth}
 \fb_i=\max_{u\in C_{\core}}\log_p|\im\ad(u)|,
 \qquad
 h_i^\ast\in\operatorname*{arg\,max}_{u\in C_{\core}}
 \log_p|\im\ad(u)|.
\end{equation}
For $2\le i\le c-1$, we have
$0\le \fb_i\le\min\{s_i,s_{i+1}\}$; also $\fb_c=0$, since
$[L_c,L_1]_L=0$.
For $2\le i\le c$, we have
\begin{equation}\label{p12:eq:adjoint-kernel-size}
 |\ker\ad(h_i^\ast)|
 =\frac{|L_i|}{|[L_i,h_i^\ast]_L|}=p^{s_i-\fb_i}.
\end{equation}
Any nonempty set of values in $L_i$ whose pairwise differences lie
in this kernel is contained in a single coset, and therefore has at
most $p^{s_i-\fb_i}$ elements.

We apply, layer by layer and with the second argument restricted to the Sims core, the counting argument from the proof of \cite[Proposition~20.16, pp.~217--218]{BNV}.
\begin{proposition}\label{p12:prop:core_map_fix_breadth}
Fix a breadth sequence $\boldsymbol{\fb}=(\fb_i)_{i=2}^{c}$
that occurs for a group with fixed recorded Sims presentation data as counted in~\cref{p12:prop:sims-generating-data} and power word array as counted in~\cref{p12:lem:power-ledger}. The base-$p$ logarithm of the number of
possible families of core bracket maps
$L_i\times C_{\core}\to L_{i+1}$, $2\le i\le c-1$, with this
breadth sequence is at most
$\sum_{i=2}^{c-1}s \fb_i(r_i+s_{i+1}-\fb_i+1)$.
\end{proposition}
\begin{proof}
  Fix $2\le i\le c-1$ and a chosen core generator
$\bar g_{1,k}=g_{1,k}G_2$, where $k\in[s]$.  The slice
$\ad(\bar g_{1,k}):L_i\to L_{i+1}$ has image of order at most $p^{\fb_i}$.
Since $\fb_i\le s_{i+1}$, a subgroup chain in the finite abelian
$p$-group $L_{i+1}$ extends this image to a subgroup of order exactly
$p^{\fb_i}$.  By \cref{p12:lem:universal-field}\textup{(iii)}
(cf.~\cite[Lemma~20.15]{BNV}), there are at most
$p^{\fb_i(s_{i+1}-\fb_i+1)}$ choices for such a subgroup.  Once it is
chosen, the images of the $r_i$ chosen generators of $L_i$ have at
most $p^{\fb_i r_i}$ possibilities.  Thus the $\bar g_{1,k}$-slice has at most
$p^{\fb_i(r_i+s_{i+1}-\fb_i+1)}$ possibilities.  The $s$ core generators
determine the whole core bracket map by additivity in the second
variable.  Applying these bounds over all $k\in[s]$ and
$2\le i\le c-1$
gives at most
\begin{equation}\label{p12:eq:core-bracket-count}
 p^{\sum_{i=2}^{c-1}s \fb_i(r_i+s_{i+1}-\fb_i+1)}
\end{equation}
possible families of higher core bracket maps.
\end{proof}

The generation property~\eqref{p12:eq:sims-core-graded-generation} also gives a useful restriction on the
breadths.  Since
\[
 L_{i+1}=[L_i,C_{\core}]_L
 =\sum_{k=1}^s[L_i,\bar g_{1,k}]_L,
\]
the addition map from $\prod_{k=1}^s[L_i,\bar g_{1,k}]_L$ onto $L_{i+1}$ is
surjective.  Hence, for $2\le i\le c-1$,
\begin{equation}\label{p12:eq:e-breadth}
 p^{s_{i+1}}=|L_{i+1}|
 \le \prod_{k=1}^{s}|[L_i,\bar g_{1,k}]_L|
 \le p^{s \fb_i}
 \qquad\Rightarrow\qquad
 s_{i+1}\le s \fb_i.
\end{equation}

It remains to count the bracket $L_1\times L_1\to L_2$ after one
family of higher core bracket maps has been fixed. This is the
Sims-core form of the Jacobi--kernel argument in
\cite[Proposition~20.16]{BNV}.  
\begin{proposition}\label{p12:prop:count_leading_maps}
Assume $c\ge3$. Fix a core bracket map
$L_2\times C_{\core}\to L_3$ of breadth $\fb_2$ that occurs
for a group with fixed recorded Sims presentation data as counted in~\cref{p12:prop:sims-generating-data} and power word array as counted in~\cref{p12:lem:power-ledger}. Among the induced graded Lie
brackets with this core bracket map, the number of possible
restrictions $L_1\times L_1\to L_2$ is at most
$p^{\binom{r_1}{2}s_2-\binom{r_1-s}{2}\fb_2}$.
\end{proposition}
\begin{proof}
We first count the restrictions $L_1\times C_{\core}\to L_2$.
It suffices to choose the values on generator pairs
$(\bar g_{1,j},\bar g_{1,k})$ with $1\le j<k\le r_1$ and $j\le s$.
There are $\binom{r_1}{2}-\binom{r_1-s}{2}$ such pairs, and each
value has at most $|L_2|=p^{s_2}$ choices. These values determine the restriction
$L_1\times C_{\core}\to L_2$ by alternation and additivity.

Fix one such restriction arising from a group with the fixed
recorded Sims presentation data and power word array.
Together with the fixed map $L_2\times C_{\core}\to L_3$,
it determines $L_2\times L_1\to L_3$ by
\cref{p12:lem:relative-graded}\textup{(i)}.
It remains to choose the values on
$(\bar g_{1,j},\bar g_{1,k})$ with $s<j<k\le r_1$;
we call these \emph{outside--outside} pairs.
Choose a maximizing element $h_2^\ast\in C_{\core}$ as in
\eqref{p12:eq:core-breadth}. For each outside--outside pair, the Jacobi
identity~\eqref{p12:eq:graded-jacobi-explicit} gives
\begin{equation}\label{p12:eq:elementary-jacobi-model}
 \ad(h_2^\ast)([\bar g_{1,j},\bar g_{1,k}]_L)
 =-[[\bar g_{1,k},h_2^\ast]_L,\bar g_{1,j}]_L-[[h_2^\ast,\bar g_{1,j}]_L,\bar g_{1,k}]_L.
\end{equation}
The inner brackets on the right are fixed by
$L_1\times C_{\core}\to L_2$, and the outer brackets are
fixed by $L_2\times L_1\to L_3$. Thus the right-hand side is fixed.
When nonempty, the set of attainable values of
$[\bar g_{1,j},\bar g_{1,k}]_L$ is therefore contained in a single
coset of $\ker\ad(h_2^\ast)$. By~\eqref{p12:eq:adjoint-kernel-size},
each outside--outside pair has at most $p^{s_2-\fb_2}$ possible values.
Multiplying this conditional bound over the $\binom{r_1-s}{2}$
outside--outside pairs and then by the number of core restrictions
gives at most
\[
 p^{(\binom{r_1}{2}-\binom{r_1-s}{2})s_2}
 p^{\binom{r_1-s}{2}(s_2-\fb_2)}
 =p^{\binom{r_1}{2}s_2-\binom{r_1-s}{2}\fb_2}
\]
possible restrictions $L_1\times L_1\to L_2$.
Further compatibility conditions can only reduce this count.
\end{proof}

When $c=2$, only $L_1\times L_1\to L_2$ remains to be chosen, and
the trivial bound $p^{\binom{r_1}{2}s_2}$ gives the same estimate
since $\fb_2=0$. For $c\ge3$, by \cref{p12:lem:relative-graded}\textup{(iii)}, the higher core bracket
maps together with $L_1\times L_1\to L_2$ determine the entire
graded Lie bracket. Multiplying the two bounds in~\cref{p12:prop:core_map_fix_breadth,p12:prop:count_leading_maps} therefore gives, in both cases,
\begin{equation}\label{p12:eq:lie-bracket-count}
 \#\{\text{possible graded Lie brackets with breadth }\boldsymbol{\fb}\}
 \le
 p^{\sum_{i=2}^{c-1}s \fb_i(r_i+s_{i+1}-\fb_i+1)
      +\binom{r_1}{2}s_2-\binom{r_1-s}{2}\fb_2}.
\end{equation}

\begin{remark}
We can also use the leading alternating-map count
$N_p(r_1,r_2,s)$ to bound the number of possible restrictions $L_1\times L_1\to L_2$. Fix the higher core bracket maps and impose the
$r_2$ defining values
$[\bar g_{1,f(j)},\bar g_{1,k}]_L=\bar g_{2,j}$, where
$d_{k-1}<j\le d_k$. In the proof
of~\cref{p12:prop:count_leading_maps}, these distinct core--core
pairs were counted freely. Omitting their prescribed values gives
the leading-layer exponent
$\bigl(\binom{r_1}{2}-r_2\bigr)s_2
-\binom{r_1-s}{2}\fb_2$.

Alternatively, recall from~\cref{p12:rem:bracket-mod-p}
that the bracket $L_1\times L_1\to L_2$ induces the surjective
linear map
$\beta_G:\bigwedge^2(L_1/pL_1)\to L_2/pL_2$.
This map has Sims rank $s$. There are at most
$p^{N_p(r_1,r_2,s)}$ such maps.
For each such map, the values on the remaining
$\binom{r_1}{2}-r_2$ pairs have at most
$|pL_2|=p^{s_2-r_2}$ lifts each; the $r_2$ defining values require
no lift choices. Thus, for each fixed family of higher core maps,
the base-$p$ logarithm of the number of possible restrictions $L_1\times L_1\to L_2$ is at most
\begin{equation}\label{p12:eq:leading-bracket-minimum}
 \displaystyle\min\left\{
 \left(\binom{r_1}{2}-r_2\right)s_2
       -\binom{r_1-s}{2}\fb_2,\quad
 N_p(r_1,r_2,s)+\left(\binom{r_1}{2}-r_2\right)(s_2-r_2)
 \right\}.
\end{equation}
When $c=2$, the family of higher core maps is empty and
$\fb_2=0$, so the same bound holds. The minimum uses two
alternative counts of the same leading-layer values.  
\end{remark}

\section{Counting deeper commutator coefficients}
\label{p12:sec:filtered-HW}
For fixed recorded Sims presentation data and a compatible power
word array, \cref{p12:subsec:lie-bracket-count} counts the induced
graded Lie brackets on $L=\bigoplus_{i=1}^c L_i$.
Each such bracket determines the leading commutator coefficients
$(c(i,j,k,i+1,v))$. For each fixed bracket, we now count the distinct
assignments of the deeper coefficients $(c(i,j,k,u,v))$ with
$i+2\le u\le c$ arising from groups with these fixed recorded Sims presentation data
and power word array. 
By a \emph{selected relation}, we mean
a commutator relation whose coefficients are recorded in the index
ranges of~\cref{thm:sims_presentation_determines_isomorphism_classes_}.

Recall the exponent saving using the breadth sequence when counting the Lie bracket (cf.~\cref{p12:prop:count_leading_maps}):
After fixing the higher core bracket maps and the bracket values involving the core,
the Jacobi identity restricts the remaining first-layer bracket
values to cosets of adjoint kernels, whose sizes are controlled by
the breadths.

For the deeper coefficients, we shall apply a similar strategy.
Fixing one graded Lie bracket on $L=\bigoplus_i L_i$, we count the deeper components
successively, conditional on the components already fixed and the
current components involving the core. Comparing groups with these
common data gives a linear identity for their commutator differences,
which restricts the remaining components to cosets of adjoint
kernels controlled by the same breadth sequence. The resulting bound
is uniform over graded brackets with these breadths. We therefore
sum over the brackets counted in
\cref{p12:subsec:lie-bracket-count} and omit the coefficient positions
fixed by the defining identities for the weight-two generators to
obtain the following full commutator bound.

\begin{theorem}[Fixed-breadth relation bound]
\label{p12:prop:total-relation-bound}
For nonabelian groups of order $p^m$, fix a choice of recorded Sims presentation data counted
in~\cref{p12:prop:sims-generating-data} and a
compatible power word array $(b(i,j,u,v))$ counted
in~\cref{p12:lem:power-ledger}. Also fix a breadth sequence
$\boldsymbol{\fb}=(\fb_i)_{i=2}^{c}$ as in \cref{p12:eq:core-breadth} that occurs for the associated graded Lie
bracket of a group with these data.
Then the base-$p$ logarithm of the number of distinct commutator
coefficient arrays arising from groups equipped with ordered Sims
generating sets realizing the fixed recorded Sims presentation data and power word
array, and whose induced graded Lie brackets have breadth sequence
$\boldsymbol{\fb}$, is at most
\begin{equation}\label{p12:eq:total-relation-bound}
 \left(\binom{r_1}{2}-r_2\right)m_2
   +\sum_{i=2}^{c-2}sr_i m_{i+2}+\sum_{i=2}^{c-1}\fb_i
  \left(s(r_i+s_{i+1}-\fb_i+1)-\binom{r_1-s}{2}\right).
\end{equation}
\end{theorem}
The rest of this section will prove~\cref{p12:prop:total-relation-bound}. We first compare commutators
in two groups with the recorded Sims presentation data and power word array fixed in~\cref{p12:prop:total-relation-bound} and inducing
the same graded Lie bracket. Assuming agreement at smaller excesses,
we then show that their commutator differences at the current excess form
an alternating biadditive map satisfying an adjoint cocycle identity.

\subsection{Commutator comparison in two groups with fixed Lie bracket}
\label{p12:subsec:comparison_of_commutators_across_valid_completions}
Assume $c\ge3$ and fix one of the graded Lie brackets counted in
\cref{p12:subsec:lie-bracket-count}. To prepare for counting the deeper
components of the selected relations, we compare two groups with the same recorded Sims presentation data and power word array fixed in~\cref{p12:prop:total-relation-bound}, whose associated graded brackets agree under
the identification of their labeled homogeneous generators.

Equation~\eqref{p12:eq:commutator-weight-blocks} writes a selected relation
$[g_{i,j},g_{1,k}]$ as the product of its Sims normal-form blocks
$\omega_u(i,j,k)$ for $i+1\le u\le c$. Since
$[g_{i,j},g_{1,k}]\in G_{i+1}$, the blocks of weight less than
$i+1$ are trivial. For each $u$, the coset
$\omega_u(i,j,k)G_{u+1}\in L_u$ is its $L_u$-component, as in
\eqref{p12:eq:L_u_component_relation}. These components form the
Sims normal-form layer decomposition of $[g_{i,j},g_{1,k}]$ as in~\eqref{eq:sims layer decomposition}, and depend on the chosen Sims generating set. We call $L_{i+1}$ the \emph{designated leading layer} of $[g_{i,j},g_{1,k}]$, and call the $L_{i+1+\xi}$-component of $[g_{i,j},g_{1,k}]$ the \emph{excess $\xi$} component for $0\le \xi\le c-i-1$. In particular, the component in the designated leading layer has excess $0$, and the positive-excess components form the \emph{tail}.  

More precisely, for a selected relation $[g_{i,j},g_{1,k}]$ with indices as
in~\cref{thm:sims_presentation_determines_isomorphism_classes_} and any
$0\le \xi\le c-i-1$, its excess-$\xi$ component as in~\eqref{p12:eq:L_u_component_relation}, that is, its
$L_{i+1+\xi}$-component, is
\begin{equation}\label{p12:eq:selected-excess-component}
 \omega_{i+1+\xi}(i,j,k)G_{i+2+\xi}
 =\sum_{v=1}^{r_{i+1+\xi}}
   c(i,j,k,i+1+\xi,v)\bar g_{i+1+\xi,v}
 \in L_{i+1+\xi}.
\end{equation}
The excess-$0$ component is
$[\bar g_{i,j},\bar g_{1,k}]_L$, so it is fixed by the chosen
graded bracket. We count the positive-excess components successively,
conditional on those of smaller excess. The comparison results first
assume agreement below the current excess for all corresponding
Sims-generator commutators;
\cref{p12:prop:selected-agreement} later shows that agreement for
the selected relations suffices to ensure this hypothesis.

\phantomsection\label{p12:par:comparison-data}
In the rest of this section, let $G^{(0)}$ and $G^{(1)}$ be two groups of order $p^m$ equipped with
ordered Sims generating sets
$(g_{i,j}^{(\nu)})_{i\in[c],j\in[r_i]}$ for $\nu\in\{0,1\}$.
Both generating sets realize the fixed integers
$c$, $s$, $(r_i)$ and $(d_k)$, and use the same auxiliary function
$f$ and the same index pairs defining the higher-weight generators.
With respect to these generating sets, the two groups have the
same prescribed relative order array $(a(i,j))$ and complete
power word array $(b(i,j,u,v))$.

Write
$(G_i^{(\nu)})_{i\ge1}$ for the lower central series of $G^{(\nu)}$.
By~\eqref{p12:eq:fibre-power-to-layer}, the common relative order array and same-layer power word subarray identify the associated graded abelian groups of $G^{(0)}$ and
$G^{(1)}$ with $L=\bigoplus_{i=1}^c L_i$
under the identifications
$g_{i,j}^{(\nu)}G_{i+1}^{(\nu)}\mapsto\bar g_{i,j}$.
Denote the graded Lie bracket fixed at the start of~\cref{p12:subsec:comparison_of_commutators_across_valid_completions} by $[-,-]_L$,
and assume that both groups induce this bracket under these
identifications.

Every $g^{(\nu)}\in G^{(\nu)}$ has a unique Sims normal form as in~\eqref{p12:eq:sims_normal_form}
\[
 g^{(\nu)}
 =\prod_{i=1}^c\prod_{j=1}^{r_i}
   (g_{i,j}^{(\nu)})^{e^{(\nu)}(i,j)},
\]
where both products are taken in increasing index order and the coordinates satisfy $0\le e^{(\nu)}(i,j)<p^{a(i,j)}$. The coordinate index set and ranges are common to the two groups.
For each $i\in[c]$ and $\nu\in\{0,1\}$, the weight-$i$ Sims
normal forms identify $G_i^{(\nu)}/G_{i+1}^{(\nu)}$ with $L_i$ by the identification
\begin{equation}\label{p12:eq:common-layer-identification}
 G_i^{(\nu)}/G_{i+1}^{(\nu)}\longrightarrow L_i,~~\left(\prod_{j=1}^{r_i}(g_{i,j}^{(\nu)})^{e^{(\nu)}(i,j)}\right)
 G_{i+1}^{(\nu)}
 \longmapsto\sum_{j=1}^{r_i}e^{(\nu)}(i,j)\bar g_{i,j}.
\end{equation}
These maps preserve addition by the relations
in~\eqref{p12:eq:fibre-power-to-layer} and are bijective by
uniqueness of the Sims coordinates. Hence they are isomorphisms
of abelian groups.

For $\nu\in\{0,1\}$ and $g\in G^{(\nu)}$,
write $\pi_u(g)\in L_u$ for the $L_u$-component obtained by extracting its weight-$u$ block, as in~\eqref{p12:eq:L_u_component_relation}. Under~\eqref{p12:eq:common-layer-identification},
\begin{equation}\label{p12:eq:pi_u_g}
 \pi_u(g)
 :=\left(\prod_{j=1}^{r_u}(g_{u,j}^{(\nu)})^{e^{(\nu)}(u,j)}\right)
   G_{u+1}^{(\nu)}=\sum_{j=1}^{r_u}e^{(\nu)}(u,j)\bar g_{u,j}\in L_u.
\end{equation}
The map $\pi_u$ is a coordinate function determined by the chosen
Sims generators. Its restriction to $G_u^{(\nu)}$ is precisely the quotient
homomorphism to $L_u$ under~\eqref{p12:eq:common-layer-identification}.

Recall from~\eqref{p12:eq:sims-lift} that, for $x\in L_u$
with Sims coordinates $e(u,v)$ satisfying
$0\le e(u,v)<p^{a(u,v)}$, its Sims lift in $G^{(\nu)}$ will be denoted by
\begin{equation}\label{p12:eq:comparison-sims-lifts}
 \widetilde{x}^{(\nu)}
 =\prod_{v=1}^{r_u}(g_{u,v}^{(\nu)})^{e(u,v)}
 \in G_u^{(\nu)}.
\end{equation}
The two lifts use the same coordinates and have trivial blocks
in every weight other than $u$.

For $\xi\ge1$ and $i\in[c]$, we call a pair 
$(g^{(0)},g^{(1)})$
\emph{matched below excess $\xi$ in degree $i$} if
\begin{equation}\label{eq:matched below excess}
  g^{(0)}\in G_i^{(0)},~~g^{(1)}\in G_i^{(1)},~~\text{and}~~\pi_{i+e}(g^{(0)})
 =\pi_{i+e}(g^{(1)})~~\forall\ 0\le e<\xi~\text{with}~i+e\le c
\end{equation}
When $i+\xi\le c$, the first component in which a matched pair may differ is its $L_{i+\xi}$-component. We call
\begin{equation}\label{p12:eq:first-possible-difference}
 \partial_{i,\xi}(g^{(0)},g^{(1)})
 :=\pi_{i+\xi}(g^{(1)})
   -\pi_{i+\xi}(g^{(0)})\in L_{i+\xi}
\end{equation}
the \emph{first possible difference} of the matched pair
$(g^{(0)},g^{(1)})$.

The goal of this subsection is to prove that taking products,
inverses and commutators preserves the required matching. These
comparisons will be used in the next subsection to compute
commutator differences using Sims lifts. The following result
establishes quotient
isomorphisms that preserve coordinates and are compatible
with conjugation.
\begin{theorem}
\label{p12:prop:sims-collection-rules}
Let $G^{(0)}$ and $G^{(1)}$ be two groups of order $p^m$ equipped with
ordered Sims generating sets realizing the same fixed recorded Sims
presentation data from~\cref{p12:prop:sims-generating-data} and having
the same full power word array $(b(i,j,u,v))$ from
\cref{p12:lem:power-ledger}.

Fix $\xi\in[c-1]$, and assume that the commutators of corresponding
pairs of Sims generators agree at every excess less than $\xi$.
That is, for all $i_1,i_2\in[c]$, $j_1\in[r_{i_1}]$,
$j_2\in[r_{i_2}]$, and $0\le e<\xi$ with $i_1+i_2+e\le c$,
we have
\[
\pi_{i_1+i_2+e}
\bigl([g_{i_1,j_1}^{(0)},g_{i_2,j_2}^{(0)}]\bigr)
=
\pi_{i_1+i_2+e}
\bigl([g_{i_1,j_1}^{(1)},g_{i_2,j_2}^{(1)}]\bigr).
\]

Then for any $k,u$ with $1\le k\le u\le c$ and $u\le k+\xi$,
the following hold.
\begin{enumerate}[label=\textup{(\roman*)},leftmargin=*]
 \item There is a unique coordinate-preserving isomorphism $\varphi_{k,u}:G_k^{(0)}/G_{u+1}^{(0)}
  \rightarrow G_k^{(1)}/G_{u+1}^{(1)}$.
 That is, for any choice of coordinates $0\le e(\ell,v)<p^{a(\ell,v)}$ with
 $k\le \ell\le u$ and $v\in[r_\ell]$,
 \begin{equation}\label{p12:eq:varphi_action}
  \varphi_{k,u}\left(
    \left(\prod_{\ell=k}^{u}\prod_{v=1}^{r_\ell}
      (g_{\ell,v}^{(0)})^{e(\ell,v)}\right)G_{u+1}^{(0)}\right)
  =\left(\prod_{\ell=k}^{u}\prod_{v=1}^{r_\ell}
      (g_{\ell,v}^{(1)})^{e(\ell,v)}\right)G_{u+1}^{(1)}.
 \end{equation}

 \item For any $t\in[c]$, $v\in[r_t]$, and elements
 $g^{(0)}\in G_k^{(0)}$ and $g^{(1)}\in G_k^{(1)}$, suppose that $\varphi_{k,u}\bigl(g^{(0)}G_{u+1}^{(0)}\bigr)
  =g^{(1)}G_{u+1}^{(1)}$. Then
 \begin{equation}\label{p12:eq:sims-conjugation-compatibility}
  \varphi_{k,u}\bigl((g_{t,v}^{(0)})^{-1}g^{(0)}g_{t,v}^{(0)}
  G_{u+1}^{(0)}\bigr)
  =(g_{t,v}^{(1)})^{-1}g^{(1)}g_{t,v}^{(1)}G_{u+1}^{(1)}.
 \end{equation}
\end{enumerate}
\end{theorem}

\begin{proof}
Fix $k,u$ in the stated range. We first prove \textup{(i)}. 
We construct $\varphi_{k,u}$ by adding the generator cosets in decreasing
weight $i=u,\ldots,k$ and, within each weight, in increasing second
index $j\in[r_i]$. More precisely, for $k\le i\le u$, $0\le j\le r_i$ and
$\nu\in\{0,1\}$, let
\[
 K_{i,j}^{(\nu)}
 =\bigl\langle
 g_{\ell,v}^{(\nu)}G_{u+1}^{(\nu)}:i<\ell\le u,\ v\in[r_\ell];\
 g_{i,v}^{(\nu)}G_{u+1}^{(\nu)}:1\le v\le j
 \bigr\rangle.
\]
Thus $K_{i,j}^{(\nu)}$ is the subgroup after adding these generator
cosets. In particular, $K_{u,0}^{(\nu)}=1$ and
$K_{i,r_i}^{(\nu)}=K_{i-1,0}^{(\nu)}$ for $k<i\le u$.
We start with the unique isomorphism
$\varphi^{(u,0)}:K_{u,0}^{(0)}\to K_{u,0}^{(1)}$ between the
trivial groups. When moving to the next weight, we set
$\varphi^{(i-1,0)}:=\varphi^{(i,r_i)}$ for $k<i\le u$.

At stage $(i,j)$, with $j\in[r_i]$, suppose that an isomorphism
$\varphi^{(i,j-1)}:K_{i,j-1}^{(0)}\to K_{i,j-1}^{(1)}$ has already
been constructed and sends every previously added generator coset
to its counterpart. We show how to extend it to an isomorphism $\varphi^{(i,j)}:K_{i,j}^{(0)}\longrightarrow K_{i,j}^{(1)}$.
Note that all generators of $G_{i+1}^{(\nu)}/G_{u+1}^{(\nu)}$
have already been added, and all generators of $K_{i,j-1}^{(\nu)}$
have weight at least $i$. Hence
\[
 [K_{i,j-1}^{(\nu)},g_{i,j}^{(\nu)}G_{u+1}^{(\nu)}]
 \le G_{i+1}^{(\nu)}/G_{u+1}^{(\nu)}
 \le K_{i,j-1}^{(\nu)}.
\]
Thus conjugation by $g_{i,j}^{(\nu)}G_{u+1}^{(\nu)}$ maps
$K_{i,j-1}^{(\nu)}$ injectively into itself. Since $K_{i,j-1}^{(\nu)}$ is finite, its
image equals $K_{i,j-1}^{(\nu)}$, so $K_{i,j-1}^{(\nu)}$ is normal in
$K_{i,j}^{(\nu)}$ for $\nu=0,1$.

By the relative order definition~\eqref{eq: relative order of gij}
and the uniqueness of Sims normal form,
$p^{a(i,j)}$ is the exact order of $g_{i,j}^{(\nu)}G_{u+1}^{(\nu)}$
modulo $K_{i,j-1}^{(\nu)}$. Consequently,
\[
 K_{i,j}^{(\nu)}
 =\bigsqcup_{e=0}^{p^{a(i,j)}-1}
       (g_{i,j}^{(\nu)}G_{u+1}^{(\nu)})^e K_{i,j-1}^{(\nu)}.
\]
Therefore, every coset $xG_{u+1}^{(\nu)}$ in
$K_{i,j}^{(\nu)}$
has a unique expression
$xG_{u+1}^{(\nu)}=(g_{i,j}^{(\nu)}G_{u+1}^{(\nu)})^e z$,
with $0\le e<p^{a(i,j)}$ and $z\in K_{i,j-1}^{(\nu)}$.
Meanwhile, the fixed power word~\eqref{p12:eq:power_word_normal_form} gives
\begin{equation}\label{p12:eq:power_coset}
  (g_{i,j}^{(\nu)}G_{u+1}^{(\nu)})^{p^{a(i,j)}}
  =\left(\prod_{v=1}^{j-1}
       (g_{i,v}^{(\nu)})^{b(i,j,i,v)}\right)
   \left(\prod_{\ell=i+1}^{u}\prod_{v=1}^{r_\ell}
       (g_{\ell,v}^{(\nu)})^{b(i,j,\ell,v)}\right)G_{u+1}^{(\nu)}
       \in K_{i,j-1}^{(\nu)}.
\end{equation}
All generator cosets on the right have already been added, and
the coefficients $b$ are the same in both groups. Hence
\[
 \varphi^{(i,j-1)}\bigl((g_{i,j}^{(0)}G_{u+1}^{(0)})^{p^{a(i,j)}}\bigr)
 =(g_{i,j}^{(1)}G_{u+1}^{(1)})^{p^{a(i,j)}}.
\]

For the conjugation comparison, observe that the relevant generator
commutators have matching coordinates through weight $u$.
For $i_1,i_2\ge k$, $j_1\in[r_{i_1}]$, $j_2\in[r_{i_2}]$, and
$i_1+i_2\le \ell\le u$, we have $\pi_\ell([g_{i_1,j_1}^{(0)},g_{i_2,j_2}^{(0)}])
 =\pi_\ell([g_{i_1,j_1}^{(1)},g_{i_2,j_2}^{(1)}])$ since $0\le \ell-i_1-i_2\le u-2k\le\xi-k<\xi$.

Let $g_{\ell,v'}^{(\nu)}G_{u+1}^{(\nu)}$
be a generator of $K_{i,j-1}^{(\nu)}$ already added.
If $\ell+i>u$, the coset of $[g_{\ell,v'}^{(\nu)},g_{i,j}^{(\nu)}]$
is trivial. Otherwise, write its Sims normal form as
\[
 [g_{\ell,v'}^{(\nu)},g_{i,j}^{(\nu)}]G_{u+1}^{(\nu)}
 =\left(\prod_{w=\ell+i}^{u}\prod_{v=1}^{r_w}
       (g_{w,v}^{(\nu)})^{e^{(\nu)}(w,v)}\right)G_{u+1}^{(\nu)}.
\]
By this observation, the commutator components agree through weight $u$.
Uniqueness of the Sims coordinates therefore gives
$e^{(0)}(w,v)=e^{(1)}(w,v)$.
Since $w\ge \ell+i>i$, every generator coset in this word has
already been added. Because these coordinates agree, the induction
hypothesis for the already constructed map gives, in either case,
\[
 \varphi^{(i,j-1)}\bigl([g_{\ell,v'}^{(0)},g_{i,j}^{(0)}]G_{u+1}^{(0)}\bigr)
 =[g_{\ell,v'}^{(1)},g_{i,j}^{(1)}]G_{u+1}^{(1)}.
\]
Using $y^{-1}xy=x[x,y]$, we obtain
\[
\varphi^{(i,j-1)}\bigl((g_{i,j}^{(0)})^{-1}g_{\ell,v'}^{(0)}g_{i,j}^{(0)}G_{u+1}^{(0)}\bigr)=\varphi^{(i,j-1)}\bigl(g_{\ell,v'}^{(0)}[g_{\ell,v'}^{(0)},g_{i,j}^{(0)}]G_{u+1}^{(0)}\bigr)=(g_{i,j}^{(1)})^{-1}g_{\ell,v'}^{(1)}g_{i,j}^{(1)}G_{u+1}^{(1)},
\]
where the homomorphism property of $\varphi^{(i,j-1)}$ applies because the
generator coset and the commutator coset both belong to $K_{i,j-1}^{(0)}$.
The two conjugation maps are homomorphisms and agree on every
generating coset of $K_{i,j-1}^{(0)}$. Hence, for any $z\in K_{i,j-1}^{(0)}$,
\[
 \varphi^{(i,j-1)}\bigl((g_{i,j}^{(0)}G_{u+1}^{(0)})^{-1}
       z(g_{i,j}^{(0)}G_{u+1}^{(0)})\bigr)
 =(g_{i,j}^{(1)}G_{u+1}^{(1)})^{-1}
       \varphi^{(i,j-1)}(z)(g_{i,j}^{(1)}G_{u+1}^{(1)}).
\]
Iterating this identity gives, for $z\in K_{i,j-1}^{(0)}$ and any integer
$e_2\ge0$,
\[
 \varphi^{(i,j-1)}\bigl((g_{i,j}^{(0)}G_{u+1}^{(0)})^{-e_2}
       z(g_{i,j}^{(0)}G_{u+1}^{(0)})^{e_2}\bigr)
 =(g_{i,j}^{(1)}G_{u+1}^{(1)})^{-e_2}
       \varphi^{(i,j-1)}(z)(g_{i,j}^{(1)}G_{u+1}^{(1)})^{e_2}.
\]

Define the extension by
\[
 \varphi^{(i,j)}\bigl((g_{i,j}^{(0)}G_{u+1}^{(0)})^e z\bigr)
 =(g_{i,j}^{(1)}G_{u+1}^{(1)})^e\varphi^{(i,j-1)}(z)
\]
for any $z\in K_{i,j-1}^{(0)}$ and $0\le e<p^{a(i,j)}$.
The unique expressions of elements in $K_{i,j}^{(\nu)}$ show that $\varphi^{(i,j)}$ is well defined
and bijective. Taking $e=0$ shows that it extends $\varphi^{(i,j-1)}$;
taking $e=1$, $z=1$ shows that it sends the new generator coset
to its counterpart.

It remains to verify that $\varphi^{(i,j)}$ preserves multiplication. Choose
$0\le e_1,e_2<p^{a(i,j)}$ and $z_1^{(0)},z_2^{(0)}\in K_{i,j-1}^{(0)}$,
and denote $z_a^{(1)}=\varphi^{(i,j-1)}(z_a^{(0)})\in K_{i,j-1}^{(1)}$ for $a=1,2$.
Write $e_1+e_2=dp^{a(i,j)}+e$, where $0\le e<p^{a(i,j)}$.
For $\nu=0,1$,
\begin{equation}\label{p12:eq:multiplication_hat_varphi}
 \begin{split}
 &\bigl((g_{i,j}^{(\nu)}G_{u+1}^{(\nu)})^{e_1}z_1^{(\nu)}\bigr)
  \bigl((g_{i,j}^{(\nu)}G_{u+1}^{(\nu)})^{e_2}z_2^{(\nu)}\bigr)\\
 ={}&(g_{i,j}^{(\nu)}G_{u+1}^{(\nu)})^{e_1+e_2}
       \bigl((g_{i,j}^{(\nu)}G_{u+1}^{(\nu)})^{-e_2}
             z_1^{(\nu)}(g_{i,j}^{(\nu)}G_{u+1}^{(\nu)})^{e_2}\bigr)z_2^{(\nu)}\\
 ={}&(g_{i,j}^{(\nu)}G_{u+1}^{(\nu)})^e
       \Bigl((g_{i,j}^{(\nu)}G_{u+1}^{(\nu)})^{dp^{a(i,j)}}
       \bigl((g_{i,j}^{(\nu)}G_{u+1}^{(\nu)})^{-e_2}
             z_1^{(\nu)}(g_{i,j}^{(\nu)}G_{u+1}^{(\nu)})^{e_2}\bigr)z_2^{(\nu)}\Bigr).
 \end{split}
\end{equation}
Note that 
\[
  (g_{i,j}^{(\nu)}G_{u+1}^{(\nu)})^{dp^{a(i,j)}}
       \bigl((g_{i,j}^{(\nu)}G_{u+1}^{(\nu)})^{-e_2}
             z_1^{(\nu)}(g_{i,j}^{(\nu)}G_{u+1}^{(\nu)})^{e_2}\bigr)z_2^{(\nu)}\in K_{i,j-1}^{(\nu)},
\]
since the power term is in $K_{i,j-1}^{(\nu)}$ by~\eqref{p12:eq:power_coset} and conjugation by $g_{i,j}^{(\nu)}G_{u+1}^{(\nu)}$ preserves $K_{i,j-1}^{(\nu)}$.
Meanwhile, $\varphi^{(i,j-1)}$ sends the displayed $\nu=0$ element to its
$\nu=1$ counterpart. Applying the definition of
$\varphi^{(i,j)}$ therefore gives
\[
\varphi^{(i,j)}\Bigl(
       \bigl((g_{i,j}^{(0)}G_{u+1}^{(0)})^{e_1}z_1^{(0)}\bigr)
       \bigl((g_{i,j}^{(0)}G_{u+1}^{(0)})^{e_2}z_2^{(0)}\bigr)\Bigr)=\bigl((g_{i,j}^{(1)}G_{u+1}^{(1)})^{e_1}z_1^{(1)}\bigr)
        \bigl((g_{i,j}^{(1)}G_{u+1}^{(1)})^{e_2}z_2^{(1)}\bigr).
\]
Thus, $\varphi^{(i,j)}$ is an isomorphism from $K_{i,j}^{(0)}$
to $K_{i,j}^{(1)}$ extending $\varphi^{(i,j-1)}$ and sending the
new generator coset to its counterpart. After all generator cosets
of weights $k,\ldots,u$ have been added, we have
$K_{k,r_k}^{(\nu)}=G_k^{(\nu)}/G_{u+1}^{(\nu)}$. Hence $\varphi_{k,u}:=\varphi^{(k,r_k)}$
is an isomorphism between the required quotients.

For $\nu\in\{0,1\}$, recall that
for $g\in G_k^{(\nu)}$ with coordinates
$0\le e^{(\nu)}(\ell,v)<p^{a(\ell,v)}$ as in~\eqref{p12:eq:sims_normal_form}, the coset $gG_{u+1}^{(\nu)}$ has a unique Sims normal form
\[
  gG_{u+1}^{(\nu)}=\left(\prod_{\ell=k}^{u}\prod_{v=1}^{r_\ell}(g_{\ell,v}^{(\nu)})^{e^{(\nu)}(\ell,v)}\right)G_{u+1}^{(\nu)}.
\]
Since $\varphi_{k,u}$ sends each generator coset to its counterpart,
applying it to this normal form gives~\eqref{p12:eq:varphi_action}. Conversely,
\eqref{p12:eq:varphi_action} determines the map on every coset and hence
proves uniqueness. This proves~\textup{(i)}.

To prove \textup{(ii)}, fix $t\in[c]$ and $v\in[r_t]$.
Both $G_k^{(\nu)}$ and $G_{u+1}^{(\nu)}$ are normal in $G^{(\nu)}$,
so conjugation by $g_{t,v}^{(\nu)}$ induces an automorphism of
$G_k^{(\nu)}/G_{u+1}^{(\nu)}$.
It is enough to show that applying $\varphi_{k,u}$ after conjugation in the
first quotient agrees on every generating coset with applying
conjugation after $\varphi_{k,u}$ in the second quotient.

Let $k\le i\le u$ and $j\in[r_i]$. If $i+t>u$, then the coset of
$[g_{i,j}^{(\nu)},g_{t,v}^{(\nu)}]$ is trivial. If $i+t\le u$, then
for $i+t\le w\le u$, $0\le w-i-t\le u-k-1<\xi$.
The hypothesis therefore gives equality of all components of the two
commutators through weight $u$. Hence, in either case,
coordinate preservation in~\textup{(i)} gives
\[
 \varphi_{k,u}\bigl([g_{i,j}^{(0)},g_{t,v}^{(0)}]G_{u+1}^{(0)}\bigr)
 =[g_{i,j}^{(1)},g_{t,v}^{(1)}]G_{u+1}^{(1)}.
\]
Since $y^{-1}xy=x[x,y]$ and $\varphi_{k,u}$ sends each generating coset
to its counterpart, it follows that
\[
\varphi_{k,u}\bigl((g_{t,v}^{(0)})^{-1}g_{i,j}^{(0)}
                 g_{t,v}^{(0)}G_{u+1}^{(0)}\bigr)=(g_{t,v}^{(1)})^{-1}g_{i,j}^{(1)}
                 g_{t,v}^{(1)}G_{u+1}^{(1)}.
\]
The two compositions are homomorphisms and agree on every generating
coset, so they agree on the whole quotient. This proves
\eqref{p12:eq:sims-conjugation-compatibility}.
\end{proof}

\paragraph{Corresponding Sims words and deeper remainders.}
Under the hypotheses of \cref{p12:prop:sims-collection-rules}, consider
$k\in[c]$ and a pair $(g^{(0)},g^{(1)})$ matched below excess
$\xi$ in degree $k$.
Let $e^{(0)}(t,v)$ be the coordinates of $g^{(0)}$ in its Sims normal form, and define
\[
 z:=\left(\prod_{t=k}^{c}\prod_{v=1}^{r_t}
       (g_{t,v}^{(1)})^{e^{(0)}(t,v)}\right)^{-1}g^{(1)}\in G^{(1)}.
\]
Then
\begin{equation}\label{p12:eq:matched-sims-word-decomposition}
 g^{(0)}=\prod_{t=k}^{c}\prod_{v=1}^{r_t}
       (g_{t,v}^{(0)})^{e^{(0)}(t,v)}~\text{and}~g^{(1)}=\left(\prod_{t=k}^{c}\prod_{v=1}^{r_t}
       (g_{t,v}^{(1)})^{e^{(0)}(t,v)}\right)z.
\end{equation}
Fix $u$ with $k\le u\le c$ and $u<k+\xi$.
By~\cref{p12:prop:sims-collection-rules} \textup{(i)}, there is a
coordinate-preserving isomorphism
$\varphi_{k,u}:G_k^{(0)}/G_{u+1}^{(0)}\to G_k^{(1)}/G_{u+1}^{(1)}$.
Factors of weight greater than $u$ vanish in these quotients.
By~\eqref{p12:eq:varphi_action} and the matching assumption on
$(g^{(0)},g^{(1)})$, we have
\[
 \left(\prod_{t=k}^{c}\prod_{v=1}^{r_t}(g_{t,v}^{(1)})^{e^{(0)}(t,v)}\right)G_{u+1}^{(1)}=\varphi_{k,u}\bigl(g^{(0)}G_{u+1}^{(0)}\bigr)=g^{(1)}G_{u+1}^{(1)}.
\]
The definition of $z$ therefore gives $z\in G_{u+1}^{(1)}$ for every
$k\le u\le c$ with $u<k+\xi$. If $k+\xi\le c$, take
$u=k+\xi-1$; if $k+\xi>c$, take $u=c$ and use
$G_{c+1}^{(1)}=1$. In either case, $z\in G_{k+\xi}^{(1)}$.

\Cref{p12:prop:sims-collection-rules} shows that products and inverses of
matched pairs remain matched and determines their first possible differences.
\begin{proposition}\label{p12:cor:sims_collection_multiplication_inverse}
Let $G^{(0)}$, $G^{(1)}$ and $\xi\in[c-1]$ satisfy the hypotheses
of~\cref{p12:prop:sims-collection-rules}; in particular, the commutators
of corresponding Sims generators are matched below excess $\xi$.
For $i\in[c]$, $a=1,2$ and $\nu\in\{0,1\}$, let
$g_a^{(\nu)}\in G_i^{(\nu)}$. Suppose that each pair
$(g_a^{(0)},g_a^{(1)})$ is matched below excess $\xi$ in degree $i$.
Then the componentwise product $\bigl(g_1^{(0)}g_2^{(0)},g_1^{(1)}g_2^{(1)}\bigr)$ and inverse pairs $\bigl((g_1^{(0)})^{-1},(g_1^{(1)})^{-1}\bigr)$
are matched below excess $\xi$ in degree $i$. More precisely, $\pi_u(g_1^{(0)}g_2^{(0)})=\pi_u(g_1^{(1)}g_2^{(1)})$ and $\pi_u((g_1^{(0)})^{-1})=\pi_u((g_1^{(1)})^{-1})$
for any $i\le u\le c$ with $u<i+\xi$. 

Moreover, when $i+\xi\le c$,
\[
\begin{aligned}
 \partial_{i,\xi}\bigl(g_1^{(0)}g_2^{(0)},g_1^{(1)}g_2^{(1)}\bigr)
 &=\partial_{i,\xi}\bigl(g_1^{(0)},g_1^{(1)}\bigr)
     +\partial_{i,\xi}\bigl(g_2^{(0)},g_2^{(1)}\bigr),\\
 \partial_{i,\xi}\bigl((g_1^{(0)})^{-1},(g_1^{(1)})^{-1}\bigr)
 &=-\partial_{i,\xi}\bigl(g_1^{(0)},g_1^{(1)}\bigr).
\end{aligned}
\]
\end{proposition}
\begin{proof}
Fix $i\in[c]$.  We prove that the product
and inverse pairs are matched below excess $\xi$.
Fix $u$ with $i\le u\le c$ and $u<i+\xi$.
\Cref{p12:prop:sims-collection-rules}\textup{(i)} gives an isomorphism
$\varphi_{i,u}:G_i^{(0)}/G_{u+1}^{(0)}\to
 G_i^{(1)}/G_{u+1}^{(1)}$
that preserves coordinates as in~\eqref{p12:eq:varphi_action}.

The layer matching assumption gives agreement through
weight $u$, so $\varphi_{i,u}\bigl(g_a^{(0)}G_{u+1}^{(0)}\bigr)
 =g_a^{(1)}G_{u+1}^{(1)}$ for $a=1,2$.
Since $\varphi_{i,u}$ preserves products and inverses, we have
\[
 \varphi_{i,u}\bigl(g_1^{(0)}g_2^{(0)}G_{u+1}^{(0)}\bigr)
 =g_1^{(1)}g_2^{(1)}G_{u+1}^{(1)}~\text{and}~
 \varphi_{i,u}\bigl((g_1^{(0)})^{-1}G_{u+1}^{(0)}\bigr)
 =(g_1^{(1)})^{-1}G_{u+1}^{(1)}.
\]
By coordinate preservation of $\varphi_{i,u}$, the product and inverse pairs agree
in their $L_u$-components for every $i\le u\le c$ with $u<i+\xi$.

Suppose $i+\xi\le c$. We compute the first possible differences of the product and inverse pairs
at weight $i+\xi$. Applying~\cref{p12:prop:sims-collection-rules}\textup{(i)}
for $k=i$ and $u=i+\xi$ gives an isomorphism
$\varphi_{i,i+\xi}:G_i^{(0)}/G_{i+\xi+1}^{(0)}\to
 G_i^{(1)}/G_{i+\xi+1}^{(1)}$ that again preserves coordinates.

For $a=1,2$, define $h_a:=\varphi_{i,i+\xi}\bigl(g_a^{(0)}G_{i+\xi+1}^{(0)}\bigr)$ and $z_a:=h_a^{-1}\bigl(g_a^{(1)}G_{i+\xi+1}^{(1)}\bigr)$ in $ G_i^{(1)}/G_{i+\xi+1}^{(1)}$.
By~\eqref{p12:eq:varphi_action}, $h_a$ is the coset of the corresponding
Sims word in~\eqref{p12:eq:matched-sims-word-decomposition}.
The preceding remainder argument with $k=i$ gives
$z_a\in G_{i+\xi}^{(1)}/G_{i+\xi+1}^{(1)}$.
Canceling the common Sims prefix below weight $i+\xi$ leaves
only the quotient of their weight-$(i+\xi)$ blocks, which may differ.

Under the identification of
$G_{i+\xi}^{(1)}/G_{i+\xi+1}^{(1)}$ with the additive group
$L_{i+\xi}$, the image of $z_a$ is
\[
 \pi_{i+\xi}(g_a^{(1)})-\pi_{i+\xi}(g_a^{(0)})
 =\partial_{i,\xi}\bigl(g_a^{(0)},g_a^{(1)}\bigr).
\]

Note that $G_{i+\xi}^{(1)}/G_{i+\xi+1}^{(1)}$ is central in $G_{i}^{(1)}/G_{i+\xi+1}^{(1)}$, since
$[G_{i+\xi}^{(1)},G^{(1)}]\le G_{i+\xi+1}^{(1)}$. Thus products of elements in $G_{i+\xi}^{(1)}/G_{i+\xi+1}^{(1)}$ and $G_{i}^{(1)}/G_{i+\xi+1}^{(1)}$ commute.
We obtain
\[
 \begin{aligned}
 g_1^{(1)}g_2^{(1)}G_{i+\xi+1}^{(1)}&=h_1z_1h_2z_2
  =h_1h_2z_1z_2=\varphi_{i,i+\xi}(g_1^{(0)}g_2^{(0)}G_{i+\xi+1}^{(0)})z_1z_2,\\
 (g_1^{(1)})^{-1}G_{i+\xi+1}^{(1)}
 &=h_1^{-1}z_1^{-1}=\varphi_{i,i+\xi}((g_1^{(0)})^{-1}G_{i+\xi+1}^{(0)})z_1^{-1}.
 \end{aligned}
\]
The factors $z_1z_2$ and $z_1^{-1}$ change only the weight-$(i+\xi)$
component of the product and the inverse, respectively. Under the
additive identification with $L_{i+\xi}$, the images of $z_1z_2$ and
$z_1^{-1}$ are, respectively, the sum of the images of $z_1,z_2$ and
the negative of the image of $z_1$. Substituting these identified
images proves
the two identities for $\partial_{i,\xi}$.
\end{proof}

We next show that the commutator of a
pair matched below excess $\xi$ in degree $i$ and a pair matched below
excess $\xi$ in degree $j$ is matched below excess $\xi$ in degree
$i+j$ when $i+j\le c$; when $i+j>c$, both commutators are trivial.
\begin{proposition}[Commutators of matched pairs]
\label{p12:prop:collection-monotonicity}
Let $G^{(0)}$, $G^{(1)}$ and $\xi\in[c-1]$ satisfy the hypotheses
of~\cref{p12:prop:sims-collection-rules}; in particular, the commutators
of corresponding Sims generators agree below excess $\xi$.
For any $i,j\in[c]$, let $g_1^{(\nu)}\in G_i^{(\nu)}$
and $g_2^{(\nu)}\in G_j^{(\nu)}$ for $\nu=0,1$.
Suppose that $(g_1^{(0)},g_1^{(1)})$ is matched below excess $\xi$
in degree $i$ and $(g_2^{(0)},g_2^{(1)})$ is matched below excess
$\xi$ in degree $j$.
Then, when $i+j\le c$,
$\bigl([g_1^{(0)},g_2^{(0)}],[g_1^{(1)},g_2^{(1)}]\bigr)$ is matched below excess
$\xi$ in degree $i+j$. More precisely,
\[
 \pi_u([g_1^{(0)},g_2^{(0)}])
 =\pi_u([g_1^{(1)},g_2^{(1)}])
\]
for any $i+j\le u\le c$ with $u<i+j+\xi$. If $i+j>c$, then $[g_1^{(\nu)},g_2^{(\nu)}]=1$ for $\nu=0,1$.
\end{proposition}

\begin{proof}
If $i+j>c$, then
$[g_1^{(\nu)},g_2^{(\nu)}]\in G_{i+j}^{(\nu)}=1$ for $\nu=0,1$.
Suppose that $i+j\le c$.

Apply~\eqref{p12:eq:matched-sims-word-decomposition} first to $(g_1^{(0)},g_1^{(1)})$ with
$k=i$, and then to $(g_2^{(0)},g_2^{(1)})$ with $k=j$.
Denote the corresponding Sims words by $w_a^{(\nu)}$ and the
remainders by $z_a$ as in~\eqref{p12:eq:matched-sims-word-decomposition}, so
that $g_a^{(0)}=w_a^{(0)}$ and $g_a^{(1)}=w_a^{(1)}z_a$ for $a=1,2$.
Here $w_a^{(0)}$ and $w_a^{(1)}$ use the same generator labels
and coordinates, namely those of $g_a^{(0)}$.

Fix a required weight $u$, so $i+j\le u\le c$ and $u<i+j+\xi$.
Since $z_1\in G_{i+\xi}^{(1)}$ and $z_2\in G_{j+\xi}^{(1)}$, we have 
$[z_1,G_j^{(1)}]\le G_{i+j+\xi}^{(1)}\le G_{u+1}^{(1)}$ and $[G_i^{(1)},z_2]\le G_{i+j+\xi}^{(1)}\le G_{u+1}^{(1)}$.
Thus, modulo $G_{u+1}^{(1)}$, $z_1$ commutes with every element
of $G_j^{(1)}$, and $z_2$ commutes with every element of
$G_i^{(1)}$.  Consequently,
\[
 [g_1^{(1)},g_2^{(1)}]G_{u+1}^{(1)}
 =[w_1^{(1)}z_1,w_2^{(1)}z_2]G_{u+1}^{(1)}
 =[w_1^{(1)},w_2^{(1)}]G_{u+1}^{(1)}.
\]

Now, it remains to compare $[w_1^{(0)},w_2^{(0)}]$ and
$[w_1^{(1)},w_2^{(1)}]$ through weight $u$.
Expand each positive power in $w_1^{(\nu)}$ and $w_2^{(\nu)}$
into individual generator factors, for $\nu=0,1$.  Apply
\[
 \begin{aligned}
 [g_1g_2,g_3]=g_2^{-1}[g_1,g_3]g_2\,[g_2,g_3],~\text{and}~[g_1,g_2g_3]=[g_1,g_3]g_3^{-1}[g_1,g_2]g_3
 \end{aligned}
\]
in the same order in both groups. Then $[w_1^{(0)},w_2^{(0)}]$ and
$[w_1^{(1)},w_2^{(1)}]$ are each written as a product of conjugates of generator commutators.
The conjugating words use the same generator labels in the same order in both groups.
We therefore only need to compare the generator commutators and
preserve this comparison under conjugation and multiplication.

For this fixed $u$, apply~\cref{p12:prop:sims-collection-rules}\textup{(i)}
with $k=i+j$
to obtain a coordinate-preserving isomorphism
$\varphi_{i+j,u}:G_{i+j}^{(0)}/G_{u+1}^{(0)}\to
G_{i+j}^{(1)}/G_{u+1}^{(1)}$.
For $h_1\ge i$, $h_2\ge j$ and indices
$a\in[r_{h_1}]$, $b\in[r_{h_2}]$, whenever $h_1+h_2\le t\le u$,
we have
$0\le t-h_1-h_2\le u-i-j<\xi$.
The components below weight $h_1+h_2$ vanish.
The generator-commutator hypothesis and coordinate preservation give
\[
 \varphi_{i+j,u}\bigl([g_{h_1,a}^{(0)},g_{h_2,b}^{(0)}]
                       G_{u+1}^{(0)}\bigr)
 =[g_{h_1,a}^{(1)},g_{h_2,b}^{(1)}]G_{u+1}^{(1)}.
\]
If $h_1+h_2>u$, both commutator cosets are trivial, so the same
equality holds.

By~\cref{p12:prop:sims-collection-rules}\textup{(ii)},
$\varphi_{i+j,u}$ respects
conjugation by corresponding Sims generators, including those
of weight below $i+j$, and hence by the corresponding words
in the expansion.  It therefore carries each conjugated
generator commutator to its counterpart.  Preservation of
products gives
\[
 \varphi_{i+j,u}\bigl([w_1^{(0)},w_2^{(0)}]G_{u+1}^{(0)}\bigr)
 =[w_1^{(1)},w_2^{(1)}]G_{u+1}^{(1)}.
\]
Coordinate preservation and the removal of the remainders now give
\[
 \pi_u([g_1^{(0)},g_2^{(0)}])=\pi_u([g_1^{(1)},g_2^{(1)}])
\]
for any $u$ with $i+j\le u\le c$ and $u<i+j+\xi$, which proves the proposition.
\end{proof}

\subsection{Commutator differences as adjoint cocycles}
\phantomsection\label{p12:par:matched-pair-subgroups}
Under the hypotheses of \cref{p12:prop:collection-monotonicity}, let
$\cF_i\subseteq G_i^{(0)}\times G_i^{(1)}$ be the set of pairs matched below excess $\xi$ in degree $i$,
for $i\in[c]$, and let $\cF_i=\{(1,1)\}$ for $i\ge c+1$. By~\cref{p12:cor:sims_collection_multiplication_inverse}, these are subgroups of the direct products. 
For $i<c$, a pair in $\cF_{i+1}$ is matched below excess $\xi$
in degree $i+1$, and both entries have zero $L_i$-component. Hence it
is also matched below excess $\xi$ in degree $i$. For $i=c$, the
inclusion $\cF_{c+1}\le\cF_c$ is immediate from
$\cF_{c+1}=\{(1,1)\}$. Thus
$\cF_{i+1}\le\cF_i$ for every $i\ge1$.

If $(g^{(0)},g^{(1)})\in\cF_i$ and
$(h^{(0)},h^{(1)})\in\cF_j$, their commutator $([g^{(0)},h^{(0)}],[g^{(1)},h^{(1)}])$ is matched below
excess $\xi$ in degree $i+j$ when $i+j\le c$, and it is trivial when
$i+j>c$, by \cref{p12:prop:collection-monotonicity}. Therefore, $[\cF_i,\cF_j]\le\cF_{i+j}$ for any $i,j\ge1$. 

The sequence $(\cF_i)$ is therefore a \emph{terminating $N$-series}:
a descending sequence satisfying the commutator condition proved
above and eventually equal to the trivial subgroup.
In particular, $\cF_{i+1}$ is normal in $\cF_i$ because
$[\cF_i,\cF_{i+1}]\le\cF_{2i+1}
\le\cF_{i+1}$. Moreover,
$[\cF_i,\cF_i]\le\cF_{2i}
\le\cF_{i+1}$, so each quotient
$\cF_i/\cF_{i+1}$ is abelian. This gives the associated graded Lie
ring $\gr\cF
 =\bigoplus_{i=1}^{c}\cF_i/\cF_{i+1}$ with Lie bracket $[\cdot,\cdot]_{\cF}$ defined by
\[
\bigl[(a^{(0)},a^{(1)})\cF_{i+1},
       (b^{(0)},b^{(1)})\cF_{j+1}\bigr]_{\cF}:=
 \bigl([a^{(0)},b^{(0)}],[a^{(1)},b^{(1)}]\bigr)
 \cF_{i+j+1}
\]
for any $(a^{(0)},a^{(1)})\cF_{i+1}\in\cF_{i}/\cF_{i+1}$ and $(b^{(0)},b^{(1)})\cF_{j+1}\in\cF_{j}/\cF_{j+1}$.

Recall that when $i+\xi\le c$, 
\cref{p12:cor:sims_collection_multiplication_inverse} proves that
\[
\begin{aligned}
 \partial_{i,\xi}\bigl(g_1^{(0)}g_2^{(0)},g_1^{(1)}g_2^{(1)}\bigr)
 &=\partial_{i,\xi}\bigl(g_1^{(0)},g_1^{(1)}\bigr)
   +\partial_{i,\xi}\bigl(g_2^{(0)},g_2^{(1)}\bigr),\\
 \partial_{i,\xi}\bigl((g_1^{(0)})^{-1},(g_1^{(1)})^{-1}\bigr)
 &=-\partial_{i,\xi}\bigl(g_1^{(0)},g_1^{(1)}\bigr).
\end{aligned}
\]
Thus, the first possible difference map
$\partial_{i,\xi}:\cF_i\to L_{i+\xi}$ as
in~\eqref{p12:eq:first-possible-difference} is a homomorphism.

We use the Sims lifts $\widetilde{x}^{(\nu)}$ from~\eqref{p12:eq:comparison-sims-lifts} to compute the first possible difference of two matched commutators.

Under the hypotheses of \cref{p12:prop:collection-monotonicity}, define
the Sims-lift commutator difference for homogeneous $x\in L_i$ and $y\in L_j$
by
\begin{equation}\label{p12:eq:sims_lift_commutator_difference}
 \delta_\xi(x,y)
 :=\pi_{i+j+\xi}\bigl([\widetilde{x}^{(1)},\widetilde{y}^{(1)}]\bigr)
   -\pi_{i+j+\xi}\bigl([\widetilde{x}^{(0)},\widetilde{y}^{(0)}]\bigr)
 \in L_{i+j+\xi}
\end{equation}
when $i+j+\xi\le c$, and set $\delta_\xi(x,y)=0\in L$ otherwise.

\begin{proposition}[Commutator comparison using Sims lifts]
\label{p12:prop:sims-lift-comparison}
Assume the hypotheses of \cref{p12:prop:collection-monotonicity} and
use the \hyperref[p12:par:matched-pair-subgroups]{matched-pair subgroups $\cF_i$} defined at the start of this subsection.
For $1\le i\le c-\xi$, there is an isomorphism of abelian groups
\begin{equation}\label{p12:eq:comparison-graded-identification}
 \cF_i/\cF_{i+1}\cong L_i\oplus L_{i+\xi},
\end{equation}
which sends the class of $(g^{(0)},g^{(1)})$ to
$\bigl(\pi_i(g^{(0)}),\partial_{i,\xi}(g^{(0)},g^{(1)})\bigr)$.

For $1\le i,j\le c$ with $i+j+\xi\le c$, let
$(g_1^{(0)},g_1^{(1)})\in\cF_i$ and
$(g_2^{(0)},g_2^{(1)})\in\cF_j$. Denote
\begin{equation}\label{p12:eq:defxyalphaeta}
\begin{aligned}
 x&=\pi_i(g_1^{(0)})=\pi_i(g_1^{(1)}),
 &\alpha&=\partial_{i,\xi}(g_1^{(0)},g_1^{(1)}),\\
 y&=\pi_j(g_2^{(0)})=\pi_j(g_2^{(1)}),
 &\eta&=\partial_{j,\xi}(g_2^{(0)},g_2^{(1)}).
\end{aligned}
\end{equation}
Then
\[
\partial_{i+j,\xi}\bigl([g_1^{(0)},g_2^{(0)}],
                           [g_1^{(1)},g_2^{(1)}]\bigr)=
 \delta_\xi(x,y)+[x,\eta]_L+[\alpha,y]_L.
\]
\end{proposition}

\begin{proof}
For $1\le i\le c-\xi$, the restriction of $\pi_i$ to $G_i^{(0)}$
is the quotient homomorphism to $L_i$, and
$\partial_{i,\xi}$ is a homomorphism by
\cref{p12:cor:sims_collection_multiplication_inverse}.
Define
\[
 \Psi_i:\cF_i\longrightarrow L_i\oplus L_{i+\xi},
 \qquad
 \Psi_i\bigl(g^{(0)},g^{(1)}\bigr)
 :=\bigl(\pi_i(g^{(0)}),
          \partial_{i,\xi}(g^{(0)},g^{(1)})\bigr).
\]
If a pair $(g^{(0)},g^{(1)})\in\cF_i$ lies in
$\ker\Psi_i$, then $\pi_i(g^{(0)})=0$, so
$g^{(0)}\in G_{i+1}^{(0)}$.  Since the pair is matched at weight
$i$, also $\pi_i(g^{(1)})=0$, and hence
$g^{(1)}\in G_{i+1}^{(1)}$.  The pair already agrees at every
weight $t$ with $i<t<i+\xi$, while
$\partial_{i,\xi}(g^{(0)},g^{(1)})=0$ adds agreement at weight
$i+\xi$.  It therefore belongs to $\cF_{i+1}$.
Conversely, every pair in $\cF_{i+1}$ has zero weight-$i$
components and agrees at weight $i+\xi$, so it lies in
$\ker\Psi_i$. Thus $\ker\Psi_i=\cF_{i+1}$.

The map $\Psi_i$ is also surjective. Indeed, for $x\in L_i$ and
$\alpha\in L_{i+\xi}$, the pairs
$(\widetilde x^{(0)},\widetilde x^{(1)})$ and
$(1,\widetilde\alpha^{(1)})$ belong to $\cF_i$. Since
$\Psi_i$ is a homomorphism, their product has image $(x,\alpha)$.
The first isomorphism theorem now proves
\eqref{p12:eq:comparison-graded-identification}.

Since $i+j+\xi\le c$, the identification
\eqref{p12:eq:comparison-graded-identification} applies in degrees $i$,
$j$ and $i+j$.  The two input classes satisfy
\begin{equation}\label{p12:eq:decomposition_g1_g2}
\begin{aligned}
 (g_1^{(0)},g_1^{(1)})\cF_{i+1}
 &=(\widetilde x^{(0)},\widetilde x^{(1)})
   (1,\widetilde\alpha^{(1)})\cF_{i+1},\\
 (g_2^{(0)},g_2^{(1)})\cF_{j+1}
 &=(\widetilde y^{(0)},\widetilde y^{(1)})
   (1,\widetilde\eta^{(1)})\cF_{j+1},
\end{aligned}
\end{equation}
because the two sides have the same images under
\eqref{p12:eq:comparison-graded-identification}.

Under~\eqref{p12:eq:comparison-graded-identification} in degree $i+j$,
the first possible difference on the left-hand side of the asserted formula is
the second component of the bracket class below. By biadditivity of
$[\cdot,\cdot]_{\cF}$ and~\eqref{p12:eq:decomposition_g1_g2},
\[
\begin{aligned}
&\bigl[(g_1^{(0)},g_1^{(1)})\cF_{i+1},
       (g_2^{(0)},g_2^{(1)})\cF_{j+1}\bigr]_{\cF}\\
={}&\bigl[(\widetilde x^{(0)},\widetilde x^{(1)})\cF_{i+1},
          (\widetilde y^{(0)},\widetilde y^{(1)})\cF_{j+1}\bigr]_{\cF}+\bigl[(\widetilde x^{(0)},\widetilde x^{(1)})\cF_{i+1},
          (1,\widetilde\eta^{(1)})\cF_{j+1}\bigr]_{\cF}\\
 &+\bigl[(1,\widetilde\alpha^{(1)})\cF_{i+1},
          (\widetilde y^{(0)},\widetilde y^{(1)})\cF_{j+1}\bigr]_{\cF}+\bigl[(1,\widetilde\alpha^{(1)})\cF_{i+1},
          (1,\widetilde\eta^{(1)})\cF_{j+1}\bigr]_{\cF}.
\end{aligned}
\]

By the definition of $[\cdot,\cdot]_{\cF}$, we have 
\[
 \bigl[(\widetilde x^{(0)},\widetilde x^{(1)})\cF_{i+1},
          (\widetilde y^{(0)},\widetilde y^{(1)})\cF_{j+1}\bigr]_{\cF}=\bigl([\widetilde x^{(0)},\widetilde y^{(0)}],
       [\widetilde x^{(1)},\widetilde y^{(1)}]\bigr)
 \cF_{i+j+1},
\]
whose second component in $L_{i+j}\oplus L_{i+j+\xi}$ is
\[
\partial_{i+j,\xi}
 \bigl([\widetilde x^{(0)},\widetilde y^{(0)}],
       [\widetilde x^{(1)},\widetilde y^{(1)}]\bigr)
 =
 \pi_{i+j+\xi}\bigl([\widetilde x^{(1)},\widetilde y^{(1)}]\bigr)
 -\pi_{i+j+\xi}\bigl([\widetilde x^{(0)},\widetilde y^{(0)}]\bigr).
\]

Meanwhile, we have
\[
\begin{aligned}
 \bigl[(\widetilde x^{(0)},\widetilde x^{(1)})\cF_{i+1},
          (1,\widetilde\eta^{(1)})\cF_{j+1}\bigr]_{\cF}
          &=\bigl(1,[\widetilde x^{(1)},\widetilde\eta^{(1)}]\bigr)\cF_{i+j+1}\\
          \bigl[(1,\widetilde\alpha^{(1)})\cF_{i+1},
                    (\widetilde y^{(0)},\widetilde y^{(1)})\cF_{j+1}\bigr]_{\cF}
                    &=\bigl(1,[\widetilde\alpha^{(1)},\widetilde y^{(1)}]\bigr)\cF_{i+j+1}, 
\end{aligned}
\]
whose second components in $L_{i+j}\oplus L_{i+j+\xi}$ are
\[
\begin{aligned}
 \partial_{i+j,\xi}\bigl(1,[\widetilde x^{(1)},\widetilde\eta^{(1)}]\bigr)
 &=\pi_{i+j+\xi}\bigl([\widetilde x^{(1)},\widetilde\eta^{(1)}]\bigr)
   =[x,\eta]_L,~\text{and}\\
 \partial_{i+j,\xi}\bigl(1,[\widetilde\alpha^{(1)},\widetilde y^{(1)}]\bigr)
 &=\pi_{i+j+\xi}\bigl([\widetilde\alpha^{(1)},\widetilde y^{(1)}]\bigr)
   =[\alpha,y]_L,
\end{aligned}
\]
respectively. 

Finally,
\[
 \bigl[(1,\widetilde\alpha^{(1)})\cF_{i+1},
       (1,\widetilde\eta^{(1)})\cF_{j+1}\bigr]_{\cF}
 =\bigl(1,[\widetilde\alpha^{(1)},\widetilde\eta^{(1)}]\bigr)
  \cF_{i+j+1}.
\]
Since $[\widetilde\alpha^{(1)},\widetilde\eta^{(1)}]
\in G_{i+j+2\xi}^{(1)}\le G_{i+j+\xi+1}^{(1)}$, the pair belongs to
$\cF_{i+j+1}$, and this bracket is zero in
$\cF_{i+j}/\cF_{i+j+1}$. Adding the four second
components and using~\eqref{p12:eq:sims_lift_commutator_difference}
proves the asserted formula.
\end{proof}
\begin{theorem}[Commutator differences as adjoint cocycles]
\label{p12:thm:first-difference}
Let $G^{(0)}$ and $G^{(1)}$ be groups
equipped with ordered Sims generating sets having the fixed
recorded Sims presentation data, the same power word array and inducing the same graded bracket $[-,-]_L$. 

Fix $\xi\in[c-2]$. Assume that 
$[g_{h_1,a}^{(\nu)},g_{h_2,b}^{(\nu)}]$, for $h_1,h_2\in[c]$,
$a\in[r_{h_1}]$ and $b\in[r_{h_2}]$, have the same
$L_{h_1+h_2+e}$-components for $\nu=0,1$ and for every
$1\le e<\xi$ with $h_1+h_2+e\le c$.
Then the Sims-lift commutator difference defined
in~\eqref{p12:eq:sims_lift_commutator_difference} extends to a well-defined
alternating biadditive map $\delta_\xi:L\times L\rightarrow L$
with $\delta_\xi(L_i,L_j)\subseteq L_{i+j+\xi}$ whenever
$i+j+\xi\le c$.
Moreover, for all homogeneous $x,y,z\in L$, it satisfies the adjoint cocycle identity
\begin{equation}\label{p12:eq:cocycle}
[x,\delta_\xi(y,z)]_L+[y,\delta_\xi(z,x)]_L
      +[z,\delta_\xi(x,y)]_L=\delta_\xi([x,y]_L,z)+\delta_\xi([y,z]_L,x)+\delta_\xi([z,x]_L,y).
\end{equation}
We call $\delta_\xi$ an \emph{adjoint cocycle} for $[-,-]_L$ when
it satisfies this identity.
\end{theorem}

\begin{proof}
The common graded bracket identifies the excess-zero components of
every pair of corresponding Sims-generator commutators.  Together
with the assumed agreement at excesses $1\le e<\xi$, this verifies
the generator-commutator hypothesis of
\cref{p12:prop:sims-collection-rules}, and hence of
\cref{p12:prop:collection-monotonicity}. We may therefore use the
comparison $N$-series $(\cF_i)$ constructed above.
For each $x\in L_i$,
the Sims lifts $\widetilde x^{(0)}$ and $\widetilde x^{(1)}$ have the
same weight-$i$ coordinates and trivial blocks in all
other weights.  Consequently,
$(\widetilde x^{(0)},\widetilde x^{(1)})\in\cF_i$.

For $i+j+\xi\le c$, apply
\cref{p12:prop:sims-lift-comparison} in degrees $i$, $j$ and $i+j$.
Transporting $[-,-]_{\cF}$ through the identifications
\eqref{p12:eq:comparison-graded-identification} gives
\begin{equation}\label{p12:eq:comparison-bracket-formula}
 \bigl[(x,\alpha),(y,\eta)\bigr]_{\cF}
 =\bigl([x,y]_L,
        \delta_\xi(x,y)+[x,\eta]_L+[\alpha,y]_L\bigr),
\end{equation}
where $x\in L_i$, $\alpha\in L_{i+\xi}$, $y\in L_j$ and
$\eta\in L_{j+\xi}$ are defined as in~\eqref{p12:eq:defxyalphaeta}. In particular, $\delta_\xi(x,y)$ is the second
component of $[(x,0),(y,0)]_{\cF}$. The Sims lifts are uniquely
determined, so this map is well-defined. Since $x\mapsto(x,0)$ is
additive, alternation and biadditivity follow from the corresponding properties of
$[-,-]_{\cF}$. Extending over the direct sum, and using the
declared zero value when $i+j+\xi>c$, gives an alternating
biadditive map $\delta_\xi:L\times L\to L$.

Finally, let $x\in L_i$, $y\in L_j$ and $z\in L_k$.
Put $t=i+j+k+\xi$. If $t>c$, the grading rule and the definition
of $\delta_\xi$ make every term in~\eqref{p12:eq:cocycle} zero.
Suppose that $t\le c$, so every term belongs to $L_t$. Apply the Jacobi identity
in $\gr\cF$, in the form of
\eqref{p12:eq:graded-jacobi-explicit}, to $(x,0)$, $(y,0)$ and $(z,0)$.
Formula~\eqref{p12:eq:comparison-bracket-formula} gives
\[
 \bigl[(x,0),(y,0)\bigr]_{\cF}
 =\bigl([x,y]_L,\delta_\xi(x,y)\bigr).
\]
Apply the formula again with inputs
$\bigl([x,y]_L,\delta_\xi(x,y)\bigr)$ and $(z,0)$.  The second
component of their bracket, in $L_t$, is
\[
\delta_\xi([x,y]_L,z)+[[x,y]_L,0]_L
      +[\delta_\xi(x,y),z]_L=\delta_\xi([x,y]_L,z)+[\delta_\xi(x,y),z]_L.
\]
The two cyclic permutations of $(x,y,z)$ give the other two terms.
Consequently, the second component of Jacobi is
\[
0=\delta_\xi([x,y]_L,z)+[\delta_\xi(x,y),z]_L+\delta_\xi([y,z]_L,x)+[\delta_\xi(y,z),x]_L+\delta_\xi([z,x]_L,y)+[\delta_\xi(z,x),y]_L.
\]
Skew-symmetry of $[-,-]_L$ now gives~\eqref{p12:eq:cocycle}.
\end{proof}

\subsection{Deeper commutator coefficients and the fixed-breadth relation bound}
\label{p12:subsec:positive-excess-comparison}

We are now ready to count the deeper components of the selected
relations, keeping the recorded Sims presentation data, the full
power word array, and one induced graded Lie bracket from
\cref{p12:subsec:lie-bracket-count} fixed. The positive-excess components in
\eqref{p12:eq:selected-excess-component}, encoded by the coefficients
$c(i,j,k,u,v)$, have the following ranges:
\[
\begin{array}{c|c|c|c}
 \text{source weight }i&\text{selected relation indices}&
 \text{target weight }u&\text{coordinate index }v\\ \hline
 1&1\le j<k\le r_1&3\le u\le c&v\in[r_u]\\
 2&1\le k\le s,\ d_{k-1}<j\le r_2&4\le u\le c&v\in[r_u]\\
 3\le i\le c-2&1\le j\le r_i,\ 1\le k\le s&i+2\le u\le c&v\in[r_u]
\end{array}
\]

We first treat the selected source-weight-two relations separately.
For $j\in[r_2]$, write $y_j:=\bar g_{2,j}$.

\begin{lemma}[Recovery in the second layer]
\label{p12:lem:triangular-second-layer-recovery}
Assume the hypotheses of~\cref{p12:thm:first-difference}. Suppose that
\[
 \delta_\xi(C_{\core},C_{\core})=0,
 \qquad
 \delta_\xi(y_j,\bar g_{1,k})=0
 \quad(k\in[s],\ d_{k-1}<j\le r_2).
\]
Then $\delta_\xi(L_2,C_{\core})=0$.
\end{lemma}

\begin{proof}
If $\xi=c-2$, the conclusion holds by definition, since $3+\xi>c$.
Suppose that $\xi\le c-3$.
We prove $\delta_\xi(L_2,\bar g_{1,k})=0$ by induction on $k\in[s]$.
The elements $y_j$ generate $L_2$. For $j>d_{k-1}$,
$\delta_\xi(y_j,\bar g_{1,k})=0$ by assumption; this covers all $j$ when $k=1$, since
$d_0=0$. For $k>1$ and $j\le d_{k-1}$, choose $\ell<k$ with
$d_{\ell-1}<j\le d_\ell$ and put $a=f(j)<\ell$.
By the choice of the weight-two generators
in~\cref{p12:par:fixed-sims-data}, $y_j=[\bar g_{1,a},\bar g_{1,\ell}]_L$.
Since $\bar g_{1,a},\bar g_{1,\ell},\bar g_{1,k}\in C_{\core}$ and
$\delta_\xi(C_{\core},C_{\core})=0$,
\eqref{p12:eq:cocycle} gives
\[
 \delta_\xi(y_j,\bar g_{1,k})
 =-\delta_\xi([\bar g_{1,\ell},\bar g_{1,k}]_L,\bar g_{1,a})
  -\delta_\xi([\bar g_{1,k},\bar g_{1,a}]_L,\bar g_{1,\ell})=0.
\]
Here both terms vanish by the induction hypothesis, since
$a,\ell<k$ and $[\bar g_{1,\ell},\bar g_{1,k}]_L,[\bar g_{1,k},\bar g_{1,a}]_L\in L_2$.
Since the $\bar g_{1,k}$ generate $C_{\core}$, biadditivity gives
$\delta_\xi(L_2,C_{\core})=0$.
\end{proof}

Denote $L_{\ge2}:=\bigoplus_{i=2}^{c}L_i$. We have the following.

\begin{corollary}[Recovery at the current excess]
\label{p12:cor:selected-complete}
Suppose that $G^{(0)}$ and $G^{(1)}$ satisfy the hypotheses
of~\cref{p12:thm:first-difference}. Assume further that their
excess-$\xi$ components agree for every selected relation having a
core input. Then the commutator difference map $\delta_\xi$ defined
in~\eqref{p12:eq:sims_lift_commutator_difference} satisfies $\delta_\xi(L_{\ge2},L)=0$ and $\delta_\xi(L,C_{\core})=0$.
Moreover, $[\delta_\xi(x,y),h]_L=0$ for every
$x,y\in L_1$ and $h\in C_{\core}$.

In particular, among the selected excess-$\xi$ components, the two
groups can differ only in the $L_{\xi+2}$-components of the
first-layer relations $[g_{1,j}^{(\nu)},g_{1,k}^{(\nu)}]$, where $s<j<k\le r_1$ and $\nu=0,1$.
\end{corollary}

\begin{proof}
For every selected relation having a core input with $i+1+\xi\le c$,
the Sims-lift definition and the assumed component agreement give
\[
\delta_\xi(\bar g_{i,j},\bar g_{1,k})=\pi_{i+1+\xi}\bigl([g_{i,j}^{(1)},g_{1,k}^{(1)}]\bigr)-\pi_{i+1+\xi}\bigl([g_{i,j}^{(0)},g_{1,k}^{(0)}]\bigr)=0.
\]

For $j\in[r_1]$ and $k\in[s]$ with $j\ne k$, the first-layer
relation on these generators is selected after ordering their
indices. Alternation therefore gives
$\delta_\xi(\bar g_{1,j},\bar g_{1,k})=0$, including the case $j=k$.
Since the $\bar g_{1,j}$ generate $L_1$ and the $\bar g_{1,k}$ generate
$C_{\core}$, biadditivity gives
$\delta_\xi(L_1,C_{\core})=0$.
For $3\le i\le c-\xi-1$, all relations corresponding to
$(\bar g_{i,j},\bar g_{1,k})$, with $j\in[r_i]$ and $k\in[s]$, are selected.
The same argument gives $\delta_\xi(L_i,C_{\core})=0$.
For $i+1+\xi>c$, the same vanishing holds by definition.

If $\xi\le c-3$, the selected source-weight-two relations give $\delta_\xi(y_j,\bar g_{1,k})=0$ for $k\in[s]$ and $d_{k-1}<j\le r_2$.
Together with $\delta_\xi(C_{\core},C_{\core})=0$,
\cref{p12:lem:triangular-second-layer-recovery} gives
$\delta_\xi(L_2,C_{\core})=0$.
If $\xi=c-2$, this vanishing holds by definition, since $3+\xi>c$.
Thus $\delta_\xi(L_i,C_{\core})=0$ for every $i\in[c]$.
Biadditivity and $L=\bigoplus_{i=1}^{c}L_i$ give
$\delta_\xi(L,C_{\core})=0$.

We next prove $\delta_\xi(w,z)=0$ for every left-normed core Lie word $w$
and homogeneous $z\in L$, by induction on the length of $w$.
For a single core letter, this follows from
$\delta_\xi(L,C_{\core})=0$ and alternation.
Write a longer word as $w=[w',h]_L$, with $h\in C_{\core}$.
Applying~\eqref{p12:eq:cocycle} to $w',h,z$ gives
\[
[w',\delta_\xi(h,z)]_L+[h,\delta_\xi(z,w')]_L
       +[z,\delta_\xi(w',h)]_L=\delta_\xi([w',h]_L,z)+\delta_\xi([h,z]_L,w')
       +\delta_\xi([z,w']_L,h).
\]
Since $h\in C_{\core}$ and
$\delta_\xi(L,C_{\core})=0$, we have $\delta_\xi(h,z)=-\delta_\xi(z,h)=0$, $\delta_\xi(w',h)=0$, and $\delta_\xi([z,w']_L,h)=0$.
Thus,
\[
 \delta_\xi([w',h]_L,z)
 =[h,\delta_\xi(z,w')]_L-\delta_\xi([h,z]_L,w').
\]
Now $w'$ is a shorter core Lie word. The induction hypothesis applies
to every homogeneous second argument, including $z$ and $[h,z]_L$.
Hence $\delta_\xi(w',z)=0$ and $\delta_\xi(w',[h,z]_L)=0$.
Therefore, we have $\delta_\xi([w',h]_L,z)=0$.

By~\eqref{p12:eq:sims-core-graded-generation}, left-normed core Lie words generate
every $L_i$ with $i\ge2$. Hence
$\delta_\xi(L_i,L)=0$ for $2\le i\le c$.

For any $x,y\in L_1$ and $h\in C_{\core}$,
applying~\eqref{p12:eq:cocycle} to $x,y,h$ gives
\[
[x,\delta_\xi(y,h)]_L
  +[y,\delta_\xi(h,x)]_L
  +[h,\delta_\xi(x,y)]_L= \delta_\xi([x,y]_L,h)
  +\delta_\xi([y,h]_L,x)
  +\delta_\xi([h,x]_L,y).
\]
Since $\delta_\xi(L_1,C_{\core})=0$, alternation gives
$\delta_\xi(y,h)=\delta_\xi(h,x)=0$, so the
first two terms on the left vanish. Moreover,
$x,y,h\in L_1$ implies
$[x,y]_L,[y,h]_L,
[h,x]_L\in L_2$. All three terms on the right
therefore vanish by $\delta_\xi(L_{\ge2},L)=0$.
Thus for any $x,y\in L_1$ and $h\in C_{\core}$,
\begin{equation}\label{p12:eq:first_layer_pairs}
 [h,\delta_\xi(x,y)]_L=[\delta_\xi(x,y),h]_L=0.
\end{equation}
Finally, every selected relation other than a first-layer pair with
$s<j<k\le r_1$ has a core input, so its excess-$\xi$ component agrees
by assumption. This proves the final assertion.
\end{proof}

We perform the count by increasing excess. Fix $1\le\xi\le c-2$, and
suppose that two groups $G^{(0)}$ and $G^{(1)}$, equipped with ordered
Sims generating sets realizing the fixed recorded Sims presentation
data and full power word array, induce the fixed graded bracket
and agree on all selected
components at positive excesses $1\le e<\xi$. Explicitly, assume that
\[
 \pi_{i+1+e}\bigl([g_{i,j}^{(0)},g_{1,k}^{(0)}]\bigr)
 =\pi_{i+1+e}\bigl([g_{i,j}^{(1)},g_{1,k}^{(1)}]\bigr)
\]
for every selected relation $[g_{i,j},g_{1,k}]$, with indices as in
the preceding table, and every $1\le e<\xi$ with $i+1+e\le c$. 
We first show that these agreements imply agreement of all corresponding
Sims-generator commutators of source weights $h_1,h_2$ at excess $e$
with $0\le e<\xi$ and $h_1+h_2+e\le c$.

\begin{proposition}[Agreement from selected relations]
\label{p12:prop:selected-agreement}
Let $G^{(0)}$ and $G^{(1)}$ be groups
equipped with ordered Sims generating sets having the fixed
recorded Sims presentation data, the same power word array and inducing the same graded bracket $[-,-]_L$. 

Fix $\xi\in[c-2]$. Assume that the excess-$e$ components agree
for every selected relation and every $1\le e<\xi$. Then the differences
$\delta_e$ defined in~\eqref{p12:eq:sims_lift_commutator_difference} satisfy $\delta_e(L_i,C_{\core})=0$ for any $i\in[c]$ and $1\leq e<\xi$.
Furthermore, for any
$x\in L_i$ and $y\in L_j$, the commutators of their corresponding Sims
lifts have identical $L_{i+j+e}$-components whenever
$0\le e<\xi$ and $i+j+e\le c$.
\end{proposition}

\begin{proof}
We prove both conclusions by induction on the excess $e$. Agreement at
excess zero follows from the common graded bracket. Let $1\le e<\xi$,
and suppose that all Sims-lift commutator components agree at every excess smaller than $e$. In particular, they satisfy the hypotheses of~\cref{p12:thm:first-difference}
at excess $e$. Let $\delta_e$ be the difference defined
in~\eqref{p12:eq:sims_lift_commutator_difference}.

By~\cref{p12:cor:selected-complete} applied at excess $e$,
$\delta_e(L,C_{\core})=0$ and $\delta_e(L_{\ge2},L)=0$.

By~\eqref{p12:eq:sims-lift}, each Sims generator is the Sims lift of its
class in the corresponding layer. Agreement of the selected components
at excess $e$ gives $\delta_e(\bar g_{i,j},\bar g_{1,k})=0$
for every selected relation with $i+1+e\le c$.
Since every first-layer relation is selected, alternation and
biadditivity give $\delta_e(L_1,L_1)=0$.
Together with $\delta_e(L_{\ge2},L)=0$, this gives
$\delta_e(L,L)=0$. By its definition, all Sims-lift
commutators therefore agree at excess $e$, completing the induction.
\end{proof}

With all selected components at smaller positive excesses fixed, we
count and fix the excess-$\xi$ components of all selected relations
having a core input. For any two groups whose coefficient arrays satisfy these
choices, \cref{p12:prop:selected-agreement} gives agreement of all
corresponding Sims-generator commutators below excess $\xi$.
Hence \cref{p12:cor:selected-complete} gives
$\delta_\xi(L_1,C_{\core})=0$ and
$\delta_\xi(L_{\ge2},L)=0$. The remaining selected components are
the $L_{\xi+2}$-components of $[g_{1,j},g_{1,k}]$ with
$s<j<k\le r_1$.

We then have the following.

\begin{proposition}[One-excess estimate]
\label{p12:prop:one-excess}
Fix the recorded Sims presentation data and a compatible full power word array, and fix one induced graded Lie bracket with breadth sequence $\boldsymbol{\fb}=(\fb_i)_{i=2}^{c}$, and fix $1\le\xi\le c-2$. 

Suppose that the layer components of all selected relations have been
fixed at every positive excess $1\le e<\xi$. Among the coefficient
arrays arising from groups with these fixed data, components, and bracket,
fix also the excess-$\xi$ components of the first-layer relations
having a core input and of all selected relations of source weight
$2\le i\le c-\xi-1$. Then, for each $[g_{1,j},g_{1,k}]$ with $s<j<k\le r_1$, its
$L_{\xi+2}$-component takes at most
$p^{s_{\xi+2}-\fb_{\xi+2}}$ values among these arrays.
Consequently, before fixing any of the current excess-$\xi$
components, all selected excess-$\xi$ components have at most
\[
 p^{\binom{r_1}{2}s_{\xi+2}
    +s\sum_{i=2}^{c-\xi-1}r_i s_{i+\xi+1}
    -\binom{r_1-s}{2}\fb_{\xi+2}}
\]
possible choices.
\end{proposition}

\begin{proof}
If no group with the fixed data realises the conditional choices, the
assertion is immediate. Otherwise, let $G^{(0)}$ and $G^{(1)}$ be any
two groups whose Sims coefficient arrays realise these choices. By~\cref{p12:prop:selected-agreement}, all corresponding
Sims-generator commutators agree below excess $\xi$, so
\cref{p12:thm:first-difference} applies. By~\cref{p12:cor:selected-complete},
their commutator difference map satisfies $\delta_\xi(L_{\ge2},L)=0$ and $\delta_\xi(L_1,C_{\core})=0$.

Choose a maximizing element $h_{\xi+2}^\ast\in C_{\core}$ as
in~\eqref{p12:eq:core-breadth}.  The map
$\ad(h_{\xi+2}^\ast):L_{\xi+2}\to L$ has image of order
$p^{\fb_{\xi+2}}$; when $\xi=c-2$, it is the zero map as $\fb_c=0$.
For $x,y\in L_1$,
\eqref{p12:eq:first_layer_pairs} gives
\[
 [h_{\xi+2}^\ast,\delta_\xi(x,y)]_L=[\delta_\xi(x,y),h_{\xi+2}^\ast]_L=0.
\]
Thus, we have $\delta_\xi(x,y)\in\ker\ad(h_{\xi+2}^\ast)$, where,
by~\eqref{p12:eq:adjoint-kernel-size},
 $|\ker\ad(h_{\xi+2}^\ast)|=p^{s_{\xi+2}-\fb_{\xi+2}}$.

Fix one group and its ordered Sims generating set as a reference. For each outside--outside
first-layer relation $[g_{1,j},g_{1,k}]$, with $s<j<k\le r_1$,
$\delta_\xi(\bar g_{1,j},\bar g_{1,k})$ is the difference between
its $L_{\xi+2}$-components in another group and in the reference
group. Thus its component lies in a coset of
$\ker\ad(h_{\xi+2}^\ast)$ determined by the reference component for
that pair, and has at most $p^{s_{\xi+2}-\fb_{\xi+2}}$ possible values.
There are $\binom{r_1-s}{2}$ such relations. Multiplying these
per-pair bounds shows that all remaining outside--outside first-layer components
together have at most
$p^{\binom{r_1-s}{2}(s_{\xi+2}-\fb_{\xi+2})}$ possible assignments.

The core-involving first-layer components have at most
$p^{(\binom{r_1}{2}-\binom{r_1-s}{2})s_{\xi+2}}$ choices.  The selected
components of the relations of source weight $2\le i\le c-\xi-1$
have at most $p^{s\sum_{i=2}^{c-\xi-1}r_i s_{i+\xi+1}}$
choices; for $i=2$ the exact number of selected source-weight-two
relations is
$\sum_{k=1}^{s}(r_2-d_{k-1})\le sr_2$.  Multiplying these three bounds
gives the stated exponent.
\end{proof}

Now we are ready to count all possible deeper commutator data for a fixed breadth sequence.
\begin{corollary}[Positive-excess relation bound]
\label{p12:prop:filtered-fibre}
Fix the recorded Sims presentation data and a compatible full power word array, and fix one induced graded Lie bracket with breadth sequence $\boldsymbol{\fb}=(\fb_i)_{i=2}^{c}$. The base-$p$ logarithm of the number of possible assignments of all positive-excess components of the
selected relations arising from groups with these data and this
graded bracket is at most
\begin{equation}\label{p12:eq:deeper-commutator-count}
{\binom{r_1}{2}m_3
    +\sum_{i=2}^{c-2}sr_i m_{i+2}
    -\binom{r_1-s}{2}\sum_{i=3}^{c-1}\fb_i}.
\end{equation}
\end{corollary}

\begin{proof}
Apply \cref{p12:prop:one-excess} successively for
$1\le\xi\le c-2$, each time conditional on all smaller positive
excesses. Multiplying these conditional bounds shows that, among groups
realizing the fixed data and graded Lie bracket, the base-$p$
logarithm of the number of possible assignments to all
positive-excess components of the selected relations is at most
\[
  {\sum_{\xi=1}^{c-2}\left(
      \binom{r_1}{2}s_{\xi+2}
      +s\sum_{i=2}^{c-\xi-1}r_i s_{i+\xi+1}
      -\binom{r_1-s}{2}\fb_{\xi+2}\right)}.
\]
Using $\fb_c=0$ and the identities
\[
 \sum_{\xi=1}^{c-2}s_{\xi+2}=m_3,~~
 \sum_{\xi=1}^{c-2}\sum_{i=2}^{c-\xi-1}r_i s_{i+\xi+1}=\sum_{i=2}^{c-2}r_i\sum_{u=i+2}^{c}s_u=\sum_{i=2}^{c-2}r_i m_{i+2},~~
 \sum_{\xi=1}^{c-2}\fb_{\xi+2}
 =\sum_{i=3}^{c-1}\fb_i
\]
prove the corollary.
\end{proof}

We are now ready to prove \cref{p12:prop:total-relation-bound}. We allow the graded Lie bracket to vary while keeping its breadth
sequence fixed, and combine the bracket count from
\cref{p12:subsec:lie-bracket-count} with the positive-excess count from
\cref{p12:prop:filtered-fibre}.

\begin{proof}[Proof of \cref{p12:prop:total-relation-bound}]
If $c=2$, there are no deeper commutator coefficients. Since
$m_2=s_2$ and all core breadths vanish,
\eqref{p12:eq:leading-bracket-minimum} gives the stated bound
directly. We may therefore assume $c\ge3$.

For each fixed graded bracket,
\cref{p12:prop:filtered-fibre} gives the same uniform bound
\eqref{p12:eq:deeper-commutator-count} for its positive-excess data.
Summing this bound over the brackets counted in
\eqref{p12:eq:lie-bracket-count} therefore multiplies the two bounds.  The
resulting base-$p$ exponent is
\[
\begin{aligned}
 &\sum_{i=2}^{c-1}s \fb_i(r_i+s_{i+1}-\fb_i+1)
   +\binom{r_1}{2}s_2-\binom{r_1-s}{2}\fb_2+\binom{r_1}{2}m_3+\sum_{i=2}^{c-2}sr_i m_{i+2}
   -\binom{r_1-s}{2}\sum_{i=3}^{c-1}\fb_i\\
 =&\binom{r_1}{2}m_2+\sum_{i=2}^{c-2}sr_i m_{i+2}
   +\sum_{i=2}^{c-1}\fb_i
    \left(s(r_i+s_{i+1}-\fb_i+1)-\binom{r_1-s}{2}\right).
\end{aligned}
\]
The preceding applications of~\cref{p12:prop:filtered-fibre}
and the bracket count in~\cref{p12:subsec:lie-bracket-count}
admit an additional saving for the coefficients $c(1,f(j),k,u,v)$, where $j\in[r_2]$ with $d_{k-1}<j\leq d_k$, $2\leq u\leq c$ and $v\in[r_u]$. Note that these coefficients come from the commutators $[g_{1,f(j)},g_{1,k}]=g_{2,j}$ for $d_{k-1}<j\leq d_k$ when choosing the Sims generating set. 

There are $r_2$ such commutators, 
whose components were counted freely in the above 
bounds. Omitting these choices saves $r_2s_2$ in the leading
components and $r_2m_3$ in the deeper components, giving a total
saving of $r_2s_2+r_2m_3=r_2m_2$.
These positions are disjoint from those yielding the breadth
savings. Subtracting $r_2m_2$ from the above bound concludes the proof.

\end{proof}

\section{Bound the choices of power word arrays and commutator coefficient arrays}
\label{p12:sec:saturated-optimisation}

We combine~\cref{p12:lem:power-ledger,p12:prop:total-relation-bound}
to bound the number of distinct pairs of power word arrays and
commutator coefficient arrays arising with fixed recorded Sims
presentation data and a fixed core breadth sequence. We first collect
the parameters and constraints used to express this bound.

Use the layer parameters $r_i$, $s_i$ and $m_i$ from~\eqref{p12:eq:layer-parameters}. We define
\begin{equation}\label{p12:eq:optimization-parameters}
 \delta=s_1-r_1\geq 0,\qquad e=s_2-r_2\geq 0,\qquad
 z_{\rmbal}=r_1-2m_2.
\end{equation}
Thus
\begin{equation}\label{p12:eq:optimization-identities}
 m=r_1+m_2+\delta,\qquad
 m_2=r_2+e+m_3,\qquad
 s_2=m_2-m_3.
\end{equation}
The nonnegative integers $\delta$ and $e$ measure the differences
between the logarithmic orders and the ranks of $L_1$ and $L_2$, respectively.
The integer $z_{\rmbal}$ measures the deviation from
$r_1=2m_2$, which maximizes the term $\frac12r_1^2m_2$ for fixed
$r_1+m_2$ in the continuous optimization
(cf.~\cite[proof of Lemma~A.1, p.~270]{BNV}).
We can derive the following relation between $e$ and $\delta$.
\begin{lemma}\label{p12:lem:second-layer-defect}
For every prime $p$, we have $e=s_2-r_2\le\binom{s_1-r_1}{2}=\binom{\delta}{2}$.
\end{lemma}

\begin{proof}
Write $L_1=\bigoplus_{j=1}^{r_1}\langle x_j\rangle$ with $|\langle x_j\rangle|=p^{\alpha_j}$ for some $\alpha_j\ge1$. Let $t_j=\alpha_j-1$. Then
$\delta=s_1-r_1=\sum_{j=1}^{r_1}t_j$.

The abelian group $L_2$ is generated by
$z_{jk}=[x_j,x_k]_L$ for $1\le j<k\le r_1$.
Biadditivity gives $p^{\alpha_j}z_{jk}=p^{\alpha_k}z_{jk}=0$,
so the order of $z_{jk}$ divides
$p^{\min\{\alpha_j,\alpha_k\}}$. Consequently, $pL_2$ is
generated by the elements $pz_{jk}$, each of order at most
$p^{\min\{t_j,t_k\}}$.
Since $r_2=\dim_{\F_p}(L_2/pL_2)$, we obtain
\[
 s_2-r_2=\log_p|pL_2|\le\sum_{j<k}\min\{t_j,t_k\}\le\sum_{j<k}t_jt_k=\binom{s_1-r_1}{2}-\sum_{j=1}^{r_1}\binom{t_j}{2}
 \le\binom{s_1-r_1}{2},
\]
where the second inequality uses $\min\{a,b\}\le ab$ for nonnegative integers $a,b$.
\end{proof}
Define
\begin{equation}\label{p12:eq:core-tail-definition}
 T_{\core}:=\sum_{i=2}^{c-2}r_i m_{i+2}.
\end{equation}
All empty sums are zero. In particular, when $c=2$, we have $m_3=T_{\core}=0$. The quantity $T_{\core}$ appears in the bound~\eqref{p12:eq:total-relation-bound} of~\cref{p12:prop:total-relation-bound}. We collect the two summations in that bound by defining
 \begin{equation}\label{p12:eq:optimization-X}
 \begin{split}
  \cX:=&s\sum_{i=2}^{c-2}r_i m_{i+2}
 +\sum_{i=2}^{c-1}\fb_i
   \left(s(r_i+s_{i+1}-\fb_i+1)-\binom{r_1-s}{2}\right)\\
   =& sT_{\core}
 +\sum_{i=2}^{c-1}\fb_i
   \left(s(r_i+s_{i+1}-\fb_i+1)-\binom{r_1-s}{2}\right)
 \end{split}
\end{equation}
In particular, $\cX=0$ when $c=2$.

Recall the core breadth sequence $\boldsymbol{\fb}=(\fb_i)_{i=2}^c$
from~\eqref{p12:eq:core-breadth}, with $\fb_c=0$. Write
\begin{equation}\label{p12:eq:breadth-total-definition}
 |\boldsymbol{\fb}|:=\sum_{i=2}^{c}\fb_i.
\end{equation}
The breadth relations~\eqref{p12:eq:e-breadth} show that $s_{i+1}\leq s\fb_i$ for $2\leq i\leq c-1$. Since $\fb_i\le s_{i+1}$, we have
\begin{equation}\label{p12:eq:breadth-total-le-deep}
 \frac{m_3}s\le |\boldsymbol{\fb}|
 \le m_3\le m_2.
\end{equation}
When $m_3>0$, we have $s_3>0$ and $s_3\le s\fb_2$, so
$\fb_2\ge1$. When $m_3=0$, we have $\fb_2=\fb_c=0$.

For fixed recorded Sims presentation data, the layer sizes $s_i$ are fixed.
Since $0\le\fb_i\le s_{i+1}$ and $\fb_c=0$, the number of possible
core breadth sequences is at most
\begin{equation}\label{p12:eq:breadth-sequence-count}
 \prod_{i=2}^{c-1}(s_{i+1}+1)\le2^{\sum_{i=2}^{c-1}s_{i+1}}
 =2^{m_3}\le2^m,
\end{equation}
where $t+1\le2^t$ for every nonnegative integer $t$.

 The breadth relations~\eqref{p12:eq:e-breadth} give two bounds for $T_{\core}$:
 \begin{lemma}\label{lem:upper bound T_core}
 For $T_{\core}$ from~\eqref{p12:eq:core-tail-definition} and the breadth sequence from~\eqref{p12:eq:core-breadth}, we have
\begin{equation}\label{p12:eq:core-tail-two-bounds}
 T_{\core}\le\frac{m_2^2-s_2^2}{2}\leq \frac{m_2^2}{2},
 \qquad
 T_{\core}\le s m_2(|\boldsymbol{\fb}|-\fb_2).
\end{equation}   
 \end{lemma}
 \begin{proof}
   Note that $r_i\le s_i$ and $m_{i+2}=\sum_{j=i+2}^cs_j$. This gives
\[
 T_{\core}\le\sum_{i=2}^{c-2}s_i \sum_{j=i+2}^cs_j\leq \sum_{2\le i<j\le c}s_i s_j
 =\frac{m_2^2-\sum_{i=2}^{c}s_i^2}{2}
 \le\frac{m_2^2-s_2^2}{2}\leq \frac{m_2^2}{2}.
\]
Using $s_j\le s \fb_{j-1}$ from~\eqref{p12:eq:e-breadth} instead gives
\begin{align*}
 T_{\core}&\le s\sum_{i=2}^{c-2}r_i
                  \sum_{j=i+2}^{c}\fb_{j-1}
 =s\sum_{k=3}^{c-1}\fb_k\sum_{i=2}^{k-1}r_i\le s m_2\sum_{k=3}^{c-1}\fb_k
 =s m_2(|\boldsymbol{\fb}|-\fb_2).
\end{align*}
\end{proof}
Together with $r_i+s_{i+1}-\fb_i+1\le2m_2+1$, \cref{lem:upper bound T_core} gives the following bound on $\cX$:
\begin{align}
 \cX
 &\le s\min\left\{\frac{m_2^2-s_2^2}{2},
                s m_2(|\boldsymbol{\fb}|-\fb_2)\right\}
       +\left(s(2m_2+1)-\binom{r_1-s}{2}\right)
        |\boldsymbol{\fb}|,
 \label{p12:eq:X-breadth-total}\\
 &\le\frac{s m_2^2}{2}
       +\left(s(2m_2+1)-\binom{r_1-s}{2}\right)
        |\boldsymbol{\fb}|.
 \label{p12:eq:X-rough}
\end{align}
 
Now we are ready to state and prove the main theorem of Part~I:
\begin{theorem}
\label{p12:thm:joint-enumeration}
Fix a prime $p$ and an integer $m\ge3$. Fix the recorded Sims presentation data, including integers $c$, $s$, $(r_i)_{i=1}^c$, $(d_k)_{k=1}^s$, the auxiliary function, the index pairs
specifying the higher-weight generators, and the
relative order array $(a(i,j))$, as counted in
\cref{p12:prop:sims-generating-data}.
Also fix a core breadth sequence
$\boldsymbol{\fb}=(\fb_i)_{i=2}^c$ that occurs for
a nonabelian group of order $p^m$ with these choices.

Let $\cE$ be the base-$p$ logarithm of the number of distinct pairs
of full power word arrays and commutator coefficient arrays arising
from nonabelian groups of order $p^m$ equipped with ordered Sims
generating sets realizing the fixed recorded Sims presentation data and having the
prescribed core breadth sequence. Then we have
\begin{equation}\label{p12:eq:joint-enumeration-global}
 \cE\le\frac12r_1^2m_2+\cX+\frac12m_2z_{\rmbal}
              +3m_2m_3+2m_2e+r_1\delta+\Delta_p,
\end{equation}
where $\cX$ is defined by~\eqref{p12:eq:optimization-X} and
\begin{equation}\label{p12:eq:leading-map-correction}
 \Delta_p:=\min\left\{0,\,
 N_p(r_1,r_2,s)-\left(\binom{r_1}{2}-r_2\right)r_2
       +\binom{r_1-s}{2}\fb_2\right\}.
\end{equation}
\end{theorem}

\begin{proof}
For a fixed power word array, the graded abelian group
$L$ is fixed. In the proof of~\cref{p12:prop:total-relation-bound},
the first-layer bracket contribution, after imposing the $r_2$
prescribed leading values, is the first exponent
in~\eqref{p12:eq:leading-bracket-minimum}. Replacing it by the
minimum in that equation changes this contribution by $\Delta_p$.
The remaining conditional counts are uniform for each graded
bracket, and the zero tails of the defining commutators are already
accounted for in~\cref{p12:prop:total-relation-bound}. Its proof
therefore gives 
\[
 \left(\binom{r_1}{2}-r_2\right)m_2+\cX+\Delta_p.
\]
as an upper bound for the base-$p$
logarithm of the number of commutator coefficient arrays.

This bound is uniform over the full power word arrays compatible
with the fixed recorded Sims presentation data and breadth sequence. By~\cref{p12:lem:power-ledger},
counting those power word arrays gives
\begin{equation}\label{p12:eq:joint-enumeration-exact}
\begin{split}
 \cE\le{}&\sum_{i=1}^{c}\left[
 r_i m_{i+1}+\sum_{v=1}^{r_i}(r_i-v)(a(i,v)-1)\right]+\left(\binom{r_1}{2}-r_2\right)m_2+\cX+\Delta_p.
\end{split}
\end{equation}

Next, we simplify the power word contribution. By
\eqref{p12:eq:power-ledger}, the weight-$i$ power relations
contribute at most
$r_i m_{i+1}+(r_i-1)(s_i-r_i)$ to the base-$p$ logarithm.
For $i=1$, using $\delta=s_1-r_1\ge0$, this gives
\[
 r_1m_2+(r_1-1)\delta\le r_1m_2+r_1\delta.
\]
For $i=2$, using $e=s_2-r_2\ge0$ and $r_2\leq m_2$, this gives
\[
 r_2m_3+(r_2-1)e\le m_2m_3+m_2e=m_2(m_3+e).
\]
There are at most $m_3$ power relations of weight at least $3$,
and each has at most $p^{m_3}$ possible right sides. Their total
contribution to the exponent is therefore at most $m_3^2$.
Together with the commutator bound, this gives
\[
\begin{aligned}
 \cE-\Delta_p
 &\le\left(\binom{r_1}{2}-r_2\right)m_2+\cX
       +r_1m_2+r_1\delta+m_2(e+m_3)+m_3^2\\
 &=\frac12r_1^2m_2+\cX+\frac12m_2z_{\rmbal}
       +2m_2(m_3+e)+r_1\delta+m_3^2\\
 &\le\frac12r_1^2m_2+\cX+\frac12m_2z_{\rmbal}
       +3m_2m_3+2m_2e+r_1\delta,
\end{aligned}
\]
where the first equality uses $r_2=m_2-m_3-e$ and $z_{\rmbal}=r_1-2m_2$, and the last inequality uses $m_3^2\le m_2m_3$. This proves~\eqref{p12:eq:joint-enumeration-global}.
\end{proof}

The quadratic terms in the power word and commutator counts have
been combined before optimization. In particular, when
$\delta=e=m_3=0$, their remaining quadratic contribution is
$m_2z_{\rmbal}/2$, which vanishes at $r_1=2m_2$.

\begin{remark}[Additional linear constraints]
\label{p12:rem:conditional-leading-constraints}
The breadth-based conditional counts also allow a saving from prescribed affine
linear equations on the core--outside leading coefficients
\[
 c(1,j,k,2,v)\bmod p,
 \qquad 1\le j\le s<k\le r_1,\quad v\in[r_2].
\]
After fixing the full power word array, the higher core bracket
maps, and the core--core leading values, suppose such equations
have rank at least $q$ over $\F_p$, uniformly over these conditional choices.
Then $q$ may be subtracted from the commutator exponent
in~\eqref{p12:eq:total-relation-bound} and from the right-hand side
of~\eqref{p12:eq:joint-enumeration-global} with $\Delta_p$ omitted.
In particular, when $\Delta_p=0$, the saving applies directly to
the displayed joint bound.
Indeed, the core--outside values are counted freely in $L_2$ in
\cref{p12:prop:count_leading_maps}, and reduction to $L_2/pL_2$ has
constant inverse-image size $p^{s_2-r_2}$. The defining core--core
positions and the outside--outside positions used for breadth
savings are disjoint from these values, and the deeper count is
uniform for every fixed graded bracket. Any choices involved in
specifying the affine equations must be counted separately.
\end{remark}

The term $\Delta_p\leq0$ defined in~\eqref{p12:eq:leading-map-correction} is the only term in the upper bound of $\cE$ involving the map count $N_p(r_1,r_2,s)$ from~\eqref{p12:eq:leading-map-count-definition}. Sims used bounds on this count in~\cite{Sims} to derive his bound on $f(p^m)$. There are $p^{\binom{r_1}{2}r_2}$ linear maps from
$\bigwedge^2\F_p^{r_1}$ to $\F_p^{r_2}$, so $N_p(r_1,r_2,s)\le\binom{r_1}{2}r_2$.
Thus, $\Delta_p<0$ 
precisely when $\binom{r_1}{2}r_2-N_p(r_1,r_2,s)
 >r_2^2+\binom{r_1-s}{2}\fb_2$.
However, when $\binom{s-1}{2}<r_2\le\binom{s}{2}$,
this inequality never holds by the following proposition.
\begin{proposition}[Leading-map count in the saturated range]
\label{p12:prop:saturated-map-estimate}
Let $2\le s\le r_1$ and
$\binom{s-1}{2}<r_2\le\binom{s}{2}$. Then $N_p(r_1,r_2,s)=\binom{r_1}{2}r_2+O(1)$.
Moreover, $\Delta_p=0$ for every core breadth sequence arising with these
parameters.
\end{proposition}

\begin{proof}
Fix an $s$-dimensional subspace $U\le\F_p^{r_1}$. The proportion
of maps whose restriction to $\bigwedge^2U$ is surjective is
\[
 \prod_{j=\binom{s}{2}-r_2+1}^{\binom{s}{2}}(1-p^{-j})
 \ge\pi_\infty(2)>\frac14,
\]
by~\eqref{p12:eq:uniform-pi} and~\eqref{p12:eq:pi-infinity-two-lower}.
Every such map has Sims rank exactly $s$: the restriction to
$\bigwedge^2U$ is surjective, whereas the exterior square of any subspace
of dimension at most $s-1$ has dimension less than $r_2$.
Consequently, $\frac14p^{\binom{r_1}{2}r_2}
 \le p^{N_p(r_1,r_2,s)}
 \le p^{\binom{r_1}{2}r_2}$.
Taking base-$p$ logarithms proves the stated estimate, with an
absolute error since $\log_p4\le2$.

To prove $\Delta_p=0$, choose a basis of $U$ and extend it to a
basis of $\F_p^{r_1}$. Prescribe the images of $r_2$ distinct
basis wedges in $\bigwedge^2U$ to be the standard basis of
$\F_p^{r_2}$, and choose all remaining images freely. This gives
$p^{(\binom{r_1}{2}-r_2)r_2}$ maps, all of Sims rank $s$ by the
same dimension argument. Hence $N_p(r_1,r_2,s)\ge
 \left(\binom{r_1}{2}-r_2\right)r_2$,
and~\eqref{p12:eq:leading-map-correction}, together with
$\fb_2\ge0$, gives $\Delta_p=0$.
\end{proof}

\phantomsection\label{p12:par:group-ranges}
In the rest of this paper, we call the range
$\binom{s-1}{2}<r_2\le\binom{s}{2}$
\emph{saturated}, where we always have $\Delta_p=0$. The complementary range $r_2\le\binom{s-1}{2}$ is
called \emph{nonsaturated}. 
\cref{p12:prop:saturated-map-estimate} suggests the following strategy.
For the nonabelian group count $f_{nab}(p^m)$ from~\eqref{p12:eq:nonabelian-count}, let $f_{ns}(p^m)$ and $f_{sat}(p^m)$ count those in the nonsaturated and saturated ranges, respectively. Then
\begin{equation}\label{p12:eq:group-count-split}
 f_{nab}(p^m)=f_{ns}(p^m)+f_{sat}(p^m).
\end{equation}
Part~II bounds $f_{ns}(p^m)$ by combining a sharper estimate of
$N_p(r_1,r_2,s)$ with the presentation counts.
Part~III treats $f_{sat}(p^m)$ using the joint array bound in
\cref{p12:thm:joint-enumeration} and the multiplicities of distinct
array pairs produced by groups in the saturated range.

\clearpage
\part{Counting alternating bilinear maps with fixed Sims rank}\label{p12:part:leading-maps}

\section{Sims rank and hollowing properties of spaces of alternating forms}\label{p12:sec:hollowing}
Part~II sharpens the estimate of $N_p(r_1,r_2,s)$ in the
nonsaturated range. We use $N_p(r_1,r_2,s)$ as defined in
\eqref{p12:eq:leading-map-count-definition}: it is the base-$p$
logarithm of the number of surjective linear maps
$\beta:\bigwedge^2\F_p^{r_1}\longrightarrow\F_p^{r_2}$ of Sims
rank $s$, with the convention $N_p(r_1,r_2,s)=-\infty$ when
there are no such maps. The Sims-rank condition is
\eqref{eq: linearised sims rank}.
In this part, we shall prove the following theorem:

\begin{theorem}
\label{p12:thm:nonsaturated-map-count}
Let $p$ be a prime and let $r_1,r_2,s$ be integers with
$2\le s\le r_1$ and $s-1\le r_2\le\binom{s-1}{2}$. We have the following:
\begin{enumerate}[label=\textup{(\roman*)}]
 \item If $p\ge3$, or if $p=2$ and $r_2<\binom{s-1}{2}$, then
 \begin{equation}\label{p12:eq:ordinary-map-bound}
  N_p(r_1,r_2,s)
  \le(r_2-s+1)\binom{r_1}{2}+O(r_1^2+r_2^2).
 \end{equation}
 \item If $p=2$ and $r_2=\binom{s-1}{2}$, denote $a_*:=\min\left\{
     \left\lfloor\frac{2(s-1)^2}{5}\right\rfloor,
     \left\lfloor\frac{r_1}{2}\right\rfloor
  \right\}$.
 Then
 \begin{equation}\label{p12:eq:binary-endpoint-map-bound}
  N_2(r_1,r_2,s)
  \le(r_2-s+1)\binom{r_1}{2}
       +(s-1)a_*(2r_1-2a_*-1)+O(r_1^2+r_2^2).
 \end{equation}
\end{enumerate}
The implied constants are absolute and independent of $p$.
\end{theorem}

The proof will be given in~\cref{p12:sec:endpoint-absorption}.
We translate the Sims-rank condition into restrictions on dual
spaces of alternating forms, then obtain the estimate in (i) by counting totally isotropic subspaces. Proving (ii) requires additional determinant constraints and Fourier analysis.
These map-counting estimates will then be combined with the
presentation bounds from Part~I to count nonabelian groups
in the nonsaturated range.

In part II, the results are most naturally expressed without group presentations.
Fix a prime power $q$, an
$R$-dimensional $\F_q$-space $X$, an $M$-dimensional $\F_q$-space
$W$, and a surjective alternating map $\beta:\wedge^2X\longrightarrow W$.
We extend the Sims-rank definition in
\eqref{eq: linearised sims rank} to this setting over $\F_q$ by
writing
\begin{equation}\label{p12:eq:linear-Sims-rank}
 s(\beta)=
 \min\{\dim U:U\le X,\ \beta(\wedge^2U)=W\}.
\end{equation}

For $0\le h<R$, let $\Gr(h,X)$ be the Grassmannian of all $h$-dimensional subspaces of $X$. We have
\begin{equation}\label{p12:eq:exact-linear-Sims-rank}
 s(\beta)=h+1
 \quad\Longleftrightarrow\quad
 \left\{
 \begin{aligned}
  &\beta(\wedge^2H)\ne W
    &&\text{for every }H\in\Gr(h,X),\\
  &\beta(\wedge^2U)=W
    &&\text{for some }U\in\Gr(h+1,X).
 \end{aligned}
 \right.
\end{equation}

\phantomsection\label{p12:par:restriction-faithfulness}
Let $\cA_\beta=\beta^*(W^*)\le\Alt(X)$. More generally, for $\cA\le\Alt(X)$ and $U\le X$, let
\[
 \res^{\cA}_U:\cA\longrightarrow\Alt(U),
 \qquad A\longmapsto A|_{U\times U}
\]
be the restriction map of $\cA$ on $U$.
We say that $\cA$ is \emph{faithful on $U$} when the restriction map $\res^{\cA}_U$
is injective. We also say $\cA$ is $(h+1)$-faithful if $\cA$ is faithful on some $(h+1)$-dimensional subspace $U\leq X$.

\phantomsection\label{p12:par:hollowing-definitions}
For $\cA\le\Alt(X)$, we say $\cA$ is \emph{$h$-hollowing} if, for every
$H\in\Gr(h,X)$, there is a nonzero $A\in\cA$ such that
\[
 A|_{H\times H}=0.
\]
Thus every $h$-space is made totally isotropic by at least one
nonzero form from $\cA$.  The space $\cA$ is called \emph{minimally
$h$-hollowing} if no proper subspace of $\cA$ is $h$-hollowing.

The faithful and hollowing properties of spaces of alternating forms characterize the Sims rank of alternating bilinear maps as follows:
\begin{proposition}[Sims rank versus faithfulness and hollowing]
\label{p12:prop:sims-hollowing-dictionary}
Let $\beta:\wedge^2X\to W$ be a surjective alternating map over
$\F_q$, and let $\cA_\beta=\beta^*(W^*)$.
\begin{enumerate}[label=\textup{(\roman*)}]
 \item For every subspace $U\le X$, one has
 \[
  \beta(\wedge^2U)=W
  \quad\Longleftrightarrow\quad
  \res^{\cA_\beta}_U\text{ is injective}
  \quad\Longleftrightarrow\quad
  \cA_\beta\text{ is faithful on }U.
 \]
 \item For $0\le h<\dim X$,
 \[
  \cA_\beta\text{ is $h$-hollowing}
  \quad\Longleftrightarrow\quad
  s(\beta)>h.
 \]
 \item For $1\le s\leq \dim X$, one has $s(\beta)=s$ precisely when
 $\cA_\beta$ is $(s-1)$-hollowing and $s$-faithful.
\end{enumerate}
\end{proposition}

\begin{proof}
For \textup{(i)}, the kernel of $\res^{\cA_\beta}_U$ consists of the
forms $\beta^*(\lambda)$ for $\lambda\in
 \Ann_{W^*}\bigl(\beta(\wedge^2U)\bigr)$.
This annihilator is zero exactly when
$\beta(\wedge^2U)=W$.  Part \textup{(ii)} follows immediately:
$h$-hollowing says that the restriction map is noninjective for
every $h$-space, which by \textup{(i)} says that no $h$-space
generates $W$.  Part \textup{(iii)} is now exactly
\eqref{p12:eq:exact-linear-Sims-rank}, translated by \textup{(i)} and
\textup{(ii)}.
\end{proof}

\section{Counting totally isotropic subspaces}\label{p12:sec:isotropic-count}
To count hollowing subspaces of $\Alt(X)$, we first study a single alternating
form.  In this section, we relate the rank of $A\in\Alt(X)$ to the
number of $h$-spaces $H\le X$ that are totally isotropic for $A$, that
is, that satisfy $A|_{H\times H}=0$.

Let $X$ be an $R$-dimensional $\F_q$-space and let $A\in\Alt(X)$.
\phantomsection\label{p12:par:radical-definition}
The \emph{radical} of $A$ is
\[
 \rad A=\{x\in X:A(x,y)=0\text{ for every }y\in X\}.
\]
\phantomsection\label{p12:par:isotropic-locus}
For $0\le h\le R$, define the $h$-dimensional isotropic locus of
$A$ by
\[
 \cZ_h(A)
 =\{H\in\Gr(h,X):A|_{H\times H}=0\},
 \qquad
 I_h(A)=|\cZ_h(A)|.
\]
It is straightforward to see that $\cZ_h(\lambda A)=\cZ_h(A)$ for every
$\lambda\in\F_q^\times$.
The following count of the isotropic locus of $A$ can be obtained from the totally isotropic specialization of Wan's
subspace-counting formula~\cite[Theorem~3.27]{Wan}.
We give a direct proof via the radical.

\begin{lemma}[Isotropic count via the radical]\label{p12:lem:isotropic-exact}
Let $X$ be an $R$-dimensional $\F_q$-space, let
$A\in\Alt(X)$ have rank $2a$, and let $0\le h\le R$.  Let
$r_0=R-2a$.  Then
\begin{equation}\label{p12:eq:isotropic-exact}
 I_h(A)=
 \sum_{\ell=\max\{0,h-r_0\}}^{\min\{h,a\}}
 \qbinom{r_0}{h-\ell}{q}\cdot
 q^{\ell(r_0-h+\ell)}\cdot
 \prod_{i=0}^{\ell-1}
 \frac{q^{2a-2i}-1}{q^{\ell-i}-1}.
\end{equation}
\end{lemma}

\begin{proof}
Let $K=\rad A$, so $\dim K=r_0$.  Choose a complementary subspace
$Y$ such that $X=K\oplus Y$.
The restriction of $A$ to $Y$ is nondegenerate: if $y\in Y$ is
orthogonal to $Y$, then it is also orthogonal to $K$, since $K$ is
the radical, and hence $y\in K\cap Y=0$.  Thus $Y$ is a
$2a$-dimensional symplectic space.  Let $p:X\to Y$ be the projection
along $K$.

Fix an isotropic $h$-space $H\le X$. Let $J=H\cap K$ and $L=p(H)$. Put $\ell=\dim L=h-\dim J$.
The space $L$ is isotropic in $Y$: removing the $K$-components of
two vectors does not change their value under $A$.

We reconstruct $H$ in three elementary choices.  First choose
$J$, an $(h-\ell)$-space in $K$.  This gives $\qbinom{r_0}{h-\ell}{q}$ possibilities.  Second choose $L$, an isotropic $\ell$-space in the
$2a$-dimensional symplectic space $Y$.  The standard symplectic
count (see~\cite[Lemma~1(a)]{KiermaierSchmidtWassermann} and
\cite[Section~8, Exercise~8.1(ii)]{Taylor}) gives
\[
 \prod_{i=0}^{\ell-1}
 \frac{q^{2a-2i}-1}{q^{\ell-i}-1}
\]
possibilities.

It remains to count the spaces $H$ having these fixed $J$ and $L$.
For every linear map
\[
 f:L\longrightarrow K/J,
\]
choose a basis of $L$, choose in $K$ a representative of the coset
$f(u)$ for each basis vector $u$, and extend these choices linearly
to a map $\widetilde f:L\to K$.  Define
\[
 H_f=J+\{u+\widetilde f(u):u\in L\}.
\]
Changing $\widetilde f$ by a map from $L$ to $J$ does not change
$H_f$, so $H_f$ depends only on $f$.  Conversely, a space $H$ with
$H\cap K=J$ and $p(H)=L$ determines such an $f$: for $u\in L$,
choose $u+k\in H$ and set $f(u)=k+J$.  This is well-defined because
two choices of $k$ differ by an element of $H\cap K=J$.  Therefore
$f\mapsto H_f$ is a bijection.

There are consequently
\[
 q^{\dim L\,\dim(K/J)}
 =q^{\ell(r_0-h+\ell)}
\]
possible spaces $H$.  Every $H_f$ is isotropic, since $K$ is the
radical and hence, for $u,v\in L$,
\[
 A\bigl(u+\widetilde f(u),v+\widetilde f(v)\bigr)=A(u,v)=0.
\]
Multiplying the three counts and
summing over the feasible values of $\ell$ proves
\eqref{p12:eq:isotropic-exact}.
\end{proof}

We next derive uniform estimates for $I_h(A)$. For $a,h\ge0$, let $a_h=\min\{a,\max\{h-1,0\}\}$ and
\begin{equation}\label{p12:eq:kappa-definition}
 \kappa_h(a)=\frac{a_h(2h-a_h-1)}2.
\end{equation}
Note that $\kappa_0(a)=\kappa_1(a)=0$ and
$\kappa_h(a)$ increases up to $\binom h2$ and then remains constant. We shall prove that the proportion $\frac{I_h(A)}{\qbinom{R}{h}{q}}$ of the isotropic locus of $A$ relative to the whole Grassmannian $\Gr(h,X)$ is uniformly bounded by $q^{-\kappa_h(a)+O(1)}$.

\begin{lemma}[Uniform and saturated isotropic estimates]\label{p12:lem:isotropic-uniform}
Let $X$ be an $R$-dimensional $\F_q$-space and
let $A\in\Alt(X)$ have rank $2a$, where $q\ge2$ is a prime power and $0\le a\le\lfloor R/2\rfloor$.  Then
for $0\le h\le R$, the size $I_h(A)$ of the $h$-dimensional isotropic locus of $A$ satisfies
\begin{equation}\label{p12:eq:isotropic-uniform}
 \frac{I_h(A)}{\qbinom{R}{h}{q}}
 \le \frac43\,\pi_\infty(2)^{-4}q^{-\kappa_h(a)}.
\end{equation}
In particular, if
$a\ge h-1$, then we have
\begin{equation}\label{p12:eq:gamma-bound}
 \frac{I_h(A)}{\qbinom{R}{h}{q}}
 \le \gamma_q q^{-\binom h2},
\end{equation}
where
\begin{align}
 \gamma_q^-&=
 \sum_{j\ge1}q^{-3j(j-1)/2}
 \cdot
 \frac{\prod_{u=j}^{\infty}(1-q^{-2u})}
      {\pi_j(q)},\label{p12:eq:gamma-minus}\\
 \gamma_q^+&=
 1+\sum_{j\ge1}q^{-j(3j+1)/2}\cdot\frac1{\pi_j(q)},
 \label{p12:eq:gamma-plus}\\
 \gamma_q&=\max\{\gamma_q^-,\gamma_q^+\}.\label{p12:eq:gamma}
\end{align}
These constants satisfy $\gamma_2<1.72$, $\gamma_q<1.563$ and $\frac{\gamma_q}{q-1}<\frac{391}{500}=0.782$ for $q\geq 3$.
\end{lemma}

\begin{proof}
If $h=0$ or $h=1$, every $h$-space is isotropic.  The first assertion
is then immediate from \eqref{p12:eq:kappa-definition} and
$\frac43\pi_\infty(2)^{-4}>1$, while
\eqref{p12:eq:gamma-bound} follows from $\gamma_q\ge1$.  For the remainder of the proof, we assume
$h\ge2$. For brevity throughout this proof, write $\pi_n=\pi_n(q)$.

From \cref{p12:lem:isotropic-exact}, let
\[
 T_{a,\ell}
 =\qbinom{r_0}{j}{q}\cdot
  q^{\ell(r_0-j)}\cdot
  \prod_{i=0}^{\ell-1}
  \frac{q^{2a-2i}-1}{q^{\ell-i}-1}
\]
be the summand in
\eqref{p12:eq:isotropic-exact} whose image in the symplectic quotient has
dimension $\ell$, where we substitute
$j=h-\ell$ and $r_0=R-2a$.

We first separate the powers of $q$ from the finite products.  The
Gaussian factor is
\[
 \qbinom{r_0}{j}{q}=\frac{\pi_{r_0}}{\pi_j\pi_{r_0-j}}\cdot q^{j(r_0-j)}.
\]
For the symplectic factor, extracting the leading power of $q$ from
the numerator gives the exponent
$2a\ell-\ell(\ell-1)$, whereas the denominator gives
$\ell(\ell+1)/2$. Indeed,
$\sum_{i=0}^{\ell-1}2i=\ell(\ell-1)$ and
$\sum_{k=1}^{\ell}k=\ell(\ell+1)/2$.  Therefore
\[
\begin{aligned}
\prod_{i=0}^{\ell-1}
 \frac{q^{2a-2i}-1}{q^{\ell-i}-1}
 &=\frac{\prod_{u=a-\ell+1}^{a}(1-q^{-2u})}{\pi_\ell}
   \cdot q^{(2a\ell-\ell(\ell-1))-\ell(\ell+1)/2}\\
 &=\frac{\prod_{u=a-\ell+1}^{a}(1-q^{-2u})}{\pi_\ell}
   \cdot q^{\ell(4a-3\ell+1)/2}.
\end{aligned}
\]

Since $j+\ell=h$, multiplication of the three factors yields
\[
 T_{a,\ell}
 =\frac{\pi_{r_0}}
       {\pi_j\pi_{r_0-j}\pi_\ell}
  \prod_{u=a-\ell+1}^{a}(1-q^{-2u})\cdot q^{E_{a,\ell}},
 \qquad\text{where}~
 E_{a,\ell}
 =h(r_0-j)+\frac{\ell(4a-3\ell+1)}2.
\]
On the other hand, the ambient Grassmannian has size $\qbinom Rhq
 =\frac{\pi_R}{\pi_h\pi_{R-h}}\cdot q^{h(R-h)}$. Therefore, the exponent of $q$ in
$q^{\binom h2}T_{a,\ell}/\qbinom Rhq$ is
 $\binom h2+E_{a,\ell}-h(R-h)$.
Using $R=r_0+2a$ and $h=j+\ell$, we have 
\[
\begin{split}
 2(\binom h2+E_{a,\ell}-h(R-h))&=h(h-1)+2h(r_0-j)+\ell(4a-3\ell+1)-2h(R-h)\\
 &=3h^2-h-2hj-4ah+4a\ell-3\ell^2+\ell\\
 &=j^2+4j\ell-j-4aj\\
 &=-j(4a-h-3\ell+1)=-j(4a-4h+3j+1).
\end{split}
\]
Combining this exponent calculation with the finite products gives
the exact normalised identity
\begin{equation}\label{p12:eq:isotropic-normalised-term}
 \frac{T_{a,\ell}}{\qbinom Rhq}
 =\frac{\pi_{r_0}\pi_h\pi_{R-h}}
 {\pi_j\pi_{r_0-j}\pi_\ell\pi_R}
 \prod_{u=a-\ell+1}^{a}(1-q^{-2u})\cdot q^{-(\binom{h}{2}+j(4a-4h+3j+1)/2)}.
\end{equation}

Thus the loss relative to the ambient Grassmannian exponent
$h(R-h)$ is
\begin{equation}\label{p12:eq:isotropic-full-loss}
 L(a,j)=\binom h2+\frac{j(4a-4h+3j+1)}2.
\end{equation}

In the sum in~\eqref{p12:eq:isotropic-exact}, $\ell$ ranges from $\max\{0,h-r_0\}$ to $\min\{h,a\}$. Substituting $j=h-\ell$ gives the range 
\[
 \max\{0,h-a\}\le j\le\min\{h,r_0\}.
\]
If $h>r_0+a$, this interval is empty.  Hence there is no isotropic $h$-space and
$I_h(A)=0$, so the desired upper bound is automatic. 

Assume now that $h\leq r_0+a$.  Note that
\[
 L(a,j+1)-L(a,j)=3j+2a-2h+2.
\]
If $a\le h-1$, the left endpoint is $j=h-a$.  At this point the
first difference is $h-a+2>0$, and it increases with $j$.
Consequently, $L(a,j)$ is increasing throughout the feasible range,
so its minimum is attained at $j=h-a$.  Substitution in
\eqref{p12:eq:isotropic-full-loss} gives
\[
 L(a,h-a)=\frac{a(2h-a-1)}2.
\]
If $a\ge h$, the left endpoint is $j=0$, and the first difference is $2a-2h+2>0$.
Thus the minimum is
\[
 L(a,0)=\binom h2.
\]
These two values are exactly $\kappa_h(a)$, so every feasible
summand has loss at least $\kappa_h(a)$.  We now sum the losses more
carefully.  The finite-product factor in
\eqref{p12:eq:isotropic-normalised-term} is at most
$\pi_\infty(2)^{-4}$.

If $a\le h-1$, put $j_0=h-a$.  For every $m\ge0$ for which
$j_0+m$ is feasible, the first-difference formula gives
\[
 L(a,j_0+m)-\kappa_h(a)\ge3m.
\]
If $a\ge h$, the left endpoint is $j_0=0$, and similarly
\[
 L(a,m)-\kappa_h(a)\ge2m.
\]
Consequently, in both cases,
\[
 \sum_{j\text{ feasible}}q^{-L(a,j)}
 \le q^{-\kappa_h(a)}\sum_{m\ge0}q^{-2m}
 \le \frac43q^{-\kappa_h(a)}.
\]
Combining this with \eqref{p12:eq:isotropic-normalised-term} proves
\eqref{p12:eq:isotropic-uniform}. 

\medskip

We now obtain a sharper constant in the high-rank regime $a\ge h-1$, where the exponent loss has saturated at $\binom h2$. We separate the threshold rank $a=h-1$ from the higher ranks $a\ge h$.
\paragraph{Rank $a=h-1$.}
Here the feasible range begins at $j=1$, and $\ell=h-j$.
By~\eqref{p12:eq:isotropic-normalised-term}, the contribution
of the $j$th summand to
$q^{\binom h2}I_h(A)/\qbinom{R}{h}{q}$ is
\[
  \frac{q^{\binom h2}T_{h-1,h-j}}{\qbinom{R}{h}{q}}
 =\frac1{\pi_j}
 \left(
 \frac{\pi_{R-2h+2}}{\pi_{R-2h+2-j}}
 \frac{\pi_h}{\pi_{h-j}}
 \frac{\pi_{R-h}}{\pi_R}
 \prod_{u=j}^{h-1}(1-q^{-2u})
 \right)\cdot q^{-3j(j-1)/2}.
\]
We claim that the parenthesised factor is at most $\prod_{u=j}^{\infty}(1-q^{-2u})$.
It is
enough to show
\begin{equation}\label{p12:eq:finite-product-comparison}
 \frac{\pi_{R-2h+2}}{\pi_{R-2h+2-j}}
 \frac{\pi_h}{\pi_{h-j}}
 \frac{\pi_{R-h}}{\pi_R}
 \le \prod_{u=h}^{\infty}(1-q^{-2u}).
\end{equation}
Since $a=h-1$, the radical has dimension
$r_0=R-2h+2$. Also note that $j\le r_0$.  
Using $\log(1-q^{-n})=-\sum_{m\ge1}\frac{q^{-mn}}m$, the logarithm of the left-hand side
of~\eqref{p12:eq:finite-product-comparison} is
\[
 -\sum_{m\ge1}\frac1m\left(
   \sum_{n=r_0-j+1}^{r_0}q^{-mn}
   +\sum_{n=h-j+1}^{h}q^{-mn}
   -\sum_{n=R-h+1}^{R}q^{-mn}
 \right),
\]
whereas the logarithm of the right-hand side is
\[
 -\sum_{m\ge1}\frac1m\sum_{u=h}^{\infty}q^{-2um}.
\]
For every integer $m\ge1$, put $x=q^{-m}\le1/2$. We have
\[
 \sum_{n=r_0-j+1}^{r_0}x^n
 \ge jx^{r_0}
 \ge h x^{r_0+h-1}
 \ge \sum_{n=R-h+1}^{R}x^n~~\text{and}~~
 \sum_{n=h-j+1}^{h}x^n\ge x^h
 \ge \frac{x^{2h}}{1-x^2}
 =\sum_{u=h}^{\infty}x^{2u}
\]
because $h x^{h-1}\le1$ for $x\le1/2$ and $h\ge2$. Thus, for every $m\ge1$,
\[
 \sum_{n=r_0-j+1}^{r_0}q^{-mn}
 +\sum_{n=h-j+1}^{h}q^{-mn}
 -\sum_{n=R-h+1}^{R}q^{-mn}
 \ge \sum_{u=h}^{\infty}q^{-2um}.
\]
This proves~\eqref{p12:eq:finite-product-comparison}. Extending the
feasible sum to all $j\ge1$ then proves the bound by $\gamma_q^-$. 

\paragraph{Ranks $a\ge h$.}
Here the feasible range begins at $j=0$, and $\ell=h-j$.
Formula~\eqref{p12:eq:isotropic-normalised-term} shows that the
contribution of the $j$th summand to
$q^{\binom h2}I_h(A)/\qbinom Rhq$ is
\[
 \frac1{\pi_j}
 \left(
  \frac{\pi_{R-2a}}{\pi_{R-2a-j}}
  \frac{\pi_h}{\pi_{h-j}}
  \frac{\pi_{R-h}}{\pi_R}
  \prod_{u=a-h+j+1}^{a}(1-q^{-2u})
 \right)
 \cdot q^{-j(4a-4h+3j+1)/2}.
\]
We claim that the parenthesised factor is at most one.  Note that
$r_0=R-2a$.  Since $j\le r_0$, its first ratio is
\[
 \frac{\pi_{r_0}}{\pi_{r_0-j}}
 =\prod_{n=r_0-j+1}^{r_0}(1-q^{-n})\le1.
\]
It remains to compare the other numerator factors with the $h$
denominator factors in $\pi_R/\pi_{R-h}$.  Note that
\[
 \frac{\pi_h}{\pi_{h-j}}
 \prod_{u=a-h+j+1}^{a}(1-q^{-2u})
 =\prod_{t=0}^{j-1}(1-q^{-(h-j+1+t)})
  \prod_{t=0}^{h-j-1}(1-q^{-2(a-h+j+1+t)}),
\]
whereas
\[
 \frac{\pi_R}{\pi_{R-h}}
 =\prod_{t=0}^{j-1}(1-q^{-(R-h+1+t)})
  \prod_{t=0}^{h-j-1}(1-q^{-(R-h+j+1+t)}).
\]
For the first pair of blocks, since $R\ge2a\geq 2h$, we have
\[
 h-j+1+t\le R-h+1+t
\]
when $0\le t<j$. For the second pair, we have
\[
 2(a-h+j+1+t)\le R-h+j+1+t 
\]
when $0\le t<h-j$, since the right-hand side minus the left-hand side is
$r_0+h-j-1-t\ge0$.  Since $1-q^{-n}$ increases with $n$, each
numerator factor is at most the corresponding denominator factor.
Thus the parenthesised factor is at most one, as claimed.

Finally, $a\ge h$ gives
\[
 \frac{j(4a-4h+3j+1)}2\ge\frac{j(3j+1)}2.
\]
Consequently, the $j$th normalized contribution is at most $\frac1{\pi_j}\cdot q^{-j(3j+1)/2}$.
This also includes $j=0$, with the usual convention that an empty
product is one.  Extending the feasible sum to every $j\ge0$ proves
the bound by $\gamma_q^+$.

\paragraph{Explicit numerical estimates.}  For $q=2$, the
$j=1$ term of \eqref{p12:eq:gamma-minus} is less than
\[
 2(1-2^{-2})(1-2^{-4})(1-2^{-6})<1.4.
\]
The $j=2$ term is less than
\[
 \frac13(1-2^{-4})(1-2^{-6})<0.308.
\]
By \eqref{p12:eq:pi-infinity-two-lower} and \eqref{p12:eq:uniform-pi},
$\pi_j(2)\ge\pi_\infty(2)>1/4$.  Thus the tail $j\ge3$ is less than
\[
 4\sum_{j\ge3}2^{-3j(j-1)/2}<0.008.
\]
Therefore $\gamma_2^-<1.716$.  Similarly,
\[
 \gamma_2^+
 \le1+\frac12+4\sum_{j\ge2}2^{-j(3j+1)/2}<1.533.
\]
Hence $\gamma_2<1.72$.

For $q\ge3$ we give a completely explicit estimate, since this
is the only numerical input in the boundary argument.  The $j=1$ and
$j=2$ terms of $\gamma_q^-$ are bounded by $\frac1{1-q^{-1}}\le\frac32$ and $\frac{q^{-3}}{(1-q^{-1})(1-q^{-2})}\le\frac1{16}$, respectively.
For every $j\ge3$ the numerator in \eqref{p12:eq:gamma-minus} is at most
one, whereas
\[
 \pi_j(q)
 \ge 1-\sum_{i\ge1}q^{-i}\ge\frac12.
\]
Consequently
\[
 \sum_{j\ge3}
 q^{-3j(j-1)/2}
 \cdot
 \frac{\prod_{u=j}^{\infty}(1-q^{-2u})}
      {\pi_j(q)}
 \le2\sum_{j\ge3}3^{-3j(j-1)/2}<\frac1{9000}.
\]
Thus
\[
 \gamma_q^-<\frac32+\frac1{16}+\frac1{9000}<1.563.
\]
Similarly, the $j=1$ term of $\gamma_q^+$ is at most $1/6$, and
\[
 \gamma_q^+
 <1+\frac16+2\sum_{j\ge2}3^{-j(3j+1)/2}<1.169.
\]
Therefore, we have $\gamma_q<1.563$, and 
\[
 \frac{\gamma_q}{q-1}
 <\frac{1.563}{2}<\frac{391}{500}=0.782.
\]
for $q\ge3$. This proves the stated bounds.  
\end{proof}

\section{Counting exponent for \texorpdfstring{$h$}{h}-hollowing spaces of alternating forms away from the binary endpoint}
\label{p12:sec:counting-exponent}

The main result of this section counts $M$-dimensional $h$-hollowing
spaces in $\Alt(X)$, where $\dim X=R$, throughout the range
$M\le\binom h2$ except at the binary endpoint
$(q,M)=(2,\binom h2)$.  Thus it covers $q\ge3$ with
$M\le\binom h2$, and $q=2$ with $M\le\binom h2-1$.  We first record
a dimension lower bound
that is valid over every finite field, translated from~\cite[Proposition~5.2]{BNV} through the alternating-form formulation given in~\cref{p12:prop:sims-hollowing-dictionary}.

\begin{lemma}[Minimum dimension]\label{p12:lem:hollowing-min-dim}
Let $X$ be a finite-dimensional $\F_q$-space, let
$0\le h<\dim X$, and let $\cA\le\Alt(X)$ be $h$-hollowing.  Then
$\dim\cA\ge h$.
\end{lemma}

\begin{proof}
From $\cA$, define the dual alternating map $\beta_{\cA}:\wedge^2X\longrightarrow\cA^*$ by $\beta_{\cA}(x\wedge y)(A)=A(x,y)$ for any $A\in\cA$. 
The dual map $\beta_{\cA}^*:\cA\to\Alt(X)$ is just the inclusion of
$\cA$ into $\Alt(X)$.  It is injective, so $\beta_{\cA}$ is
surjective.  By \cref{p12:prop:sims-hollowing-dictionary}, we have $h<s(\beta_{\cA})$.
The original Sims rank bound \cite[Proposition~5.2]{BNV} then gives
\[
 h<s(\beta_{\cA})\le\dim\cA^*+1=\dim\cA+1.
  \]
\end{proof}

We now assume $h\le M\le\binom h2$ and either $q\ge3$, or $q=2$ and
$M\le\binom h2-1$. Logarithmic bounds below are interpreted with
$\log_q0=-\infty$.

\begin{proposition}[Counting hollowing spaces away from the binary endpoint]
\label{p12:prop:exponent-away}
Let $X$ be an $R$-dimensional
$\F_q$-space, and let
$\cH_q(R,M,h)$ be the set of $M$-dimensional
$h$-hollowing subspaces of $\Alt(X)$, where
$1\le h<R$ and $h\le M\le\binom h2$.  Assume either that $q\ge3$, or
that $q=2$ and $M\le\binom h2-1$.  Then
\begin{equation}\label{p12:eq:exponent-away}
 \log_q|\cH_q(R,M,h)|
 \le (M-h)\binom R2+O(R^2+M^2).
\end{equation}
\end{proposition}

\begin{proof}
The range $h\le M\le\binom h2$ is nonempty only when $h\ge3$.
By the \hyperref[p12:par:hollowing-definitions]{hollowing definition}
and the \hyperref[p12:par:isotropic-locus]{definition of $\cZ_h(A)$},
an $h$-hollowing $M$-space $\cA\le\Alt(X)$ satisfies
\begin{equation}\label{p12:eq:covering-property}
\Gr(h,X)=\bigcup_{0\ne A\in\cA}\cZ_h(A).
\end{equation}
\cref{p12:lem:isotropic-uniform} shows that if $\rk(A)=2a$, then
\[
\frac{|\cZ_h(A)|}{|\Gr(h,X)|}
\le q^{-\kappa_h(a)+O(1)}.
\]
Since $\kappa_h(a)$ is nondecreasing in $a$, forms of higher rank have uniformly smaller isotropic loci, up to the point where the exponent loss saturates at $\binom h2$. Consequently, once the combined contribution of the high-rank forms is shown to be small, the low-rank forms must cover a positive proportion of $\Gr(h,X)$. This leads to a better upper bound on $\cH_q(R,M,h)$ if we can isolate sufficiently many low-rank basis elements in $\cA$ through the above discussion.

Fix an $h$-hollowing $M$-space $\cA\le \Alt(X)$.  For
$1\le a\le h-2$, let
\[
 \cL_a(\cA)=\linspan\{A\in\cA:\rk A\le 2a\},
 \qquad
 k_a=\dim\cL_a(\cA).
\]
For nonzero $A\in\cL_a(\cA)$, write $\rk(A)=2r$.  Then $r\ge1$, so
$\kappa_h(r)\ge\kappa_h(1)=h-1$.  Hence
\cref{p12:lem:isotropic-uniform} gives
\[
 \frac{|\cZ_h(A)|}{|\Gr(h,X)|}\le q^{-h+O(1)}.
\]
There are fewer than $q^{k_a}$ such forms, so their isotropic loci
cover at most $q^{k_a-h+O(1)}$ of the Grassmannian. 

On the other hand, every element of $\cA\setminus\cL_a(\cA)$ has rank at least $2a+2$.
There are fewer than $q^M/(q-1)$ projective directions in $\cA$.~\footnote{We count projective directions here since $\cZ_h(A)=\cZ_h(\lambda A)$ for every $\lambda\in\F_q^\times$.}
Thus, by the first estimate in
\cref{p12:lem:isotropic-uniform}, the directions outside
$\cL_a(\cA)$ cover at most
\[
 \frac{1}{q-1}\cdot q^{M-\kappa_h(a+1)+O(1)}
\]
of the Grassmannian.  The covering property
\eqref{p12:eq:covering-property} therefore gives
\begin{equation}\label{p12:eq:incidence-split}
 1\le q^{k_a-h+O(1)}
 +\frac{1}{q-1}\cdot q^{M-\kappa_h(a+1)+O(1)}.
\end{equation}

We now choose the cutoff $a$ based on $h$ and $M$. Fix a sufficiently large absolute
constant $C$ such that replacing $M-\kappa_h(a+1)$ by $-C$ in the second term of~\eqref{p12:eq:incidence-split} makes the second term less than $1/2$. If $\binom h2-M\ge C$, choose the least $a$ such that $\kappa_h(a+1)-M\ge C$.
Such an $a\le h-2$ exists because
$\kappa_h(h-1)=\binom h2$. In this setting, 
\[
 \frac{1}{2}\le q^{k_a-h+O(1)}~~\Rightarrow~~k_a\ge h-O(1).
\]

If instead $\binom h2-M<C$, choose $a=h-2$ and put
$\ell=\binom h2-M$.  The directions outside
$\cL_{h-2}(\cA)$ have rank at least $2h-2$, so the estimate
\eqref{p12:eq:gamma-bound} shows that they cover at most
\[
 \frac{\gamma_q}{q-1}q^{-\ell}
\]
of $\Gr(h,X)$.  If $q\ge3$, this is less than
$391/500<43/50$.  If $q=2$, the hypothesis gives $\ell\ge1$, so it
is at most $\gamma_2/2<43/50$.  Hence the low-rank directions cover
more than $7/50$, and the first-group estimate gives again
\[
 k_{h-2}\ge h-O(1).
\]

Note that the cutoff of $a$ chosen above satisfies $ha=O(M)$.
Indeed, in the first case, the minimal choice of $a$ gives either
$a=1$, or $\kappa_h(a)<M+C$.
If $a\le h/2$, then
$\kappa_h(a)=a(2h-a-1)/2=\Omega(ha)$.  If $a>h/2$, then
$M=\Omega(h^2)$, and again $ha=O(M)$.  In the second case,
$a=h-2$ and $M>\binom h2-C$, so the same conclusion holds.  The
bounded values of $h$ are absorbed into the implied constant.

It remains to count the low-rank choices.  Put
$\rho=\min\{a,\lfloor R/2\rfloor\}$.  By
\cref{p12:lem:universal-field}, the number of alternating forms on $X$
of rank at most $2a$ is at most
\begin{equation}\label{p12:eq:rank-variety}
 q^{\rho(2R-2\rho-1)+O(1)}=q^{O(aR)}.
\end{equation}
When $a>R/2$, this also includes the trivial count of all alternating
forms on $X$.

There is an absolute nonnegative integer $C_0$ such that
$k_a\ge h-C_0$. Put $d=\max\{0,h-C_0\}$, so
$h-C_0\le d\le h$ and $d\le k_a$ for every space $\cA$.
Every space $\cA$ has an ordered basis whose
first $d$ members are independent rank-$\le2a$ forms from
$\cL_a(\cA)$.  For an upper bound, choose these first $d$ forms from
the rank-$\le2a$ variety and choose the remaining $M-d$ forms freely
from $\Alt(X)$, allowing dependent choices.  Counting these ordered
lists only enlarges the number of possible spaces.  Therefore
\[
\begin{split}
 \log_q|\cH_q(R,M,h)|
 &\le d\,O(aR)+(M-d)\binom R2\\
 &\le (M-h)\binom R2+O(haR+R^2)\\
 &\le (M-h)\binom R2+O(MR+R^2),
\end{split}
\]
where the last inequality holds since $haR=O(MR)$. Moreover, since $MR\le\frac{R^2+M^2}{2}$, this proves \eqref{p12:eq:exponent-away}.
\end{proof}

\section{The binary endpoint}\label{p12:sec:binary-endpoint}

Throughout this section, we deal with the case when the matrix field
is $\F_2$, $1\le h<R$, and $M=\binom h2$.  By
\cref{p12:lem:hollowing-min-dim}, the endpoint family is empty unless
$M=\binom h2\ge h$, thus
$h\ge3$. 

The difficulty is precise.  In the proof of
\cref{p12:prop:exponent-away}, take the cutoff $a=h-2$.  The projective
directions of rank at least $2h-2$ cover at most $\gamma_2\cdot 2^{M-\binom h2}$
of $\Gr(h,X)$.  If $M\le\binom h2-1$, this is at most
$\gamma_2/2<43/50$, so the low-rank directions must cover a positive
proportion of the Grassmannian.  When $M=\binom h2$, however, the bound
becomes $\gamma_2>1$ and gives no information.  

We replace the failed covering estimate by the additional structure
provided by minimality, using the
\hyperref[p12:par:hollowing-definitions]{definition of minimally
$h$-hollowing spaces} from Section~\ref{p12:sec:hollowing}.
The counting task can be divided into two parts: 
\begin{itemize}
  \item $M$-dimensional non-minimally $h$-hollowing subspaces; and 
  \item $M$-dimensional minimally $h$-hollowing subspaces. 
\end{itemize}
For the former, we can still use~\cref{p12:prop:exponent-away} to derive the desired bound.

\begin{proposition}[Nonminimal endpoint spaces]\label{p12:lem:minimal-reduction}
Let $X$ be an $R$-dimensional $\F_2$-space, let $3\le h<R$, and put
$M=\binom h2$.  Let $\fA_M^{\rmnonmin}$ be the family
of $M$-dimensional $h$-hollowing subspaces of $\Alt(X)$ that are not
minimally $h$-hollowing.  Then
\[
 \log_2|\fA_M^{\rmnonmin}|
 \le (M-h)\binom R2+O(R^2+M^2).
\]
\end{proposition}

\begin{proof}
Every $\cA\in\fA_M^{\rmnonmin}$ contains a proper
minimally $h$-hollowing subspace $\cB$.  Put $t=\dim\cB$.  By
\cref{p12:lem:hollowing-min-dim},
\[
 h\le t<M=\binom h2.
\]
For each fixed $t$, \cref{p12:prop:exponent-away} shows that the number of
possible $\cB$ is at most
\[
 2^{(t-h)\binom R2+O(R^2+t^2)}
 \le
 2^{(t-h)\binom R2+O(R^2+M^2)}.
\]
For a fixed $\cB$, the number of $M$-dimensional subspaces of
$\Alt(X)$ containing $\cB$ is
\[
 \qbinom{\binom R2-t}{M-t}{2}
 \le \pi_\infty(2)^{-1}
 2^{(M-t)(\binom R2-M)}
 \le 2^{(M-t)\binom R2+2}.
\]
Multiplying these estimates and summing over $h\le t<M$ gives
\[
\begin{split}
 |\fA_M^{\rmnonmin}|
 &\le
 \sum_{t=h}^{M-1}
 2^{(t-h)\binom R2+O(R^2+M^2)}
 \cdot 2^{(M-t)\binom R2+2}\le
 2^{(M-h)\binom R2+O(R^2+M^2)}.
\end{split}
\]
\end{proof}

\subsection{Unique hollowing forms of minimally hollowing subspaces}
It remains to count $M$-dimensional minimally $h$-hollowing subspaces over $\F_2$, with $M=\binom h2$.  We
first observe that these spaces are spanned by specific alternating forms.
\begin{proposition}[Unique hollowing forms of minimally hollowing spaces]\label{p12:lem:private-span}
Let $X$ be a finite-dimensional $\F_q$-space, and let
$\cB\le\Alt(X)$ be minimally $h$-hollowing.
\phantomsection\label{p12:par:hollowing-subspace}
For $H\in\Gr(h,X)$, define 
\[
 \cB_H:=\{A\in\cB:A|_{H\times H}=0\}.
\]
Then those spaces $\cB_H$ for which $\dim\cB_H=1$ span $\cB$.
\end{proposition}

\begin{proof}
Let $\cB'$ be their span and suppose that $\cB'$ is a proper subspace of 
$\cB$.  Choose a hyperplane $\cC$ of $\cB$ which contains $\cB'$. 

Consider $H\in\Gr(h,X)$.  If $\dim\cB_H=1$, then $\cB_H\leq\cB'\leq\cC$ by the definition of $\cB'$. If $\dim\cB_H\ge2$, since $\cC$ is a hyperplane in $\cB$, it has codimension $1$ and thus
\[
 \dim(\cC\cap\cB_H)
 \ge 1.
\]
In either case $\cC\cap\cB_H$ contains a nonzero matrix for any $H\in\Gr(h,X)$.  Hence $\cC$
is $h$-hollowing, contradicting the minimal $h$-hollowing property
of $\cB$.
\end{proof}

The unique hollowing forms obtained from~\cref{p12:lem:private-span} have rank bounded above by a function of $h$: 
\begin{proposition}[Rank bound for unique hollowing forms]\label{p12:prop:private-rank}
Let $X$ be an $R$-dimensional $\F_2$-space, let
$\cB\le\Alt(X)$ be minimally $h$-hollowing with
$\dim\cB=\binom h2$, and let $H\in\Gr(h,X)$ be such that
$\cB_H=\linspan_{\F_2}\{A\}$ for some nonzero $A\in\cB$.  Then
\[
 \rk A\le \frac{10-\sqrt{10}}9 \cdot h^2+O(h).
\]
\end{proposition}

We shall prove this rank bound in the rest of this section.  Before
doing so, we record the counting estimate that follows from it and why we need that coefficient of $h^2$. This
is the analogue of \cref{p12:lem:minimal-reduction} for minimally
$h$-hollowing endpoint spaces.

\begin{proposition}[Minimal endpoint spaces]\label{p12:prop:endpoint-exponent}
There is an absolute integer $h_0$ such that the following holds.  Let
$X$ be an $R$-dimensional $\F_2$-space, let $h_0\le h<R$, and put
\[
 M=\binom h2,
 \qquad
 \kappa=\frac45,
 \qquad
 a_*=\min\left\{
       \left\lfloor\frac{\kappa h^2}{2}\right\rfloor,
       \left\lfloor\frac R2\right\rfloor
      \right\}.
\]
Let $\fA_M^{\rmmin}$ be the family of $M$-dimensional
minimally $h$-hollowing subspaces of $\Alt(X)$.  Then
\begin{equation}\label{p12:eq:endpoint-exponent}
 \log_2|\fA_M^{\rmmin}|
 \le
 (M-h)\binom R2
 +h a_*(2R-2a_*-1)
 +O(hR+h^4).
\end{equation}
\end{proposition}

\begin{proof}
Assume \cref{p12:prop:private-rank}. 
Since $\frac{10-\sqrt{10}}9<\kappa=4/5$, we may choose an absolute $h_0\ge3$ such that
\cref{p12:prop:private-rank} gives
\[
 \rk A\le\kappa h^2
\]
for every unique hollowing form $A$ arising from
\cref{p12:lem:private-span}.  Since these forms span $\cB$, we may choose
$h$ linearly independent ones.  Their ranks are at most $2a_*$.

By \cref{p12:lem:universal-field}, the number of alternating forms on $X$
of rank at most $2a_*$ is at most
$2^{a_*(2R-2a_*-1)+O(1)}$.  Choose an ordered $h$-tuple of unique
hollowing forms from this set, and then choose the remaining $M-h$
ordered basis forms arbitrarily from the $\binom R2$-dimensional
space $\Alt(X)$.  Hence
\[
 \log_2|\fA_M^{\rmmin}|
 \le
 (M-h)\binom R2
 +h a_*(2R-2a_*-1)+O(h).
\]
This proves the proposition.
\end{proof}

\subsection{Determinant constraints from the minimal hollowing property}
Now it is more convenient to work with a fixed basis of $X$, so that vectors are columns in $\F_2^R$
and elements of $\Alt(X)$ are alternating $R\times R$ matrices over $\F_2$.  

Fix an alternating matrix space $\cB\leq \F_2^{R\times R}$ and a subspace $H\leq\F_2^R$.
The \hyperref[p12:par:hollowing-subspace]{subspace $\cB_H$ introduced
in~\cref*{p12:lem:private-span}} has the coordinate description
\[
  \cB_H=\{A\in\cB:~\trans{u}Av=0~\forall\ u,v\in H\}.
\]

\paragraph{Block-matrix structure of hollowing matrices.}
We first derive a block matrix structure for an alternating matrix $A\in\cB_H$ and obtain a rank bound with respect to such a structure.
\begin{proposition}
\label{p12:lem:rank-bound-private-span}
For every $A\in\cB_H$, let $J=H\cap\rad A$ and choose $H_1\le H$ such that $H=H_1\oplus J$. Let $d=\dim H-\dim J$.
\phantomsection\label{p12:par:private-orthogonal-space}
Define
\[
 H^{\perp_A}=\{x\in\F_2^R:\trans{x}Ah=0\text{ for every }h\in H\},
\]
\phantomsection\label{p12:par:private-complement}
and choose $X_0\le H^{\perp_A}$ such that
\[
 H^{\perp_A}=\rad A\oplus H_1\oplus X_0.
\]
Then there exist a $d$-dimensional complement $Y$ of $H^{\perp_A}$ in
$\F_2^R$ and bases adapted to
\[
 \F_2^R=\rad A\oplus H_1\oplus X_0\oplus Y
\]
for which
\[
 A=
 \bordermatrix{
          &\rad A&H_1&X_0&Y\cr
  \rad A  &0&0&0&0\cr
  H_1     &0&0&0&I_d\cr
  X_0     &0&0&B_0&D\cr
  Y       &0&I_d&\trans{D}&E\cr
 },
\]
where $B_0$ is the coefficient matrix of $A|_{X_0\times X_0}$ and is invertible, $D\in\F_2^{\dim X_0\times d}$ and $E\in\F_2^{d\times d}$. Consequently,
\begin{equation}\label{p12:eq:rank-A-X0}
 \rk A=2d+\dim X_0\le 2\dim H+\dim X_0.
\end{equation}
\end{proposition}

\begin{proof}
Since $J\le\rad A$ and $H=H_1\oplus J$, we have $\rad A+H=\rad A\oplus H_1$,
and hence $H^{\perp_A}=\rad A\oplus H_1\oplus X_0$.
 
The restriction of $A$ to $X_0$ is nondegenerate.  Indeed, suppose
$u\in X_0$ satisfies
\[
 \trans{u}Av=0\qquad\text{for every }v\in X_0.
\]
Since $u\in H^{\perp_A}$, we also have $\trans{u}Ah'=0$ for every $h'\in H_1\leq H$.
Moreover, $\trans{u}Ar=0$ for every $r\in\rad A$.  Every
$w\in H^{\perp_A}$ can be written as $w=r+h'+v$ with
$r\in\rad A$, $h'\in H_1$, and $v\in X_0$.  Therefore
\[
 \trans{u}Aw=\trans{u}Ar+\trans{u}Ah'+\trans{u}Av=0.
\]
Thus $u\in(H^{\perp_A})^{\perp_A}$.  Since it is straightforward to verify that $(H^{\perp_A})^{\perp_A}=\rad A\oplus H_1$,
we have $u\in (\rad A\oplus H_1)\cap X_0$. The direct-sum choice of $X_0$ gives
$u=0$.  Therefore $B_0$ is invertible.

Consider the linear map $T:\F_2^R\longrightarrow H^*$ defined by mapping $x$ to $T(x):~H\to\F_2$, where $T(x)(h)=\trans{x}Ah$ for any $h\in H$.
Its kernel is precisely $H^{\perp_A}$. The transpose map of $T$ is $\trans{T}:H\longrightarrow(\F_2^R)^*$, defined by mapping $h$ to $\trans{T}(h):\F_2^R\to\F_2$, where $\trans{T}(h)(x)=\trans{h}\trans{A}x$ for any $x\in\F_2^R$. Since $A$ is alternating, $\trans{A}=A$ over $\F_2$. Thus, the kernel of $\trans{T}$ is $J=H\cap\rad A$.
Hence
\[
 \codim_{\F_2^R}H^{\perp_A}
 =\rk T =\rk \trans{T}
 =\dim H-\dim J=d.
\]
Choose a $d$-dimensional complement $Y$ of $H^{\perp_A}$ in $\F_2^R$.  The
pairing between $H_1$ and $Y$ is nondegenerate.  Indeed, if $y\in Y$
is orthogonal to $H_1$, then $\trans{y}Ah'=0$ for any $h'\in H_1$.
Since $J\le\rad A$, we also have $\trans{y}Ar=0$ for any $r\in J$.
Every $h\in H=H_1\oplus J$ can be written as $h=h'+r$, and hence
\[
 \trans{y}Ah=\trans{y}Ah'+\trans{y}Ar=0.
\]
Thus $y\in H^{\perp_A}$.  Since $Y\cap H^{\perp_A}=0$, we have $y=0$.
As $\dim H_1=\dim Y=d$, the pairing is nondegenerate, and we may
choose dual bases for which its coefficient matrix is $I_d$.

Choose bases of $\rad A,H_1,X_0,Y$ in this order.  We may write 
\[
 A=
 \bordermatrix{
          &\rad A&H_1&X_0&Y\cr
  \rad A  &0&0&0&0\cr
  H_1     &0&0&0&I_d\cr
  X_0     &0&0&B_0&D\cr
  Y       &0&I_d&\trans{D}&E\cr
 },
\]
where $B_0$ is the coefficient matrix of $A|_{X_0\times X_0}$ and is invertible, $D\in\F_2^{\dim X_0\times d}$ is responsible for the pairing between $X_0$ and $Y$ and $E\in\F_2^{d\times d}$ is the coefficient matrix of $A|_{Y\times Y}$. 

Applying Gaussian elimination to the block form then gives
\[
 \rk A
 =2d+\dim X_0
 \le2\dim H+\dim X_0,
\]
which proves the result.
\end{proof}
\begin{remark}
\cref{p12:lem:rank-bound-private-span} also holds over an arbitrary field. In this case, we replace the two lower-left blocks $I_d$ and $\trans{D}$ by $-I_d$ and $-\trans{D}$, respectively, and the proof proceeds similarly.
\end{remark}
\cref{p12:lem:rank-bound-private-span} suggests bounding $\rk A$ by bounding $\dim X_0$.
\phantomsection\label{p12:par:private-restricted-form}
For this, define the restricted alternating bilinear form $b:X_0\times X_0\to\F_2$ by $b(u,v)=\trans{u}Av$ for any $u,v\in X_0$. We first establish a connection between the set 
\[
  \{(u,v)\in X_0\times X_0:~b(u,v)=1\}
\] 
and the hollowing property of $\cB$ through the evaluation matrix.

\paragraph{Determinant constraint on the evaluation matrix.}

\phantomsection\label{p12:par:evaluation-row}
Fix a basis $A_1,A_2,\ldots,A_M$ of $\cB\leq\F_2^{R\times R}$ with $\dim(\cB)=M$. For $x,y\in X\cong\F_2^R$, define their evaluation as the (row) vector
\[
 \mathbf r(x,y)=
 \bigl(\trans{x}A_1y,\ldots,\trans{x}A_My\bigr)\in\F_2^{1\times M}.
\]
Clearly, $\mathbf r(\cdot,\cdot)$ is bilinear. Moreover,
since every $A_r$ is alternating over $\F_2$, we also have $\mathbf r(x,y)=\mathbf r(y,x)$ for any $x,y\in X$.

\phantomsection\label{p12:par:evaluation-matrix}
Fix an ordered basis $y_1,\ldots,y_h$ of $U\in\Gr(h,X)$, and define the evaluation matrix of $\cB$ on $U$ by
\[
 \cR_{\cB,U}
 =\bigl(\trans{y_i}A_ry_j\bigr)_{
      1\le i<j\le h,\ 1\le r\le M}\in\F_2^{\binom h2\times M}.
\]
In particular, it is straightforward to verify that
\[
 \left.\left(\sum_{r=1}^M a_rA_r\right)\right|_{U\times U}=0
 \quad\Longleftrightarrow\quad
 \cR_{\cB,U}\trans{(a_1,\ldots,a_M)}=0.
\]

\begin{lemma}[Row basis at a unique hollowing matrix]\label{p12:lem:private-row-basis}
Let $\cB$ be a minimally $h$-hollowing alternating matrix space in $\F_2^{R\times R}$ of
dimension $M=\binom h2$.  Fix an $H\in\Gr(h,X)$ satisfying
$\cB_H=\linspan_{\F_2}\{A\}$.
Let $A=A_1,A_2,\ldots,A_M$ be a basis of $\cB$. Then there exists an ordered basis $z_1,\ldots,z_h$ of $H$, such that the rows
\[
 \begin{aligned}
 &\mathbf r(z_i,z_j), &&1\le i<j\le h-2,\\
 &\mathbf r(z_i,z_{h-1}),\quad\mathbf r(z_i,z_h),
     &&1\le i\le h-2,
 \end{aligned}
\]
form a basis of
\[
 V_0=\{(c_1,\ldots,c_M)\in\F_2^{1\times M}:c_1=0\}.
\]
\end{lemma}

\begin{proof}
Since $\cB_H=\linspan_{\F_2}\{A\}$, the right kernel of
$\cR_{\cB,H}$ is spanned by $(1,0,\ldots,0)$, and hence
$\rk\cR_{\cB,H}=M-1$.  Thus the left kernel is
one-dimensional.  For any ordered basis $z_1,\ldots,z_h$ of
$H$, there exists a unique nonzero solution $(\tau_{ij})_{1\le i<j\le h}$ satisfying
\[
 \sum_{1\le i<j\le h}\tau_{ij}\mathbf r(z_i,z_j)=0.
\]
After reordering the $z_i$, we may assume that $\tau_{h-1,h}=1$.
If the first $M-1$ rows were dependent, their relation would give
a second element of the left kernel whose $(h-1,h)$-coordinate is
zero, contradicting the uniqueness of $(\tau_{ij})$.  Hence, the first $M-1$ rows are linearly independent.  Since
$A=A_1$ vanishes on $H\times H$, these $M-1$ linearly independent rows lie in $V_0$ and thus also form a basis of $V_0$.
\end{proof}

Now we focus on a fixed minimally $h$-hollowing alternating matrix space $\cB\leq \F_2^{R\times R}$ of
dimension $M=\binom h2$ and a fixed $H\in\Gr(h,X)$ satisfying
$\cB_H=\linspan_{\F_2}\{A\}$.
Fix the ordered basis $z_1,\dots,z_h$ of $H$ as obtained in~\cref{p12:lem:private-row-basis}. 

\phantomsection\label{p12:par:private-test-subspaces}
From now on, denote
\[
 K=\linspan_{\F_2}\{z_1,\ldots,z_{h-2}\}<H,
 \qquad
 W=\linspan_{\F_2}\{\mathbf r(z_i,z_j):1\le i<j\le h-2\}<V_0.
\]
The chosen ordered basis of $H$ obtained in~\cref{p12:lem:private-row-basis} allows us to perturb $H$ with vectors in $X_0$.
\begin{lemma}[Perturbed $h$-spaces]\label{p12:lem:perturbed-test-spaces}
Let $H=\linspan_{\F_2}\{z_1,\dots,z_h\}$ with the ordered basis obtained from~\cref{p12:lem:private-row-basis}, and retain
\hyperref[p12:par:private-complement]{the chosen complement $X_0$}
and \hyperref[p12:par:private-test-subspaces]{the subspace $K$}.
For $u,v\in X_0$, define $H_{u,v}
 =K+\linspan_{\F_2}\{z_{h-1}+u\}
  +\linspan_{\F_2}\{z_h+v\}$.
Then the sum is direct and $\dim H_{u,v}=h$. 

Moreover, let $A=A_1,\ldots,A_M$ be the basis of $\cB$ fixed in
\cref{p12:lem:private-row-basis}. With respect to the displayed ordered basis, the first $M-1$ rows of
$\cR_{\cB,H_{u,v}}$ lie in $V_0$, while the remaining row has
first coordinate $b(u,v)=\trans{u}Av$, using the restricted form
\hyperref[p12:par:private-restricted-form]{$b$ defined above}.
\end{lemma}

\begin{proof}
We first verify the direct sum. Suppose that
\[
 k+a(z_{h-1}+u)+c(z_h+v)=0
\]
for some $k\in K$ and $a,c\in\F_2$.  Then
\[
 k+az_{h-1}+cz_h=au+cv.
\]
The left-hand side belongs to $H$, whereas the right-hand side belongs
to $X_0$.  Since $X_0\cap H=0$, both sides are zero.  The fact that
$z_1,\ldots,z_h$ is a basis of $H$ now gives $k=0$ and $a=c=0$.
Thus the displayed sum is direct, and
$\dim H_{u,v}=(h-2)+1+1=h$.

For the ``Moreover'' part, note that 
the first $M-1$ rows of $\cR_{\cB,H_{u,v}}$ are
\[
 \begin{aligned}
 &\mathbf r(z_i,z_j)\in W<V_0, &&1\le i<j\le h-2,\\
 &\mathbf r(z_i,z_{h-1}+u),\quad\mathbf r(z_i,z_h+v),
     &&1\le i\le h-2.
 \end{aligned}
\]
The first $\binom{h-2}{2}$ rows already lie in $V_0$. For $1\le i\le h-2$, since $A=A_1$ is alternating over $\F_2$, $A$ vanishes on $H$ and $u,v\in X_0\leq H^{\perp_A}$, we have
\[
 \trans{z_i}A(z_{h-1}+u)=\trans{u}Az_i=0,
 \qquad
 \trans{z_i}A(z_h+v)=\trans{v}Az_i=0.
\]
Thus all these rows lie in $V_0$.  The remaining row $\mathbf r(z_{h-1}+u,z_h+v)$ has first
coordinate
\[
 \trans{(z_{h-1}+u)}A(z_h+v)=\trans{u}Av=b(u,v),
\]
because $A|_{H\times H}=0$ and $u,v\in H^{\perp_A}$.
\end{proof}
\cref{p12:lem:perturbed-test-spaces} provides a constructive way to generate $h$-spaces depending on vectors $u,v\in X_0$. We shall use these $h$-spaces to place constraints on $\dim X_0$ through the minimal hollowing property of $\cB$. For this, we first see how $u\in X_0$ influences the other coordinates of $\mathbf r(z_i,u)$.

\begin{lemma}[Coefficient maps]\label{p12:lem:private-coefficient-maps}
Keep $\cB=\linspan_{\F_2}\{A=A_1,\dots,A_M\}$ and the ordered
basis $z_1,\dots,z_h$ of $H$ from~\cref{p12:lem:private-row-basis},
with $\cB_H=\linspan_{\F_2}\{A\}$.
Use \hyperref[p12:par:private-complement]{the chosen complement $X_0$}
and \hyperref[p12:par:private-test-subspaces]{the subspace $W$}.
For every $u\in X_0$ and $1\le i, j\le h-2$, there are unique
$w_j(u)\in W$ and unique coefficients
$P(u)_{ij},Q(u)_{ij}\in\F_2$ such that
\[
 \mathbf r(z_j,u)
 =w_j(u)
  +\sum_{i=1}^{h-2}P(u)_{ij}\mathbf r(z_i,z_{h-1})
  +\sum_{i=1}^{h-2}Q(u)_{ij}\mathbf r(z_i,z_h).
\]
Writing $P(u)=(P(u)_{ij})$ and $Q(u)=(Q(u)_{ij})$, the maps
$P,Q:X_0\to\F_2^{(h-2)\times(h-2)}$ are linear and satisfy $P(0)=Q(0)=0$.
\end{lemma}

\begin{proof}
Since $u\in H^{\perp_A}$ and $z_j\in H$, and $A$ is alternating over
$\F_2$, we have $\trans{z_j}Au=\trans{u}Az_j=0$, and hence
$\mathbf r(z_j,u)\in V_0$.  By
\cref{p12:lem:private-row-basis}, every vector of $V_0$ has a unique
expansion in the displayed basis.  This gives $w_j(u)$ and the
coefficients $P(u)_{ij},Q(u)_{ij}$.  Bilinearity of $\mathbf r$ and
uniqueness of the basis coefficients show that $P$ and $Q$ are
linear and satisfy $P(0)=Q(0)=0$.
\end{proof}
These coefficient maps are used to derive determinant constraints for vectors $u,v\in X_0$ satisfying $b(u,v)=1$.

\begin{proposition}[Determinant constraint]\label{p12:prop:private-determinant}
Let $P,Q:X_0\to\F_2^{(h-2)\times(h-2)}$ be the coefficient maps
defined in~\cref{p12:lem:private-coefficient-maps}.
For $u,v\in X_0$ such that $b(u,v)=1$, we have
\begin{equation}\label{p12:eq:compatible-minor}
 F(u,v)=
 \det\begin{pmatrix}
 I+P(u)&P(v)\\
 Q(u)&I+Q(v)
 \end{pmatrix}=0.
\end{equation}
\end{proposition}

\begin{proof}
Fix $u,v\in X_0$ such that $b(u,v)=1$.  By
\cref{p12:lem:perturbed-test-spaces}, $H_{u,v}$ is an $h$-dimensional
subspace. Meanwhile, the first $M-1$ rows of $\cR_{\cB,H_{u,v}}$ lie in $V_0$ and the last row has first coordinate $1$.

Since $\cB$ is $h$-hollowing, there is a nonzero vector
$a=(a_1,\ldots,a_M)\in\F_2^M$ such that
$\cR_{\cB,H_{u,v}}\trans{a}=0$.  Thus the square matrix
$\cR_{\cB,H_{u,v}}$ is singular, and its rows satisfy a
nontrivial linear dependence.  The first $M-1$ rows lie in $V_0$,
whereas the last row has first coordinate one.  Comparing first
coordinates shows that the coefficient of the last row is zero.
Hence the first $M-1$ rows are linearly dependent.  Equivalently, there
are coefficients $\{\alpha_{ij}\}$, $\{\beta_i\}$, and $\{\gamma_i\}$,
not all zero, such that
\begin{equation}\label{p12:eq:linear-dependence-row}
 \sum_{1\le i<j\le h-2}\alpha_{ij}\mathbf r(z_i,z_j)
 +\sum_{i=1}^{h-2}\beta_i\mathbf r(z_i,z_{h-1}+u)
 +\sum_{i=1}^{h-2}\gamma_i\mathbf r(z_i,z_h+v)=0.
\end{equation}
Note that the coefficients $\beta_i,\gamma_i$ cannot all be zero, since the rows
$\mathbf r(z_i,z_j)$ are linearly independent by
\cref{p12:lem:private-row-basis}.  We now compute the coordinates of the
other two families of rows. Since
\[
\mathbf r(z_i,z_{h-1}+u)=\mathbf r(z_i,z_{h-1})+\mathbf r(z_i,u)~~\text{and}~~\mathbf r(z_i,z_h+v)=\mathbf r(z_i,z_h)+\mathbf r(z_i,v),
\]
substituting the expressions of $\mathbf r(z_i,u)$ and
$\mathbf r(z_i,v)$ in~\cref{p12:lem:private-coefficient-maps} gives
\[
 \begin{aligned}
 \mathbf r(z_i,z_{h-1}+u)
  &=w_i(u)
    +\sum_{j=1}^{h-2}\bigl(\delta_{ji}+P(u)_{ji}\bigr)
       \mathbf r(z_j,z_{h-1})
    +\sum_{j=1}^{h-2}Q(u)_{ji}\mathbf r(z_j,z_h),~~\text{and}\\
 \mathbf r(z_i,z_h+v)
  &=w_i(v)
    +\sum_{j=1}^{h-2}P(v)_{ji}\mathbf r(z_j,z_{h-1})
    +\sum_{j=1}^{h-2}\bigl(\delta_{ji}+Q(v)_{ji}\bigr)
       \mathbf r(z_j,z_h),
 \end{aligned}
\]
where $\delta_{ji}$ is the Kronecker delta and the vectors $w_i(u)$ and
$w_i(v)$ belong to $W$ and involve only the basis vectors
$\mathbf r(z_j,z_\ell)$ with $1\le j<\ell\le h-2$.  By
\cref{p12:lem:private-row-basis}, these basis vectors together with
\[
 \mathbf r(z_j,z_{h-1}),\quad\mathbf r(z_j,z_h),
 \qquad 1\le j\le h-2,
\]
are linearly independent. Substitute the above equations for $\mathbf r(z_i,z_{h-1}+u)$ and $\mathbf r(z_i,z_h+v)$ into~\eqref{p12:eq:linear-dependence-row}. For each $1\le j\le h-2$, the coefficients of
$\mathbf r(z_j,z_{h-1})$ and $\mathbf r(z_j,z_h)$ become
\[
 \sum_{i=1}^{h-2}\bigl(\delta_{ji}+P(u)_{ji}\bigr)\beta_i
 +\sum_{i=1}^{h-2}P(v)_{ji}\gamma_i~~\rmand~~ \sum_{i=1}^{h-2}Q(u)_{ji}\beta_i
 +\sum_{i=1}^{h-2}\bigl(\delta_{ji}+Q(v)_{ji}\bigr)\gamma_i
\]
respectively.
These coefficients must be zero by linear independence.
Writing
\[
 \beta=\trans{(\beta_1,\ldots,\beta_{h-2})},
 \qquad
 \gamma=\trans{(\gamma_1,\ldots,\gamma_{h-2})},
\]
these equations are equivalent to
\[
 \begin{pmatrix}
 I+P(u)&P(v)\\
 Q(u)&I+Q(v)
 \end{pmatrix}
 \binom{\beta}{\gamma}=0.
\]
Since $(\beta,\gamma)\ne0$ when $b(u,v)=1$, the block matrix is singular, and hence
$F(u,v)=0$.
\end{proof}

\subsection{\texorpdfstring{Bound $\dim X_0$ using low-degree polynomials and Fourier analysis}{Upper bound for dim X0 using low-degree polynomials and Fourier analysis}}

\paragraph{Determinant constraints and low-degree polynomials.}
Retain the minimally $h$-hollowing space $\cB$, its fixed subspace
$H$ with $\cB_H=\linspan_{\F_2}\{A\}$, and
\hyperref[p12:par:private-complement]{the chosen complement $X_0$}.
For the \hyperref[p12:par:private-restricted-form]{restricted form $b$},
\eqref{p12:eq:compatible-minor} gives, whenever $b(u,v)=1$, 
\[
  F(u,v)=
 \det\begin{pmatrix}
 I+P(u)&P(v)\\
 Q(u)&I+Q(v)
 \end{pmatrix}=0,
\]
where $F$ is the determinant from~\eqref{p12:eq:compatible-minor}
and $P,Q$ are the linear coefficient maps from
\cref{p12:lem:private-coefficient-maps}.

Note that the upper-left block $I+P(u)$ in $F(u,v)$ depends only on $u$.  For those $u\in X_0$ such that $I+P(u)$ is invertible, the Schur complement rewrites $F(u,v)$ as 
\[
  F(u,v)=\det(I+P(u))\det\bigl(I+Q(v)-Q(u)(I+P(u))^{-1}P(v)\bigr),
\]
which is a polynomial of degree at most $h-2$ evaluated at the
linear map $R_u$ of $v$, where
\begin{equation}\label{p12:eq:R_uv}
  G_u:=Q(u)(I+P(u))^{-1},
  \qquad R_u(v):=Q(v)-G_uP(v).
\end{equation}
We use
the following support bound to derive a quantitative relation between $\dim X_0$ and $\dim R_u(X_0)$.

\begin{lemma}[Reed--Muller support]\label{p12:lem:RM-support}
Let $n,d\ge0$.  If a nonzero polynomial $f:\F_2^n\to\F_2$ has degree at most $d$, then
\[
 |\supp f|\ge2^{n-d}.
\]
If, in addition, $f(0)=1$, then $\codim\linspan(\supp f)\le d$.
\end{lemma}

\begin{proof}
If $d\ge n$, the first assertion follows immediately from
$|\supp f|\ge1$.  If $0\le d\le n$, the evaluation vector of $f$ is a
nonzero codeword in the binary Reed--Muller code $\operatorname{RM}(d,n)$,
whose minimum distance is $2^{n-d}$; see
\cite{Reed,DelsarteGoethalsMacWilliams} and
\cite[Chapter~13]{MacWilliamsSloane}.

For the second assertion, let $S=\linspan(\supp f)$ and $c=\codim S$. 
After an invertible linear change of coordinates, which preserves the
algebraic degree and fixes the origin, we may assume that
\[
 S=\{x\in\F_2^n:x_1=\cdots=x_c=0\}.
\]
Thus $f(x)=0$ whenever at least one of $x_1,\ldots,x_c$ is equal to
$1$.  In the Boolean polynomial ring
\[
 \F_2[x_1,\ldots,x_n]/(x_1^2-x_1,\ldots,x_n^2-x_n),
\]
the factor $1+x_i$ is the indicator of the condition $x_i=0$.
Define
\[
 g(x_{c+1},\ldots,x_n)
 =f(0,\ldots,0,x_{c+1},\ldots,x_n).
\]
It follows that, as functions on $\F_2^n$,
\[
 f(x_1,\ldots,x_n)
 =(1+x_1)\cdots(1+x_c)g(x_{c+1},\ldots,x_n).
\]
Indeed, when $x_1=\cdots=x_c=0$, the product of the first $c$
factors is $1$ and the right-hand side equals $f(x)$.  If some
$x_i=1$, then the product and $f(x)$ are both zero.  Finally,
$f(0)=1$ gives $g(0)=1$, so $g$ has constant term $1$.  The displayed
product therefore contains the monomial $x_1\cdots x_c$ with
coefficient $1$.  Hence $\deg f\ge c$, and the assumption
$\deg f\le d$ gives $\codim\linspan(\supp f)=c\le d$.
\end{proof}

Consider the polynomial
$D(u)=\det(I+P(u))$ for $u\in X_0$.
Since $P$ is linear, $D$ has degree at most $h-2$.  Moreover,
$D(0)=\det(I)=1$, so $D$ is nonzero. Therefore, the support 
\begin{equation}\label{p12:eq:S_0}
 S_0:=\{u\in X_0:\det(I+P(u))=1\}
\end{equation}
of $D$ satisfies, by \cref{p12:lem:RM-support},
\begin{equation}\label{p12:eq:S0-lower}
 |S_0|\ge 2^{\dim X_0-(h-2)}.
\end{equation}

Fix $u\in S_0$.  Since $\det(I+P(u))=1$, the Schur-complement formula
gives $F(u,v)=\det(I+R_u(v))$, where the map
$R_u:X_0\to\F_2^{(h-2)\times(h-2)}$ given in~\eqref{p12:eq:R_uv} is linear.
Recall that \cref{p12:prop:private-determinant} gives
\begin{equation}\label{p12:eq:Riccati-constraint}
 b(u,v)=1\quad\Longrightarrow\quad \det(I+R_u(v))=0.
\end{equation}
Equivalently, since both sides take values in $\F_2$,
\begin{equation}\label{p12:eq:support-orthogonality}
 \det(I+R_u(v))=1\quad\Longrightarrow\quad b(u,v)=0.
\end{equation}
Indeed, if $R_u(v)=0$, then $\det(I+R_u(v))=1$, so
\eqref{p12:eq:support-orthogonality} gives $b(u,v)=0$.  Thus $b(u,-)$
vanishes on $\ker R_u$.
\phantomsection\label{p12:par:lambda-functional}
Consequently, the rule
\[
 \lambda_u(R_u(v))=b(u,v)
\]
defines a linear functional on $R_u(X_0)$.  By
\eqref{p12:eq:support-orthogonality}, this functional vanishes on the
support of $M\mapsto\det(I+M)$ in $R_u(X_0)$.  This is the form of the
determinant constraint used below.

\begin{proposition}[Fibre kernel bound]\label{p12:lem:fibre-kernel}
For every $u\in S_0$, the linear map $R_u:X_0\longrightarrow\F_2^{(h-2)\times (h-2)}$ defined in~\eqref{p12:eq:R_uv}
satisfies $\dim X_0\leq \dim R_u(X_0)+2(h-2)$.
\end{proposition}

\begin{proof}
Fix $u\in S_0$, with $S_0$ as in~\eqref{p12:eq:S_0}, and let
$U=\{w\in S_0:G_w=G_u\}$, using $G_u$ from~\eqref{p12:eq:R_uv}.
We first use~\cref{p12:lem:RM-support} to lower bound $|U|$.
By~\eqref{p12:eq:R_uv}, $R_u(w)=Q(w)-G_uP(w)$. Meanwhile, since $w\in S_0$, $I+P(w)$ is invertible. The condition $G_w=G_u$ then gives
\[
  Q(w)(I+P(w))^{-1}=G_u~~\Leftrightarrow~~Q(w)=G_u(I+P(w))~~\Leftrightarrow~~R_u(w)=G_u.
\]
Thus, 
\[
 U=\{w\in R_u^{-1}(G_u):\det(I+P(w))=1\}.
\]
The affine space $R_u^{-1}(G_u)$ contains $u$. Indeed,
\[
  R_u(u)=Q(u)-G_uP(u)=Q(u)-G_u(I+P(u))+G_u=G_u,
\]
where the second equality uses the fact that $G_u(I+P(u))=Q(u)$.
Since $R_u^{-1}(G_u)$ is nonempty, it is a translate of $\ker R_u$
and therefore has dimension $\dim\ker R_u$.

 The polynomial $w\mapsto\det(I+P(w))$ has degree at
most $h-2$ and takes the value $1$ at $u$.  Its restriction to this
affine space is therefore nonzero.  Hence \cref{p12:lem:RM-support} gives
$|U|\ge2^{\dim\ker R_u-(h-2)}$.

We next upper bound $|U|$. Note that for any $w\in U$, we have $R_w(v)=R_u(v)$ for any $v\in X_0$. Thus, if $v\in \ker R_u$, then $\det(I+R_w(v))=\det I=1$, so
\eqref{p12:eq:support-orthogonality} gives $b(w,v)=0$.  Hence, for every
$w\in U$, the \hyperref[p12:par:lambda-functional]{functional
$\lambda_w$ already constructed on $R_w(X_0)$} is a functional on
$R_u(X_0)$, since $R_w=R_u$:
\[
 \lambda_w:R_u(X_0)\longrightarrow\F_2,
 \qquad
 \lambda_w(R_u(v))=b(w,v).
\]

Apply the second assertion of \cref{p12:lem:RM-support} to the polynomial
$M\mapsto\det(I+M)$ on $R_u(X_0)$.  It has degree at most $h-2$ and takes
the value $1$ at $M=0$, so
\[
 L=\linspan\{M\in R_u(X_0):\det(I+M)=1\}
\]
has codimension at most $h-2$ in $R_u(X_0)$.  
On the other hand, every $\lambda_w$ belongs to the annihilator
\[
 L^\circ=\{\lambda\in R_u(X_0)^*: \lambda(M)=0
          \text{ for every }M\in L\}.
\]
Indeed, for any $M=R_u(v)$ satisfying $\det(I+M)=1$, since $R_u(v)=R_w(v)$ for any $v\in X_0$, we have $\det(I+R_w(v))=1$. By~\eqref{p12:eq:support-orthogonality}, $b(w,v)=0$. Thus,
\[
 \lambda_w(M)=\lambda_w(R_u(v))=b(w,v)=0.
\]
Since $\dim L^\circ=\codim_{R_u(X_0)}L\le h-2$,
the space $L^\circ$ contains at most $2^{h-2}$ functionals.

The map $w\mapsto\lambda_w$ is injective on $U$.  Indeed, if
$\lambda_w=\lambda_{w'}$, then $b(w+w',v)=0$ for every $v\in X_0$,
and the nondegeneracy of $b$ gives $w=w'$.  Therefore $|U|\le2^{h-2}$.
Together with $|U|\ge2^{\dim\ker R_u-(h-2)}$, this gives $\dim\ker R_u\le2(h-2)$. 
By rank--nullity,
\[
 \dim R_u(X_0)=\dim X_0-\dim\ker R_u\ge \dim X_0-2(h-2).
\]
\end{proof}

\paragraph{Fourier analysis on matrix spaces.}
\cref{p12:lem:fibre-kernel} bounds $\dim X_0$ by
$\dim R_u(X_0)+2(h-2)$. We next bound $\dim R_u(X_0)$ using Fourier analysis on matrix spaces. The following result is the full-rank specialization of the eigenvalue formula
for the association scheme of bilinear forms; see
\cite{DelsarteBilinearForms}. For completeness, we include its proof here.

\begin{lemma}[Fourier transform of the invertible-matrix indicator]\label{p12:lem:GL-Fourier}
For $k\geq1$, let $1_{\GL_k(2)}$ be the indicator function of
$\GL_k(2)\subseteq\F_2^{k\times k}$.  Put $\langle C,A\rangle=\tr(\trans{C}A)$
and define the unnormalised Fourier transform of $1_{\GL_k(2)}$ as
\[
 \widehat{1_{\GL_k(2)}}(C)=
 \sum_{A\in \F_2^{k\times k}}1_{\GL_k(2)}(A)(-1)^{\langle C,A\rangle}.
\]
If $\rk C=r$, then
\begin{equation}\label{p12:eq:GL-Fourier}
 \widehat{1_{\GL_k(2)}}(C)=
 (-1)^r2^{k(k-r)+\binom r2}
 \pi_{k-r}(2).
\end{equation}
In particular, $|\widehat{1_{\GL_k(2)}}(C)|$ is nonincreasing as a function of
$r$.
\end{lemma}

\begin{proof}
For a subspace $U\le\F_2^k$, the M\"obius function of the subspace
lattice is $\mu(0,U)=(-1)^{\dim U}2^{\binom{\dim U}{2}}$.
Therefore
\[
 1_{\GL_k(2)}(A)
 =1_{\ker A=0}
 =\sum_{U\le\ker A}\mu(0,U).
\]
Substitution into the Fourier sum gives
\[
\widehat{1_{\GL_k(2)}}(C)=
 \sum_{U\le\F_2^k}\mu(0,U)
 \sum_{A:\,AU=0}(-1)^{\langle C,A\rangle}.
\]
For the inner sum, let $d=\dim U\leq k$. Write $a_i,c_i\in\F_2^k$
for the transposes of the $i$th rows of $A,C$, respectively.  The
condition $AU=0$ is equivalent to $a_i\in U^\perp$ for every $i$, and
\[
 \langle C,A\rangle
 =\tr(\trans{C}A)
 =\sum_{i=1}^k \trans{c_i}a_i.
\]
Thus the character $A\mapsto(-1)^{\langle C,A\rangle}$ is trivial on
the space $\{A:AU=0\}$ if and only if $c_i\in(U^\perp)^\perp=U$ for every $i$. 
This condition is equivalent to the row space of $C$ being
contained in $U$.  In this case every summand is $1$, so the inner
sum equals $2^{k(k-d)}$.

On the other hand, if the row space of $C$ is not contained in $U$, choose $i$ with
$c_i\notin U$.  Then there is $b\in U^\perp$ such that
$\trans{c_i}b=1$.  Let $B$ be the matrix whose $i$th row is
$\trans{b}$ and whose other rows are zero.  Then $BU=0$ and
$\langle C,B\rangle=1$.  Translation by $B$ permutes the matrices
satisfying $AU=0$ and changes the sign of every summand:
\[
 (-1)^{\langle C,A+B\rangle}
 =-(-1)^{\langle C,A\rangle}.
\]
The terms therefore cancel in pairs, and the inner character sum is zero. Thus,
\[
\widehat{1_{\GL_k(2)}}(C)=
 \sum_{\substack{U\le\F_2^k\\
                  \row(C)\le U}}
 \mu(0,U)\,2^{k(k-\dim U)}.
\]

Now suppose that $\rk C=r$.  For each $r\le d\le k$, the number of
$d$-dimensional subspaces $U$ containing $\row(C)$
is $\qbinom{k-r}{d-r}{2}$.
Grouping the above sum according to $d=\dim U$ and substituting
the formula for $\mu(0,U)$ gives
\[
 \widehat{1_{\GL_k(2)}}(C)=
 \sum_{d=r}^{k}
 \qbinom{k-r}{d-r}{2}
 (-1)^d2^{\binom d2}2^{k(k-d)}.
\]
Writing $d=r+j$ and $n=k-r$ gives
\[
 \widehat{1_{\GL_k(2)}}(C)=
 (-1)^r2^{k(k-r)+\binom r2}
 \sum_{j=0}^{n}
 \qbinom nj2(-1)^j2^{\binom j2}2^{-nj}=(-1)^r2^{k(k-r)+\binom r2}\pi_{k-r}(2),
\]
where the last equality uses the Gaussian binomial theorem
\[
 \sum_{j=0}^{n}\qbinom nj2(-1)^j2^{\binom j2}x^j
 =\prod_{i=0}^{n-1}(1-2^ix)
\]
with $x=2^{-n}$.  

The ratio of the absolute values at ranks $r+1$ and $r$ is $\frac1{2^{k-r}-1}\le1$, so these absolute values are nonincreasing.

\end{proof}

The \hyperref[p12:par:lambda-functional]{functionals $\lambda_u$}
encode the determinant constraint in~\eqref{p12:eq:support-orthogonality};
\cref{p12:lem:fibre-kernel} used this constraint with the
Reed--Muller support bound. For Fourier analysis, we extend these linear
functionals using the trace pairing on the full matrix space.
The associated characters have the form
$M\mapsto(-1)^{\tr(\trans{C}M)}$. We record the trace representation next.

\begin{lemma}[Trace representation of $b(u,-)$]\label{p12:lem:trace-representation}
For every $u\in S_0$, use the
\hyperref[p12:par:lambda-functional]{functional $\lambda_u$ constructed above}.
\phantomsection\label{p12:par:trace-extension-space}
Define
\[
 \cC_u=\left\{C\in\F_2^{(h-2)\times(h-2)}:
 \tr(\trans{C}M)=\lambda_u(M)
 \text{ for every }M\in R_u(X_0)\right\}.
\]
Then $\cC_u$ is a nonempty affine space of dimension $(h-2)^2-\dim R_u(X_0)$.
\end{lemma}

\begin{proof}
The \hyperref[p12:par:lambda-functional]{construction of $\lambda_u$}
already establishes its well-definedness and linearity.
Consider the linear map
\[
 \F_2^{(h-2)\times(h-2)}\longrightarrow R_u(X_0)^*,
 \qquad
 C\longmapsto\bigl(M\mapsto\tr(\trans{C}M)\bigr).
\]
This map is surjective.  Indeed, choose a basis of $R_u(X_0)$ and
extend it to a basis of $\F_2^{(h-2)\times(h-2)}$.  Any linear
functional on $R_u(X_0)$ can then be extended to the full matrix space
by assigning the value zero to the additional basis matrices.

Every such extension is represented by the trace pairing.  To see
this, first note that the trace pairing on the full matrix space is
nondegenerate.  If $C=(c_{ij})\ne0$, choose an entry with $c_{ij}=1$ and let $E_{ij}$
be the matrix with its only nonzero entry in position $(i,j)$.  Then
\[
 \tr(\trans{C}E_{ij})=c_{ij}=1.
\]
Thus the map from the full matrix space to its dual induced by the
trace pairing is injective and, since the two spaces have the same
dimension, is an isomorphism. This proves the claimed surjectivity.~\footnote{The extension of a linear functional
from a subspace and its representation by a nondegenerate bilinear
form are standard finite-dimensional duality results; see
\cite[Section~3.5, Exercise~12, and Section~10.1]{HoffmanKunze}.}

Applying surjectivity to $\lambda_u$ gives a matrix $C_0\in\cC_u$.
Moreover,
\[
 \cC_u=C_0+
 \left\{C\in\F_2^{(h-2)\times(h-2)}:
 \tr(\trans{C}M)=0
 \text{ for every }M\in R_u(X_0)\right\}.
\]
Hence $\cC_u$ is an affine space whose translation space is the kernel
of the displayed surjective map.  Its image has dimension
$\dim R_u(X_0)^*=\dim R_u(X_0)$, so rank--nullity gives
\[
 \dim\cC_u=(h-2)^2-\dim R_u(X_0).
\]
\end{proof}

With the invertible-matrix indicator, we derive the upper bound for $\dim R_u(X_0)$.

\begin{proposition}[Fourier bound for $R_u(X_0)$]\label{p12:lem:affine-singular}
For every $0\ne u\in S_0$, let $\lambda_u$ and $\cC_u$ be as in
\cref{p12:lem:trace-representation}. Let $\rho_u=\min\{\rk C:~C\in\cC_u\}$. Then 
\begin{equation}\label{p12:eq:affine-singular}
 \dim R_u(X_0)
 \le (h-2)^2+(h-2)+1-(h-2)\rho_u+\binom{\rho_u}{2}.
\end{equation}
\end{proposition}

\begin{proof}
Since $u\ne0$ and $b$ is nondegenerate, there is $v\in X_0$ such
that $b(u,v)=1$.  Hence $\lambda_u(R_u(v))=1$, so $\lambda_u$ is
nonzero.  Thus $0\notin\cC_u$, which gives $\rho_u\ge1$.  Since
$\cC_u$ is nonempty and consists of $(h-2)\times(h-2)$ matrices, we
also have $\rho_u\le h-2$.

Let $T_u=\ker\lambda_u=\{M\in R_u(X_0):\lambda_u(M)=0\}$.  Since $\lambda_u$ is nonzero, its image is
$\F_2$. Thus,
\[
 \dim T_u=\dim R_u(X_0)-1.
\]
The Boolean polynomial $M\mapsto\det(I+M)$ restricted to $T_u$ has
degree at most $h-2$ and takes the value $1$ at zero.  Let
\[
 N_0=|\{M\in T_u:I+M\in\GL_{h-2}(2)\}|.
\]
Then \cref{p12:lem:RM-support} gives
\begin{equation}\label{p12:eq:N0-lower}
 N_0\ge
 2^{\dim R_u(X_0)-(h-2)-1}.
\end{equation}

Note that we can write $N_0$ as the indicator sum $N_0=\sum_{M\in T_u}1_{\GL_{h-2}(2)}(I+M)$.
Recall from \eqref{p12:eq:Riccati-constraint} that, for $M\in R_u(X_0)$, $\lambda_u(M)=1$ implies $\det(I+M)=0$.
Thus, the coset $\{M\in R_u(X_0):\lambda_u(M)=1\}$ contains no
$M$ for which $I+M$ is invertible. We obtain
\begin{equation}\label{p12:eq:character-sum-N0}
\begin{split}
&\sum_{M\in R_u(X_0)}1_{\GL_{h-2}(2)}(I+M)(-1)^{\lambda_u(M)}\\
=&
 \sum_{\substack{M\in R_u(X_0)\\ \lambda_u(M)=0}}
 1_{\GL_{h-2}(2)}(I+M)
 -
 \sum_{\substack{M\in R_u(X_0)\\ \lambda_u(M)=1}}
 1_{\GL_{h-2}(2)}(I+M)\\
 =&\sum_{M\in T_u}
 1_{\GL_{h-2}(2)}(I+M)\\
 =& N_0.
 \end{split}
\end{equation}
With the unnormalized Fourier transform in \cref{p12:lem:GL-Fourier},
Fourier inversion gives, for every $M\in\F_2^{(h-2)\times(h-2)}$,
\[
\begin{split}
 1_{\GL_{h-2}(2)}(I+M)
 &=2^{-(h-2)^2}\sum_{C\in \F_2^{(h-2)\times(h-2)}}
 \widehat{1_{\GL_{h-2}(2)}}(C)
 (-1)^{\tr(\trans{C}(I+M))}\\
&=2^{-(h-2)^2}\sum_{C\in \F_2^{(h-2)\times(h-2)}}
 \widehat{1_{\GL_{h-2}(2)}}(C)(-1)^{\tr C}
 (-1)^{\tr(\trans{C}M)}.
\end{split}
\]
Substitute this into the character-sum expression for $N_0$ as in~\eqref{p12:eq:character-sum-N0}.
Interchanging the sums gives
\[
 N_0=2^{-(h-2)^2}\sum_{C\in \F_2^{(h-2)\times(h-2)}}
 \widehat{1_{\GL_{h-2}(2)}}(C)(-1)^{\tr C}
 \sum_{M\in R_u(X_0)}
 (-1)^{\tr(\trans{C}M)+\lambda_u(M)}.
\]
Note that the inner character sum is $2^{\dim R_u(X_0)}$ precisely when
$C\in\cC_u$, and vanishes otherwise.  Therefore,
\begin{equation}\label{p12:eq:N0-fourier-exact}
 N_0=
 2^{\dim R_u(X_0)-(h-2)^2}
 \sum_{C\in\cC_u}
 \widehat{1_{\GL_{h-2}(2)}}(C)(-1)^{\tr C}.
\end{equation}
By \cref{p12:lem:trace-representation}, $|\cC_u|=2^{(h-2)^2-\dim R_u(X_0)}$.
Applying the triangle inequality to
\eqref{p12:eq:N0-fourier-exact} gives
\[
 N_0\le
 \max_{C\in\cC_u}
 \bigl|\widehat{1_{\GL_{h-2}(2)}}(C)\bigr|.
\]
By the explicit formula and rank monotonicity in
\cref{p12:lem:GL-Fourier}, the largest absolute Fourier coefficient in the maximum
therefore occurs at rank $\rho_u$.  Using also
$\pi_{h-2-\rho_u}(2)\le1$, we obtain
\[
 N_0\le
 \max_{C\in\cC_u}
 \bigl|\widehat{1_{\GL_{h-2}(2)}}(C)\bigr|
 =2^{(h-2)(h-2-\rho_u)+\binom{\rho_u}{2}}
   \pi_{h-2-\rho_u}(2)
 \le
 2^{(h-2)(h-2-\rho_u)+\binom{\rho_u}{2}}.
\]
Comparing this with \eqref{p12:eq:N0-lower} proves
\eqref{p12:eq:affine-singular}.
\end{proof}

\subsection{\texorpdfstring{Proof of~\cref{p12:prop:private-rank}}{Proof of the private-witness rank estimate}}
Now we are ready to prove~\cref{p12:prop:private-rank}, which we restate here.
\begin{trivlist}
\item[\hskip\labelsep{\bfseries \Cref*{p12:prop:private-rank} (restated).}]
\itshape
Let $X$ be an $R$-dimensional $\F_2$-space, let
$\cB\le\Alt(X)$ be minimally $h$-hollowing with
$\dim\cB=\binom h2$, and let $H\in\Gr(h,X)$ satisfy
$\cB_H=\linspan_{\F_2}\{A\}$ for some nonzero $A\in\cB$.  Then
\[
 \rk A\le \frac{10-\sqrt{10}}9 \cdot h^2+O(h).
\]
\end{trivlist}

\begin{proof}[Proof of \cref{p12:prop:private-rank}]
By~\eqref{p12:eq:rank-A-X0}, $\rk A\leq 2h+\dim X_0$, using
\hyperref[p12:par:private-complement]{the chosen complement $X_0$}. If $\dim X_0\le h-2$, then
\eqref{p12:eq:rank-A-X0} immediately gives
\[
 \rk A\le2h+\dim X_0\le2h+(h-2)=O(h).
\]

Suppose from now on that $\dim X_0\ge h-1$. Then \eqref{p12:eq:S0-lower} gives
$|S_0|\ge 2^{\dim X_0-(h-2)}\ge2$. Since $P(0)=0$, we have $\det(I+P(0))=1$, so $0\in S_0$ and hence $S_0\setminus\{0\}$ is nonempty.
Transferring the hollowing property of $\cB$ to determinant constraints on the evaluation matrix of $\cB$, we then use a low-degree polynomial argument to prove $\dim X_0\leq \dim R_u(X_0)+2(h-2)$ (\cref{p12:lem:fibre-kernel}) for any nonzero $u\in S_0$. The Fourier-analysis argument then shows that for any nonzero $u\in S_0$,
\begin{equation}\label{p12:eq:X0-rho-one}
   \dim X_0\leq \dim R_u(X_0)+2(h-2)
 \le (h-2)^2+3(h-2)+1-(h-2)\rho_u+\binom{\rho_u}{2},
\end{equation}
where $\rho_u$ is the minimal rank of matrices in the affine space $\cC_u$ (\cref{p12:lem:affine-singular}). 

We now discuss possible choices of $\rho_u$ and derive the stated bound for $\rk A$. For each $0\neq u\in S_0$, with $S_0$ defined in~\eqref{p12:eq:S_0},
choose one minimum-rank representative $C_u\in\cC_u$. Since $\lambda_u\ne0$, we have $0\notin\cC_u$, and hence
$1\leq\rho_u=\rk C_u\leq h-2$. 

Recall from~\eqref{p12:eq:R_uv} that $G_u=Q(u)(I+P(u))^{-1}$
and $R_u(v)=Q(v)-G_uP(v)$. Let 
$D_u=\trans{G_u}C_u$.  By~\cref{p12:lem:trace-representation},
\[
 b(u,v)=\tr(\trans{C_u}(Q(v)-G_uP(v)))=\tr(\trans{C_u}Q(v))-\tr(\trans{D_u}P(v)).
\]
Note that the map
\[
 S_0\setminus\{0\}\longrightarrow
 \F_2^{(h-2)\times(h-2)}\times\F_2^{(h-2)\times(h-2)},
 \qquad u\longmapsto(C_u,D_u),
\]
is injective.  Indeed, suppose that
$(C_u,D_u)=(C_{u'},D_{u'})$.  Then, for every $v\in X_0$,
\[
\begin{split}
 b(u,v)
 &=\tr(\trans{C_u}Q(v))-\tr(\trans{D_u}P(v))\\
 &=\tr(\trans{C_{u'}}Q(v))-\tr(\trans{D_{u'}}P(v))
 =b(u',v).
\end{split}
\]
Thus $b(u-u',v)=0$ for every $v\in X_0$.  Since $b$ is
nondegenerate on $X_0$, we have $u-u'=0$, and hence $u=u'$.

Every matrix $C\in\F_2^{(h-2)\times(h-2)}$ of rank at most
$\rho$ has its row space contained in some $\rho$-dimensional
subspace of $\F_2^{h-2}$.
There are $\qbinom{h-2}{\rho}{2}$ such subspaces.
For each one, the $h-2$ rows of $C$ have at most
$2^{(h-2)\rho}$ possible choices. Hence,
by~\cref{p12:lem:universal-field}\textup{(i)},
\[
 \#\{C:\rk C\le\rho\}
 \le \qbinom{h-2}{\rho}{2}\,2^{(h-2)\rho}
 \le \pi_\infty(2)^{-1}
       2^{\rho(2(h-2)-\rho)}.
\]
Since $D_u=\trans{G_u}C_u$, every pair $(C_u,D_u)$ is counted in
this set.

Partition $S_0\setminus\{0\}$ into the $h-2$
possible exact ranks.  Some class has at least
\[
 \frac{|S_0|-1}{h-2}
 \ge\frac{2^{\dim X_0-(h-2)}-1}{h-2}
 \ge\frac{2^{\dim X_0-(h-2)-1}}{h-2}
\]
elements.  If its common rank is $\rho$, injectivity of
$u\mapsto(C_u,D_u)$ and the preceding pair count give
\begin{equation}\label{p12:eq:X0-rho-two}
 \dim X_0\le\rho(3(h-2)-\rho)+O(h).
\end{equation}
Combining \eqref{p12:eq:X0-rho-one} and \eqref{p12:eq:X0-rho-two}, and putting
$x=\rho/(h-2)$, gives
\[
 \frac{\dim X_0}{(h-2)^2}
 \le
 \min\left\{3x-x^2,1-x+\frac{x^2}{2}\right\}+O(h^{-1}).
\]
For $0\le x\le1$, the first function in the minimum is increasing and
the second is decreasing. Hence the maximum of this minimum is attained
at their intersection $x_0=\frac{4-\sqrt{10}}3$, with value
$\frac{10-\sqrt{10}}9$.
Thus we have
\[
 \rk A\le2h+\dim X_0\le \frac{10-\sqrt{10}}9\cdot h^2+O(h).
\]
\end{proof}

\section{Bound \texorpdfstring{$f(p^m)$}{f(pm)} in the nonsaturated regime}
\label{p12:sec:endpoint-absorption}

We now prove~\cref{p12:thm:nonsaturated-map-count} using the
matrix-space estimates established above. We first relate the map
count to the number of dual spaces of alternating forms
(cf.~\cite[Proposition~5.5]{BNV}).

\phantomsection\label{p12:par:nonsaturated-substitution}
Specialize the spaces from Section~\ref{p12:sec:hollowing} to
$X=\F_p^{r_1}$ and $W=\F_p^{r_2}$, and put $h=s-1$.
By \cref{p12:prop:sims-hollowing-dictionary}, a surjective linear map
$\beta:\wedge^2X\to W$ corresponds to an alternating bilinear map of
Sims rank $s$ if and only if its dual
image $\cA=\beta^*(W^*)\le\Alt(X)$
is an $r_2$-dimensional $h$-hollowing space that is
$(h+1)$-faithful.  Conversely, every such space $\cA$, together with
an isomorphism $W^*\to\cA$, determines one such map.  Consequently,
\begin{equation}\label{p12:eq:Np-hollowing-correspondence}
\begin{split}
 p^{N_p(r_1,r_2,s)}
 &={}|\GL_{r_2}(p)|
 \mathrel{}\cdot
 \left|\left\{\cA\le\Alt(X):
 \begin{array}{l}
  \dim X=r_1,\ \dim\cA=r_2,\ \\
  \cA\text{ is $h$-hollowing and $(h+1)$-faithful}
 \end{array}
 \right\}\right|
 \\
 &\leq{}|\GL_{r_2}(p)|
 \mathrel{}\cdot
 \left|\left\{\cA\le\Alt(X):
 \begin{array}{l}
  \dim X=r_1,\ \dim\cA=r_2,\ \\
  \cA\text{ is $h$-hollowing}
 \end{array}
 \right\}\right|.
\end{split}
\end{equation}
We are now ready to prove~\cref{p12:thm:nonsaturated-map-count}.
\begin{proof}[Proof of~\cref{p12:thm:nonsaturated-map-count}]
For each fixed dual image $\cA$, the choices of the isomorphism
$\beta^*:W^*\to\cA$ number $|\GL_{r_2}(p)|\le p^{r_2^2}$.
In case~\textup{(i)}, apply~\cref{p12:prop:exponent-away} with
$(q,R,M,h)=(p,r_1,r_2,s-1)$ in
\eqref{p12:eq:Np-hollowing-correspondence}. The factor $|\GL_{r_2}(p)|$ is absorbed by
$O(r_1^2+r_2^2)$, proving~\eqref{p12:eq:ordinary-map-bound}.

For case~\textup{(ii)}, we work with $h=s-1$. First suppose that $h\ge h_0$, where
$h_0$ is the absolute constant in~\cref{p12:prop:endpoint-exponent}.
Split the dual images into those that are minimally $h$-hollowing
and those that are not. By~\cref{p12:lem:minimal-reduction}, the
nonminimal subfamily, including the target-isomorphism choices,
has logarithmic count at most
\[
 (r_2-h)\binom{r_1}{2}+O(r_1^2+r_2^2).
\]
For the minimal subfamily, \cref{p12:prop:endpoint-exponent} and
the same target-isomorphism count give the upper exponent
\[
 (r_2-h)\binom{r_1}{2}
 +h a_*(2r_1-2a_*-1)+O(h r_1+h^4).
\]
Here the target factor is absorbed by $O(h^4)$ because
$r_2=\binom h2$. Moreover,
$h r_1+h^4=O(r_1^2+r_2^2)$, and the rank-correction term
$h a_*(2r_1-2a_*-1)$ is nonnegative. Adding the two subfamily
counts therefore proves~\eqref{p12:eq:binary-endpoint-map-bound}.
The additional logarithmic cost of adding two counts is at most
$1$.

If $h<h_0$, the trivial count gives
\[
 N_2(r_1,r_2,s)\le r_2\binom{r_1}{2}
 =(r_2-h)\binom{r_1}{2}+O(r_1^2),
\]
with an absolute implied constant since $h_0$ is absolute.
As the rank-correction term is nonnegative, this also proves
\eqref{p12:eq:binary-endpoint-map-bound} in the remaining cases by replacing $h$ with $s-1$.
\end{proof}

We now use our bound on $N_p(r_1,r_2,s)$ to count isomorphism classes of nonabelian groups of order $p^m$ in the nonsaturated range $r_2\leq \binom{s-1}{2}$.
For large $s$, it yields a saving of order $sm^2$ in the
group-counting exponent. For bounded $s$, the nonsaturated
condition $r_2\le\binom{s-1}{2}$ makes the second-layer rank bounded.
We handle that range separately using~\eqref{p12:eq:total-relation-bound}
from Part~I. We first record an estimate for the binary-endpoint
exponent obtained from Sims' presentation bound~\eqref{p12:eq:book-master},
followed by a bound for groups with bounded second-layer rank.

\begin{lemma}[Endpoint absorption]\label{p12:lem:endpoint-absorption}
Let $m,R,h$ be integers satisfying $m\ge2$, $h\geq 1$, $R\geq 0$ and $R+\binom h2\le m$.
Denote $w=m-R-\binom h2$
and suppose $0\le\theta\le4/5$.
Then
\begin{equation}\label{p12:eq:endpoint-full-exponent}
\begin{split}
 &\frac12R^2(m-R-h)
 +\left(\frac12h^3-h^2\right)w
 +\frac h2w^2
 +\theta h^3R-\frac{\theta^2}{2}h^5\\
 \le& \frac{2}{27}m^3-\frac{11}{1602}hm^2+O(m^2).
 \end{split}
\end{equation}
\end{lemma}

\begin{proof}
Let $E$ denote the left-hand side of
\eqref{p12:eq:endpoint-full-exponent}.  Substituting
$w=m-R-\binom h2$ and expanding gives
\[
 E=\frac12R^2(m-R)
 +\frac h2(m-\theta h^2)(m+\theta h^2-2R)+\frac{h^2}{2}(R-m)
 -\frac{h^3}{8}(h-1)(h-3).
\]
Note that
\[
 \frac12R^2(m-R)
 =\frac{2}{27}m^3
  -\frac12\left(R+\frac m3\right)
   \left(R-\frac{2m}{3}\right)^2.
\]
For the remaining terms, write
$R=(R-2m/3)+2m/3$ and expand.  This gives
\begin{align*}
 &\frac h2(m-\theta h^2)(m+\theta h^2-2R)
 +\frac{h^2}{2}(R-m)
 -\frac{h^3}{8}(h-1)(h-3)\\
 ={}&-h\left(m-\frac h2-\theta h^2\right)
       \left(R-\frac{2m}{3}\right)-\frac h{24}\left(
       4(m-2\theta h^2)^2+(3-4\theta^2)h^4
      \right)
 -\frac16h^2m+\frac12h^4-\frac38h^3.
\end{align*}
Adding these two identities yields
\begin{align*}
 E={}\frac{2}{27}m^3&+\underbrace{
   -\frac12\left(R+\frac m3\right)
       \left(R-\frac{2m}{3}\right)^2-h\left(m-\frac h2-\theta h^2\right)
       \left(R-\frac{2m}{3}\right)
  }_{\text{$R$-dependent group}}\\
 &+\underbrace{
  -\frac h{24}\left(
       4(m-2\theta h^2)^2+(3-4\theta^2)h^4
      \right)
  }_{\text{$\theta$-square group}}\\
 &+\underbrace{
  -\frac16h^2m+\frac12h^4-\frac38h^3
  }_{\text{lower-order remainder}}.
\end{align*}
For the negative contribution of the ``$\theta$-square group'', we further have
\begin{align*}
 &4(m-2\theta h^2)^2+(3-4\theta^2)h^4=(12\theta^2+3)
 \left(h^2-\frac{8\theta m}{12\theta^2+3}\right)^2
 +\frac{4(3-4\theta^2)}{12\theta^2+3}m^2
 \ge\frac{44}{267}m^2.
\end{align*}
For the last inequality, use
$3-4\theta^2\ge11/25$ and $12\theta^2+3\le267/25$.
Thus the ``$\theta$-square group'' is at most
$-11hm^2/1602$.  The ``lower-order remainder'' is at most
$\frac12h^4$ since $h,m\geq0$.

We now deal with the ``$R$-dependent group''.
First note that $h\leq m$: this follows from $m\geq2$ when $h\leq2$, and from
$h\leq\binom h2\leq m$ when $h\geq3$.  Since
$h(h-1)\leq2m$, it follows that
$h^2\leq2m+h\leq3m$.
Note that \[
\left|h\left(m-\frac h2-\theta h^2\right)\right|\leq h(m+\frac h2+\theta h^2)
\leq h\left(m+\frac{13}{10}h^2\right)
\leq\frac{49}{10}hm<5hm,
\]
where we use $h/2\leq h^2/2$, $h^2\leq3m$, and $0\leq \theta\leq4/5$. Since $R+m/3\ge m/3$, the ``$R$-dependent group'' term
\[
  -\frac12\left(R+\frac m3\right)
       \left(R-\frac{2m}{3}\right)^2-h\left(m-\frac h2-\theta h^2\right)
       \left(R-\frac{2m}{3}\right)
\]
is upper bounded by
\begin{align*}
 -\frac m6\left(R-\frac{2m}{3}\right)^2
 +5hm\left|R-\frac{2m}{3}\right|.
\end{align*}
The right-hand side equals
\[
 -\frac m6\left(\left|R-\frac{2m}{3}\right|-15h\right)^2
 +\frac{75}{2}h^2m,
\]
so it is at most $\frac{75}{2}h^2m$.

Combining these estimates, we obtain that
\begin{align*}
 E&\le\frac{2}{27}m^3-\frac{11}{1602}hm^2
       +\frac{75}{2}h^2m+\frac12h^4\le\frac{2}{27}m^3-\frac{11}{1602}hm^2+117m^2,
\end{align*}
where the second inequality again uses $h^2\leq3m$.
\end{proof}

For bounded Sims rank $s$, the negative term in
\eqref{p12:eq:endpoint-full-exponent} with $h=s-1$ need not absorb the quadratic
cost of the remaining presentation parameters.  The same issue
arises with the ordinary map bound.  We use the following lemma
for all these bounded-parameter cases.

\begin{lemma}[Bounded second-layer rank]\label{p12:lem:bounded-second-rank}
For every prime $p$ and integers $K\ge1$ and $m\ge3$, the number of
isomorphism classes of nonabelian groups of order $p^m$ with
$r_2\le K$ is at most $p^{\frac{2}{27}\frac{K}{K+1}m^3+(5K+11)m^2}$.
\end{lemma}

\begin{proof}
Fix $r_2\le K$.  By~\cref{p12:lem:hollowing-min-dim}, we have
$s\le r_2+1\le K+1$. Use the layer parameters from
\eqref{p12:eq:layer-parameters}, the core breadth sequence from
\eqref{p12:eq:core-breadth} and its total from
\eqref{p12:eq:breadth-total-definition}, and the defect
$\delta=s_1-r_1$ from~\eqref{p12:eq:optimization-parameters}.
In particular, $m_2=\sum_{i=2}^c s_i$ and
$|\boldsymbol{\fb}|=\sum_{i=2}^{c-1}\fb_i$.
We shall apply \cref{p12:prop:total-relation-bound} in this setting.
Fix recorded Sims presentation data, a compatible full power word
array, and a core breadth sequence as in that theorem. Let $\cP$ be
the set of distinct commutator coefficient arrays arising from
nonabelian groups of order $p^m$ with ordered Sims generating sets
realizing the recorded Sims presentation data and power word array, with this core
breadth sequence.
By~\eqref{p12:eq:core-tail-two-bounds} and
\eqref{p12:eq:breadth-total-le-deep}, we have
$T_{\core}\le m^2/2$ and $|\boldsymbol{\fb}|\le m$.
Using also $r_i+s_{i+1}-\fb_i+1\le2m+1\le3m$ and $\binom{r_1-s}{2}\ge\frac12r_1^2-\left(s+\frac12\right)r_1$, the bound in~\cref{p12:eq:total-relation-bound} can be applied and simplified as
\begin{equation}\label{p12:eq:bounded-second-rank-relations}
 \log_p|\cP|
 \le\frac12r_1^2(m_2-|\boldsymbol{\fb}|)+5sm^2.
\end{equation}
Indeed, discard the nonpositive term $-r_2m_2$ in
\eqref{p12:eq:total-relation-bound} and use
$\binom{r_1}{2}\le r_1^2/2$. The resulting positive error terms
are at most
\[
 \left(\frac s2+3s+s+\frac12\right)m^2\le5sm^2.
\]
Now we bound $m_2-|\boldsymbol{\fb}|$. Recall that $|\boldsymbol{\fb}|\geq\frac{m_3}s$ as in~\eqref{p12:eq:breadth-total-le-deep}. Thus,
\begin{equation}\label{eq: m_2-|b|}
  m_2-|\boldsymbol{\fb}|\leq m_2-\frac{m_3}{s}=\frac{sm_2-m_3}{s}=\frac{(s-1)m_2+s_2}{s},
\end{equation}
where the last equality uses $m_2-m_3=s_2$. 
To bound $s_2$, use the Frattini-quotient identification in
\eqref{p12:eq:frattini-quotients}: $|L_1/pL_1|=p^{r_1}$, so
$|pL_1|=p^{s_1-r_1}=p^\delta$.
Thus, every element $y$ of $pL_1$ satisfies $p^\delta y=0$. Moreover, for $x\in L_1$, $px\in pL_1$. Therefore, 
$p^{\delta+1}x=0$ for every $x\in L_1$. By biadditivity,
\[
 p^{\delta+1}[x,y]_L=[p^{\delta+1}x,y]_L=0
\]
for any $x,y\in L_1$. Since $L_2=\langle[x,y]_L:~x,y\in L_1\rangle$ as in~\eqref{p12:eq:sims-core-graded-generation}, the chosen $r_2$
generators of $L_2$ are killed by $p^{\delta+1}$. Every element of $L_2$
is consequently a sum of these generators with coefficients in
$\{0,\ldots,p^{\delta+1}-1\}$. Hence
$|L_2|\le p^{r_2(\delta+1)}$, giving $s_2\le r_2(\delta+1)$.

Together with~\eqref{eq: m_2-|b|} and 
$m_2=m-r_1-\delta$, we have
\begin{align*}
 m_2-|\boldsymbol{\fb}|
 &\le\min\left\{m-r_1-\delta,
 \frac{(s-1)(m-r_1)+(r_2-s+1)\delta+r_2}{s}\right\}\le\frac{r_2}{r_2+1}(m-r_1+1).
\end{align*}
To see the last inequality, the first affine function of $\delta$
is decreasing, and the second is nondecreasing since $r_2\ge s-1$.
They agree at $\delta=(m-r_1-r_2)/(r_2+1)\ge0$, where their
common value is $r_2(m-r_1+1)/(r_2+1)$. 

Substituting in \eqref{p12:eq:bounded-second-rank-relations} gives
\[
 \log_p|\cP|
 \le\frac{r_2}{r_2+1}\cdot\frac12r_1^2(m-r_1)
       +\left(5s+\frac12\right)m^2
 \le\frac{2}{27}\frac{r_2}{r_2+1}m^3
       +\left(5s+\frac12\right)m^2.
\]

Multiplying the explicit estimates in the proof of
\cref{p12:prop:sims-generating-data}, the number of choices of
recorded Sims presentation data is at most $(m^2 2^{2m})\,m^{2m}\,2^{m-1}
 =m^{2m+2}2^{3m-1}$. The breadth sequences contribute at most
$2^m$, by~\eqref{p12:eq:breadth-sequence-count}.
Together, these choices have at most $m^{2m+2}2^{4m-1}\le p^{4m^2}$
possibilities, where we use $p\ge2$, $\log_2m\le m$ and $m\ge3$.

The \hyperref[p12:par:naive-power-count]{naive power-word count in
Section~\ref{sec:bijuv}} gives at most $p^{m^2}$ further choices
for the full power word arrays. By
\cref{thm:sims_presentation_determines_isomorphism_classes_}, the
recorded Sims presentation data and these two arrays determine the group up to
isomorphism. Thus, for $1\le r_2\le K$ and $s\le K+1$, multiplying
these bounds gives a group-counting exponent of at most
\[
 \frac{2}{27}\frac{K}{K+1}m^3
 +\left(5(K+1)+\frac12+4+1\right)m^2
 \le\frac{2}{27}\frac{K}{K+1}m^3+(5K+11)m^2.
\]
\end{proof}

We are now ready to bound the number of isomorphism classes of nonabelian groups of order $p^m$ in the nonsaturated range $r_2\leq\binom{s-1}{2}$.
\begin{theorem}[The nonsaturated range]
\label{p12:cor:nonsaturated-group}
For every fixed $A>0$, there is an integer $m_A$, independent of
$p$, such that, for every prime $p$ and every $m\ge m_A$, the
number $f_{ns}(p^m)$ of isomorphism classes of nonabelian groups of order
$p^m$ with $r_2\le\binom{s-1}{2}$ satisfies
\begin{equation}\label{eq:nonsat group enumeration}
  f_{ns}(p^m)\leq p^{\frac{2}{27}m^3-Am^2}.
\end{equation}
\end{theorem}

\begin{proof}[Proof]
Fix recorded Sims presentation data as counted
in~\cref{p12:prop:sims-generating-data}, and a compatible full power
word array $(b(i,j,u,v))$ as counted
in~\cref{p12:lem:power-ledger}. Let $\cP$ denote the set of distinct
commutator coefficient arrays arising from nonabelian groups of
order $p^m$ with ordered Sims generating sets realizing these
recorded Sims presentation data and this power word array.
By \cref{p12:lem:hollowing-min-dim}, every dual image occurring here
satisfies $h=s-1\le r_2$.  

First suppose that $p\ge3$, or that $p=2$ and
$r_2<\binom h2$. By~\cref{p12:thm:nonsaturated-map-count}\textup{(i)},
we have $N_p(r_1,r_2,s)\leq(r_2-h)\binom{r_1}{2}+O(m^2)$,
where the $O(r_1^2+r_2^2)$ error has been absorbed into $O(m^2)$,
since $r_1+r_2\le m$.
Let $w=m-r_1-r_2$. Inserting this estimate
in~\eqref{p12:eq:book-master} gives
\begin{equation}\label{eq:log cp}
\begin{split}
  \log_p|\cP|
 \le{}&(r_2-h)\binom{r_1}{2}
 +\left(\binom{r_1}{2}+hr_2-\frac{h^2}{2}\right)w
 +\frac h2w^2+O(m^2)\\
 ={}&\binom{r_1}{2}(r_2+w-h)
 +\left(hr_2-\frac{h^2}{2}\right)w
 +\frac h2w^2+O(m^2)\\
 ={}&\frac12r_1^2(r_2+w-h)
 +\left(hr_2-\frac{h^2}{2}\right)w
 +\frac h2w^2+O(m^2),
\end{split}
\end{equation}
where the last equality absorbs the term $-r_1(r_2+w-h)/2$ into $O(m^2)$ since it
has absolute value at most $m^2/2$.

Following \cite[Section 5.5]{BNV}, let $x=r_1/m$, $y=r_2/m$, $z=(m-r_1-r_2)/m$ and $u=h/m$. Thus, $x+y+z=1$ and $0\le u\le\min\{x,y\}$. Define
\[
 B(x,y,z,u):=
 \frac{x^2}{2}(y+z-u)
 +(uy-\frac{u^2}{2})z+\frac{uz^2}{2}
 =\frac{1-x-u}{2}\bigl(x^2+u(1-x)\bigr)
  -\frac{uy(y-u)}{2}.
\]
We then have
$\log_p|\cP|\le m^3B(x,y,z,u)+O(m^2)$.
We strengthen the cubic comparison in \cite[Lemma~A.2]{BNV} as follows. Note that
\begin{equation}\label{p12:eq:ordinary-cubic-deficit}
 \frac2{27}-B(x,y,z,u)
 =\frac12\left(x-\frac23\right)^2\left(x+\frac13\right)
  +\frac u2(2x-1+y^2+uz).
\end{equation}
For $x\ge7/12$, the right-hand side is at least $u/12$.
For $1/2\le x\le7/12$, its second term is nonnegative, while
the first is at least $\frac2{27}-\frac12\left(\frac7{12}\right)^2\frac5{12}
 =\frac{11}{3456}$.

For $x\le1/2$, we have $x^2+u(1-x)-\frac{x+u}{2}
 =\left(x-\frac12\right)(x-u)\le0$.
Thus,
\[
B(x,y,z,u)\le\frac{(x+u)(1-x-u)}{4}\le\frac{1}{16}.
\]
Since $u\le1$, these three ranges give
\begin{equation}\label{p12:eq:ordinary-group-deficit}
 m^3 B(x,y,z,u)
 \le\frac2{27}m^3-\frac{11}{3456}hm^2.
\end{equation}
Thus, when $p\ge3$, or when $p=2$ and
$r_2<\binom h2$, we have
\begin{equation}\label{p12:eq:ordinary-relation-deficit}
 \log_p|\cP|\le\frac{2}{27}m^3-\frac{11}{3456}hm^2+O(m^2).
\end{equation}

It remains to treat $p=2$ and $r_2=\binom h2$, using
$h=s-1$ from the \hyperref[p12:par:nonsaturated-substitution]{substitution
at the start of this section}. By~\cref{p12:thm:nonsaturated-map-count}\textup{(ii)}, we have
\[
N_2(r_1,\binom{h}{2},h+1)\leq (\binom{h}{2}-h)\binom{r_1}2
 +h a_*(2r_1-2a_*-1)
 +O(m^2),
\]
where $a_*=\min\left\{
       \left\lfloor\frac{2h^2}{5}\right\rfloor,
       \left\lfloor\frac{r_1}2\right\rfloor
      \right\}$; here the $O(r_1^2+r_2^2)$ error is absorbed into
$O(m^2)$ because $r_1+r_2\le m$. 

Let $w=m-r_1-\binom{h}{2}$. Inserting the bound for
$N_2(r_1,\binom{h}{2},h+1)$ into~\eqref{p12:eq:book-master} gives
\begin{align*}
 \log_2|\cP|
 \leq(\binom{h}{2}-h)\binom{r_1}2+h a_*(2r_1-2a_*-1)+\binom{r_1}2 w
  +\left(h\binom{h}{2}-\frac{h^2}{2}\right)w
  +\frac h2w^2+O(m^2).
\end{align*}
Put $\theta=2a_*/h^2\leq4/5$. Then the rank-variety term
$h a_*(2r_1-2a_*-1)$ becomes 
\[
 h a_*(2r_1-2a_*-1)
 =\theta h^3r_1-\frac{\theta^2}{2}h^5-\frac\theta2h^3.
\]
Also, since $w=m-r_1-\binom{h}{2}$,
\begin{align*}
 \bigl((\binom{h}{2}-h)+w\bigr)\binom{r_1}{2}
 &= (m-r_1-h)\frac{r_1(r_1-1)}2=\frac12r_1^2(m-r_1-h)-\frac12r_1(m-r_1-h),\\
 h\binom{h}{2}-\frac{h^2}{2}
 &=\frac{h^2(h-1)}2-\frac{h^2}{2}
  =\frac12h^3-h^2.
\end{align*}
Consequently, the above upper bound on $\log_2|\cP|$ can be rewritten as
\begin{align*}
\frac12r_1^2(m-r_1-h)
 +\left(\frac12h^3-h^2\right)w
 +\frac h2w^2+\theta h^3r_1-\frac{\theta^2}{2}h^5-\frac12r_1(m-r_1-h)-\frac\theta2h^3
 +O(m^2).
\end{align*}
Applying~\cref{p12:lem:endpoint-absorption}, and note that $r_1\le m$ and $h^2=O(m)$, we have 
\begin{equation}\label{p12:eq:endpoint-group-deficit}
 \log_2|\cP|\le\frac{2}{27}m^3-\frac{11}{1602}hm^2+O(m^2).
\end{equation}

The recorded Sims presentation data
have $p^{O(m\log m)}$ choices by
\cref{p12:prop:sims-generating-data}. The full power word arrays contribute at
most $p^{m^2}$ choices by~\cref{p12:lem:power-ledger}.

For the commutator coefficient array, fix an integer $H\ge3$ and let $K=\binom{H-1}{2}$. For $s-1=h\ge H$,
the commutator coefficient arrays contribute at most
$p^{\frac{2}{27}m^3-\frac{11}{3456}hm^2+O(m^2)}$.
This follows from~\eqref{p12:eq:ordinary-relation-deficit} and
\eqref{p12:eq:endpoint-group-deficit}, since $11/1602\ge11/3456$.
Thus there is an absolute constant $C>0$, independent of $p$, $m$
and $H$, such that the
number of isomorphism classes of nonabelian groups in the nonsaturated range with $h\ge H$ is at most $p^{2m^3/27-(11H/3456-C)m^2}$.

For $h<H$, the nonsaturated condition gives
$r_2\le\binom{H-1}{2}=K$, so
\cref{p12:lem:bounded-second-rank} applies. Combining these two
group counts, both of which include all recorded Sims presentation
data and compatible relation arrays, gives
\begin{equation}\label{p12:eq:nonsaturated-split-count}
 f_{\rmns}(p^m)
 \le p^{\frac{2}{27}m^3-(\frac{11}{3456}H-C)m^2}+p^{\frac{2}{27}\frac{K}{K+1}m^3+(5K+11)m^2}.
\end{equation}

Now for any $A>0$, fix $H\ge3$ such that $\frac{11}{3456}H-C\ge A+1$. Then we choose an integer $m_A\ge3$ satisfying $m_A\ge\frac{27}{2}(K+1)(A+5K+12)$ with $K=\binom{H-1}{2}$.
For every $m\ge m_A$, the first summand in
\eqref{p12:eq:nonsaturated-split-count} is at most
$p^{2m^3/27-(A+1)m^2}$. The choice of $m_A$ then gives
\[
 \frac{2m^3}{27(K+1)}-(5K+11)m^2\ge(A+1)m^2,
\]
so the second summand satisfies the same bound. Hence
\[
 f_{\rmns}(p^m)
 \le2p^{\frac{2}{27}m^3-(A+1)m^2}
 \le p^{\frac{2}{27}m^3-Am^2}.
\]
All constants and choices are independent of $p$. This proves the theorem.
\end{proof}

\clearpage

\clearpage
\addtocontents{toc}{\protect\newpage}
\part[Sharp enumeration of \texorpdfstring{$p$}{p}-groups and domination by \texorpdfstring{$2$}{2}-groups of Frattini class \texorpdfstring{$2$}{2}]{Sharp enumeration and domination by\\
\texorpdfstring{$2$}{2}-groups of Frattini class \texorpdfstring{$2$}{2}}
\label{part:sharp-enumeration}

\section{Optimizing the joint array bound}
\label{p3:sec:statement}
Our primary goal is to estimate the nonabelian group count
$f_{nab}(p^m)$ defined in~\eqref{p12:eq:nonabelian-count}. Following Sims' enumeration framework
in~\cref{p12:par:fixed-sims-data}, we count complete choices of
recorded Sims presentation data, power word arrays, and
commutator coefficient arrays arising from these groups.
We use the \hyperref[p12:par:recorded-sims-data]{recorded Sims presentation data}
defined before~\cref{p12:prop:sims-generating-data}, including the
relative order array.
For each fixed choice of recorded Sims presentation data, we count the distinct pairs
consisting of a full power word array and a commutator coefficient
array, with the index ranges in
\cref{thm:sims_presentation_determines_isomorphism_classes_}.
Each complete choice of recorded Sims presentation data and these two arrays determines
one group isomorphism class.

We retain the lower central series and layer parameters from
\eqref{p12:eq:lower-central-series}--\eqref{p12:eq:layer-parameters}:
$L_i=G_i/G_{i+1}$, $r_i=d(L_i)$, $s_i=\log_p|L_i|$, and
$m_i=\log_p|G_i|=\sum_{j=i}^c s_j$, with the same convention
$m_i=0$ for $i>c$.

In Part~I, we analyze the number of distinct pairs
consisting of a full power word array and a commutator coefficient
array, conditioned on fixed recorded Sims presentation data
and a fixed core breadth sequence. We use the two disjoint
families introduced in~\eqref{p12:eq:group-count-split}:
\begin{itemize}
\item The nonsaturated range, consisting of nonabelian
$p$-groups with
$r_2\le\binom{s-1}{2}$;
\item The saturated range, consisting of nonabelian
$p$-groups with
$\binom{s-1}{2}<r_2\le\binom{s}{2}$.
\end{itemize}

Part~II bounds the nonsaturated count $f_{ns}(p^m)$ from
\eqref{p12:eq:group-count-split}. \cref{p12:cor:nonsaturated-group}
shows that, for every fixed
$A>0$, there exists $m_A\geq 0$, such that when $m\geq m_A$,
\[
 f_{ns}(p^m)\le p^{\frac{2}{27}m^3-Am^2}
\]
Choose $A>4/9$. Together with Higman's lower bound on $f(p^m)$~\cite{Higman} (cf.~\eqref{eq: old estimate}), 
we have 
\[
 \frac{f_{ns}(p^m)}{f(p^m)}
 \le p^{-(A-\frac49)m^2}
 \longrightarrow0
\]
when $m\to\infty$. Thus, nonabelian $p$-groups in the nonsaturated range contribute a negligible proportion among all $p$-groups, uniformly in the prime $p$.

In this section, we estimate the saturated count $f_{sat}(p^m)$
from~\eqref{p12:eq:group-count-split}, recalling that its parameters satisfy
\begin{equation}\label{p3:eq:saturation}
 \binom{s-1}{2}<r_2\le\binom{s}{2},~~2\le s\le r_1.
\end{equation}
Since $r_2=d(L_2)\le\log_p |L_2|= s_2=m_2-m_3$, the saturated range gives
\[
 s^2-3s+2=2\binom{s-1}{2}<2r_2\leq 2s_2\le2(m_2-m_3).
\]
Consequently, $s^2<2(m_2-m_3)+3s$, and for an absolute constant $C$,
\begin{equation}\label{p3:eq:saturated-core-size}
 s\leq C\sqrt{m_2}.
\end{equation}

We shall apply~\cref{p12:thm:joint-enumeration} to estimate $f_{sat}(p^m)$. Fix a prime $p$ and an integer $m\ge3$. As
in~\cref{p12:thm:joint-enumeration}, fix recorded Sims presentation
data as counted in~\cref{p12:prop:sims-generating-data}. 
We consider groups $G$ of order $p^m$ equipped with Sims
generating sets $(g_{i,j})_{i\in[c],j\in[r_i]}$ realizing
these recorded Sims presentation data and having a fixed core breadth sequence
$\boldsymbol{\fb}=(\fb_i)_{i=2}^c$ as defined
in~\eqref{p12:eq:core-breadth}.
In each realization, $H=\langle g_{1,1},\ldots,g_{1,s}\rangle$
is the chosen Sims core from~\cref{p12:par:fixed-sims-data}, and its
image $C_{\core}=HG_2/G_2$ is as in~\eqref{p12:eq:core-image}.

Use $\cE$ as in~\cref{p12:thm:joint-enumeration} for the base-$p$ logarithm of the
number of distinct pairs of full power word arrays and commutator
coefficient arrays arising from groups with compatible ordered Sims
generating sets realizing these fixed recorded Sims presentation data and breadths.
By~\cref{p12:prop:saturated-map-estimate}, the correction $\Delta_p$
in~\eqref{p12:eq:leading-map-correction} is zero in this range.
Thus~\cref{p12:thm:joint-enumeration} gives
\[
      \cE\le\frac12r_1^2m_2+\cX+3m_2m_3+\frac12m_2z_{\rmbal}
              +2m_2e+r_1\delta,
\]
where $\delta$, $e$ and $z_{\rmbal}$ are the parameters
from~\eqref{p12:eq:optimization-parameters}, and
$|\boldsymbol{\fb}|$ is the breadth sum
in~\eqref{p12:eq:breadth-total-definition}.
Recall $\cX$ and $T_{\core}$ from
\eqref{p12:eq:optimization-X} and~\eqref{p12:eq:core-tail-definition}:
\[
\cX=sT_{\core}
 +\sum_{i=2}^{c-1}\fb_i
   \left(s(r_i+s_{i+1}-\fb_i+1)-\binom{r_1-s}{2}\right)~~\text{with}~~T_{\core}=\sum_{i=2}^{c-2}r_i m_{i+2}.
\]
Note that $e\le\binom\delta2$ by~\cref{p12:lem:second-layer-defect}.
The breadth relations~\eqref{p12:eq:e-breadth} give
$s_{i+1}\le s\fb_i$ and
$\fb_i\le\min\{s_i,s_{i+1}\}$ for $2\le i\le c-1$.
From~\cref{lem:upper bound T_core}, we know that
\begin{equation}\label{p3:eq:breadth}
 T_{\core}\le\min\left\{
 \frac{m_2^2-s_2^2}{2},\quad
 s m_2(|\boldsymbol{\fb}|-\fb_2)\right\},
\end{equation}

We also recall the consequences
\eqref{p12:eq:X-breadth-total}--\eqref{p12:eq:X-rough}:
\begin{equation}\label{p3:eq:roughX}
\begin{split}
       \cX\leq& s\min\left\{\frac{m_2^2-s_2^2}{2},
                s m_2(|\boldsymbol{\fb}|-\fb_2)\right\}
       +\left(s(2m_2+1)-\binom{r_1-s}{2}\right)
        |\boldsymbol{\fb}|\\
        \leq&\frac{s m_2^2}{2}
       +|\boldsymbol{\fb}|
        \left(s(2m_2+1)-\binom{r_1-s}{2}\right).
\end{split}
\end{equation}

These estimates give rough estimate of the upper bound of $\cE$ as follows: For $\frac12r_1^2m_2$, since $r_1+m_2=m-\delta$,
\[
     \frac12r_1^2m_2=2\left(\frac{r_1}{2}\right)^2m_2\leq2\left(\frac{\frac{r_1}{2}+\frac{r_1}{2}+m_2}{3}\right)^3=\frac{2}{27}(m-\delta)^3. 
\]
For $\cX$, combining $s\leq C\sqrt{m_2}$,
$m_3/s\leq |\boldsymbol{\fb}|\le m_3\le m_2$ as in~\eqref{p12:eq:breadth-total-le-deep} and
\eqref{p3:eq:roughX} gives $\cX\le O(m^{5/2})$.
The remaining terms
$3m_2m_3+\frac12m_2z_{\rmbal}+2m_2e+r_1\delta$
are $O(m^2)$. Consequently,
\[
 \cE\le\frac{2}{27}(m-\delta)^3+O(m^{5/2})
 \le\frac{2}{27}m^3+O(m^{5/2}).
\]
If $\delta$ is bounded, then
$e\le\binom{\delta}{2}$ is also bounded, and
$2m_2e+r_1\delta=O(m)$.

This rough estimate identifies the correct cubic coefficient.
To obtain the target exponent
$\frac{2}{27}m^3-\frac49m^2+O(m)$ for the group count,
we must control the remaining terms using the breadth
constraints and the loss from the cubic maximum. 

For each fixed $\eta\in\R$, define
\begin{equation}\label{p3:eq:balance-objective}
 \cF_\eta
 :=\frac12r_1^2m_2+\cX+3m_2m_3+\eta m_2z_{\rmbal}.
\end{equation}
The original bound on $\cE$ corresponds to $\eta=1/2$:
\begin{equation}\label{p3:eq:joint-balance-objective}
 \cE\le\cF_{1/2}+2m_2e+r_1\delta.
\end{equation}

We first establish the estimate for $\cF_\eta$:
\begin{proposition}\label{p3:lem:positive-depth}
There is an absolute constant $\kappa>0$ such that, for every fixed
$\eta\in\R$ and every choice of recorded Sims presentation
data and core breadth sequence arising in the saturated range,
the following hold:
\begin{enumerate}[label=\textup{(\roman*)}]
\item If $m_3=0$, then
\[
 \cF_\eta
 \le\cc(r_1+m_2)^3+\frac{27\eta^2}{2}m_2.
\]
\item If $m_3>0$ and $r_1+m_2$ is larger than some constant depending on $\eta$, then
\begin{equation}\label{p3:eq:positive-depth}
 \cF_\eta
 \le\cc(r_1+m_2)^3-\frac{1}{128}(r_1+m_2)^2.
\end{equation}
\end{enumerate}
\end{proposition}
\begin{proof}
Since $z_{\rmbal}=r_1-2m_2$, we have
\begin{equation}\label{p3:eq:cubic-deficit}
 \frac{2}{27}(r_1+m_2)^3-\frac12r_1^2m_2
 =\frac{z_{\rmbal}^2(4r_1+m_2)}{54}.
\end{equation}

For~\textup{(i)}, $m_3=0$ implies $c=2$, so the sums defining
$T_{\core}$ and $\cX$ are empty. Then we have
\[
 \begin{aligned}
 \cF_\eta
 &=\cc(r_1+m_2)^3-\frac{z_{\rmbal}^2(4r_1+m_2)}{54}
      +\eta m_2z_{\rmbal}\\
 &\le\cc(r_1+m_2)^3
      +m_2\left(\eta z_{\rmbal}
                   -\frac{z_{\rmbal}^2}{54}\right)\\
 &\le\cc(r_1+m_2)^3+\frac{27\eta^2}{2}m_2,
 \end{aligned}
\]
where the first inequality uses the fact that $4r_1+m_2\ge m_2$, and the second inequality uses $\eta z_{\rmbal}-\frac{z_{\rmbal}^2}{54}\leq \frac{27}{2}\eta^2$ for any $z_{\rmbal}$.
This proves~\textup{(i)}.

For~\textup{(ii)}, assume $m_3>0$. Use
$s\leq C\sqrt{m_2}$ from~\eqref{p3:eq:saturated-core-size}.
Dropping the negative term in the second inequality of \eqref{p3:eq:roughX} and using
$|\boldsymbol{\fb}|\le m_2$ gives $\cX\le (\frac{C}{2}+2C)m_2^{5/2}+Cm_2^{3/2}$.

If $r_1<m_2$, then $|z_{\rmbal}|=2m_2-r_1>\frac{1}{2}(r_1+m_2)$; if $r_1>3m_2$, then $|z_{\rmbal}|=r_1-2m_2>\frac{1}{4}(r_1+m_2)$. In these cases, 
\[
    \frac{z_{\rmbal}^2(4r_1+m_2)}{54}\geq \frac{(r_1+m_2)^2(4r_1+m_2)}{864}\geq \frac{(r_1+m_2)^3}{864},
\]
where the second inequality uses $4r_1\geq r_1$. Substituting these into $\cF_\eta$ gives
\[
\begin{split}
  \cF_\eta&= \cc(r_1+m_2)^3-\frac{z_{\rmbal}^2(4r_1+m_2)}{54}+\cX+3m_2m_3+\eta m_2z_{\rmbal}\\
  &\leq \cc(r_1+m_2)^3-\frac{(r_1+m_2)^3}{864}+\left(\frac{C}{2}+2C\right)m_2^{5/2}+Cm_2^{3/2}+3m_2m_3+\eta m_2z_{\rmbal}\\
  &\leq \cc(r_1+m_2)^3-\frac{(r_1+m_2)^3}{864}+\left(\frac{C}{2}+2C\right)(r_1+m_2)^{5/2}+C(r_1+m_2)^{3/2}+(2|\eta|+3)(r_1+m_2)^2,
\end{split}
\]
where the second inequality holds since $|z_{\rmbal}|=|r_1-2m_2|\leq 2(r_1+m_2)$ and $m_3\leq m_2\leq r_1+m_2$. Thus, when $r_1+m_2$ is large enough such that
\[
  -\frac{(r_1+m_2)^3}{864}+\left(\frac{C}{2}+2C\right)(r_1+m_2)^{5/2}+C(r_1+m_2)^{3/2}+(2|\eta|+3)(r_1+m_2)^2\leq -\frac{1}{128}(r_1+m_2)^2.
\]
We obtain
\[
 \cF_\eta\leq \cc(r_1+m_2)^3-\frac{1}{128}(r_1+m_2)^2.    
\]

Now we consider $m_2\le r_1\le3m_2$. Then $|z_{\rmbal}|\le m_2$, $\frac{z_{\rmbal}^2(4r_1+m_2)}{54}\geq \frac{5m_2z_{\rmbal}^2}{54}$ and $s\leq C\sqrt{m_2}$ for some constant $C$.
Thus, 
\[
      \binom{r_1-s}{2}-s(5m_2+1)\ge \binom{m_2-s}{2}-s(5m_2+1)\ge\frac12m_2^2-6sm_2-\frac12m_2\geq \frac{m_2^2}{4}
\]
when $m_2\geq\max\{(48C)^2,4\}$.

Since $m_3>0$, we have $\fb_2\ge1$.
Moreover, since $m_3\le s|\boldsymbol{\fb}|$ as in~\eqref{p12:eq:breadth-total-le-deep}, this gives
$3m_2m_3\le3s m_2|\boldsymbol{\fb}|$.
Combining this with the first inequality of~\eqref{p3:eq:roughX} yields
\begin{equation}\label{p3:eq:absorbed-depth}
 \begin{split}
 \cX+3m_2m_3\leq s\min\left\{\frac{m_2^2-s_2^2}{2},
          s m_2(|\boldsymbol{\fb}|-\fb_2)\right\}-\left(\binom{r_1-s}{2}-s(5m_2+1)\right)
            |\boldsymbol{\fb}|
 \end{split}
\end{equation}
when $m_2\geq\max\{(48C)^2,4\}$. 

We bound the right-hand side of~\eqref{p3:eq:absorbed-depth} by
maximizing over all real values
$|\boldsymbol{\fb}|\ge\fb_2$, which enlarges the
set of feasible values.
Since $s_2=m_2-m_3$, the two arguments of the minimum agree when
\[
 |\boldsymbol{\fb}|
 =\fb_2+\frac{m_3(2m_2-m_3)}{2sm_2}.
\]
When $\fb_2\leq |\boldsymbol{\fb}|\leq \fb_2+\frac{m_3(2m_2-m_3)}{2sm_2}$, the right-hand side is
\[
 \begin{split}
 -\left(\binom{r_1-s}{2}-s(5m_2+1)\right)\fb_2+\left(s^2m_2-\binom{r_1-s}{2}+s(5m_2+1)\right)
       \left(|\boldsymbol{\fb}|-\fb_2\right).
 \end{split}
\]
This is affine, so its maximum on that interval occurs at one
of the two endpoints.
When $|\boldsymbol{\fb}|\geq \fb_2+\frac{m_3(2m_2-m_3)}{2sm_2}$, the right-hand side is
\[
 s\frac{m_2^2-s_2^2}{2}
 -\left(\binom{r_1-s}{2}-s(5m_2+1)\right)
       |\boldsymbol{\fb}|.
\]
This branch is decreasing because
$\binom{r_1-s}{2}-s(5m_2+1)\ge m_2^2/4>0$.
The maximum over the whole range is therefore attained at
$|\boldsymbol{\fb}|=\fb_2$ or at the breakpoint.
The difference between the value at the breakpoint and the value
at $\fb_2$ is
\[
 \frac{m_3(2m_2-m_3)}{2sm_2}
 \left(s^2m_2-\binom{r_1-s}{2}+s(5m_2+1)\right).
\]
Taking the larger of these two endpoint values gives
\[
 \begin{split}
 &-\left(\binom{r_1-s}{2}-s(5m_2+1)\right)\fb_2+\frac{m_3(2m_2-m_3)}{2sm_2}
 \cdot\max\left\{s^2m_2-\binom{r_1-s}{2}+s(5m_2+1),0\right\}.
 \end{split}
\]
Using $\fb_2\ge1$ and
$\binom{r_1-s}{2}-s(5m_2+1)\ge m_2^2/4$, we obtain
\begin{equation}\label{p3:eq:depth-max}
 \begin{split}
 &\cX+3m_2m_3\\
 \le&-\frac{m_2^2}{4}
       +\frac{m_3(2m_2-m_3)}{2s m_2}
        \cdot\max\left\{s^2m_2-\binom{r_1-s}{2}+s(5m_2+1),0\right\}\\
  \le&-\frac{m_2^2}{4}
       +\frac{m_3}{s}
        \max\left\{s^2m_2-\binom{r_1-s}{2}+s(5m_2+1), 0\right\},
 \end{split}
\end{equation}
where the second inequality uses $m_3(2m_2-m_3)\leq 2m_2m_3$. Moreover, if $s^2m_2-\binom{r_1-s}{2}+s(5m_2+1)\leq 0$, we have $\cX+3m_2m_3\leq -\frac{m_2^2}{4}$. Otherwise, 
\[
 s^2m_2>\binom{r_1-s}{2}-s(5m_2+1)\ge\frac14m_2^2,
\]
implying that $s>\sqrt{m_2}/2$. Moreover, substitute $r_1=2m_2+z_{\rmbal}$, we have
\[
\binom{r_1-s}{2}
=2m_2^2+2m_2(z_{\rmbal}-s)
+\frac12(z_{\rmbal}-s)^2
-m_2-\frac{z_{\rmbal}}{2}+\frac{s}{2}.
\]
Thus, 
\[
\begin{aligned}
&s^2m_2-\binom{r_1-s}{2}+s(5m_2+1)\\
=&s^2m_2-2m_2^2-2m_2z_{\rmbal}
+7sm_2-\frac12(z_{\rmbal}-s)^2
+m_2+\frac{z_{\rmbal}}{2}+\frac s2\\
\le& -2m_2(m_3+z_{\rmbal})+10s m_2
       -\frac12(z_{\rmbal}-s)^2
       +m_2+\frac{z_{\rmbal}}{2}+\frac s2\\
\le& -2m_2(m_3+z_{\rmbal})+12s m_2,
\end{aligned}
\]
where the first inequality uses $s^2\le2(m_2-m_3)+3s$, and the second inequality drops the negative term and uses $|z_{\rmbal}|\le sm_2$, $s\leq sm_2$ and $m_2\leq sm_2$ since $s\geq 2$ and $m_2\geq 1$.

Substitute this bound to~\eqref{p3:eq:depth-max}, we have
\begin{align}\label{p3:eq:depth-square}
 \cX+3m_2m_3
 &\le-\frac{m_2^2}{4}+\frac{m_3}{s}\max\left\{-2m_2(m_3+z_{\rmbal})+12s m_2, 0\right\}
       \notag\\
 &=-\frac{m_2^2}{4}+\frac{m_2}{s}\max\left\{-2m_3(m_3+z_{\rmbal})+12s m_3, 0\right\}
       \notag\\
 &=-\frac{m_2^2}{4}+\frac{m_2}{s}\max\left\{-2\left(m_3+\frac{1}{2}(z_{\rmbal}-6s)\right)^2+\frac{1}{2}(z_{\rmbal}-6s)^2, 0\right\}
       \notag\\
 &\le-\frac{m_2^2}{4}+\frac{m_2}{2s}(z_{\rmbal}-6s)^2\notag\\
 &\le-\frac{m_2^2}{4}+\frac{m_2z_{\rmbal}^2}{2s}
       +6m_2|z_{\rmbal}|+18s m_2
       \notag\\
 &\le-\frac{m_2^2}{4}+\sqrt{m_2}\,z_{\rmbal}^2
       +6m_2|z_{\rmbal}|+18C m_2^{3/2},
\end{align}
where the first equality uses $m_2\geq m_3>0$, and the last inequality uses $\sqrt{m_2}/2\leq s\leq C\sqrt{m_2}$. Its right-hand side is at least
$-m_2^2/4$, so the final bound also holds when $s^2m_2-\binom{r_1-s}{2}+s(5m_2+1)\leq 0$. Thus this bound
is valid for every $s$.

Now we bound $\cF_{\eta}$ when $m_2\leq r_1\leq 3m_2$. Note that
\[
\begin{split}
   \cF_\eta-\cc(r_1+m_2)^3&=\cX+3m_2m_3+\eta m_2z_{\rmbal}
       -\frac{z_{\rmbal}^2(4r_1+m_2)}{54}\\
       &\leq -\frac{m_2^2}{4}+\sqrt{m_2}\,z_{\rmbal}^2
       +(6+|\eta|)m_2|z_{\rmbal}|
       -\frac{5m_2z_{\rmbal}^2}{54}+18C m_2^{3/2},   
\end{split}
\]
where the last inequality also uses $4r_1+m_2\ge5m_2$.
When $m_2\geq (\frac{108}{5})^2$, we have
$\sqrt{m_2}\le5m_2/108$. In this regime, we have
\[
\sqrt{m_2}\,z_{\rmbal}^2
+(6+|\eta|) m_2|z_{\rmbal}|
\le\frac{5m_2z_{\rmbal}^2}{108}
+(6+|\eta|) m_2|z_{\rmbal}|
\le
\frac5{54}m_2z_{\rmbal}^2
+\frac{27(6+|\eta|)^2}{5}m_2.
\]
Thus,
\[
 \begin{aligned}
 \cF_\eta-\cc(r_1+m_2)^3\le\frac{27(6+|\eta|)^2}{5}m_2
       +18C m_2^{3/2}-\frac14m_2^2\le-\frac18m_2^2
 \end{aligned}
\]
when $m_2\ge
 \max\left\{(288C)^2,\frac{432}{5}(6+|\eta|)^2\right\}$.
Indeed, $(6+|\eta|)\ge6$ ensures $m_2\geq (\frac{108}{5})^2$,
and we have $\frac{27(6+|\eta|)^2}{5}m_2\leq m_2^2/16$ and $18C m_2^{3/2}\leq m_2^2/16$.

Finally, since $r_1+m_2\le4m_2$, we have
\[
 \cF_\eta
 \le\cc(r_1+m_2)^3-\frac18m_2^2
 \le\cc(r_1+m_2)^3-\frac{(r_1+m_2)^2}{128}.
\]
All the required lower bounds on $m_2$ depend only on $\eta$
and the absolute constant $C$. Since
$m_2\ge(r_1+m_2)/4$ in the present range, they hold once
$r_1+m_2$ is sufficiently large depending only on $\eta$.
Combining this with the estimate outside the present range
proves~\textup{(ii)}.
\end{proof}

Taking $\eta=1/2$ in~\cref{p3:lem:positive-depth} and using
\eqref{p3:eq:joint-balance-objective}, together with
$r_1+m_2=m-\delta$, we have
\[
 \cE\le\frac{2}{27}(m-\delta)^3+2m_2e+r_1\delta+
 \begin{cases}
  \displaystyle\frac{27}{8}m_2,
   & m_3=0,\\[2mm]
  \displaystyle-\frac{1}{128}(m-\delta)^2,
   & m_3>0,\quad m-\delta\text{ sufficiently large}.
 \end{cases}
\]
On the other hand, if $m-\delta$ is bounded, then $\cF_{1/2}$ is also bounded  independently of $p$.
Absorbing these cases into an absolute constant therefore gives
the uniform bound
\[
 \cE\le\frac{2}{27}(m-\delta)^3
       +2m_2e+r_1\delta+\frac{27}{8}m_2+O(1).
\]
Since $r_1,m_2,e,\delta\le m$, there is an absolute constant
$C_0$ such that
\begin{equation}\label{p3:eq:coarse}
 \cE\le\frac{2}{27}(m-\delta)^3+C_0m^2.
\end{equation}
Meanwhile, 
\begin{equation}\label{p3:eq:first-layer-defect-deficit}
 \begin{split}
 m^3-(m-\delta)^3
 &=\delta(3m^2-3m\delta+\delta^2)=\delta\bigl(m^2+(m-\delta)(2m-\delta)\bigr)
 \ge\delta m^2.
 \end{split}
\end{equation}
Thus, if $\delta$ is bounded, then
\cref{p12:lem:second-layer-defect} gives $e=O(1)$, implying that $2m_2e+r_1\delta=O(m)$.

Summing the fixed-data array bounds gives an upper bound on
$f_{sat}(p^m)$. By
\cref{thm:sims_presentation_determines_isomorphism_classes_},
each complete choice of recorded Sims presentation data and a pair
of power word and commutator coefficient arrays determines a group
up to isomorphism. The recorded Sims presentation data have $p^{O(m\log m)}$ choices
by~\cref{p12:prop:sims-generating-data}. For each such choice,
\eqref{p12:eq:breadth-sequence-count} gives at most $2^{m_3}\le2^m$
choices for the core breadth sequences $\boldsymbol{\fb}$.
Summing the fixed-data joint array bounds over these choices gives
\[
 \log_p f_{\rmsat}(p^m)
 \le\max\cE+O(m\log m)
 \le\frac{2}{27}m^3+O(m^2),
\]
for sufficiently large $m$, where the maximum is over recorded Sims presentation data and core
breadth sequences realized by groups of order $p^m$ in the saturated
range. The second inequality follows from~\eqref{p3:eq:coarse}
with $m-\delta\le m$. Note that this bound, together with the bound for $p$-groups in the nonsaturated range (cf.~\cref{p12:cor:nonsaturated-group}) and the bound for abelian $p$-groups (cf.~\cref{p12:lem:universal-field}(iv)), shows that 
\begin{equation}\label{eq: partial}
 \log_p f(p^m)\leq \frac{2}{27}m^3+O(m^2),
\end{equation}
which already answers Sims conjecture~\cite{Sims} and Question 22.3 in~\cite{BNV} in affirmative.

Suppose $0\le\delta\le D$ for a fixed nonnegative integer $D$.
By~\cref{p12:lem:second-layer-defect},
$e\le\binom D2$, and hence
$2m_2e+r_1\delta\le D^2m$.
Taking $\eta=1/2$ in~\cref{p3:lem:positive-depth} and using
\eqref{p3:eq:joint-balance-objective} gives
\[
 \cE\le\frac{2}{27}(m-\delta)^3
       +\left(D^2+\frac{27}{8}\right)m
\]
for sufficiently large $m$, uniformly in $p$.

The proof of~\cref{p12:prop:sims-generating-data} bounds the
number of recorded Sims presentation data choices by
$m^{2m+2}2^{3m-1}$, and there are at most $2^{m_3}\le2^m$
core breadth sequences.
Consequently, the number of isomorphism classes of nonabelian
groups of order $p^m$ in the saturated range with $\delta\le D$
is at most
\[
 p^{\frac{2}{27}m^3
 +(2m+2)\log_p m+(4m-1)\log_p2
 +(D^2+\frac{27}{8})m}
 =p^{\frac{2}{27}m^3+O_D(m\log m)}.
\]

\section{Counting multiplicities of Sims presentation}
\label{p3:sec:normalization}

The preceding argument counts complete choices of recorded Sims
presentation data, power word arrays, and commutator coefficient
arrays, giving the bound~\eqref{eq: partial} on $f(p^m)$.
This improves the bounds of Sims~\cite{Sims} and Newman and
Seeley~\cite[Section~5]{BNV}. Matching Higman's lower bound
\cite{Higman} up to a linear term in the exponent requires further
savings from the number of distinct choices produced by each group.

Recall the construction of Sims generators
in~\cref{p12:par:fixed-sims-data}: the first $s$ weight-$1$
generators generate a Sims core, and all higher-weight generators
are iterated commutators in these core generators. We use the
\hyperref[p12:par:recorded-sims-data]{recorded Sims presentation data}
and the relation arrays with the index ranges in
\cref{thm:sims_presentation_determines_isomorphism_classes_}.

In this count, the relative order array is already part of the
recorded Sims presentation data. For fixed $p$ and $m$,
\cref{p12:prop:sims-generating-data} counts the choices of these
recorded Sims presentation data. For each such choice and each core breadth sequence,
\cref{p12:thm:joint-enumeration} bounds the number of distinct pairs
of power word arrays and commutator coefficient arrays arising from
nonabelian groups of order $p^m$ equipped with compatible ordered
Sims generating sets. Different generating sets producing the same
recorded Sims presentation data and array pair contribute a single choice to this count.
Summing these bounds gives~\eqref{eq: partial}.

By~\cref{thm:sims_presentation_determines_isomorphism_classes_},
two nonisomorphic groups cannot produce the same complete choice of
recorded Sims presentation data and arrays. Hence, if every group under consideration
produces at least $T$ distinct complete choices, dividing the total
number of such choices by $T$ gives an upper bound on the number of
isomorphism classes. To obtain this saving, we construct, for each
group, a family of ordered Sims generating sets whose members share
recorded Sims presentation data and a core breadth sequence. We then
bound from below the number of distinct pairs of power word arrays
and commutator coefficient arrays produced within that family.

\subsection{Refining Sims generating sets using the first layer structure}
\label{p3:subsec:source-bases}
Fix a nonabelian $p$-group $G$ and a Sims core $H=\langle g_{1,1},\ldots,g_{1,s}\rangle$, with its ordered generators chosen as in \cref{p12:par:fixed-sims-data}. Also fix the auxiliary function $f$ as in~\eqref{eq:auxiliary function} and the index pairs specifying the higher-weight generators. These choices determine all higher-weight generators as iterated commutators in the fixed core generators. Extending $g_{1,1},\ldots,g_{1,s}$ to a minimum generating set $g_{1,1},\ldots,g_{1,r_1}$ of $G$ then gives an ordered Sims generating set as described in~\cref{p12:par:fixed-sims-data}.

With these choices fixed, the remaining freedom lies in choosing the additional weight-$1$ generators $g_{1,s+1},\ldots,g_{1,r_1}$. The relative orders of the core generators and all higher-weight generators are already fixed. We shall use the structure of $L_1=G/G_2$ to construct a family of choices for the additional weight-$1$ generators such that their relative orders are constant at each position throughout the family. All resulting Sims generating sets will therefore share the same recorded Sims presentation data.

Recall from~\eqref{eq: relative order of gij} that the relative-order exponent
$a(1,j)$ of $g_{1,j}$ is the least integer $a\in\N$ for which $g_{1,j}^{p^a}\in
 \langle g_{1,1},\ldots,g_{1,j-1},G_2\rangle$.
Equivalently, $p^{a(1,j)}$ is the order of $\bar g_{1,j}=g_{1,j}G_2\in L_1$ in the quotient
$L_1/\langle\bar g_{1,1},\ldots,\bar g_{1,j-1}\rangle$.

Note that $L_1$ is an abelian $p$-group. Let $\bigoplus_{j=1}^{r_1}\Z_{p^{\alpha_j}}$ be the cyclic decomposition of $L_1$, where $\sum_{j=1}^{r_1}(\alpha_j-1)=s_1-r_1=\delta$.
Define $Z_{\geq 2}=\bigoplus_{\alpha_j\ge2}\Z_{p^{\alpha_j}}$ and $Z_1=\bigoplus_{\alpha_j=1}\Z_{p^{\alpha_j}}$. Then $L_1\cong Z_{\geq 2}\oplus Z_1$, with $Z_1$ elementary abelian.
Each cyclic factor of $Z_{\geq 2}$ contributes at least one to $\delta$,
so $d(Z_{\geq 2})\le\delta$ and $\dim Z_1=r_1-d(Z_{\geq 2})\ge r_1-\delta$.
Consequently, $Z_1$ has dimension close to $r_1$ when 
$\delta=s_1-r_1$ is small.

Recall from~\eqref{p12:eq:core-image} that
$C_{\core}=HG_2/G_2=\langle\bar g_{1,1},\ldots,\bar g_{1,s}\rangle\leq L_1$.
By the decomposition $L_1\cong Z_{\geq 2}\oplus Z_1$, write
$\bar g_{1,k}=b_k+e_k$ with $b_k\in Z_{\geq 2}$ and $e_k\in Z_1$.
Let $Z_{\core}=\linspan_{\F_p}\{e_1,\ldots,e_s\}$.
Since $\dim Z_{\core}\le s$ and $Z_1$ is elementary abelian, a complement $Z_{0}$ of $Z_{\core}$ in $Z_1$
satisfies
\[
 \dim Z_0
 =r_1-d(Z_{\geq 2})-\dim Z_{\core}
 \ge r_1-\delta-s.
\]

Using $Z_1=Z_0\oplus Z_{\core}$, we have
$L_1\cong(Z_{\geq 2}\oplus Z_{\core})\oplus Z_0$, with
$C_{\core}\le Z_{\geq 2}\oplus Z_{\core}$.
This decomposition allows the basis of $Z_0$ to vary while
the core image $C_{\core}$ remains fixed.

Consider the previously fixed ordered core
generators $g_{1,1},\ldots,g_{1,s}$.
Their images are independent modulo $pL_1$ by their choice
in~\cref{p12:par:fixed-sims-data}, so their images in
$(Z_{\geq 2}\oplus Z_{\core})/p(Z_{\geq 2}\oplus Z_{\core})$ are also independent.
By the Burnside basis theorem~\cite[Theorem~5.3.2]{Robinson},
the images $\bar g_{1,1},\dots,\bar g_{1,s}$ can be extended to a minimum generating set of $Z_{\geq 2}\oplus Z_{\core}$,
consisting of $r_1-\dim Z_0$ elements.
Choose and fix lifts of the additional elements as
$g_{1,s+1},\ldots,g_{1,r_1-\dim Z_0}$.
For any ordered basis of $Z_0$, lifts of its basis vectors can
serve as the remaining weight-$1$ generators $g_{1,r_1-\dim Z_0+1},\ldots,g_{1,r_1}$.
These weight-$1$ generators clearly satisfy the requirements for
a Sims generating set in~\cref{p12:par:fixed-sims-data}, and all higher-weight generators remain those fixed above since we don't change the fixed Sims core generators.

Keep the first $r_1-\dim Z_0$ weight-$1$ generators and all
higher-weight generators fixed as chosen above, and allow only
the last $\dim Z_0$ weight-$1$ generators to vary.
With this convention, the construction gives the family of
ordered Sims generating sets for $G$
\begin{equation}\label{eq: Sims generating set family}
\cS_{Z_0}:=\left\{
 (g_{i,j})_{i\in[c],j\in[r_i]}:
 \left\langle
 g_{1,r_1-\dim Z_0+1}G_2,\ldots,g_{1,r_1}G_2
 \right\rangle=Z_0
\right\},
\end{equation}
where we refer to the generators $g_{1,r_1-\dim Z_0+1},\ldots,g_{1,r_1}$ as \emph{free generators}. The dimension of $Z_0$ can be bounded by $\max\{0,r_1-s-\delta\}\le\dim Z_0\le r_1-s$.
All these ordered Sims generating sets share the fixed integers $c$, $s$, $(r_i)_{i\in[c]}$
and $(d_k)_{k\in[s]}$, and use the chosen auxiliary function $f$
from~\eqref{eq:auxiliary function} and the index pairs specifying
the higher-weight generators.

Moreover, for each Sims generating set $(g_{i,j})_{i\in[c],j\in[r_i]}\in\cS_{Z_0}$, the relative orders of the fixed weight-$1$ prefix and all higher-weight generators remain constant throughout the family. For the last $\dim Z_0$ weight-$1$ generators $g_{1,r_1-\dim Z_0+1},\ldots,g_{1,r_1}$, their images in $L_1$ form a basis of $Z_0\leq Z_1$. Each of these basis vectors then has order $p$ and lies outside the
subgroup generated by the preceding weight-$1$ images, using
the direct-sum decomposition and basis independence.
Thus, for each Sims generating set $(g_{i,j})_{i\in[c],j\in[r_i]}\in\cS_{Z_0}$, $a(1,j)=1$ for $r_1-\dim Z_0<j\le r_1$. Summarizing the above, for any nonabelian group $G$, the family $\cS_{Z_0}$ constructed above has common recorded Sims presentation data.

\subsection{Multiplicity of relation arrays}

We now fix the common recorded Sims presentation data of
$\cS_{Z_0}$.
The core breadth sequence also remains fixed throughout this
family, since $G$ and its chosen Sims core $H$ are both fixed.
Consequently, \cref{p12:thm:joint-enumeration} provides an upper
bound on the number of distinct pairs of power word arrays and
commutator coefficient arrays arising from ordered Sims generating sets in $\cS_{Z_0}$.
It remains to bound from below the number of distinct pairs of
power word arrays and commutator coefficient arrays arising
from this family.
Different choices of the last $\dim Z_0$ weight-$1$ generators
can yield the same pair, and we use the alternating bilinear map
$\beta_G$ to control this overlap.

Use the Frattini quotients $V$ and $W$ from
\eqref{p12:eq:frattini-quotients}, of dimensions $r_1$ and $r_2$
over $\F_p$, respectively.
Recall from~\eqref{p12:eq:group-alternating-map} that
$\beta_G:V\times V\to W$ is the alternating bilinear map induced
by group commutation.
Recall the core image $U=H\Phi(G)/\Phi(G)\le V$ associated with
$H$ after~\eqref{eq: linearised sims rank}; it has dimension $s$.
Define its \emph{core centralizer} with respect to $\beta_G$ by
\begin{equation}\label{p3:eq:core-centralizer}
 K=\{v\in V:\beta_G(v,U)=0\}.
\end{equation}
In particular, every ordered Sims generating set in $\cS_{Z_0}$ has the same core centralizer $K$.

Note that the core
centralizer $K$ must be disjoint from the core image $U$ in $V$.
Otherwise,
suppose $0\ne v\in K\cap U$. Extend $v$ to a basis
$v,u_2,\ldots,u_s$ of $U$ and let
$U'=\linspan_{\F_p}\{u_2,\ldots,u_s\}$.
Since $v\in K$ and $\beta_G$ is alternating, all bracket values
involving $v$ and an element of $U$ vanish. Bilinearity therefore
gives $\beta_G(U',U')=\beta_G(U,U)=W$.
But $\dim U'=s-1$, contradicting the minimality of $s$ in
\eqref{eq: linearised sims rank}. Consequently, we have
\begin{equation}\label{eq: bound on dim K}
\dim K\leq r_1-s.
\end{equation}

We bound the number of distinct array
pairs produced by $\cS_{Z_0}$ using $\dim K$ and $\dim Z_0$.

\begin{proposition}[Multiplicity of relation arrays]
\label{p3:prop:source-normalization}
For a fixed nonabelian $p$-group $G$, the family $\cS_{Z_0}$ of ordered Sims
generating sets described in~\eqref{eq: Sims generating set family} 
produces at least 
\[
      \pi_\infty(2)
 p^{(\dim Z_0)^2-(\dim Z_0)(\dim K)}
\]
distinct pairs of power word arrays and commutator coefficient arrays,
with the common recorded Sims presentation data and core breadth
sequence fixed above.
\end{proposition}
\begin{proof}

By the decomposition
$L_1\cong(Z_{\geq 2}\oplus Z_{\core})\oplus Z_0$ and the equality
$pZ_1=pZ_{\core}\oplus pZ_0=0$, we have $pL_1=pZ_{\geq 2}$ and thus $Z_0\cap pL_1=0$.
The quotient map $L_1\to L_1/pL_1\cong V$ has kernel $pL_1$,
so its restriction to $Z_0$ is injective. We therefore regard
$Z_0$ as a subspace of $V$.

Let $(x_\ell)_{\ell=1}^{\dim Z_0}$ and
$(y_\ell)_{\ell=1}^{\dim Z_0}$ be two ordered bases of $Z_0$.
Choose corresponding tuples of free generators $g_{1,r_1-\dim Z_0+1},\ldots,g_{1,r_1}$ in
$\cS_{Z_0}$ as lifts under $G\to L_1=G/G_2$, such that
\[
 g_{1,r_1-\dim Z_0+\ell}G_2=x_\ell\in L_1,~~g'_{1,r_1-\dim Z_0+\ell}G_2=y_\ell\in L_1
\]
for any $\ell\in[\dim Z_0]$.
Suppose that these two tuples of free generators produce the same
pair of power word and commutator coefficient arrays.
For every $k\in[s]$, $\ell\in[\dim Z_0]$ and $v\in[r_2]$, the coefficients $c(1,k,r_1-\dim Z_0+\ell,2,v)\bmod p$
therefore agree for the two choices. The core generators and the ordered basis
$(g_{2,v}G_3\Phi(G_2))_{v\in[r_2]}$ of $W$ remain fixed
throughout the family. By~\eqref{p12:eq:group-alternating-map},
these coefficients are precisely the coordinates of
$\beta_G(g_{1,k}\Phi(G),x_\ell)$ and
$\beta_G(g_{1,k}\Phi(G),y_\ell)$ in that common basis.
Consequently, for any $k\in[s]$, $\beta_G(g_{1,k}\Phi(G),x_\ell)
 =\beta_G(g_{1,k}\Phi(G),y_\ell)$.
Bilinearity and alternation of $\beta_G$ now give
\[
 \beta_G(x_\ell-y_\ell,g_{1,k}\Phi(G))
 =-\beta_G(g_{1,k}\Phi(G),x_\ell-y_\ell)=0
\]
for any $k\in[s]$.
Since the core images $g_{1,1}\Phi(G),\dots,g_{1,s}\Phi(G)$ span $U$,
we have $\beta_G(x_\ell-y_\ell,U)=0$.
Thus $x_\ell-y_\ell\in K$ for every $\ell$.

Fix an array pair occurring in the family and one ordered basis
$(x_\ell)$ whose lifts produce that pair.
Every other basis $(y_\ell)$ whose lifts produce the same pair
must satisfy $y_\ell\in(x_\ell+K)\cap Z_0$ for every $\ell$.
There are at most $p^{\dim K}$ possibilities for each vector,
so at most $p^{(\dim Z_0)(\dim K)}$ ordered bases can
produce a specified array pair.
Every ordered basis of $Z_0$ has a tuple of lifts belonging
to the family. Since there are $|\GL(\dim Z_0,p)|$ such
bases, \eqref{p12:eq:uniform-pi} implies that the number of
distinct array pairs in the family is at least
\[
 \frac{|\GL(\dim Z_0,p)|}
      {p^{(\dim Z_0)(\dim K)}}
 \ge\pi_\infty(2)
 p^{(\dim Z_0)^2-(\dim Z_0)(\dim K)}.
\]

\end{proof}

We now use~\cref{p12:rem:conditional-leading-constraints} to bound
the number of distinct pairs of power word arrays and commutator
coefficient arrays with recorded Sims presentation data, a core
breadth sequence, and $\dim K$ fixed. 

\begin{proposition}
\label{p3:prop:centralizer-array-count}
Fix a prime $p$ and an integer $m\ge3$.
Fix recorded Sims presentation data as counted
in~\cref{p12:prop:sims-generating-data}, a core breadth sequence
$\boldsymbol{\fb}$, and a value of $\dim K$ realized
by a family $\cS_{Z_0}$ of the
form~\eqref{eq: Sims generating set family} for a nonabelian group
of order $p^m$.
Use the joint-count notation $\cE$ from~\cref{p12:thm:joint-enumeration}
for the base-$p$ logarithm of the number of distinct pairs of power word
arrays and commutator coefficient arrays produced by all such
families with the fixed recorded Sims
presentation data, core breadth sequence, and prescribed value
of $\dim K$.
Then
\begin{equation}\label{p3:eq:centralizer-joint-arrays}
 \begin{split}
 \cE\le{}&\cF_{1/2}+2m_2e+r_1\delta+(\dim K)(r_1-\dim K-sr_2)
       +\log_p\bigl(\pi_\infty(2)^{-1}\bigr),
 \end{split}
\end{equation}
where $\cF_{1/2}$ is defined
in~\eqref{p3:eq:balance-objective}.
\end{proposition}

\begin{proof}
For each ordered Sims generating set $(g_{i,j})_{i\in[c],j\in[r_i]}\in \cS_{Z_0}$ under consideration, the cosets $(g_{1,j}\Phi(G))_{j\in[r_1]}$ form an ordered basis of
$V=G/\Phi(G)\cong \F_p^{r_1}$. Write
$v_j=g_{1,j}\Phi(G)\in V$ for $j\in[r_1]$, all of which form an ordered basis of $V$.
Since $H=\langle g_{1,1},\ldots,g_{1,s}\rangle$, its image
$U=H\Phi(G)/\Phi(G)\leq V$ is a dimension-$s$ subspace of $V$ and spanned by $v_1,\dots, v_s$.
By~\cref{p12:lem:universal-field}\textup{(i)}, the number of
possible subspaces $K\leq V$ of the prescribed dimension $\dim K$ is
at most $\pi_\infty(2)^{-1}p^{(\dim K)(r_1-\dim K)}$.
In particular, this bound also bounds all subspaces $K$ of that dimension 
satisfying $K\cap U=0$.

Let $X=\linspan_{\F_p}\{v_{s+1},\ldots,v_{r_1}\}$. Then $V=U\oplus X$.
For each subspace $K$ disjoint from $U$,
Choose a basis $v'_1,\ldots,v'_{\dim K}$ of $K$, and write
$v'_\ell=u_\ell+x_\ell$ with $u_\ell\in U$ and $x_\ell\in X$.
Since $K\cap U=0$, projection onto $X$ is injective on $K$.
Thus $x_1,\ldots,x_{\dim K}$ are linearly independent and can
be extended to a basis of $X$.

Since $K$ is the core centralizer, we have 
\[
\beta_G(v_j,v'_\ell)=\beta_G(v_j,u_\ell)+\beta_G(v_j,x_\ell)=0
\]
for any $j\in[s]$ and $\ell\in[\dim K]$. Equivalently, $\beta_G(v_j,u_\ell)=-\beta_G(v_j,x_\ell)$. 
For each $\ell\in[\dim K]$, $\beta_G(v_j,u_\ell)$ is determined by the commutator coefficients $c(1,a,b,2,v)\bmod p$ with $1\leq a<b\leq s$ and $v\in[r_2]$. Once these commutator coefficients are fixed, the value $\beta_G(v_j,x_\ell)$ is fixed.

Extend $x_1,\ldots,x_{\dim K}$ to a basis of $X$.
In this basis, the values $\beta_G(v_j,x_\ell)$, with
$j\in[s]$ and $\ell\in[\dim K]$, are $s(\dim K)$ independent
$W$-valued entries among the freely counted core--outside leading
data. The preceding equations prescribe each of these values
once the core--core coefficients are fixed.
Since $\dim W=r_2$, prescribing these values imposes exactly
$s(\dim K)r_2$ independent scalar conditions.
Changing from the original outside basis
$v_{s+1},\ldots,v_{r_1}$ to this basis of $X$ gives an invertible
linear transformation of the core--outside leading coordinates.
Thus the prescribed values give an affine system of rank
$s(\dim K)r_2$ on the original coefficients
\[
 c(1,j,k,2,v)\bmod p,
 \qquad j\in[s],\quad s<k\le r_1,\quad v\in[r_2].
\]
By~\cref{p12:rem:conditional-leading-constraints}, we may
subtract $s(\dim K)r_2$ from the right-hand side of
\eqref{p12:eq:joint-enumeration-global} with $\Delta_p$ omitted.
For each coordinate subspace $K$, the number of joint array pairs
is therefore at most
\[
 p^{\cF_{1/2}+2m_2e+r_1\delta-s(\dim K)r_2}.
\]
Summing over the possible coordinate subspaces and taking
base-$p$ logarithms gives~\eqref{p3:eq:centralizer-joint-arrays}.
\end{proof}

\section{Proofs of~\texorpdfstring{\cref{intro:sharp,intro:concentration}}{thm1.11.4}}\label{p3:sec:linear}

We first collect two counting estimates that will also be used in
Section~\ref{p4:sec:extensions}.

\begin{corollary}[Group count with fixed presentation parameters]
\label{p3:cor:fixed-parameter-count}
Fix a prime $p$ and an integer $m\ge3$. For a nonabelian group
$G$ of order $p^m$, let $G_1=G$ and $G_{i+1}=[G_i,G]$, and write
\[
 r_i=d(G_i/G_{i+1}),\qquad s_i=\log_p|G_i/G_{i+1}|,
 \qquad m_i=\log_p|G_i|.
\]
In particular, $r_1=d(G)$ and $m_2=\log_p|G_2|$. Put
\[
 \delta=s_1-r_1,\qquad e=s_2-r_2,
 \qquad m=r_1+m_2+\delta.
\]
Let $H$ be a Sims core of rank $s=d(H)$, and assume the
saturated condition $\binom{s-1}{2}<r_2\le\binom{s}{2}$.
Use the objective $\cF_\eta$ from~\eqref{p3:eq:balance-objective}.

For each such group and chosen core, the construction in
Section~\ref{p3:subsec:source-bases} gives an elementary abelian
direct summand $Z_0\le G/G_2$ and a family $\cS_{Z_0}$ of ordered
Sims generating sets with common recorded presentation data and
core breadth sequence. The complementary summand of $Z_0$
contains $HG_2/G_2$. With
\[
 V=G/\Phi(G),\qquad W=G_2/(G_3\Phi(G_2)),\qquad
 U=H\Phi(G)/\Phi(G),
\]
let $\beta_G:V\times V\to W$ be the alternating map induced by
commutation, and put
\[
 K=\{v\in V:\beta_G(v,U)=0\}.
\]
These choices satisfy
\[
 \max\{0,r_1-s-\delta\}\le\dim Z_0\le r_1-s,
 \qquad K\cap U=0,\qquad 0\le\dim K\le r_1-s.
\]
In particular,
\begin{equation}\label{p3:eq:source-dimension-loss}
 0\le r_1^2-(\dim Z_0)^2
 =(r_1-\dim Z_0)(r_1+\dim Z_0)
 \le2r_1(s+\delta).
\end{equation}

Fix recorded Sims presentation data, a core breadth sequence,
and the values of $\dim Z_0$ and $\dim K$ realized by these
families. The base-$p$ logarithm of the number of group
isomorphism classes admitting a family with these fixed choices
is at most
\begin{equation}\label{p3:eq:normalizedraw}
 \begin{split}
 &\cF_{1/2}+2m_2e+r_1\delta-(\dim Z_0)^2+(\dim K)
       \bigl(r_1+\dim Z_0-\dim K-sr_2\bigr)+O(1).
 \end{split}
\end{equation}
The additive constant is at most
$2\log_2(\pi_\infty(2)^{-1})$, uniformly in $p$ and in the
fixed choices. There are at most $p^{O(m\log m)}$ possible
choices of the recorded data, breadth sequence, and two
dimensions. Consequently, summing over any collection of
realized choices adds at most $O(m\log m)$ to the largest
fixed-choice exponent.
\end{corollary}
\begin{proof}
The construction of $Z_0$ and its family gives the asserted
bounds on $\dim Z_0$, while the argument preceding
\cref{p3:prop:source-normalization} gives $K\cap U=0$.
Since $\dim U=s$, this also gives $\dim K\le r_1-s$.
The bound~\eqref{p3:eq:source-dimension-loss} follows from
$r_1-\dim Z_0\le s+\delta$ and $r_1+\dim Z_0\le2r_1$.

For each isomorphism class admitting a family with the fixed
choices, choose one such family. By
\cref{p3:prop:source-normalization}, it produces at least
\[
 \pi_\infty(2)
 p^{(\dim Z_0)^2-(\dim Z_0)(\dim K)}
\]
distinct joint array pairs. Different group isomorphism classes
cannot produce the same recorded Sims presentation data, power
word array and commutator coefficient array, by
\cref{thm:sims_presentation_determines_isomorphism_classes_}.
Subtracting the base-$p$ logarithm of this multiplicity from the
joint array bound~\eqref{p3:eq:centralizer-joint-arrays} proves
\eqref{p3:eq:normalizedraw}, with the stated additive constant.

Every nonabelian group in the saturated range admits a family
of the form~\eqref{eq: Sims generating set family}.
The recorded Sims presentation data have $p^{O(m\log m)}$
possible choices by~\cref{p12:prop:sims-generating-data}.
For each such choice,~\eqref{p12:eq:breadth-sequence-count}
gives at most $2^{m_3}\le2^m$ core breadth sequences, and each
of the two dimensions has at most $m+1$ possible values.
This proves the parameter-count assertion. Taking the largest
fixed-choice exponent and multiplying by the number of choices
proves the final assertion, also for any restricted collection
of realized choices.
\end{proof}

\begin{lemma}[Small-core bound]
\label{p3:lem:small-core-count}
There is an absolute constant $C_5>0$ such that, for every fixed
$\Lambda>0$, $A>0$, and nonnegative integer $D$, and every prime
$p$,
\begin{equation}\label{p3:eq:small-core-count}
 \#\left\{
 \begin{array}{c}
  \text{isomorphism classes of nonabelian groups}\\
  \text{of order }p^m\text{ in the saturated range}\\
  \text{with }\delta\le D\text{ and }sr_2<\Lambda r_1
 \end{array}
 \right\}
 \le p^{\frac{2}{27}m^3-Am^2}
\end{equation}
whenever
\[
 m\ge C_5\bigl(1+\Lambda^2+D^2
       +(1+\sqrt{\Lambda})(A+1)^{3/2}\bigr).
\]
Here $s$ is the Sims rank, $r_1=d(G)$,
$r_2=d(G_2/G_3)$, and $\delta=\log_p|G/G_2|-r_1$.
\end{lemma}
\begin{proof}
Fix recorded Sims presentation data and a core breadth sequence
in the range under consideration. Let $\cE$ be the base-$p$
logarithm of the number of compatible pairs of power word and
commutator coefficient arrays, as in
\cref{p12:thm:joint-enumeration}.
Suppose $sr_2<\Lambda r_1$.
By $\binom{s-1}{2}<r_2\le\binom{s}{2}$ and $r_1\le m$,
\[
 (s-2)^3\le s(s-1)(s-2)
 =2s\binom{s-1}{2}<2sr_2<2\Lambda m.
\]
Consequently, $s<2+(2\Lambda m)^{1/3}$ and $r_2\le\binom{s}{2}
 \le\frac12\bigl(2+(2\Lambda m)^{1/3}\bigr)^2$.
Using $|\boldsymbol{\fb}|\le m_2$
from~\eqref{p12:eq:breadth-total-le-deep} and dropping the
negative term in the second inequality of~\eqref{p3:eq:roughX},
we obtain, for $m\ge3$,
\[
 \cX\le\frac{s m_2^2}{2}+s(2m_2+1)m_2
 =s\left(\frac52m_2^2+m_2\right)\le3sm^2
 \le6m^2+3(2\Lambda)^{1/3}m^{7/3}.
\]

If $r_1<m/12$ or $m_2<m/12$, then
\[
 \frac12r_1^2m_2\le\frac{m^3}{24}
 =\frac{2}{27}m^3-\frac{7}{216}m^3.
\]
The remaining terms in~\eqref{p3:eq:joint-balance-objective},
apart from $\cX$, satisfy
\[
 3m_2m_3+\frac12m_2z_{\rmbal}+2m_2e+r_1\delta
 \le\frac72m^2+D^2m.
\]
Here $m_2,m_3,r_1\le m$, $z_{\rmbal}=r_1-2m_2\le m$,
and $2m_2e+r_1\delta\le D^2m$ follows from
$e\le\binom D2$ in~\cref{p12:lem:second-layer-defect}.
Combining these estimates gives the explicit fixed-data bound
\[
 \cE\le\frac{m^3}{24}
       +3(2\Lambda)^{1/3}m^{7/3}
       +\frac{19}{2}m^2+D^2m.
\]
Summing over the recorded Sims presentation data and core breadth
sequences adds $O(m\log m)$, with an absolute implied constant,
by~\cref{p12:prop:sims-generating-data} and the breadth-sequence
count in Section~\ref{p3:sec:statement}.
The cubic saving $7m^3/216$ therefore shows that, for every fixed
$A>0$, the groups in this subcase contribute at most
$p^{\frac{2}{27}m^3-Am^2}$ whenever $m\ge C_1\bigl(1+\sqrt{\Lambda}+D+A\bigr)$
for a sufficiently large absolute constant $C_1$, uniformly in $p$.

Otherwise $r_1,m_2\ge m/12$.
The preceding bounds for $r_2$ and $e$ give
\[
 r_2+e\le\frac12\bigl(2+(2\Lambda m)^{1/3}\bigr)^2
              +\binom D2\le\frac{m}{24}
\]
whenever $m\ge C_2(1+\Lambda^2+D^2)$, for a sufficiently large
absolute constant $C_2$. Hence
$m_3=m_2-r_2-e\ge m/24$.
Moreover, $r_1\ge m/12$ and $r_1,m_2\le m$ imply
\[
 s(2m_2+1)-\binom{r_1-s}{2}\le-\frac{m^2}{288}+s(3m+1)+\frac m2\le-\frac{m^2}{576}.
\]
The last inequality follows from $s<2+(2\Lambda m)^{1/3}$
whenever $m\ge C_3(1+\Lambda^2+D^2)$, where $C_3\ge C_2$ is a
sufficiently large absolute constant.
Using this coefficient bound in the second inequality
of~\eqref{p3:eq:roughX}, together with
$|\boldsymbol{\fb}|\ge m_3/s$
from~\eqref{p12:eq:breadth-total-le-deep}, gives
\[
 \cX\le\frac12sm^2-\frac{m^2m_3}{576s}\le\frac12\bigl(2+(2\Lambda)^{1/3}\bigr)m^{7/3}
       -\frac{m^{8/3}}{13824\bigl(2+(2\Lambda)^{1/3}\bigr)}\le-c_2m^{8/3},
\]
where $c_2=\frac{1}{27648\bigl(2+(2\Lambda)^{1/3}\bigr)}>0$.
The last inequality holds whenever
$m\ge C_4(1+\Lambda^2+D^2)$, where $C_4\ge C_3$ is a
sufficiently large absolute constant.
Since $r_1+m_2=m-\delta\le m$,
\eqref{p3:eq:cubic-deficit} gives $\frac12r_1^2m_2\le2m^3/27$.
Combining this with~\eqref{p3:eq:joint-balance-objective} and the bound
$3m_2m_3+\frac12m_2z_{\rmbal}+2m_2e+r_1\delta
\le\frac72m^2+D^2m$ established above, we obtain
\[
 \cE\le\frac{2}{27}m^3-c_2m^{8/3}+\frac72m^2+D^2m.
\]
Summing over the recorded Sims presentation data and core breadth
sequences adds only $O(m\log m)$, with an absolute implied constant.
Applying both subcase bounds with $A+1$ and adding their
contributions proves~\eqref{p3:eq:small-core-count} whenever
\[
 m\ge C_5\bigl(1+\Lambda^2+D^2
       +(1+\sqrt{\Lambda})(A+1)^{3/2}\bigr),
\]
where $C_5\ge\max\{C_1,C_4\}$ is a sufficiently large absolute
constant. All estimates are uniform in $p$.
\end{proof}

With the preceding estimates, we are ready to prove~\cref{intro:concentration}.
Recall the family $\fH_p$ of groups of Frattini class at most $2$
from~\eqref{intro:eq:frattini-family} and its isomorphism-class count
$f_{\fH_p}(p^m)$ from~\eqref{intro:eq:frattini-counts}.
As noted after~\eqref{intro:eq:frattini-family}, its members have
nilpotency class at most $2$, and its nonabelian members have both
Frattini class and nilpotency class exactly $2$.
\begin{theorem}[Restatement of Theorem~\ref{intro:concentration}]\label{p3:prop:frattini-concentration}
There is an absolute $\varepsilon>0$ such that, for all sufficiently large $m$, uniformly in $p$,
\begin{equation}\label{p3:eq:complement-gap}
 f(p^m)-f_{\fH_p}(p^m)
 \le p^{\frac{2}{27}m^3-\frac49m^2-\varepsilon m^2}.
\end{equation}
Consequently,
\[
 1\le\frac{f(p^m)}{f_{\fH_p}(p^m)}
 \le1+p^{-\varepsilon m^2}.
\]
\end{theorem}
\begin{proof}
We count the groups outside $\fH_p$. The number of Abelian groups of
order $p^m$ is at most $p^m$ (cf.~\cref{p12:lem:universal-field}(iv)), which is negligible for \eqref{p3:eq:complement-gap}. For nonabelian $p$-groups in the nonsaturated range, apply
\eqref{eq:nonsat group enumeration}
in~\cref{p12:cor:nonsaturated-group} with $A>4/9+1$.
This gives the bound
\[
 f_{ns}(p^m)
 \le p^{\frac{2}{27}m^3-Am^2}
 \le p^{\frac{2}{27}m^3-\frac49m^2-m^2}
\]
for sufficiently large $m$, uniformly in $p$.
For nonabelian $p$-groups in the saturated range, fix recorded Sims
presentation data and a core breadth sequence.
The coarse bound~\eqref{p3:eq:coarse}, together with
$m^3-(m-\delta)^3\ge\delta m^2$
from~\eqref{p3:eq:first-layer-defect-deficit}, gives
\[
 \cE\le\frac{2}{27}(m-\delta)^3+C_0m^2\le\frac{2}{27}m^3-\left(\frac{2\delta}{27}-C_0\right)m^2.
\]
Choose an absolute integer $\delta_0$ with $\frac{2\delta_0}{27}>C_0+\frac49+2$.
For $\delta\ge\delta_0$, the fixed-data exponent is therefore at most $\frac{2}{27}m^3-\left(\frac49+2\right)m^2$.

The recorded Sims presentation data have $m^{O(m)}$ choices
by~\cref{p12:prop:sims-generating-data}, and the breadth-sequence
count~\eqref{p12:eq:breadth-sequence-count} gives at most $2^{m_3}\le2^m$
choices for each such choice of recorded Sims presentation data.
Summing these conditional array bounds therefore adds only
$O(m\log m)=o(m^2)$ to the base-$p$ logarithm, uniformly in $p$.
For sufficiently large $m$, this cost is at most $m^2$, so the total
number of groups in the saturated range with $s_1-r_1=\delta\ge\delta_0$
is at most $p^{\frac{2}{27}m^3-\frac49m^2-m^2}$.

Consequently, to prove~\eqref{p3:eq:complement-gap}, it remains to
consider recorded Sims presentation data and core breadth sequences
in the saturated range with $0\le\delta<\delta_0$.
For these choices, \cref{p12:lem:second-layer-defect} gives $s_2-r_2=e\le\binom{\delta}{2}\le\binom{\delta_0-1}{2}=O(1)$ since $\delta_0$ is an absolute constant.

Apply~\cref{p3:lem:small-core-count} with $\Lambda=2$,
$D=\delta_0-1$, and $A=4/9+1$. The total contribution from
$sr_2<2r_1$ with $\delta<\delta_0$ is therefore at most
$p^{\frac{2}{27}m^3-\frac49m^2-m^2}$ for sufficiently large $m$.

We now consider the saturated range with
$0\le\delta\le D=\delta_0-1$ and $sr_2\ge2r_1$.
For every fixed choice occurring in~\eqref{p3:eq:normalizedraw}, since $\dim Z_0\leq r_1-s$ and $\dim K\geq 0$,
we have
\[
 r_1+\dim Z_0-\dim K-sr_2
 \le2r_1-sr_2\le r_1+r_1-sr_2=2r_1-sr_2\leq 0.
\]
The dimension estimate~\eqref{p3:eq:source-dimension-loss}
in~\cref{p3:cor:fixed-parameter-count} gives
$0\le r_1^2-(\dim Z_0)^2\le2r_1(s+\delta)$.
Since $\binom{s-1}{2}<r_2\le\binom{s}{2}$ and $2\le s\le r_1$,
$(s-2)^2\le(s-1)(s-2)<2r_2\le2m_2$.
Thus $s<2+\sqrt{2m_2}\le2+\sqrt{2m}$, and the right-hand side of~\eqref{p3:eq:source-dimension-loss}
is at most $2\sqrt2\,m^{3/2}+2(D+2)m$.

The estimate~\eqref{p3:eq:normalizedraw} in
\cref{p3:cor:fixed-parameter-count} therefore gives a uniform
bound for each fixed choice of recorded Sims presentation data,
core breadth sequence, and the two dimensions. By~\cref{p12:prop:sims-generating-data}, the recorded Sims
presentation data have at most $p^{O(m\log m)}$ choices.
The core breadth sequences have at most $2^{m_3}$ choices,
and the pair $(\dim Z_0,\dim K)$ has at most $(m+1)^2$
choices. Thus summing over all these choices adds
$O(m\log m)$ to the base-$p$ logarithm.
Consequently, the base-$p$ logarithm of the total contribution from
any remaining range of choices is at most the maximum of
\begin{equation}\label{p3:eq:normalized-for-gap}
 \cF_{1/2}+2m_2e+r_1\delta-r_1^2+C_6m^{3/2},
\end{equation}
over choices realized by groups in that range.
Here $C_6$ is an absolute constant accounting for both the dimension
bound and the summation cost.

Using $r_1+m_2=m-\delta$ and
$z_{\rmbal}=r_1-2m_2$, we have the exact identity
\begin{equation}\label{p3:eq:identity}
 -r_1^2=-\frac49(m-\delta)^2-\frac43m_2z_{\rmbal}
              -\frac59z_{\rmbal}^2.
\end{equation}
The definition~\eqref{p3:eq:balance-objective} gives
$\cF_{1/2}-\frac43m_2z_{\rmbal}=\cF_{-5/6}$.
Moreover, \cref{p12:lem:second-layer-defect} gives
$2m_2e+r_1\delta\le D^2m$.
It follows that the expression in~\eqref{p3:eq:normalized-for-gap}
is at most
\begin{equation}\label{p3:eq:normalized-gap-identity}
 \cF_{-5/6}-\frac49(m-\delta)^2
       -\frac59z_{\rmbal}^2+C_7m^{3/2},
\end{equation}
where $C_7=C_6+D^2$ is absolute.

For $1\le\delta\le D$, apply~\cref{p3:lem:positive-depth} with $\eta=-5/6$.
When $m_3=0$, part~\textup{(i)} gives the error term $75m_2/8$.
When $m_3>0$, part~\textup{(ii)} gives a stronger bound once
$m-\delta$ exceeds its absolute threshold.
Since $\delta\le D$, for sufficiently large $m$ both cases imply
\[
 \cF_{-5/6}
 \le\frac{2}{27}(m-\delta)^3+\frac{75}{8}m,
\]
uniformly in $p$.
Discard the nonpositive square
in~\eqref{p3:eq:normalized-gap-identity} and expand
\[
 \frac{2}{27}(m-\delta)^3-\frac49(m-\delta)^2
 =\frac{2}{27}m^3-\frac49m^2-\frac{2\delta}{9}m^2+\frac{2\delta^2+8\delta}{9}m-\frac{2\delta^3}{27}-\frac{4\delta^2}{9}.
\]
The exponent for these choices is therefore at most
\[
 \frac{2}{27}m^3-\frac49m^2-\frac{2\delta}{9}m^2+C_8m^{3/2}
 \le\frac{2}{27}m^3-\frac49m^2-\frac19m^2,
\]
where
$C_8=C_7+\frac{75}{8}+\frac{2D^2+8D}{9}$ is absolute.
The last inequality holds whenever $m\ge(9C_8)^2$.
Together with the threshold in~\cref{p3:lem:positive-depth},
this bounds the total contribution from $1\le\delta\le D$,
uniformly in $p$.

It remains to consider $\delta=0$.
Then $e=0$ by~\cref{p12:lem:second-layer-defect}.
If $m_3>0$, \cref{p3:lem:positive-depth}\textup{(ii)} gives
$\cF_{-5/6}\le2m^3/27-m^2/128$.
Thus~\eqref{p3:eq:normalized-gap-identity} is at most
\[
 \frac{2}{27}m^3-\frac49m^2-\frac{m^2}{128}+C_7m^{3/2}
 \le\frac{2}{27}m^3-\frac49m^2-\frac{m^2}{256}
\]
whenever $m\ge(256C_7)^2$ and the threshold
in~\cref{p3:lem:positive-depth}\textup{(ii)} is satisfied.

The remaining nonabelian groups have $\delta=m_3=0$.
The equality $\delta=0$ makes $G/G_2$ elementary abelian, so
$G^p\le G_2$.
Also, $m_3=0$ gives $G_3=1$, and $e=0$ makes
$G_2=G_2/G_3$ elementary abelian.
Consequently, $\Phi(G)=G^pG_2=G_2\le Z(G)$ and $\Phi(G)^p=1$.
Every such group therefore belongs to $\fH_p$.

All groups outside $\fH_p$ occur among the cases already
bounded: the abelian groups, the nonsaturated groups, the saturated
groups with $\delta\ge\delta_0$, the range $sr_2<2r_1$ with
$\delta<\delta_0$, the remaining choices with $\delta\ge1$,
and those with $\delta=0$ and $m_3>0$.
For sufficiently large $m$, each of these six contributions is at
most $p^{2m^3/27-4m^2/9-m^2/256}$.
Their sum satisfies
\[
 6p^{\frac{2}{27}m^3-\frac49m^2-\frac{m^2}{256}}
 \le p^{\frac{2}{27}m^3-\frac49m^2-\frac{m^2}{1024}},
\]
after increasing the absolute threshold on $m$.
This proves~\eqref{p3:eq:complement-gap} with $\varepsilon=1/1024$.
Higman's lower bound in~\eqref{p3:eq:BLQR} then gives
\[
 1\le\frac{f(p^m)}{f_{\fH_p}(p^m)}
 =1+\frac{f(p^m)-f_{\fH_p}(p^m)}
          {f_{\fH_p}(p^m)}
 \le1+p^{-m^2/1024}.
\]
\end{proof}

Recall~\eqref{intro:eq:BLQR}, supplied by Higman's lower bound
\cite[Theorem~2.3]{Higman} and the upper bound
of~\cite[Theorem~1.3 and Section~4.3]{BLQR}:
\begin{equation}\label{p3:eq:BLQR}
 p^{\frac{2}{27}m^3-\frac49m^2}\le f_{\fH_p}(p^m)\le p^{\frac{2}{27}m^3-\frac49m^2+C_{\!H}m},
\end{equation}
where $C_{\!H}$ is an absolute constant. Together with~\cref{p3:prop:frattini-concentration}, this proves~\cref{intro:sharp}.

\begin{theorem}[Restatement of Theorem~\ref{intro:sharp}]\label{p3:thm:linear}
There is an absolute constant $C$ such that, for every prime $p$
and every integer $m\ge1$,
\begin{equation}\label{p3:eq:sharp}
 p^{\frac{2}{27}m^3-\frac49m^2}\le f(p^m)\le p^{\frac{2}{27}m^3-\frac49m^2+Cm}.
\end{equation}
Thus $\log_p f(p^m)=2m^3/27-4m^2/9+O(m)$ uniformly in $p$.
\end{theorem}

\begin{proof}
Let $m_0\ge3$ be an absolute threshold for
\eqref{p3:eq:complement-gap}.
For $m\ge m_0$, combine that estimate with~\eqref{p3:eq:BLQR} to obtain
\[
 \begin{aligned}
 f(p^m)
 &=f_{\fH_p}(p^m)
       +\bigl(f(p^m)-f_{\fH_p}(p^m)\bigr)\\
 &\le p^{\frac{2}{27}m^3-\frac49m^2+C_{\!H}m}
       +p^{\frac{2}{27}m^3-\frac49m^2-\varepsilon m^2}\\
 &\le2p^{\frac{2}{27}m^3-\frac49m^2+C_{\!H}m}
 \le p^{\frac{2}{27}m^3-\frac49m^2+(C_{\!H}+1)m}.
 \end{aligned}
\]
The last inequality uses $2\le p^m$.

To cover $m<m_0$ uniformly in $p$, the presentation counts
in~\cref{p12:prop:sims-generating-data,p12:lem:power-ledger}
and~\eqref{p12:eq:book-master}, together with the abelian-group
count, give an absolute constant $C_9$ such that
$f(p^m)\le p^{C_9m^3}$.
The constant $C$ in the statement can therefore be chosen so that
\[
 C\ge\max\left\{C_{\!H}+1,\,
                C_9m_0^2+\frac49m_0\right\}.
\]
Indeed, for $1\le m<m_0$ this gives
\[
 C_9m^3
 \le Cm-\frac49m^2
 \le\frac{2}{27}m^3-\frac49m^2+Cm,
\]
which proves the upper bound for the remaining exponents.
The lower bound follows from~\eqref{p3:eq:BLQR} and
$f(p^m)\ge f_{\fH_p}(p^m)$ for $m\ge2$.
For $m=1$, it follows directly from $f(p)=1$.
\end{proof}

We adapt Higman's lower bound proof in~\cite{Higman} to show that the linear term in the counting exponent is necessary:
\begin{proposition}[Refinement of~{\cite[Theorem 2.3]{Higman}}]\label{p3:prop:linear-necessary}
For every prime $p$ and all sufficiently large $m$,
\[
 \log_p f(p^m)\ge \frac{2}{27}m^3-\frac49m^2+\frac29m-2.
\]
\end{proposition}
\begin{proof}
For $r\ge1$ and $1\le d_{\Phi}\le r(r+1)/2$, Higman's bound
\cite[Theorem~2.1]{Higman} for groups whose Frattini subgroup has
order $p^{d_{\Phi}}$ and index $p^r$ gives
\[
 \log_p f(p^m)\ge E(r,d_{\Phi}),\qquad m=r+d_{\Phi},
\]
where
\[
 E(r,d_{\Phi})
 =\tfrac12r^2d_{\Phi}+\tfrac12r d_{\Phi}-d_{\Phi}^2-r^2.
\]
Put $z_{\Phi}=r-2d_{\Phi}$.
Solving together with $m=r+d_{\Phi}$ gives
$r=(2m+z_{\Phi})/3$ and $d_{\Phi}=(m-z_{\Phi})/3$.
Substitution yields
\[
 E(r,d_{\Phi})-\frac{2}{27}m^3+\frac49m^2
 =-\frac{z_{\Phi}^2+5z_{\Phi}}{18}m
       -\frac{z_{\Phi}^3}{54}-\frac{5z_{\Phi}^2}{18}.
\]
Choose $z_{\Phi}=-3,-2,-1$ according as $m\equiv0,1,2\pmod3$.
Writing $m=3t,3t+1,3t+2$, respectively, gives
\[
 (r,d_{\Phi})=(2t-1,t+1),\quad
 (2t,t+1),\quad
 (2t+1,t+1).
\]
For $m\ge6$, these are integers satisfying
$1\le d_{\Phi}\le r\le r(r+1)/2$, so Higman's bound applies.
The corresponding values of the preceding expression are
\[
 E(r,d_{\Phi})-\frac{2}{27}m^3+\frac49m^2
 =
 \begin{cases}
 m/3-2,&m\equiv0\pmod3,\\
 m/3-26/27,&m\equiv1\pmod3,\\
 2m/9-7/27,&m\equiv2\pmod3.
 \end{cases}
\]
Each value is at least $2m/9-2$.
Combining this with $\log_p f(p^m)\ge E(r,d_{\Phi})$ proves the
claimed lower bound, uniformly in $p$.
\end{proof}

\begin{corollary}[Restatement of Corollary~\ref{cor:pyber-sharpening}]
\label{p3:cor:pyber-sharpening}
With $f(n)$ and $\mu(n)$ as defined in~\eqref{intro:eq:group-counts},
we have
\[
 f(n)\le n^{\frac{2}{27}\mu(n)^2+O(\mu(n))},
\]
where the implied constant is absolute.
\end{corollary}
\begin{proof}
The case $n=1$ follows from $f(1)=1$.
For $n>1$, write $n=\prod_{p\mid n}p^{a_p}$, so that
$1\le a_p\le\mu(n)$ for every prime divisor $p$ of $n$.
By~\eqref{p3:eq:sharp}, there is an absolute constant $C$ such that
\[
 f(p^{a_p})
 \le p^{\frac{2}{27}a_p^3+Ca_p}
 \le (p^{a_p})^{\frac{2}{27}\mu(n)^2+C}.
\]
Multiplying over the prime divisors of $n$ gives $\prod_{p\mid n}f(p^{a_p}) \le n^{\frac{2}{27}\mu(n)^2+C}$.
Pyber's prescribed-Sylow estimate~\cite[p.~204]{Pyber} therefore yields
\[
 f(n)
 \le n^{75\mu(n)+16}\prod_{p\mid n}f(p^{a_p})
 \le n^{\frac{2}{27}\mu(n)^2+75\mu(n)+16+C}
 =n^{\frac{2}{27}\mu(n)^2+O(\mu(n))},
\]
where the implied constant is absolute because $\mu(n)\ge1$.
\end{proof}

\section{Cumulative domination by \texorpdfstring{$2$}{2}-groups of Frattini class \texorpdfstring{$2$}{2}}
\label{p4:sec:extensions}

\subsection{Cumulative domination}\label{p4:sec:statement}

Use the isomorphism-class notation $[G]$ from~\eqref{intro:eq:group-counts}
and the family of $p$-groups of Frattini class $2$, $\fH_p$, from~\eqref{intro:eq:frattini-family}.
Recall the cumulative counts from
\eqref{intro:eq:cumulative-counts} and~\eqref{intro:eq:frattini-counts}:
\[
 \begin{aligned}
 F(x)&=\sum_{1\le n\le x}f(n),\\
 F_2(x)&=\#\{[G]:|G|\le x,\ G\text{ is a $2$-group}\}
       =\sum_{k=0}^{\lfloor\log_2x\rfloor}f(2^k),\\
 F_{\fH_p}(x)&=\#\{[G]:|G|\le x,\ G\in\fH_p\}
       =\sum_{k=0}^{\lfloor\log_p x\rfloor}f_{\fH_p}(p^k).
 \end{aligned}
\]
The sums starting at $k=0$ include the trivial group, with
$f(1)=f_{\fH_p}(1)=1$.
The family $\fH_p$ has abelian members, and every member has
nilpotency class at most two. Its nonabelian members have both
Frattini class and nilpotency class exactly two, as noted after
\eqref{intro:eq:frattini-family}.

\begin{theorem}[Restatement of Theorem~\ref{intro:domination}]\label{p4:thm:main}\label{p4:thm:class-2}
Let $F(x)$ and $F_{\mathfrak{H}_2}(x)$ be from above. Then we have
\[
 \lim_{x\to\infty}\frac{F_{\fH_2}(x)}{F(x)}=1.
\]
Moreover, almost all finite groups, counted up to isomorphism and
ordered by their cardinalities, are nonabelian $2$-groups of Frattini
class exactly 2. 
\end{theorem}

In particular, if we can write
\[
 F_{2,2}(x)=\#\{[G]:|G|\le x,\ |G|\text{ is a power of }2
                         \ \text{ of nilpotency class }2\},
\]
then \cref{p4:thm:class-2} implies that 
\[
 \lim_{x\to\infty}\frac{F_{2,2}(x)}{F(x)}=1.
\]

For the cumulative estimates, put
\begin{equation}\label{p4:eq:cumulative-polynomial}
 Q(t)=\frac{2}{27}t^3-\frac49t^2.
\end{equation}
When $m=\lfloor\log_2x\rfloor$, the groups of order $2^m$ alone
give the lower bound $F(x)\ge f(2^m)\ge2^{Q(m)}$ for large $m$.
We compare all exceptional counts with this lower bound.

We first remove groups whose orders have nontrivial odd part.
For a bounded odd part, the core of a Sylow $2$-subgroup gives a
normal $2$-subgroup with quotient in a fixed finite list.
We therefore begin by counting extensions of $2$-groups by a
fixed quotient. The argument combines the joint array bounds and
source multiplicities used in
Sections~\ref{p3:sec:statement}--\ref{p3:sec:linear}, conditioned on
fixed recorded Sims presentation data, with an additional weight
for the extension choices.
For large odd parts, the uniform bound~\eqref{p3:eq:sharp} and
Pyber's prescribed-Sylow estimate give a sufficient saving
directly. After these two steps prove $F_2(x)/F(x)\to1$,
\cref{p3:prop:frattini-concentration} removes the $2$-groups outside
$\fH_2$. The final subsection deduces the cumulative estimates and
order comparisons in~\cref{intro:bnv-2218}.

\subsection{Outer actions and factor systems}

Fix a finite group $\Gamma$ whose order has nontrivial odd part, that is, $|\Gamma|=2^jb$ for some $j=geq 0$ and $b>1$ being odd. Denote $v=j+\lfloor\log_2 b\rfloor$.
An extension of $N$ by $\Gamma$ is a group $G$ in an exact sequence
$1\to N\to G\to\Gamma\to1$.
We count the isomorphism classes of the groups $G$, so different
choices of the embedding of $N$ or the quotient map need not give
different objects in this count.
For a finite $2$-group $N$ of order $2^n$, use the rank notation
from~\eqref{p12:eq:layer-parameters}, so $r_1=d(N/[N,N])=d(N)$ is the minimum size of a generating set of $N$, and put
\begin{equation}\label{p4:eq:automorphism-image}
 A(N)=\im\bigl(\Aut(N)\longrightarrow\GL(N/\Phi(N))\bigr),
 \qquad a(N)=|A(N)|.
\end{equation}
Thus $N/\Phi(N)$ is an $r_1$-dimensional vector space over
$\F_2$.

We also write $\Out(N)=\Aut(N)/\Inn(N)$.
The following lemmas bound the number of groups obtained from one
fixed kernel $N$ and one fixed quotient $\Gamma$.

We first recall an elementary fact.
\begin{lemma}[An odd first generator]\label{p4:lem:generators}
Let $\Gamma$ be a group of order $2^jb$ with $j\geq 0$ and $b>1$ odd. Let $v=j+\lfloor \log_2 b\rfloor$. Then $\Gamma$ has a generating tuple $\gamma_1,\ldots,\gamma_{h+1}$
in which $\gamma_1$ has odd prime order, with $h\ge0$ satisfying
\begin{equation}\label{p4:eq:h}
 h\le\lfloor\log_2(|\Gamma|/3)\rfloor\le v-1.
\end{equation}
\end{lemma}

For each fixed group $\Gamma$, choose and fix one generating tuple from
\cref{p4:lem:generators}. Its length $h+1$ and the odd prime order
$\lambda$ of $\gamma_1$ are fixed in all subsequent counts for this
quotient.

A result of P.~Hall (cf. \cite[Theorem~5.3.3]{Robinson}) gives the following.
\begin{lemma}[The kernel of the Frattini action]
Let \(J_N=\ker(\operatorname{Aut}(N)\to A(N))\).
Then \(J_N\) is a \(2\)-group containing \(\operatorname{Inn}(N)\),
and \(|J_N|\mid 2^{r_1(n-r_1)}\).
Consequently, the Frattini action induces a natural surjection
\(\operatorname{Out}(N)\to A(N)\) whose kernel
\(T_N=J_N/\operatorname{Inn}(N)\) is a \(2\)-group satisfying
\begin{equation}\label{p4:eq:frattini}
 |T_N|\le2^{r_1(n-r_1)},\qquad |\Out(N)|=a(N)|T_N|.
\end{equation}
\end{lemma}
The following bound combines the Schur--Zassenhaus theorem
\cite[Theorem~9.1.2]{Robinson} with classical extension theory
\cite[Theorem~11.4.10]{Robinson}.
We give an elementary counting proof using outer actions and factor systems.
\begin{lemma}[Outer actions and factor systems]\label{p4:lem:extensions}
Let \(h\) be as in Lemma~15.2.
The number of isomorphism classes of groups admitting an extension
of \(N\) by \(\Gamma\) is at most
\begin{equation}\label{p4:eq:extension-cost}
 a(N)^{h+1}2^{h r_1(n-r_1)+n|\Gamma|^2}.
\end{equation}
\end{lemma}
\begin{proof}
Choose a generating tuple \(\gamma_1,\ldots,\gamma_{h+1}\)
of \(\Gamma\) as in Lemma~15.2, with \(\gamma_1\) of odd prime
order \(\lambda\).
Every extension induces an outer action, namely a homomorphism
$\theta:\Gamma\to\Out(N)$ obtained from conjugation on $N$.
We first count these homomorphisms up to conjugacy in $\Out(N)$,
and then count the factor systems for each chosen outer action.

For $\alpha\in A(N)$ with $\alpha^\lambda=1$, consider the elements
$y\in\Out(N)$ satisfying $y^\lambda=1$ and mapping to $\alpha$.
There is at most one $T_N$-conjugacy class of such elements.
Indeed, if $\alpha=1$, then $y\in T_N$; as $T_N$ is a $2$-group
and $\lambda$ is odd, this forces $y=1$.
Otherwise $\alpha$ has order $\lambda$. In the inverse image of
$\langle\alpha\rangle$, the subgroup $T_N$ is a normal Hall
$2$-subgroup. Every such $y$ has order $\lambda$ and generates a
complement to $T_N$.
The Schur--Zassenhaus conjugacy theorem
\cite[Theorem~9.1.2]{Robinson} applies, since the quotient
is cyclic of odd prime order, so these complements are conjugate
by elements of $T_N$. In each complement there is exactly one
element mapping to $\alpha$, because projection onto
$\langle\alpha\rangle$ is an isomorphism. The asserted conjugacy
of the elements $y$ follows.

For any outer action $\theta$, the element $\theta(\gamma_1)$ has
order dividing $\lambda$. The preceding argument gives at most
$a(N)$ choices for it up to conjugacy in $\Out(N)$.
Choose a representative for each possible conjugacy class of this
first image. Every homomorphism is conjugate to one with first
image among these representatives, and its remaining $h$ generator
images have at most $|\Out(N)|^h$ choices.
Some tuples fail the relations of $\Gamma$, and some represent the
same conjugacy class; retaining them only increases the count.
Consequently, by~\eqref{p4:eq:frattini}, the number of outer actions
up to conjugacy is at most
\[
 a(N)|\Out(N)|^h
 \le a(N)^{h+1}2^{h r_1(n-r_1)}.
\]

Fix one representative outer action $\theta$.
For each $\gamma\in\Gamma$, choose a lift
$\alpha_\gamma\in\Aut(N)$ of $\theta(\gamma)$, taking $\alpha_1=1$.
In a realizing extension, choose coset representatives
$t_\gamma$ with $t_1=1$, and let
$\psi_\gamma(x)=t_\gamma x t_\gamma^{-1}$ for $x\in N$.
The automorphisms $\psi_\gamma$ and $\alpha_\gamma$ have the same
outer class, so
$\psi_\gamma=\Inn(x_\gamma)\alpha_\gamma$ for some $x_\gamma\in N$,
where $\Inn(x)(u)=xux^{-1}$.
Replacing $t_\gamma$ by $x_\gamma^{-1}t_\gamma$ makes its
conjugation action exactly $\alpha_\gamma$; for $\gamma=1$ we may
take $x_1=1$.

For these adjusted representatives, the factors
$c(\gamma,\omega)=t_\gamma t_\omega t_{\gamma\omega}^{-1}$ lie in
$N$ and satisfy
\[
 \alpha_\gamma\alpha_\omega
 =\Inn(c(\gamma,\omega))\alpha_{\gamma\omega}.
\]
The chosen automorphisms therefore fix the inner automorphism
induced by each factor. The elements of $N$ inducing any one inner
automorphism form a coset of $Z(N)$, so there are at most $|Z(N)|$
choices for each $c(\gamma,\omega)$.
There are $|\Gamma|^2$ ordered pairs, giving at most
$|Z(N)|^{|\Gamma|^2}\le2^{n|\Gamma|^2}$ possible factor systems.
To see why this counts all realizing groups, identify the element
$ut_\gamma$ with $(u,\gamma)\in N\times\Gamma$.
Multiplication is then determined by
\[
 (u,\gamma)(w,\omega)=
 \bigl(u\alpha_\gamma(w)c(\gamma,\omega),\gamma\omega\bigr).
\]
The identity and associativity impose further conditions on the
factors, so counting all the preceding choices is an upper bound.
Finally, a conjugator in $\Out(N)$ lifts to an automorphism of $N$;
relabeling the kernel by this automorphism gives an isomorphic
group. Thus it suffices to count the representative outer actions
already chosen. Multiplying their number by the factor-system
bound proves~\eqref{p4:eq:extension-cost}.
\end{proof}

\subsection{Counting the extension families}\label{p4:sec:fixed-quotient-count}

We first prove the weighted enumeration estimate needed for the
extension count. It depends only on a fixed integer $h$, and its
statement uses only the generator number and the automorphism
image of the kernel.

\begin{lemma}[Weighted enumeration of kernels]
\label{p4:lem:weighted-kernels}
For a finite $2$-group $N$, let $d(N)$ be its minimum number of
generators and put
\[
 A(N)=\im\bigl(\Aut(N)\longrightarrow\GL(N/\Phi(N))\bigr),
 \qquad a(N)=|A(N)|.
\]
For every fixed integer $h\ge0$ and $n\geq 1$,
\begin{equation}\label{p4:eq:weighted-kernels}
 W_h(n):=\sum_{\substack{[N]\\|N|=2^n}}
 a(N)^{h+1}2^{h d(N)(n-d(N))}
 \le2^{\frac{2}{27}n^3+\frac{2h-4}{9}n^2+O_h(n^{3/2})}.
\end{equation}
The sum is over isomorphism classes of $2$-groups.
\end{lemma}
\begin{proof}
Fix $h\ge0$. For a kernel $N$ of order $2^n$, write $r_1=d(N)$
and call
\[
 w_h(N):=a(N)^{h+1}2^{h r_1(n-r_1)}
\]
its weight. We first remove ranges for which coarse bounds on the
weight suffice. For the remaining kernels, we combine the
fixed-parameter group count with a sharper automorphism bound,
then optimize the resulting exponent and sum over the parameters.

\paragraph{Notation and numerical inputs.}
For a nonabelian kernel $N$ of nilpotency class $c$, set
$N_1=N$, $N_{i+1}=[N_i,N]$, and
\[
 r_i=d(N_i/N_{i+1}),\qquad
 s_i=\log_2|N_i/N_{i+1}|,\qquad
 m_i=\log_2|N_i|\quad(i\in[c]),
\]
with $m_i=0$ for $i>c$. Thus $m_1=n$ and $r_1=d(N)$.
A Sims core is a subgroup $H$ with as few generators as possible
such that $N_2=[H,H]N_3$; write $s=d(H)$.
Put
\[
 \delta=s_1-r_1,\qquad e=s_2-r_2,\qquad
 z_{\rmbal}=r_1-2m_2.
\]
Then $n=r_1+m_2+\delta$, and
\cref{p12:lem:second-layer-defect} gives $e\le\binom\delta2$.
The saturated range is
\[
 \binom{s-1}{2}<r_2\le\binom s2,\qquad 2\le s\le r_1;
 \qquad s<2+\sqrt{2m_2}\le2+\sqrt{2n}.
\]

Recall from~\eqref{p3:eq:balance-objective} that, for each fixed
$\eta\in\R$,
\[
 \cF_\eta
 :=\frac12r_1^2m_2+\cX+3m_2m_3+\eta m_2z_{\rmbal}.
\]
Here $m_3=\log_2|N_3|$, $z_{\rmbal}=r_1-2m_2$, and $\cX$
is the correction term in the commutator count defined
in~\eqref{p12:eq:optimization-X}. The fixed-parameter count below
uses $\cF_{1/2}$. We need only the following two properties of
$\cF_\eta$, which follow from its definition and
\cref{p3:lem:positive-depth}: for every fixed $\eta\in\R$ and
every saturated choice of recorded Sims presentation data and
core breadth sequence,
\begin{equation}\label{p4:eq:balance-input}
 \begin{aligned}
 \cF_\eta-\cF_{1/2}
   &=(\eta-\tfrac12)m_2z_{\rmbal},\\
 \cF_\eta
   &\le\frac{2}{27}(r_1+m_2)^3+O_\eta(n).
 \end{aligned}
\end{equation}
For the second inequality, when $m_3=0$ the error is at most
$27\eta^2m_2/2$. When $m_3>0$ and $r_1+m_2$ is sufficiently
large, the cited proposition gives a quadratic saving; the
remaining bounded values are absorbed into the error.
All constants in the proof below may depend on $h$.

\paragraph{Step 1: remove ranges with a larger saving.}
Since $a(N)\le|\GL(r_1,2)|\le2^{r_1^2}$ and
$r_1(n-r_1)\le n^2/4$, every weight satisfies
\begin{equation}\label{p4:eq:crude-cost}
 w_h(N)\le2^{(h+1+h/4)n^2}.
\end{equation}
Choose a fixed $A>0$ such that
\[
 A>h+1+h/4+\left|\frac{2h-4}{9}\right|+2.
\]
By~\cref{p12:cor:nonsaturated-group}, for sufficiently large $n$
the nonsaturated kernels contribute to $W_h(n)$ at most
\[
 2^{\frac{2}{27}n^3-(A-h-1-h/4)n^2}
 \le2^{\frac{2}{27}n^3+\frac{2h-4}{9}n^2-2n^2}.
\]
There are at most $2^n$ abelian kernel isomorphism classes by
\cref{p12:lem:universal-field}\textup{(iv)}. Their weighted
contribution is at most $2^{(h+1+h/4)n^2+n}$, which is also smaller
than the last display for sufficiently large $n$.

For saturated kernels with large $\delta$, fix recorded Sims
presentation data and a core breadth sequence. As in
\cref{p12:thm:joint-enumeration}, let $\cE$ be the base-$2$
logarithm of the number of distinct pairs of full power word
arrays and commutator coefficient arrays arising from kernels
of order $2^n$ with ordered Sims generating sets realizing these
fixed choices. The coarse bound~\eqref{p3:eq:coarse} gives, for
an absolute constant $C_0$,
\[
 \cE\le\frac{2}{27}(n-\delta)^3+C_0n^2
 \le\frac{2}{27}n^3-
       \left(\frac{2\delta}{27}-C_0\right)n^2,
\]
where the second inequality uses
$n^3-(n-\delta)^3\ge\delta n^2$ from
\eqref{p3:eq:first-layer-defect-deficit}.
Choose a fixed nonnegative integer $D$, depending only on $h$,
such that
\[
 \frac{2(D+1)}{27}>
 C_0+h+1+h/4+\left|\frac{2h-4}{9}\right|+3.
\]
Each array pair together with the recorded data determines at
most one kernel isomorphism class, by
\cref{thm:sims_presentation_determines_isomorphism_classes_}.
Thus the array count times the maximum weight bounds the
corresponding weighted group count. For $\delta\ge D+1$, its
fixed-choice exponent is at most
\[
 \frac{2}{27}n^3-
 \left(\left|\frac{2h-4}{9}\right|+3\right)n^2.
\]
There are $2^{O(n\log n)}$ choices of recorded Sims presentation
data and core breadth sequence, by
\cref{p12:prop:sims-generating-data} and
\eqref{p12:eq:breadth-sequence-count}. Their logarithmic cost is
at most $n^2$ for sufficiently large $n$. Consequently, all kernels
with $\delta\ge D+1$ contribute at most
$2^{2n^3/27+(2h-4)n^2/9-2n^2}$ to $W_h(n)$.

For $0\le\delta\le D$, apply
\cref{p3:lem:small-core-count} with $p=2$, $m=n$,
$\Lambda=h+3$, and the preceding $A$.
It bounds the number of saturated kernels with
$sr_2<(h+3)r_1$ by $2^{2n^3/27-An^2}$.
Multiplication by~\eqref{p4:eq:crude-cost} again leaves a weighted
contribution at most $2^{2n^3/27+(2h-4)n^2/9-2n^2}$.
It therefore remains to count saturated kernels satisfying
\begin{equation}\label{p4:eq:remaining-kernels}
 0\le\delta\le D,\qquad e\le\binom D2,\qquad
 sr_2\ge(h+3)r_1.
\end{equation}

\paragraph{Step 2: combine the group count and the automorphism weight.}
For each isomorphism class of nonabelian kernels $N$ of order
$2^n$ satisfying
\[
 \binom{s-1}{2}<r_2\le\binom{s}{2},\qquad
 2\le s\le r_1,\qquad
 0\le\delta\le D,\qquad
 sr_2\ge(h+3)r_1,
\]
choose one representative $N$, a Sims core $H$, and auxiliary
choices as in~\cref{p3:cor:fixed-parameter-count}. Let
\[
 V=N/\Phi(N),\qquad W=N_2/(N_3\Phi(N_2)),\qquad
 U=H\Phi(N)/\Phi(N).
\]
Here $\dim V=r_1$, $\dim W=r_2$, and $\dim U=s$.
Let $\beta_N:V\times V\to W$ be the alternating map induced by
commutation, and put
\[
 K=\{v\in V:\beta_N(v,U)=0\},\qquad k=\dim K.
\]

An automorphism whose induced action fixes $U$ pointwise also
fixes $W$ pointwise, since the values of $\beta_N$ on $U\times U$
span $W$. For its action $g$ on $V$, compatibility with $\beta_N$
therefore gives
\[
 \beta_N((g-1)v,u)
 =\beta_N(gv,gu)-\beta_N(v,u)=0
 \qquad(v\in V,\ u\in U).
\]
Thus $(g-1)V\subseteq K$ and $(g-1)U=0$, so $g-1$ factors
through a linear map $V/U\to K$. There are $2^{k(r_1-s)}$
such maps. The ordered core basis has at most $2^{sr_1}$ possible
images under $A(N)$, and orbit--stabilizer gives
\begin{equation}\label{p4:eq:autoimage}
 a(N)\le2^{sr_1+k(r_1-s)}.
\end{equation}

Group the chosen kernels by their recorded Sims presentation
data, core breadth sequence, and the two dimensions
$u=\dim Z_0$ and $k=\dim K$. For each such choice,
\cref{p3:cor:fixed-parameter-count} gives
\[
 \max\{0,r_1-s-\delta\}\le u\le r_1-s
\]
and bounds the logarithm of the number of kernel isomorphism
classes with any one fixed choice by
\[
 \cF_{1/2}+2m_2e+r_1\delta-u^2
       +k(r_1+u-k-sr_2)+O(1).
\]
The bound~\eqref{p4:eq:autoimage} is uniform for these same
choices. Multiplying the number of kernel classes by the maximum
weight therefore bounds the logarithm of their contribution to
$W_h(n)$ by
\begin{align}\label{p4:eq:normalized-extension}
 &\cF_{1/2}+2m_2e+r_1\delta-u^2
       +(h+1)sr_1+hr_1(n-r_1)\notag\\
 &\quad+k\bigl(r_1+u-k-sr_2+(h+1)(r_1-s)\bigr)+O(1).
\end{align}
Since $u\le r_1$ and $sr_2\ge(h+3)r_1$, the coefficient of $k$
is at most
\[
 (h+3)r_1-sr_2-k-(h+1)s\le0.
\]
This is the reason for the threshold $h+3$ in Step~1.
We may omit the entire $k$ term in an upper bound.

The bounds on $u$ give
\[
 0\le r_1^2-u^2\le2r_1(s+\delta)=O_h(n^{3/2}),
\]
using $s<2+\sqrt{2n}$ and $\delta\le D$.
Also $(h+1)sr_1=O_h(n^{3/2})$ and $n-r_1=m_2+\delta$.
Consequently,~\eqref{p4:eq:normalized-extension} is at most
\[
 \cF_{1/2}-r_1^2+hr_1m_2
       +2m_2e+(h+1)r_1\delta+O_h(n^{3/2}).
\]
Here $2m_2e+(h+1)r_1\delta\le(D^2+hD)n$ by
\eqref{p4:eq:remaining-kernels}. Thus the fixed-choice weighted
exponent has been reduced to
\begin{equation}\label{p4:eq:weighted-objective}
 \cF_{1/2}-r_1^2+hr_1m_2+O_h(n^{3/2}).
\end{equation}

\paragraph{Step 3: optimize and sum over the choices.} Step~2 bounds the weighted contribution for each fixed choice
of parameters. We now bound this exponent uniformly in those
parameters and then sum over all choices.

Since $r_1=2m_2+z_{\rmbal}$ and
$n-\delta=3m_2+z_{\rmbal}$, direct expansion gives
\begin{equation}\label{p4:eq:quadratic-identity}
 -r_1^2+hr_1m_2
 =\frac{2h-4}{9}(n-\delta)^2
  -\frac{h+4}{3}m_2z_{\rmbal}
  -\frac{2h+5}{9}z_{\rmbal}^2.
\end{equation}
Apply the first identity in~\eqref{p4:eq:balance-input} with
$\eta=-(2h+5)/6$.
The expression in~\eqref{p4:eq:weighted-objective} becomes
\[
 \cF_{-(2h+5)/6}
 +\frac{2h-4}{9}(n-\delta)^2
 -\frac{2h+5}{9}z_{\rmbal}^2+O_h(n^{3/2}).
\]
The square term is nonpositive. By~\eqref{p4:eq:balance-input}
and $r_1+m_2=n-\delta$, this is at most
\[
 \frac{2}{27}(n-\delta)^3
 +\frac{2h-4}{9}(n-\delta)^2+O_h(n^{3/2}).
\]
The polynomial $2y^3/27+(2h-4)y^2/9$ has derivative
$2y(y+2h-4)/9$, which is nonnegative for $y\ge4$.
For $n\ge D+4$, replacing $n-\delta$ by $n$ therefore gives
\[
 \frac{2}{27}n^3+\frac{2h-4}{9}n^2+O_h(n^{3/2})
\]
for every fixed choice still under consideration.

By~\cref{p3:cor:fixed-parameter-count}, there are
$2^{O(n\log n)}$ choices of recorded Sims presentation data,
core breadth sequence, and the two dimensions. Summing over
them adds only $O(n\log n)$ to the logarithm, which is absorbed
into $O_h(n^{3/2})$. The ranges removed in Step~1 satisfy
stronger bounds, so their addition has only a bounded logarithmic
cost for sufficiently large $n$. Increasing the implied constant
to cover the finitely many remaining $n$ proves
\eqref{p4:eq:weighted-kernels}.
\end{proof}

We now apply the weighted estimate to extensions with a fixed
quotient. For each integer $n\ge0$, let $g_\Gamma(n)$ be the
number of isomorphism classes of groups admitting a normal
subgroup of order $2^n$ with quotient isomorphic to $\Gamma$.
A group admitting several such normal subgroups may be counted
more than once, which gives an upper bound. In particular,
$g_\Gamma(0)=1$.

\begin{proposition}\label{p4:prop:fixedQ}
Fix $\Gamma$ of order $2^jb$, with $j\ge0$ and $b$ odd.
Let $\gamma_1,\ldots,\gamma_{h+1}$ be a generating tuple of
$\Gamma$ as in~\cref{p4:lem:generators}, where
$h\le j+\lfloor\log_2b\rfloor$. Then for $n\ge 1$,
\begin{equation}\label{p4:eq:gQ}
 g_\Gamma(n)\le
 2^{\frac{2}{27}n^3+\frac{2h-4}{9}n^2+O_\Gamma(n^{3/2})}.
\end{equation}
\end{proposition}
\begin{proof}
Choose the kernel isomorphism class first. By
\cref{p4:lem:extensions}, with $r_1=d(N)$,
\begin{equation}\label{p4:eq:kernel-sum}
 g_\Gamma(n)
 \le2^{n|\Gamma|^2}
 \sum_{\substack{[N]\\|N|=2^n}}
 a(N)^{h+1}2^{h r_1(n-r_1)}
 =2^{n|\Gamma|^2}W_h(n).
\end{equation}
Apply~\cref{p4:lem:weighted-kernels}. Since $\Gamma$ and $h$ are
fixed, the additional exponent $n|\Gamma|^2=O_\Gamma(n)$ is
absorbed into $O_\Gamma(n^{3/2})$. This proves~\eqref{p4:eq:gQ}.
\end{proof}

\subsection{The cumulative cutoff and bounded odd cofactors}
\label{p4:sec:odd-part}

We next compare the extension count with the groups of order
$2^{\lfloor\log_2x\rfloor}$, all of which are counted by $F(x)$.
The order of the fixed quotient reduces the largest possible
kernel order. Together with~\eqref{p4:eq:h}, this gives the
quadratic saving needed below.

\begin{proposition}[The cumulative cutoff]\label{p4:prop:cutoff}
Let $\Gamma$ be a group whose order has a nontrivial odd part. Let $x$ be a sufficiently large integer, and set $m=\lfloor\log_2x\rfloor$. Then the proportion of groups of order at most $x$ admitting
a normal $2$-subgroup with quotient isomorphic to $\Gamma$ is at most
\begin{equation}\label{p4:eq:margin}
 2^{-\frac29m^2+O_{\Gamma}(m^{3/2})}.
\end{equation}
\end{proposition}
\begin{proof}
Recall from~\cref{p4:lem:generators} notations $|\Gamma|=2^jb$, $v=j+\lfloor\log_2b\rfloor$, and the
integer $h$.
An extension of $N$ by $\Gamma$ with $N$ of order $2^n$ has order $b2^{n+j}$.
The inequalities
\[
 2^{n+j+\lfloor\log_2b\rfloor}\le b2^{n+j}\le x<2^{m+1}
\]
give $n+j+\lfloor\log_2b\rfloor<m+1$. Since these exponents are integers,
$n\le m-v$. Consequently, the number of groups in the numerator
is at most
\[
 \sum_{n=0}^{m-v}g_{\Gamma}(n),
\]
where $g_{\Gamma}(0)=1$, since a trivial kernel gives the group
$\Gamma$ itself.

The polynomial $2y^3/27+(2h-4)y^2/9$ has derivative
$2y(y+2h-4)/9$, so it is increasing for $y\ge4$.
By~\eqref{p4:eq:gQ} in \cref{p4:prop:fixedQ}, the error at every $n\le m-v$ is at most
$C_{\Gamma,1}m^{3/2}$ for a constant $C_{\Gamma,1}$ depending
only on $\Gamma$.
The finitely many terms with $n<4$ can be absorbed by increasing
this constant once $m$ is sufficiently large.
Bounding every remaining term at the endpoint $n=m-v$, and
summing at most $m+1$ terms, gives
\[
 \log_2\left(\sum_{n=0}^{m-v}g_{\Gamma}(n)\right)
 \le\frac{2}{27}(m-v)^3+\frac{2h-4}{9}(m-v)^2
       +O_{\Gamma}(m^{3/2}).
\]
Here the additional summation cost is at most $\log_2(m+1)$.
Expanding the two polynomial terms gives
\begin{align*}
 \frac{2}{27}(m-v)^3+\frac{2h-4}{9}(m-v)^2
 &=\frac{2}{27}m^3
   -\left(\frac49+\frac29(v-h)\right)m^2
   +O_{\Gamma}(m)\\
 &\le\frac{2}{27}m^3-\frac69m^2+O_{\Gamma}(m),
\end{align*}
because $v-h\ge1$ by~\eqref{p4:eq:h}.
On the other hand, Higman's lower bound in~\eqref{p3:eq:BLQR}
gives, for $m\ge2$,
\begin{equation}\label{p4:eq:denominator}
 F(x)\ge f(2^m)\ge2^{Q(m)}
 =2^{\frac{2}{27}m^3-\frac49m^2}.
\end{equation}
Subtracting this denominator exponent from the numerator exponent
proves~\eqref{p4:eq:margin}.
\end{proof}

\begin{corollary}[Bounded odd parts]\label{p4:cor:boundedodd}
For every fixed $b_{\max}>1$, the proportion of groups of order at most
$x$ whose odd part lies in $1<b\le b_{\max}$ tends to zero.
\end{corollary}
\begin{proof}
Let $|G|=b2^k$ with $b$ odd and $1<b\le b_{\max}$, and choose
a Sylow $2$-subgroup $S\le G$.
The action of $G$ on the $b$ cosets of $S$ has kernel
\[
 N=\bigcap_{g\in G}S^g.
\]
Thus $N\le S$ is a normal $2$-subgroup, and the faithful induced
action embeds $\Gamma=G/N$ into $\operatorname{Sym}(b)$.
Since $|N|$ is a power of $2$, the order of $\Gamma$ is $2^jb$
for some $j\ge0$; in particular its odd part is nontrivial.

Only finitely many permutation groups of these bounded degrees
occur. Hence there is a finite list of quotient isomorphism types
$\Gamma_1,\ldots,\Gamma_t$, depending only on $b_{\max}$, such that
every group under consideration belongs to at least one of the
extension families counted by~\cref{p4:prop:cutoff}.
If this list is empty, the assertion is immediate.
Otherwise, for $m=\lfloor\log_2x\rfloor$ and all sufficiently
large $m$, the proportion belonging to the $i$th family is at most
\[
 2^{-\frac29m^2+C_{\Gamma_i}m^{3/2}}
\]
with a constant $C_{\Gamma_i}\ge0$.
The constants and thresholds have a common bound because the list
is finite. Put $C_{\max}=\max_i C_{\Gamma_i}$.
For $m\ge(9C_{\max})^2$, we have
$C_{\Gamma_i}m^{3/2}\le m^2/9$ for every $i$.

For each \(i\in[t]\), let \(\mathcal E_i(x)\) be the set of
isomorphism classes of groups of order at most \(x\) admitting a
normal \(2\)-subgroup with quotient isomorphic to \(\Gamma_i\).
The preceding construction shows that every group whose order
has odd part \(b\) with \(1<b\le b_{\max}\) belongs to at least
one of these families. After imposing all the finitely many
thresholds above, the union bound therefore gives
\[
\begin{aligned}
\frac{\#\{[G]:|G|\le x,\ 1<\text{the odd part of }|G|
                         \le b_{\max}\}}{F(x)}
&\le \frac{\left|\bigcup_{i=1}^t\mathcal E_i(x)\right|}{F(x)}\\
&\le \sum_{i=1}^t\frac{|\mathcal E_i(x)|}{F(x)}
\le t\,2^{-m^2/9}\longrightarrow0. \qedhere
\end{aligned}
\]
\end{proof}

\subsection{Large odd cofactors and domination by \texorpdfstring{$2$}{2}-groups}

For large odd parts, the sharp uniform estimate from Section~\ref{p3:sec:linear}
and Pyber's prescribed-Sylow estimate suffice; no extension
multiplicity is needed.

\begin{proposition}[Large odd parts]\label{p4:prop:largeodd}
The proportion of groups of order at most $x$ whose orders have
odd part at least $2^{342}$ tends to zero.
\end{proposition}
\begin{proof}
For $n\ge2$ and each prime $p$, let $v_p(n)$ be the exponent of
$p$ in $n$, taking $v_p(n)=0$ when $p\nmid n$.
Thus $n=\prod_{p\mid n}p^{v_p(n)}$, and the parameter from
\eqref{intro:eq:group-counts} is
$\mu(n)=\max_{p\mid n}v_p(n)$.
We use Pyber's prescribed-Sylow estimate~\cite[p.~204]{Pyber},
as in the proof of~\cref{p3:cor:pyber-sharpening}:
\begin{equation}\label{p4:eq:pyber}
 f(n)\le n^{75\mu(n)+16}\prod_{p\mid n}f(p^{v_p(n)}).
\end{equation}
The uniform bound~\eqref{p3:eq:sharp} in \cref{p3:thm:linear} gives an absolute constant
$C$ such that $f(p^a)\le p^{2a^3/27+Ca}$ for every prime $p$
and $a\ge1$. Taking logarithms in~\eqref{p4:eq:pyber} consequently gives
\[
 \log_2f(n)
 \le\frac{2}{27}\sum_{p\mid n}v_p(n)^3\log_2p
      +(75\mu(n)+16+C)\log_2n.
\]

Put $\ell=\log_2x$, and write $n=2^ab\le x$, where
$a=v_2(n)\ge0$ and $b$ is odd.
Let $B=\log_2b$ and $\rho=\log_2 3>1$.
For each odd prime divisor $p$ of $n$, put
$w_p=v_p(n)\log_2p$. Then
$\sum_{p\text{ odd}}w_p=B$, and nonnegativity of the $w_p$ gives
\[
 \sum_{p\text{ odd}}v_p(n)^3\log_2p
 =\sum_{p\text{ odd}}\frac{w_p^3}{(\log_2p)^2}
 \le\frac{\sum_{p\text{ odd}}w_p^3}{\rho^2}
 \le\frac{B^3}{\rho^2}.
\]
Also $a+B=\log_2n\le\ell$, so $a^3\le(\ell-B)^3$.
Finally, $\mu(n)\le\log_2n\le\ell$ bounds the remaining term
by $75\ell^2+(16+C)\ell$. Thus
\begin{equation}\label{p4:eq:large-odd-exponent}
 \log_2 f(n)\le
 \frac{2}{27}\left((\ell-B)^3+\frac{B^3}{\rho^2}\right)
 +75\ell^2+O(\ell),
\end{equation}
with an absolute implied constant.

For the groups in the proposition, $342\le B\le\ell$.
The second derivative of $(\ell-B)^3+B^3/\rho^2$ is
$6(\ell-B)+6B/\rho^2\ge0$. Its maximum on this interval therefore
occurs at $B=342$ or $B=\ell$.
At the first endpoint,
\[
 \frac{2}{27}\left((\ell-342)^3+\frac{342^3}{\rho^2}\right)
 +75\ell^2
 =\frac{2}{27}\ell^3-\ell^2+O(\ell),
\]
since the quadratic coefficient is $75-2\cdot342/9=-1$.
At the second endpoint, the difference from
$2\ell^3/27-\ell^2$ is
\[
 -\frac{2}{27}(1-\rho^{-2})\ell^3+76\ell^2\le0
 \qquad\text{if }\ell\ge\frac{1026}{1-\rho^{-2}}.
\]
Both endpoints consequently give
$\log_2f(n)\le2\ell^3/27-\ell^2+O(\ell)$ for sufficiently large
$\ell$. There are at most $\lfloor x\rfloor\le2^\ell$ possible
integer orders, so their total contribution is at most
\[
 2^{\frac{2}{27}\ell^3-\ell^2+O(\ell)}.
\]
Here summation adds at most $\ell$ to the logarithm.

To compare this count with~\eqref{p4:eq:denominator}, put
$m=\lfloor\ell\rfloor$ and $t=\ell-m\in[0,1)$.
An expansion gives
\begin{align*}
 Q(m)=Q(\ell-t)
 &=\frac{2}{27}\ell^3-\left(\frac49+\frac{2t}{9}\right)\ell^2
   +\frac{2t^2+8t}{9}\ell-\frac{2t^3}{27}-\frac{4t^2}{9}\\
 &\ge\frac{2}{27}\ell^3-\frac23\ell^2-\frac{14}{27}.
\end{align*}
The last inequality uses $0\le t<1$ and $\ell\ge0$.
Therefore the required proportion is at most
\[
 2^{-\ell^2/3+O(\ell)}\longrightarrow0. \qedhere
\]
\end{proof}

\begin{corollary}[Domination by $2$-groups]\label{p4:cor:two-groups}
\[
 \lim_{x\to\infty}\frac{F_2(x)}{F(x)}=1.
\]
\end{corollary}
\begin{proof}
Every group counted by $F(x)-F_2(x)$ has an order with nontrivial
odd part. Those with odd part at least $2^{342}$ contribute
$o(F(x))$ by~\cref{p4:prop:largeodd}.
Those with $1<b<2^{342}$ contribute $o(F(x))$ by
\cref{p4:cor:boundedodd}, applied with $b_{\max}=2^{342}$.
Adding these two bounds gives
\[
 0\le\frac{F(x)-F_2(x)}{F(x)}\longrightarrow0,
\]
which proves the result.
\end{proof}

\subsection{Concentration in Frattini class two}\label{p4:sec:class-2}

\begin{proof}[Proof of Theorem~\ref{p4:thm:main}]
By~\cref{p4:cor:two-groups}, $F_2(x)/F(x)\to1$.
It remains to remove the $2$-groups outside $\fH_2$ and then its
abelian members.
For $k\ge0$, put
\[
 e_k=f(2^k)-f_{\fH_2}(2^k).
\]
Using $Q(t)=\frac{2}{27}t^3-\frac49t^2$,
\cref{p3:prop:frattini-concentration} gives absolute
constants $\varepsilon>0$ and $k_0$ such that
$e_k\le2^{Q(k)-\varepsilon k^2}$ for $k\ge k_0$.
The derivative of $Q(y)-\varepsilon y^2$ is
$2y(y-4-9\varepsilon)/9$, so increasing $k_0$ makes this
polynomial increasing for $y\ge k_0$.
For $m=\lfloor\log_2x\rfloor\ge k_0$, we obtain
\[
 F_2(x)-F_{\fH_2}(x)=\sum_{k=0}^{m}e_k
 \le\sum_{k=0}^{k_0-1}e_k
       +(m+1)2^{Q(m)-\varepsilon m^2}.
\]
Dividing by the denominator bound~\eqref{p4:eq:denominator} gives
\[
 0\le\frac{F_2(x)-F_{\fH_2}(x)}{F(x)}
 \le2^{-Q(m)}\sum_{k=0}^{k_0-1}e_k
       +(m+1)2^{-\varepsilon m^2}\longrightarrow0.
\]
The first sum is a fixed finite number, and $Q(m)\to\infty$.
Together with $F_2(x)/F(x)\to1$, this proves
$F_{\fH_2}(x)/F(x)\to1$.

By~\cref{p12:lem:universal-field}\textup{(iv)}, the abelian
$2$-groups of order at most $x$ number at most
\[
 \sum_{k=0}^{m}2^k=2^{m+1}-1.
\]
Their proportion is therefore at most $2^{m+1-Q(m)}$, which tends
to zero. The nonabelian members of $\fH_2$ consequently still
account for a proportion tending to one of all finite groups.
As noted after~\eqref{intro:eq:frattini-family}, these groups have both
Frattini class and nilpotency class exactly two.
In particular,
\[
 F_{\fH_2}(x)-2^{m+1}\le F_{2,2}(x)\le F(x).
\]
Dividing by $F(x)$ proves the remaining assertion.
\end{proof}

\subsection{Consequences for BNV Questions 22.16--22.18}
\label{p4:sec:bnv-2218}

\begin{corollary}[Restatement of Corollary~\ref{intro:bnv-2218}]
\label{p4:cor:bnv-2218}
As $x\to\infty$ through real values,
\[
 F(x)\sim f\bigl(2^{\lfloor\log_2x\rfloor}\bigr).
\]
Consequently, as $m\to\infty$ through integers,
\[
 F(2^m-1)\sim f(2^{m-1}),\qquad
 \frac{F(2^m-1)}{f(2^m)}
 =2^{-2m^2/9+O(m)}\longrightarrow0.
\]
Moreover, there is an absolute integer $m_0$ such that, for every
integer $m\ge m_0$,
\[
 \begin{aligned}
 f(n)&<f(2^m)&&\text{for every integer }1\le n<2^m,\\
 f(2^m p)&<f(2^{m+2})&&\text{for every prime }p.
 \end{aligned}
\]
In particular, $F(2^m-1)<f(2^m)$ for all sufficiently large $m$.
\end{corollary}
\begin{proof}
Choose an absolute constant $C\ge1$ in~\eqref{p3:eq:sharp}.
With $Q(t)=\frac{2}{27}t^3-\frac49t^2$ as
in~\eqref{p4:eq:cumulative-polynomial}, we have
$f(2^k)\le2^{Q(k)+Ck}$ for every $k\ge0$, including $k=0$
since $f(1)=1$.
The polynomial $Q(y)+Cy$ is increasing for $y\ge0$: its derivative
is $2(y-2)^2/9+C-8/9>0$.
Using also the lower bound $f(2^m)\ge2^{Q(m)}$, we obtain
\[
 \frac{\sum_{k=0}^{m-1}f(2^k)}{f(2^m)}
 \le m\,2^{Q(m-1)+C(m-1)-Q(m)}
 =2^{-2m^2/9+O(m)},
\]
where the last equality follows from
\[
 Q(m-1)-Q(m)=-\frac29m^2+\frac{10}{9}m-\frac{14}{27}.
\]
Thus, for $m=\lfloor\log_2x\rfloor$,
\[
 F_2(x)=\sum_{k=0}^{m}f(2^k)\sim f(2^m).
\]
Combining this with~\cref{p4:cor:two-groups} gives
$F(x)\sim f(2^{\lfloor\log_2x\rfloor})$.
In particular, substituting $x=2^m-1$ yields
$F(2^m-1)\sim f(2^{m-1})$.
By~\eqref{p3:eq:sharp},
\[
 \log_2\frac{f(2^{m-1})}{f(2^m)}
 =Q(m-1)-Q(m)+O(m)=-\frac29m^2+O(m).
\]
These two assertions give the claimed ratio and hence
$F(2^m-1)<f(2^m)$ for all sufficiently large $m$.
For every $1\le n<2^m$, it follows that
$f(n)\le F(2^m-1)<f(2^m)$.

For the last assertion, let $p$ be an odd prime.
Pantelidakis's estimate~\cite{Pantelidakis2003}, recorded
in~\cite[p.~265]{BNV}, gives
\[
 f(2^m p)
 \le2^{3m/2}f(2^m)
       +2^{(m^2+4m-3)/2}f(2^{m-1}).
\]
Apply the upper bound in~\eqref{p3:eq:sharp} to each numerator and
the lower bound to $f(2^{m+2})$.
The two polynomial exponent differences are
\begin{align*}
 \frac{3m}{2}+Q(m)-Q(m+2)
 &=-\frac49m^2+\frac{43}{18}m+\frac{32}{27},\\
 \frac{m^2+4m-3}{2}+Q(m-1)-Q(m+2)
 &=-\frac16m^2+4m-\frac56.
\end{align*}
Consequently,
\[
 \frac{f(2^m p)}{f(2^{m+2})}
 \le2^{-4m^2/9+O(m)}+2^{-m^2/6+O(m)}\longrightarrow0,
\]
with absolute implied constants independent of $p$.
For $p=2$, the same uniform bound gives
\[
 \log_2\frac{f(2^{m+1})}{f(2^{m+2})}
 =Q(m+1)-Q(m+2)+O(m)=-\frac29m^2+O(m),
\]
which also tends to $-\infty$.
The thresholds in these estimates are independent of the prime,
so a single absolute $m_0$ proves all the strict inequalities.
\end{proof}

\nocite{*}
\bibliographystyle{alpha}
\bibliography{reference}

\end{document}